\documentclass[11pt]{article}

\usepackage{silence}
\usepackage[left=1in,top=1in,right=1in,bottom=1in,head=.1in]{geometry}

\usepackage{mystyle}

\defcitealias{das2023understanding}{DS23}
\defcitealias{pareek2024understanding}{PDO24}
\defcitealias{ildiz2025high}{IGT+25}
\defcitealias{moniri2025on}{MH25}
\defcitealias{garg2025preventing}{GBS25}

\renewcommand{\bm}[1]{#1}

\renewcommand{\mathbf}[1]{#1}

\newcommand{\fresh}{{\mathrm{fr}}}
\newcommand{\freshmix}{{\mathrm{frmix}}}
\newcommand{\freshaffine}{{\mathrm{fraff}}}
\newcommand{\same}{{\mathrm{same}}}
\newcommand{\barR}{{\bar{R}}}
\newcommand{\barC}{{\bar{C}}}

\newcommand{\titletext}{Optimal Unconstrained Self-Distillation in Ridge Regression:\\Strict Improvements, Precise Asymptotics, and One-Shot Tuning}

\title{\titletext}

\author{
    Hien Dang\footremember{utsds}{Department of Statistics and Data Sciences, University of Texas, Austin, TX 78712, USA.} \\ {\footnotesize \texttt{\url{hiendang@utexas.edu}}}
    \and
    Pratik Patil\footrecall{utsds} \\ {\footnotesize \texttt{\url{pratikpatil@utexas.edu}}} 
    \and
    Alessandro Rinaldo\footrecall{utsds} \\ {\footnotesize \texttt{\url{alessandrorinaldo@utexas.edu}}} 
}

\date{\vspace{-15pt}}

\begin{document}

\maketitle

\begin{abstract}
Self-distillation (SD) is the process of retraining a student on a mixture of ground-truth labels and the teacher’s own predictions using the same architecture and training data.  Although SD has been empirically shown to often improve generalization, its formal guarantees remain limited. We study SD for ridge regression in unconstrained setting in which the mixing weight $\xi$ may be outside the unit interval. Conditioned on the training data and without any distributional assumptions, we prove that for any squared prediction risk (including out-of-distribution), the optimally mixed student strictly improves upon the ridge teacher for every regularization level $\lambda > 0$ at which the teacher ridge risk $R(\lambda)$ is nonstationary (i.e., $R'(\lambda) \neq 0$). We obtain a closed-form expression for the optimal mixing weight $\xi^\star(\lambda)$ for any value of $\lambda$ and show that it obeys the sign rule: $\operatorname{sign}(\xi^\star(\lambda))=-\operatorname{sign}(R'(\lambda))$. In particular, $\xi^\star(\lambda)$ can be negative, which is the case in over-regularized regimes. To quantify the risk improvement due to SD, we derive exact deterministic equivalents for the optimal SD risk in the proportional asymptotics regime (where the sample and feature sizes $n$ and  $p$ both diverge but their the aspect ratio $p/n$ converges) under general anisotropic covariance and deterministic signals. Our asymptotic analysis extends standard second-order ridge deterministic equivalents to their fourth-order analogs using block linearization, which may be of independent interest. From a practical standpoint, we propose a consistent one-shot tuning method to estimate $\xi^\star$ without grid search, sample splitting, or refitting. Experiments on real-world datasets and pretrained neural network features support our theory and the one-shot tuning method.
\end{abstract}

\section{Introduction}
\label{sec:introduction}

Knowledge distillation (KD), introduced by \citet{bucilua2006model, ba2014deep, hinton2015distilling}, is conventionally used for model compression, transferring knowledge from a large teacher to a smaller student.
Recently, this paradigm has been adapted to the setting where teacher and student share the same architecture and training data, a process known as \emph{self-distillation} (SD) \citep{furlanello2018born, zhang2021self}.
While it may seem counterintuitive that a model would improve by learning from its own predictions, extensive empirical evidence shows that SD can in fact boost generalization \citep{chen2017learning, chen2022knowledge, li2017learning, ahn2019variational, li2021align, gou21knowledge}.
Despite these successes, it remains unclear whether and when such improvements can be guaranteed.

Formally, let $f_{\te}$ be a teacher trained on $\{(x_i,y_i)\}_{i=1}^n$ using a loss function $\ell$.
Self-distillation trains a student $f_{\sd}$ on the \emph{same} data by minimizing a mixed objective that is an affine interpolation of the losses incurred with respect to the ground-truth labels $y_i$'s and the teacher’s predictions $f_{\te}(x_i)$'s. In detail, the SD procedure seeks to find the student model  $f_{\sd}$ minimizing
\begin{equation}
\label{eq:sd_mixed_loss}
\frac{1}{n}\sum_{i=1}^{n}\Big[(1-\xi) \cdot \ell\big(y_i,f_{\sd}(x_i)\big)
+\xi \cdot \ell\big(f_{\te}(x_i),f_{\sd}(x_i)\big)\Big],
\end{equation}
where $\xi$ is the mixing parameter \citep{lopez2015unifying}; see \Cref{fig:illustration}.
When $\xi=1$, the student learns solely from the teacher’s predictions; we call this \emph{pure-distillation} (PD) and denote the resulting predictor by $f_{\pd}$.

\begin{figure}[!t]
    \centering
    \includegraphics[width=0.8\textwidth]{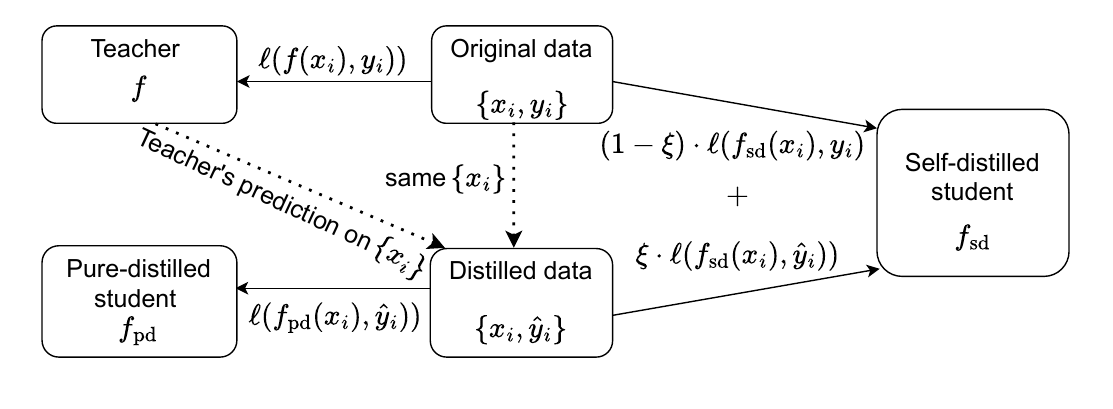}
    \caption{Visual illustration of the self-distillation process.}
    \label{fig:illustration}
\end{figure}

The mixing parameter $\xi$ balances the influence of ground-truth labels against teacher predictions.
Standard distillation methods restrict $\xi$ to lie in $[0,1]$, interpreting the target loss as a convex combination.
Recent work by \citet{das2023understanding} shows that this constraint can be suboptimal under high label noise, where the optimal mixing weight $\xi^\star$ may in fact be found to be greater than $1$.
Motivated by this, we adopt a fully unconstrained perspective and allow $\xi\in\mathbb{R}$, including negative values.
Note that setting $\xi=0$ recovers the teacher predictor, hence optimizing over $\xi$ cannot perform worse than the teacher. With this in mind, we pose the key questions  about SD:

\begin{itemize}[nosep,itemsep=2pt]
    \item[(\Qsnoparen{1})]
    When does the optimally mixed student $f_{\sd}$ trained  using an optimal $\xi^\star \in \RR$ strictly outperform the teacher $f$, and how large can the gain be?
    \item[(\Qsnoparen{2})]
    Can optimal SD from a suboptimal teacher achieve performance comparable to an optimally tuned teacher?
    \item[(\Qsnoparen{3})]
    How can we efficiently tune the optimal $\xi^\star\in\mathbb{R}$ without computationally expensive grid search?
\end{itemize}

We provide complete answers to all these questions for ridge regression, a model in which SD admits an explicit affine path (in the response) and the risk of both the teacher and the students can be characterized sharply, capturing the interplay between regularization and distillation.

\begin{figure*}[!t]
  \centering
    \begin{subfigure}[t]{0.32\textwidth}
    \centering
    \includegraphics[width=\textwidth]{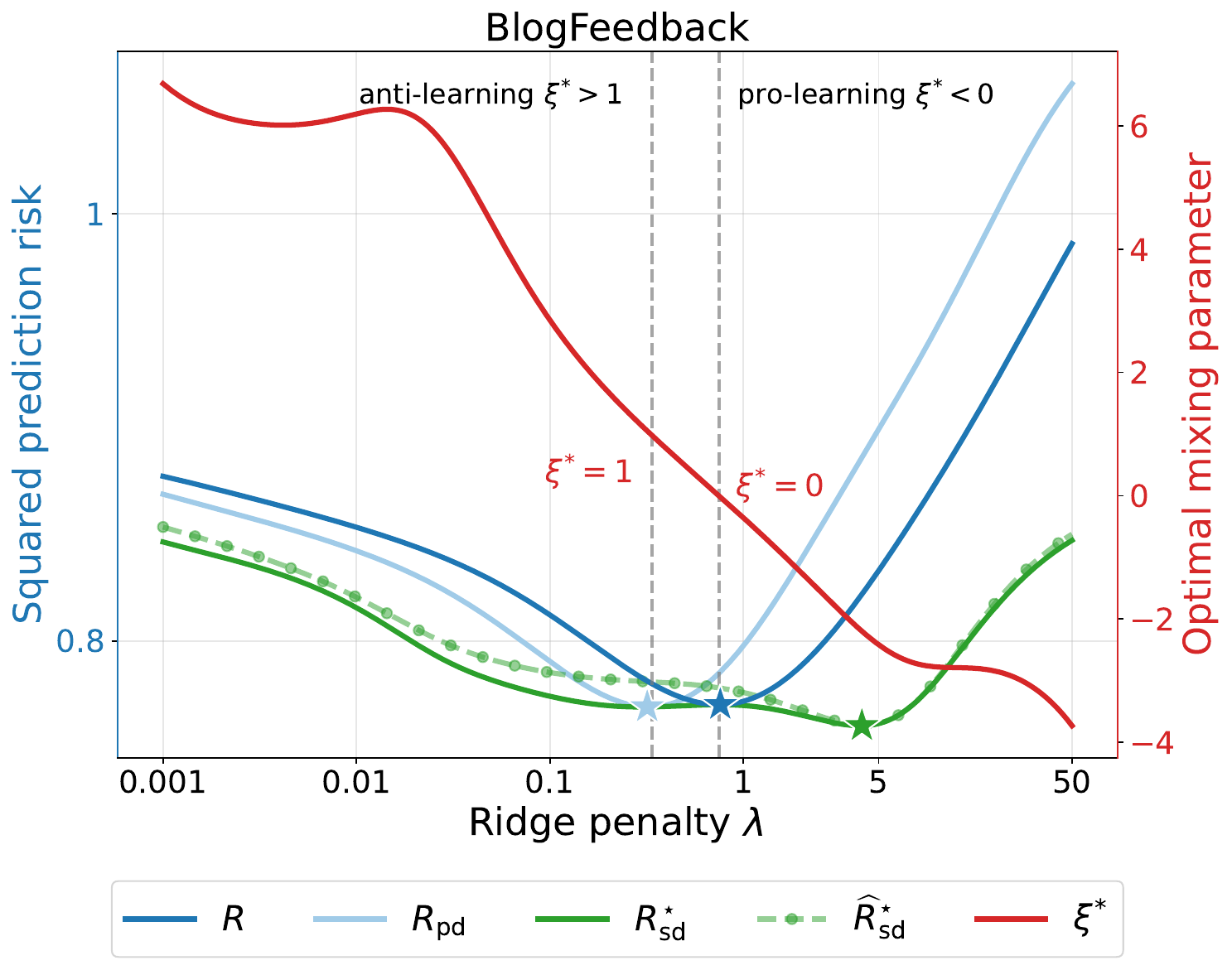}
    \caption{$n = 2619, p = 280, \gamma \approx 0.1$}
    \label{fig:blog_main}
  \end{subfigure}
  \hfill
    \begin{subfigure}[t]{0.32\textwidth}
    \centering
    \includegraphics[width=\textwidth]{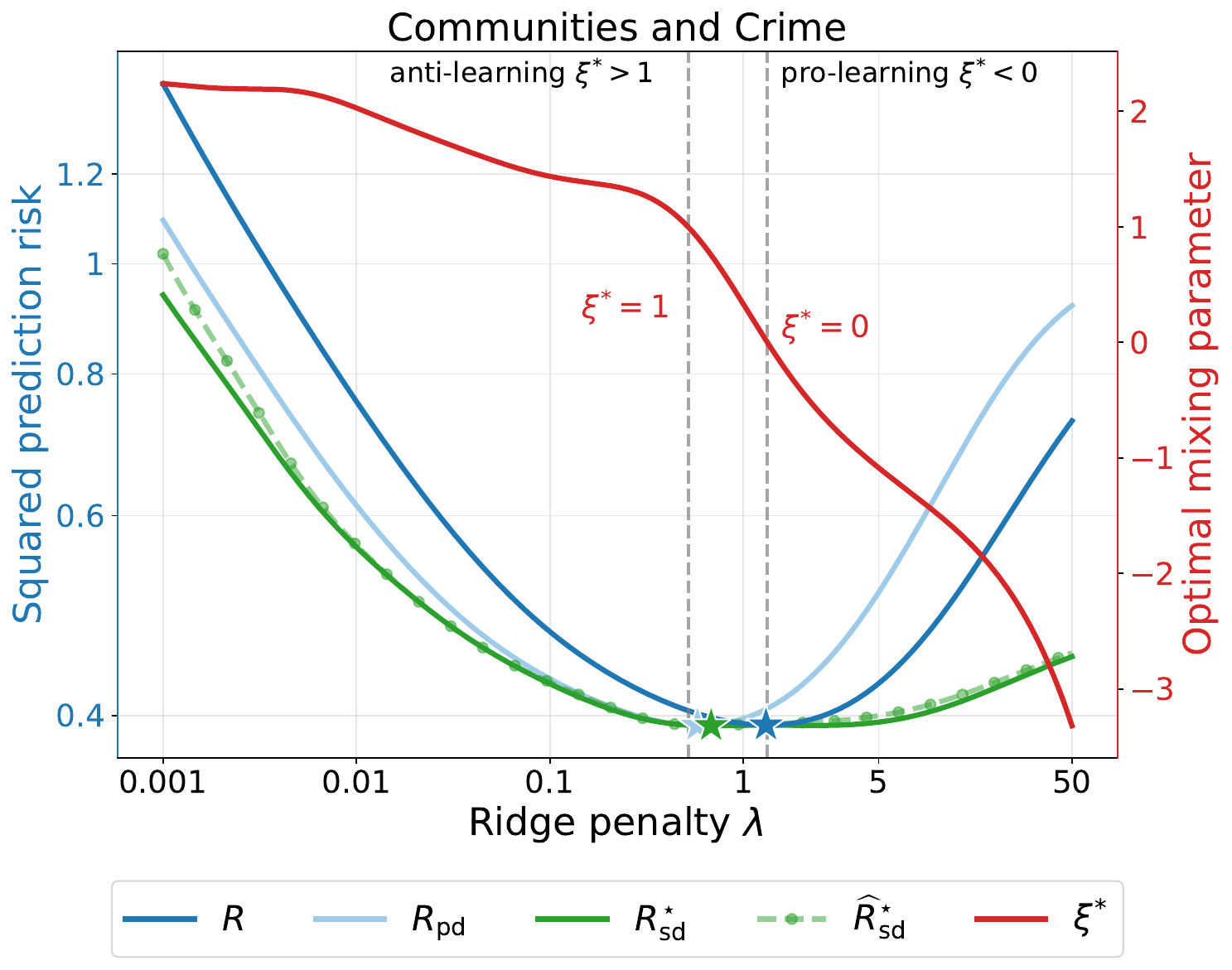}

    \caption{$n = 99, p = 99, \gamma = 1$}
    \label{fig:CC_main}
  \end{subfigure}
  \hfill
  \begin{subfigure}[t]{0.32\textwidth}
    \centering
    \includegraphics[width=\textwidth]{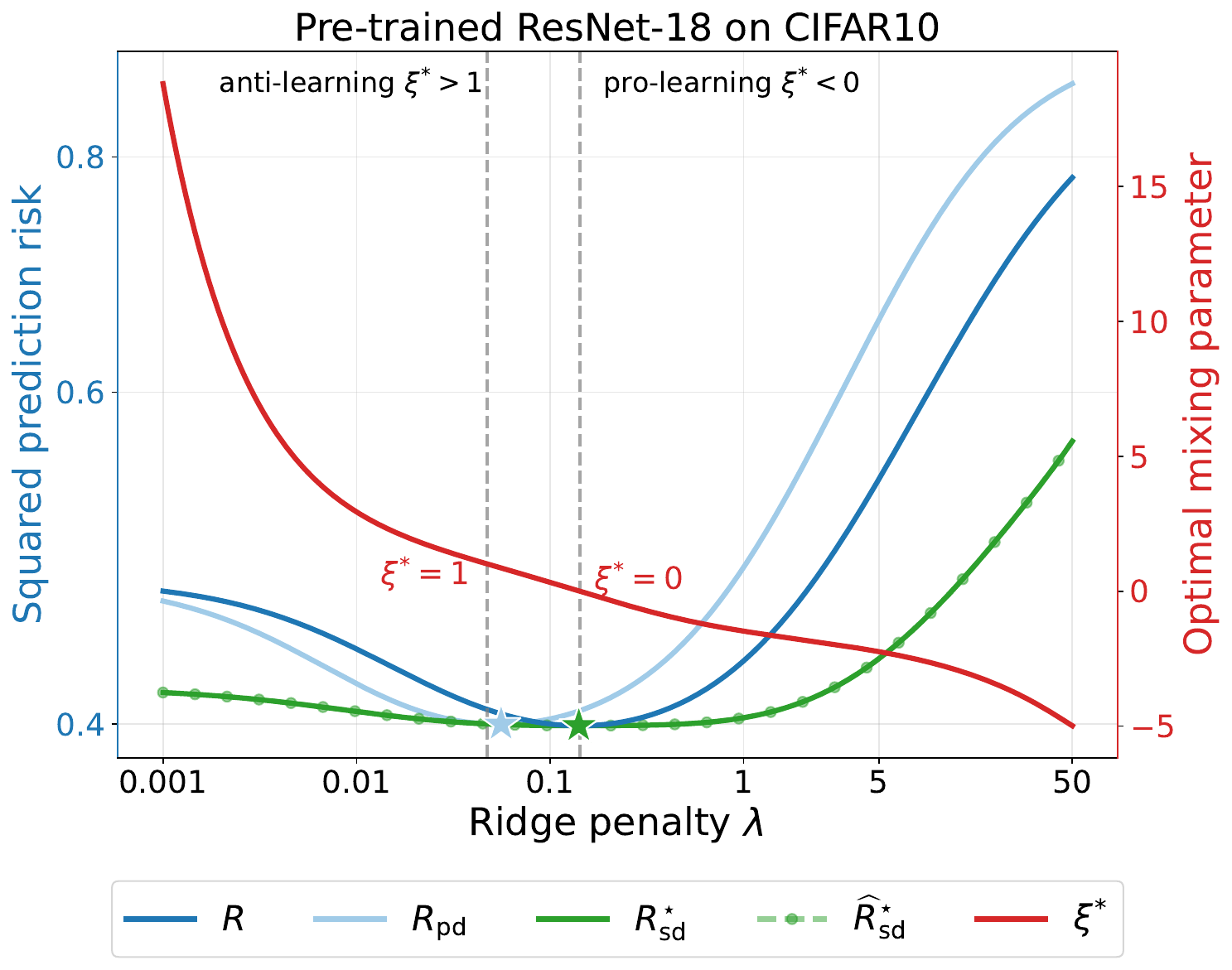}

    \caption{$n = 2000, p = 5120, \gamma \approx 2.6$}
    \label{fig:cifar_main}
  \end{subfigure}
  \caption{
  \textbf{Strict improvement of SD risk with unconstrained mixing.} 
  Test squared prediction risk of ridge regression ($R$, in {\color{pyblue} blue}), pure-distilled ridge ($R_{\pd}$, in {\color{pylightblue} light blue}) and optimal self-distilled ridge ($R^{\star}_{\sd}$, in {\color{pygreen} green}) as functions of the ridge penalty $\lambda$. 
  Results are shown on raw features from real-world datasets: \dataset{BlogFeedback} and \dataset{Communities and Crime} datasets, and on pretrained ResNet-18 features. 
  The optimal mixing parameter $\xi^{\star}(\lambda)$ is in {\color{pyred} red} and the one-shot risk estimate $\widehat R_{\sd}^\star(\lambda)$ computed from the training data is shown in {\color{pygreen} green dashed} line. 
  Note that $\xi^\star(\lambda)$ lies in $[0,1]$ only for a narrow range of $\lambda$ and can be strongly negative for large $\lambda$.
  We also observe that: (i) $R_{\sd}^{\star}(\lambda)$ is strictly smaller than $R(\lambda)$ at every $\lambda$ that is not the stationary point of $R(\lambda)$, (ii) the sign of $\xi^{\star}(\lambda)$ is opposite to the sign of $R^{\prime}(\lambda)$, and (iii) the sign change of $\xi^{\star}$ happens at the stationary point of $R(\lambda)$.
  (Experiments with $\xi$ restricted to $[0,1]$ appear in \Cref{fig:exp_real_world_restricted}.)
  } 
  \label{fig:exp_real_main}
\end{figure*}

\subsection{Summary of Paper Contributions and Outline}

Below we describe in detail the main contributions of the paper; see \Cref{fig:exp_real_main} for a visual summary.

\textbf{Structural nonasymptotic guarantees (\Cref{sec:structural}).}
Addressing \Qs{1}, we derive deterministic identities for self-distilled ridge that hold conditionally on the observed training data, without any distributional assumptions, and for \emph{any} squared prediction risk (including the out-of-distribution risk).
In particular, we show that for every $\lambda>0$ such that the teacher ridge-path risk $\lambda \mapsto R(\lambda)$ is nonstationary (i.e., $R'(\lambda) \neq 0$), optimal mixing yields a \emph{strict} improvement over the teacher (\Cref{thm:tangent_sign}).
Addressing \Qs{2}, we provide a curvature-based sufficient condition under which the global minimum over $\lambda$ of the SD risk $R_{\sd}^\star(\lambda)$, obtained using the optimal mixing $\xi^\star(\lambda)$,  is strictly smaller than the smallest ridge risk of the teacher (\Cref{prop:curvature_test_main}).

\textbf{Precise proportional asymptotics (\Cref{sec:asymptotics}).}
Returning to \Qs{1} in the proportional regime -- in which the sample and feature sizes $n,p\to\infty$ but their aspect ratio $p/n\to\gamma \in (0, \infty)$ --  we derive exact deterministic equivalents for the optimal SD risk and mixing weight under general anisotropic covariance and deterministic signals (\Cref{thm:risk-asymptotics}).
These formulas quantify the SD gains in terms of $\gamma$, the signal-to-noise ratio, and the signal--covariance alignment. In particular, they characterize precisely when the optimal mixing weight $\xi^\star$ becomes \emph{negative} (\Cref{cor:xi_asymptotic}).
Regarding \Qs{2}, under an isotropic random signal we show that even for extremely under- or over-regularized teachers, optimal SD can closely approach the risk of the optimally tuned ridge predictor (\Cref{prop:compare_extreme_lambda}).

\textbf{One-shot tuning (\Cref{sec:tuning}).}
Addressing \Qs{3}, we propose a consistent one-shot estimator of $\xi^\star$ based on generalized cross-validation (\Cref{thm:oneshot_consistency}).
The method avoids sample splitting and grid search over $\xi$ (and hence does not require any refitting across candidate mixing weights), while remaining consistent in proportional asymptotics.
We validate our theory and tuning methodology through extensive experiments on synthetic data, UCI regression benchmarks, and pretrained ResNet feature representations.

\textbf{Extensions and variants (\Cref{sec:extensions}).}
We analyze several natural variants of the SD ridge framework, including multi-round distillation with a risk monotonicity property (\Cref{app:multiround_proofs}), self-distillation using teacher predictions on fresh features, which we show to be dominated by the same-$X$ setup in an isotropic setting (\Cref{sec:supp_freshX_isotropic}), and extensions to generalized ridge and kernel ridge regression (\Cref{sec:supp_extensions_ridge_variants}).

\subsection{Related Works and Comparisons}
\label{sec:related_work}

\textbf{Fixed-design analyses.}
Early theoretical work on distillation focused on when and how well a student can mimic a teacher \citep{phuong2019towards, ji2020knowledge}; for regression problems, this has typically been addressed using fixed-design.
In kernel ridge regression, \citet{mobahi2020self} show that repeated SD shrinks the effective function class, resembling an increase in regularization; their emphasis is on the implicit regularization induced by SD rather than on guaranteeing performance improvements.
In ridge regression, \citet{das2023understanding} analyze the bias--variance trade-off induced by the mixing weight $\xi$ and give conditions under which the \emph{globally} optimal SD risk (obtained by optimizing over $\lambda$) can go below the optimally tuned teacher ridge risk.
\citet{pareek2024understanding} extend this line of work to multiple SD rounds, quantifying gains under an alignment condition when $\lambda$ and $\xi$ are tuned in each round.
A key gap in these analyses is the question of pointwise improvement: whether SD can \emph{strictly} improve a mis-regularized teacher at a fixed, suboptimal $\lambda$, which we show in \Cref{sec:structural}.
Such pointwise gains are important because finding global optimal $\lambda$ is often challenging; for example, \citet{stephenson2021can} show that the leave-one-out cross-validation loss is generally neither convex nor even quasi-convex in $\lambda$.

A further distinction is the data model.
Both the linear regression analysis in \citet{das2023understanding, pareek2024understanding} operate under fixed design and a well-specified linear response model.
While analytically convenient, fixed-design analyses do not capture the contribution of test-feature randomness to prediction error, which is central in high-dimensional generalization \citep{hastie2022surprises}.
In contrast, our structural results hold conditionally on the observed training data under random design and for \emph{any} squared prediction risk, including out-of-distribution risk.
We also work in proportional random-design asymptotics ($n,p\to\infty$ with $p/n\to\gamma$) in \Cref{sec:asymptotics} and provide exact characterizations of generalization effects driven by the feature and signal structures.

\textbf{Random-design analyses.}
Closer to our setting, \citet{ildiz2025high} analyze KD under random design with anisotropic covariance, focusing on scaling laws and ``weak-to-strong'' generalization, a phenomenon where a strong student model is supervised by a weaker teacher \citep{burns2023weak}. 
They provide nonasymptotic characterizations of pure-distilled ridge ($\xi = 1$ in our notation) and its scaling behavior. 
They primarily focus on the \textit{ridgeless} case (minimum $\ell_2$-norm interpolator where $\lambda \to 0^+$). 
While interesting, such analysis obscures the interaction between explicit ridge penalty ($\lambda > 0$) and distillation weight $\xi$, which is central to our results.
In particular, we identify over-regularized regimes where optimal SD requires a {\it negative} optimal mixing $\xi^\star$ to account for excessive shrinkage and thus strictly improve the teacher.
In contrast, in pure-distillation ridge with $\lambda > 0$, the student can only improve upon the teacher for a range of small $\lambda$, and can fail to improve under over-regularization (or even being worse than the teacher)  \citep{moniri2025on}; see \Cref{fig:exp_real_main}.

\begin{table*}[!t]
\centering
\tiny
\caption{\textbf{Settings overview in the self-distillation ridge/ridgeless literature}. Unless stated otherwise, risks in the proportional asymptotics setting are in-distribution squared prediction risks.}
\label{tab:comparison}

\begin{tabularx}{\textwidth}{
  C{0.6cm}    %
  C{0.5cm}    %
  C{1.5cm}    %
  C{1.0cm}    %
  C{0.5cm}    %
  C{1.7cm}    %
  C{1.1cm}    %
  C{1.3cm}    %
  C{0.5cm}    %
  C{0.3cm}    %
  C{1.5cm}    %
  C{1.0cm}    %
}
\toprule

\textbf{Paper} & \multicolumn{2}{c}{\textbf{Focus}} & \multicolumn{3}{c}{\textbf{Nonasymptotic setting}} & \multicolumn{3}{c}{\textbf{Proportional asymptotics}} & \multicolumn{2}{c}{\textbf{Weights}} & \textbf{Multi-rounds} \\

\cmidrule(lr){2-3} \cmidrule(lr){4-6} \cmidrule(lr){7-9} \cmidrule(lr){10-11}
& \textbf{Type} & \textbf{Family} & \textbf{Design} & \textbf{Response} & \textbf{Risk} & \textbf{Features} & \textbf{Signal} & \textbf{Model} & \textbf{Range} & \textbf{Tuning} & \\

\midrule \arrayrulecolor{black!10}
   
\citetalias{das2023understanding}
& SD & Ridge 
  & Fixed 
  & Linear 
  & Estimation
  & \colorgray 
  & \colorgray 
  & \colorgray 
  & $\mathbb{R}$ 
  & \colorgray 
  & \colorgray \\

\midrule

\citetalias{pareek2024understanding}  & SD & Ridge 
  & Fixed 
  & Linear 
  & Fixed-X prediction
  & \colorgray 
  & \colorgray 
  & \colorgray 
  & $\mathbb{R}$ 
  & Split CV (three refits)
  & Yes \\

\midrule

\citetalias{ildiz2025high}   & PD & Ridgeless 
  & Random 
  & Linear 
  & In-distribution prediction 
  & \colorgray 
  & \colorgray 
  & \colorgray 
  & $\{ 1 \}$ 
  & \colorgray
  & \colorgray \\

\midrule

\citetalias{moniri2025on}  & PD & Ridge (generalized) 
  & \colorgray 
  & \colorgray 
  & \colorgray 
  & Isotropic
  & Isotropic
  & Linear 
  & $\{ 1 \}$ 
  & \colorgray 
  & \colorgray \\

\midrule

\citetalias{garg2025preventing}  & SD & Ridge (infinite rounds)
  & \colorgray 
  & \colorgray 
  & \colorgray 
  & Anisotropic
  & Deterministic
  & Linear
  & $[0,1]$ 
  & \colorgray 
  & Infinite rounds only \\

\midrule

\colorhighlight Ours    & \colorhighlight SD & \colorhighlight Ridge 
  & \colorhighlight Random 
  & \colorhighlight Any 
  & \colorhighlight Any squared prediction
  & \colorhighlight Anisotropic
  & \colorhighlight Deterministic 
  & \colorhighlight Any 
  & \colorhighlight $\mathbb{R}$ 
  & \colorhighlight GCV (one-shot) 
  & \colorhighlight Yes \\

\arrayrulecolor{black} 
\bottomrule
\end{tabularx}
\end{table*}

A large body of recent work has focussed ridge regression risk and its variants under proportional asymptotics using tools from random matrix theory and statistical physics (e.g., \citealp{dobriban2018high, hastie2022surprises, patil2024optimal} and references therein), shedding light on phenomena such as benign overfitting \citep{bartlett2020benign} and double descent \citep{belkin2019reconcling}.
We extend this literature by deriving exact deterministic equivalents for optimal SD risk under anisotropic feature covariance and deterministic signals.
Because the teacher and student are trained on the \emph{same} data, SD risks involve dependent resolvent-type quantities; unlike settings with independent refits where one can directly combine known deterministic equivalents, our analysis requires higher-order deterministic equivalents obtained via block linearization (see \Cref{app:asymptotics_proofs-equivalents-outline,app:DE_Q_U}).

\textbf{Unconstrained mixing and ``anti-learning''.}
Standard KD typically restricts $\xi\in[0,1]$.
In the presence of label noise, \citet{das2023understanding} show that optimal mixing can satisfy $\xi^\star>1$, corresponding to a negative weight on the noisy ground truth (``anti-learning'').
In binary classification, \citet{javanmard2025self} derive a nonlinear Bayes-optimal aggregation effect (using approximate message passing tools) that also effectively subtracts noisy labels during retraining.
In our ridge setting, closed-form risks allow a clean three-regime picture: interpolation ($0\le\xi\le 1$), extrapolation/anti-learning ($\xi>1$), and over-regularization correction ($\xi<0$), which we term ``pro-learning''.

\textbf{Synthetic-label retraining and model collapse.}
Beyond one-round SD, recent work has raised concerns about recursively training on model-generated labels leading to performance degradation, referred to as \emph{model collapse} \citep{shumailov2024ai, alemohammad2023self, dohmatob2024model, gerstgrasser2024model}.
In contrast to the notion of ``strong model collapse'' in \citet{dohmatob2024strong} (in which a model trained in the presence of synthetic data has worse asymptotic risk than a model trained solely on ground-truth data), optimal SD in our ridge setting exhibits no such degradation. Provided the nondegeneracy condition in \Cref{thm:tangent_sign} holds, the optimally mixed student strictly improves upon the teacher whenever $R'(\lambda) \neq 0$, with the two risks coinciding only at stationary points.

Closer to our setting, \cite{he2025golden, garg2025preventing} analyze infinite-round schemes that mix ground-truth and synthetic labels with a fixed weight $w$ (analogous to our $\xi$) to prevent degradation under repeated synthetic training; the optimal mixing weight $w^\star$ in their settings lies in $[0, 1]$.
In contrast, we characterize the risk-minimizing mixing for one-round SD and show that $\xi^\star$ can lie outside $[0,1]$, including $\xi^\star<0$ in over-regularized regimes.
To highlight the distinction, we run a synthetic experiment (\Cref{app:vs_multi_round}) comparing optimal one-round SD risk to $20$-round SD with optimal fixed constrained weight in $[0,1]$ for every round.
As shown in \Cref{fig:vs_multi_round}, unconstrained mixing can be crucial for improving the teacher's performance, particularly in over-regularized regimes where negative weights becomes necessary.
Finally, we also study a recursive multi-round variant with per-round unconstrained optimal mixing that yields a monotone (weakly) decreasing risk sequence (\Cref{sec:extensions}).

\textbf{Tuning and risk estimation.}
A notable gap in existing theoretical results about SD is the choice of the optimal mixing hyperparameter $\xi$ in practice.
Most analyses assume oracle access to population risks or detailed spectral knowledge.
In practice, cross-validation over $\xi$ is computationally costly due to grid search and repeated refitting, and it can be statistically inefficient in high dimensions due to sample splitting (see, e.g., \citet{rad2020scalable}).
Building on consistent risk estimation for high-dimensional ridge/ridgeless regression (see, e.g., \citealp{patil2021uniform, patil2022estimating, wei2022more, han2023distribution, bellec2025corrected, koriyama2024precise} and references therein), we propose a one-shot generalized cross-validation estimator for SD ridge and prove its consistency in proportional asymptotics (\Cref{thm:oneshot_consistency}).
This enables data-efficient selection of unconstrained $\xi$ without grid search over candidate mixing weights or hold-out sets.

Overall, relative to prior SD ridge analyses, we  provide the following novel results: (i) pointwise strict improvement guarantees at any nonstationary $\lambda>0$ together with a sign characterization of $\xi^\star$ (\Cref{thm:tangent_sign}); (ii) exact proportional-asymptotic risk characterizations under anisotropic feature covariance and deterministic signals (\Cref{thm:risk-asymptotics}); and (iii) a consistent one-shot GCV-based tuning method (\Cref{thm:oneshot_consistency}).
At the same time, our settings are more general than those considered in the literature so far: (a) a nonasymptotic setting with no distributional assumptions and any squared prediction risk; and (b) proportional asymptotics with general feature covariance and deterministic signals, without imposing a well-specified response model.
A compact comparison of settings and assumptions across several related works is provided in \Cref{tab:comparison}.

\section{Structural Nonasymptotic Results}
\label{sec:structural}

This section described deterministic ``structural'' identities for self-distillation in ridge regression under any squared prediction risk (including the out-of-distribution risk).  
All the statements hold with virtually no distributional assumptions (other than the existence of second moments) and conditionally on the  training data  $\mathcal{D} = (X, y)$, where $X \in \mathbb{R}^{n \times p}$ is the design matrix containing $p$-dimensional observations $x_i \in \RR^p$ as rows and $y \in \mathbb{R}^{n}$ is the response vector containing scalar ground-truth labels $y_i$ for $i = 1, \dots, n$.

\subsection{Self-Distillation with Ridge Regression}

Given a regularization parameter $\lambda > 0$, define the \emph{teacher} ridge predictor $x \in \mathbb{R}^p \mapsto f_{\lambda}(x) \in \RR$ trained on $\mathcal{D}$ as
\begin{equation}
    \label{eq:ridge_predictor}
x^\top \argmin_{\beta \in \RR^p}\big\{ \| y - X \beta \|_2^2 / n + \lambda \| \beta \|_2^2\big\}.
\end{equation}
Let \( \hat{ y}_\lambda = f_{\lambda}(X)  = (f_\lambda(x_1) ,\ldots, f_\lambda(x_n) )^\top\in\RR^n\) denote the teacher's training predictions. %
The \emph{pure-distilled} predictor \( f_{\pd, \lambda} \) is ridge regression trained on the pseudo-labeled or ``distilled'' dataset $(X, \hat{y}_\lambda)$ with the same penalty $\lambda$.
For a mixing parameter \(\xi\in\RR\), the  \emph{self-distilled} predictor $x \in \R^{p} \mapsto f_{\sd, \lambda}(x) \in \RR$ is the ridge fit obtained using the mixed-label loss:
\begin{equation}
    \label{eq:sd_predictor}
    x^\top \argmin_{\beta \in \RR^p}\big\{ (1 - \xi) \, \| y - X \beta \|_2^2 / n + \xi \, \| \hat{y}_{\lambda} - X \beta \|_2^2 / n + \lambda \| \beta \|_2^2\big\}.
\end{equation}
A key simplification in ridge regression is that the ridge map is linear in the response vector. 
As a result, the SD predictor lies on a linear path between the teacher and the PD fit (see \Cref{app:structural_proofs}):
\begin{equation}
\label{eq:affine_path_pred}
f_{\sd,\lambda,\xi}(x)=(1-\xi)\,f_\lambda(x)+\xi\,f_{\pd,\lambda}(x).
\end{equation}

For a test point $(x_0, y_0)$ (possibly out-of-distribution) of finite (data) conditional $L_2$ norm\footnote{We do not need $(x_0,y_0)$ to be independent of $\cD$ for results in this section. Thus they also apply when $(x_0,y_0)$ is drawn from the empirical measure on $\cD$, in which case $R(f)$ is simply the training error of $f$.}, we measure the performance of any (possibly data-dependent) predictor $f$ by the conditional squared prediction risk
\begin{equation}
  \label{eq:pred-risk}
  R(f):=\EE\!\big[(y_0-f(x_0))^2\mid \cD\big].
\end{equation}
For convenience, write $R(\lambda):=R(f_\lambda)$, $R_{\pd}(\lambda):=R(f_{\pd,\lambda})$, and $R_{\sd}(\lambda,\xi):=R(f_{\sd,\lambda,\xi})$. We note that these different types of risk are random, as they depend on training data.

\subsection{Optimal SD Risk Decomposition}
\label{sec:structural_closed_form}

We first provide a useful decomposition of the optimal SD risk in terms of three key quantities, defined conditionally on $\mathcal{D}$\footnote{Throughout this section, any identity, inequality, or stated property involving such quantities is implicitly assumed to hold with probability one with respect to the distribution of the training data.}.
First, we let 
\begin{equation}
    \label{eq:Clambda}
  C(\lambda) :=\EE\big[(y_0 - f_{\lambda}(x_0)) (y_0-f_{\pd,\lambda}(x_0))\mid \cD\big],
\end{equation}
be the conditional correlation between the residual errors of the teacher and PD predictors and 
\begin{equation}
    \label{eq:Dlambda}
    D(\lambda) = \EE\big[(f_\lambda(x_0)-f_{\pd,\lambda}(x_0))^2\mid \cD\big] \geq 0.
\end{equation}
be the conditional expected squared difference between the predictions of the teacher and PD student predictors (which is also the expected squared difference between their residuals).

The difference $D(\lambda)$ also admits an equivalent form $D(\lambda) = R(\lambda) + R_{\pd}(\lambda) - 2C(\lambda)$ (see \Cref{app:structural_proofs}).
Let $\xi^\star(\lambda) \in \argmin_{\xi \in \RR} R_{\sd}(\lambda, \xi)$, and define $R_{\sd}^\star(\lambda) := R_{\sd}(\lambda, \xi^\star)$, referred to as the optimal SD risk.  

\begin{proposition}[Optimal SD risk decomposition]
\label{prop:closed_form_xi_R1}
Fix $\lambda > 0$ and  assume $D(\lambda) > 0$. 
Then,
\begin{align}
  \xi^\star(\lambda)
  =\frac{R_{\te}(\lambda)- C(\lambda)}{D(\lambda)},
  \quad
  R_{\sd}^\star(\lambda)
  =
  R_{\te}(\lambda) -
  \frac{(R_{\te}(\lambda)-C(\lambda))^2}{D(\lambda)}.
  \label{eq:R1_star_short}
\end{align}
In particular, \(R_{\sd}^\star(\lambda) \le R_{\te}(\lambda)\) for all \(\lambda\), and \(\xi^\star(\lambda)\) may be negative.
\end{proposition}

The magnitude of the risk improvement $R(\lambda) - R_{\sd}^\star(\lambda)$ depends on 
(i) the numerator $R(\lambda)-C(\lambda)$, which also determines the sign of $\xi^\star(\lambda)$, and (ii) the denominator $D(\lambda)$.
From \eqref{eq:Dlambda}, note that $D(\lambda)=0$ if and only if $f_\lambda$ and $f_{\pd,\lambda}$ coincide almost surely on the test distribution, in which case $R_{\sd}(\lambda,\xi)\equiv R(\lambda)$ for all $\xi$, so that each $\xi$ is trivially optimal.
Thus the nondegeneracy assumption $D(\lambda)>0$ in \Cref{prop:closed_form_xi_R1} is innocuous.

\subsection{Strict Pointwise Improvement and Sign of Optimal Mixing Weight}
\label{sec:structural_tangent}

Building on \Cref{prop:closed_form_xi_R1}, we next characterize when optimal SD is \emph{strictly} better than the teacher, and determine the sign of the optimal mixing weight in terms of the derivative of the teacher risk.
(It is worth mentioning that under finite conditional second moment conditions, the teacher risk $R(\lambda)$ is a smooth function of $\lambda$ and hence the derivative is well-defined; see \Cref{lem:app_risk_smoothness}.)

\begin{theorem}[Strict improvement and sign rule]
\label{thm:tangent_sign}
Fix $\lambda>0$ and assume $D(\lambda)>0$.
Then
\begin{equation}
\label{eq:xi_star_slope_main}
\xi^\star(\lambda)
=
-\frac{\lambda}{2}\,\frac{R'(\lambda)}{D(\lambda)},
\quad
R_{\sd}^\star(\lambda)
=
R(\lambda)-\frac{\lambda^2}{4}\,\frac{(R'(\lambda))^2}{D(\lambda)}.
\end{equation}
In particular, if $R'(\lambda)\neq 0$, then $R_{\sd}^\star(\lambda)<R(\lambda)$ and
$\operatorname{sign}(\xi^\star(\lambda))=-\operatorname{sign}(R'(\lambda))$.
\end{theorem}

\begin{figure}[!t]
  \centering
\includegraphics[width=0.6\textwidth]{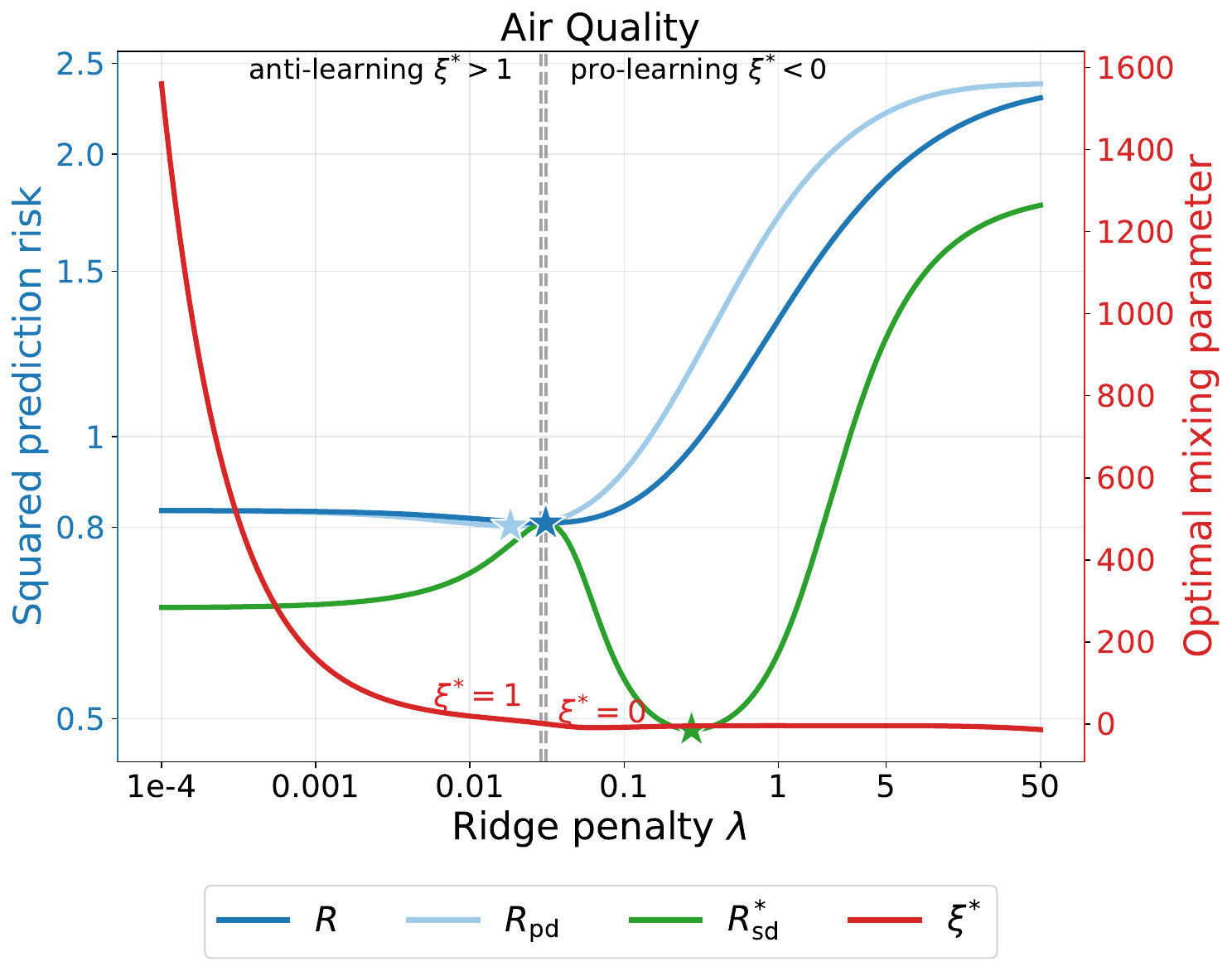}
  \caption{\textbf{Out-of-distribution SD risk improvement} Test prediction risk of ridge and optimal SD ridge on \dataset{Air Quality} dataset (see \Cref{sec:additional_details} for more details). SD yields strict improvements across $\lambda$ and achieves a substantially smaller global minimum.
  }
  \label{fig:air_quality}
\end{figure}

Plainly, optimal self-distillation strictly improves the teacher at every \emph{nonstationary} point of the teacher ridge risk (i.e., whenever $R'(\lambda)\neq 0$).
Moreover, $\xi^\star(\lambda)$ is positive in under-regularized regimes (where $R'(\lambda)<0$) and negative in over-regularized regimes (where $R'(\lambda)>0$), a behavior we refer to as ``pro-learning'' to contrast it with ``anti-learning'' \citep{das2023understanding}.
While prior work has noted improvements of optimal SD, to our knowledge, \Cref{thm:tangent_sign} is the first result establishing pointwise strict improvements for optimal self-distilled ridge, along with the sign rule for optimal mixing, at any nonstationary $\lambda$, for any squared prediction risk and without distributional assumptions.
The strict improvement and sign rule are consistently observed in \Cref{fig:exp_real_main,fig:air_quality}.

\subsection{Can Self-Distillation Beat Optimally Tuned Ridge?}
\label{sec:structural_global_min_short}

\Cref{thm:tangent_sign} guarantees improvement at any suboptimal $\lambda$, but it also implies that at a ridge-optimal penalty $\lambda^\star$, we must have $R_{\sd}^\star(\lambda^\star)=R(\lambda^\star)$.
It is therefore natural to ask whether the \emph{global} minimum of the SD risk curve can be strictly smaller than the teacher's global minimum.
We give a sufficient condition based on the local curvature of $R(\lambda)$.

\begin{proposition}[Curvature test at the ridge-optimal \(\lambda\)]
\label{prop:curvature_test_main}
Let $\lambda^\star\in\arg\min_{\lambda>0}R_{\te}(\lambda)$.
If the following curvature test at $\lambda^\star$ holds, i.e. if
\begin{equation}
  \label{eq:curv_cond_main}
  D(\lambda^\star) 
  <
  \frac{\lambda^{\star\,2}}{2}\,R''(\lambda^\star),
\end{equation}
then \(R_{\sd}^\star\) has negative curvature at \(\lambda^\star\), and, consequently,
\begin{equation}
  \min_{\lambda > 0} R_{\sd}^{\star}(\lambda) < \min_{\lambda > 0} R_{\te}(\lambda).
\end{equation}
\end{proposition}
\Cref{prop:curvature_test_main} formalizes a phenomenon visible in our experiments:
the curves \( R_{\te}(\lambda)\) and \(R_{\sd}^\star(\lambda)\) necessarily \emph{touch} at ridge-optimal value \(\lambda^{\star}\) of the regularization parameter, but \(R_{\sd}^\star\) may bend downward elsewhere and achieve a strictly smaller global minimum.
The curvature test \eqref{eq:curv_cond_main} can be viewed as a generalization of related curvature-based sufficient conditions by \citet[Theorem 3.8]{das2023understanding}, as it applies to arbitrary squared prediction risk and does not require that the response follows a well-specified linear model.

\begin{table}[!t]
\centering
\caption{Curvature test at the ridge-optimal $\lambda^{\star}$.}
\begin{tabular}{cccc}
\toprule
Dataset & Illustration & \eqref{eq:curv_cond_main} holds? & Global gain? \\
\midrule
BlogFeedback & \Cref{fig:blog_main}  &  \cmark & \cmark \\
Communities and Crime & \Cref{fig:CC_main} & \xmark & \xmark \\
CIFAR10  & \Cref{fig:cifar_main} & \xmark & \xmark \\
Air Quality & \Cref{fig:air_quality} & \cmark  & \cmark \\
\bottomrule
\end{tabular}
\label{tab:curvature_test}
\end{table}

As mentioned above, the results in this section are concerned with the to out-of-distribution prediction risk, where $(x_0, y_0)$ may be drawn from a distribution different from that of the training samples.
In \Cref{fig:air_quality}, we showcase a real-data example on the \dataset{Air Quality} dataset with a distribution shift between the training and test sets due to time-dependent effects.
We also numerically verify the curvature test in \Cref{tab:curvature_test} for all of our real-data experiments and observe exact agreement in all cases.

While the results above guarantee a strict improvement at every nonstationary $\lambda$ and require no distributional assumptions, they do not quantify the \emph{magnitude} of the gain, i.e., how much smaller $R_{\sd}^\star(\lambda)$ can be than $R(\lambda)$.
In the next section, we consider standard distributional assumptions in the proportional-asymptotics literature and derive deterministic equivalents for these risks and their improvements as explicit functions of the problem parameters.

\section{Proportional Asymptotic Results}
\label{sec:asymptotics}

We now refine the structural identities from \Cref{sec:structural} by deriving deterministic limits for the optimal SD risk $R_{\sd}^\star(\lambda)$ and the optimal mixing weight $\xi^\star(\lambda)$ in the proportional asymptotics regime in which $n,p\to\infty$ with $p/n\to\gamma\in(0,\infty)$.
Our goals are to (i) obtain computable asymptotic expressions for $R_{\sd}^\star(\lambda)$ and its gain over the teacher risk $R(\lambda)$, and (ii) quantify how these quantities depend on the aspect ratio $\gamma$, the signal-to-noise ratio $\SNR$, and signal--covariance alignment (to be defined below).

\subsection{Data Assumptions}

We assume $\{(x_i,y_i)\}_{i=1}^n$ are drawn i.i.d.\ from a distribution $P_{x,y}$ satisfying the following:

\begin{assumption}
    [Data distribution]
    \label{def:dist}
    The covariate vector $x \sim P_{x}$ admits the representation $x = \Sigma^{1/2}\bz$, where $\Sigma \in \RR^{p \times p}$ is deterministic and positive definite with eigenvalues uniformly bounded away from $0$ and $\infty$, and $\bz \in \RR^{p}$ has i.i.d.\ entries with mean $0$, variance $1$, and uniformly bounded $(4+\mu)$-th moment for some $\mu > 0$.
    The response $y \sim P_{y}$ has mean $0$ and uniformly bounded $(4+\nu)$-th moment, for some $\nu > 0$.
\end{assumption}

The feature structure imposed in \Cref{def:dist} is standard in random matrix analyses of high-dimensional regression; see, e.g., \cite{baisilverstein,bartlett2021deep, misiakiewicz2024six}.
We do \emph{not} assume a well-specified linear model for the response $y$.
Instead, we parameterize $P_{x,y}$ by $(\Sigma, \beta, \sigma^2)$, where $\beta := \Sigma^{-1} \EE[x y]$ denotes the parameter of the population linear $L_2$ projection of $y$ onto $x$ and $\sigma^2:=\Var(y - x^\top \beta)$ denotes the corresponding residual variance.
Finally, We denote the signal energy by $r^2 := \| \beta \|^2_2$ and $\SNR:=r^2/\sigma^2$.

Throughout this section, we focus on the in-distribution squared prediction risk in which the test point $(x_0,y_0)$ in \eqref{eq:pred-risk} is an independent copy drawn from the same distribution as the training samples.
Extensions to out-of-distribution risks is possible building on results of \citet{tripuraneni2021covariate,patil2024optimal} but is technically involved and left for future work.

\subsection{Asymptotics of Optimal Self-Distillation Risk and Mixing Weight}
\label{sec:asymptotics_main}

\begin{figure}[!t]
    \centering
    \includegraphics[width=0.9\textwidth]{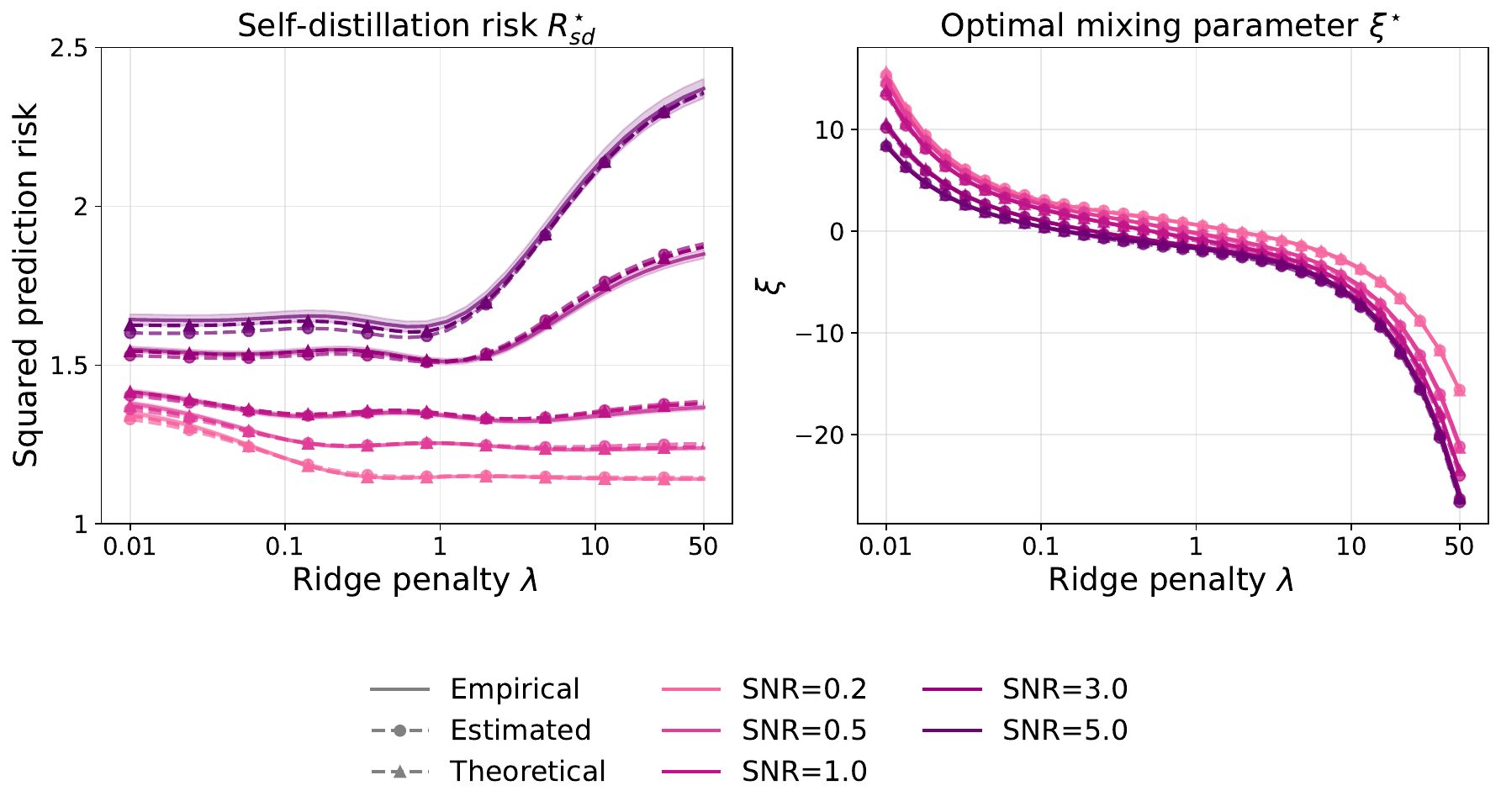}
    \caption{\textbf{Theoretical versus empirical risks.} Asymptotic SD risks and optimal mixing weights versus $\lambda$ across multiple $\SNR$ values. Empirical curves are averaged over $30$ simulations. Estimated curves are obtained using the proposed one-shot tuning method (\Cref{sec:tuning}), and the theoretical curves are obtained from \Cref{thm:risk-asymptotics}. Setting: $n = 400$, $p =200$, $\sigma^2 = 1$, $r^2 = \sigma^2 \SNR$; $\Sigma$ is AR1, $\beta$ is a deterministic signal aligned with the top $10\%$ eigenvectors of $\Sigma$ (alignment factor 0.9). (See \Cref{sec:additional_details} for more details.)}
    \label{fig:asymptotic_ar1_top_aligned}
\end{figure}

Recall from \Cref{sec:structural_closed_form} that, for each fixed $\lambda>0$, the optimal SD risk \eqref{eq:R1_star_short} can be expressed in terms of three random quantities: the teacher risk $R(\lambda)$, the PD risk $R_{\pd}(\lambda)$, and the residual correlation term $C(\lambda)$.
Under proportional asymptotics, these quantities concentrate and converge to deterministic limits, which we denote by $\sR(\lambda)$, $\sR_{\pd}(\lambda)$, and $\sC(\lambda)$, respectively, that only depend on $(\Sigma,\beta,\sigma^2)$, $\gamma$, and $\lambda$.
To characterize these limits, we introduce some scalar parameters.

Write $\otr(A):=\tr(A)/p$ for the normalized trace.
For a fixed $\lambda>0$, let $\kappa=\kappa(\lambda)>0$ be the unique solution to the fixed-point equation
\begin{equation}
\label{eq:kappa_fp_mainpaper}
\kappa=\lambda+\gamma\kappa\,\otr\!\big(\Sigma(\Sigma+\kappa I_p)^{-1}\big).
\end{equation}
Let  $G:=(\Sigma+\kappa I_p)^{-1}$ be the resolvent at $\kappa$ and for any $k\in\{2,3,4\}$ let the trace and signal--covariance alignment functionals be
\begin{equation}
\label{eq:t234_mainpaper}
t_k:=\gamma\,\otr(\Sigma^2G^k),
\quad
q_k:=\beta^\top G^k\Sigma\,\beta,
\end{equation}
respectively.
Next, setting $b:=(1-t_2)^{-1}$, define the variance-trace combinations as
\begin{equation}
u_2:=t_2 b,\quad
u_3:=t_3 b^3,\quad
u_4:=t_4 b^4+2t_3^2 b^5.
\end{equation}
Finally, define the coefficients
\begin{equation}
a_2:=bE^2+b^4\kappa^2\lambda^2 t_4+b^5\kappa^2\lambda^2 t_3^2,
\quad
a_3:=2b^2\kappa\lambda E,
\quad
a_4:=b^3\kappa^2\lambda^2.
\end{equation}
where $E:=\kappa-b\lambda+b^2\kappa\lambda\,t_3$. (The dependence on $\lambda$ of here and throughout is uppressed for readability.)
We are now ready to state the main result of this section.

\begin{theorem}
    [Risk asymptotics]
    \label{thm:risk-asymptotics}
    Under \Cref{def:dist},
    as $n, p \to \infty$ with $p / n \to \gamma \in (0, \infty)$, for each fixed $\lambda > 0$, we have
    \begin{align}
      \xi^\star(\lambda)
      \pto \frac{\sR_{\te}(\lambda) - \sC(\lambda)}{\sR(\lambda) + \sR_{\pd}(\lambda) - 2 \sC(\lambda)},
      \quad
      R_{\sd}^\star(\lambda)
      \pto
      \sR_{\te}(\lambda) -
      \frac{(\sR_{\te}(\lambda)-\sC(\lambda))^2}{\sR(\lambda) + \sR_{\pd}(\lambda) - 2 \sC(\lambda)},
    \end{align}
    where the individual component limits are as follows:
    \begin{align*}
  \sR(\lambda)
  &:=
  \kappa^2 b\,q_2
  +
  \sigma^2\,u_2
  + 
  \sigma^2,
  \\
  \sC(\lambda)
  &:=
  2\kappa^2 b\,q_2-(\kappa bE\,q_2+\kappa^2 b^2\lambda\,q_3)
  +
  \sigma^2\,(u_2-\lambda u_3)
  +
  \sigma^2,
  \\
  \sR_{\pd}(\lambda)
  &:=
  4\kappa^2 b\,q_2-2(2\kappa bE\,q_2+2\kappa^2 b^2\lambda\,q_3) +(a_2q_2+a_3q_3+a_4q_4) + \sigma^2\,(u_2-2\lambda u_3+\lambda^2 u_4)
  +
  \sigma^2.
    \end{align*}
\end{theorem}

To our knowledge, this is the first precise characterization of the optimal SD risk in the proportional asymptotics regime that holds under the general setting of anisotropic covariance and deterministic signal.
The spectrum of $\Sigma$ enters through resolvent traces such as $t_2,t_3,t_4$, while the signal enters through the alignment functionals $q_2,q_3,q_4$.
As illustrated in \Cref{fig:asymptotic_ar1_top_aligned}, the theoretical predictions closely match empirical risks even for moderate $n$ and $p$.

In the isotropic-signal specialization when $\beta \sim \mathcal{N}(\bm{0}, (r^2/p)  I_p)$, the quadratic forms $q_k$ simplify to trace functionals; see \Cref{app:thm:risk-asymptotics-isotropic-signal}.
In this setting, the strict-improvement and sign behavior becomes particularly transparent, as illustrated in our next result.

\begin{corollary}
    \label{cor:xi_asymptotic}
    Under \Cref{def:dist} with $\beta\sim\mathcal{N}(0,(r^2/p)I_p)$, we have
    \begin{enumerate}[label=(\alph*),leftmargin=7mm,nosep,itemsep=2pt]
        \item $\sR_{\sd}^{\star}(\lambda) < \sR(\lambda) \text{ for all } \lambda \neq \lambda^{\star} := \gamma \frac{\sigma^2}{r^2}$.
        \item $\xi^{\star}(\lambda) < 0$ iff $\lambda > \lambda^{\star}$, and $\xi^{\star}(\lambda) > 0$ iff $\lambda < \lambda^{\star}$, i.e., $\mathrm{sign}(\xi^{\star}(\lambda)) = \mathrm{sign}(\lambda^{\star} - \lambda)$ for all $\lambda \neq \lambda^{\star}$.
    \end{enumerate}
\end{corollary}

Thus, in the isotropic-signal setting, optimal SD strictly improves upong ridge at every suboptimal value of $\lambda$, and it matches the (statistically optimal, for the isotropic design) ridge risk only at the ridge-optimal value $\lambda^\star$.
Moreover, the sign of the optimal mixing weight is determined solely by whether the ridge model is under-regularized ($\lambda<\lambda^\star$) or over-regularized ($\lambda>\lambda^\star$).
Simulation studies show in \Cref{fig:asymptotic_gain_ar1_top_aligned} are in close agreement with these findings: (i) $\sR_{\sd}^\star(\lambda)$ lies below $\sR(\lambda)$ for all nonstationary $\lambda$, (ii) $\xi^\star(\lambda)$ flips sign at $\lambda^\star$, and (iii) the gain is largest in strongly under- or over-regularized regimes.
Additional illustrations for other covariance and signal geometries are provided in \Cref{app:additional-illustrations-propasymp}.

\begin{figure}[!t]
    \centering
    \includegraphics[width=0.49\textwidth]{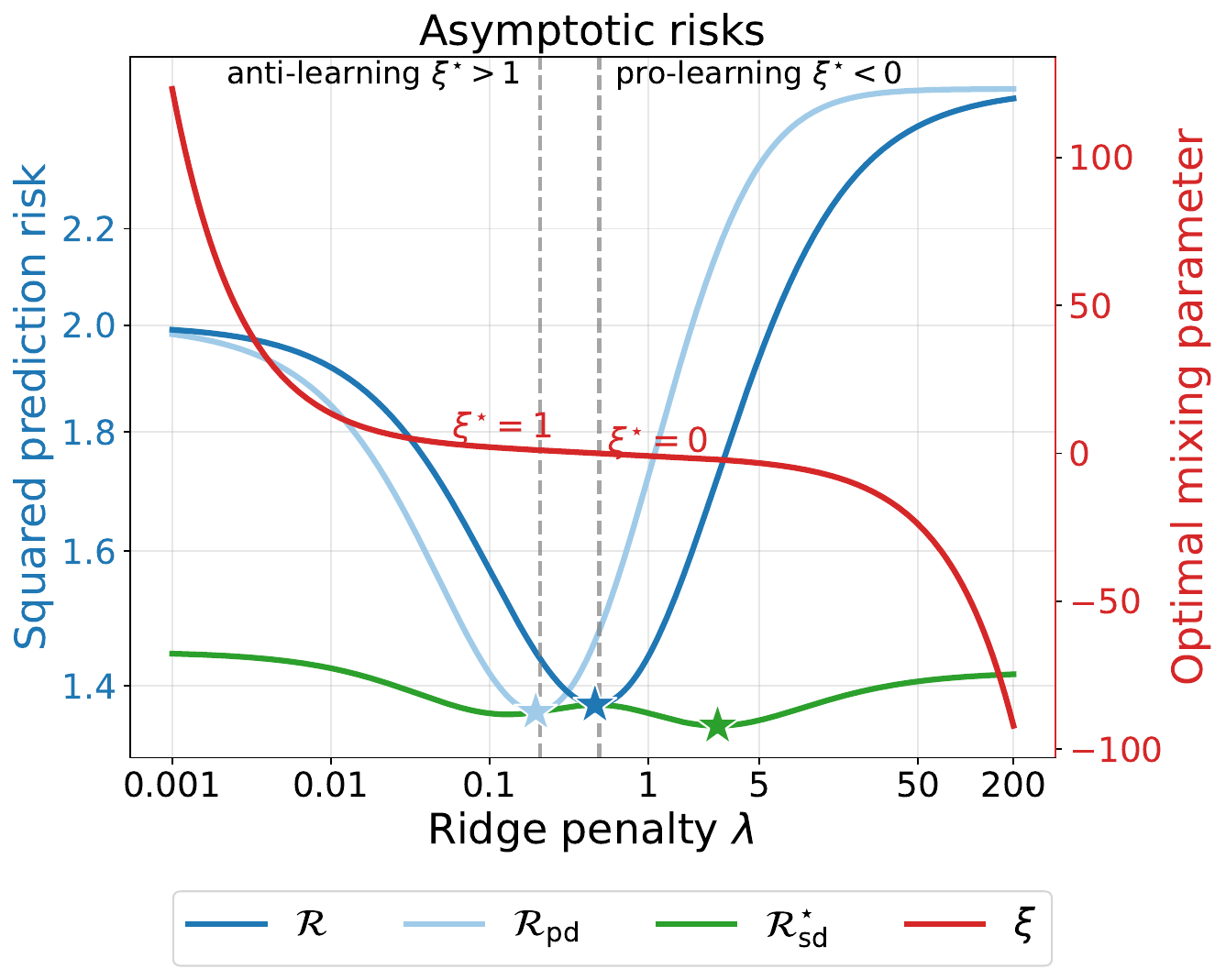}
    \includegraphics[width=0.49\textwidth]{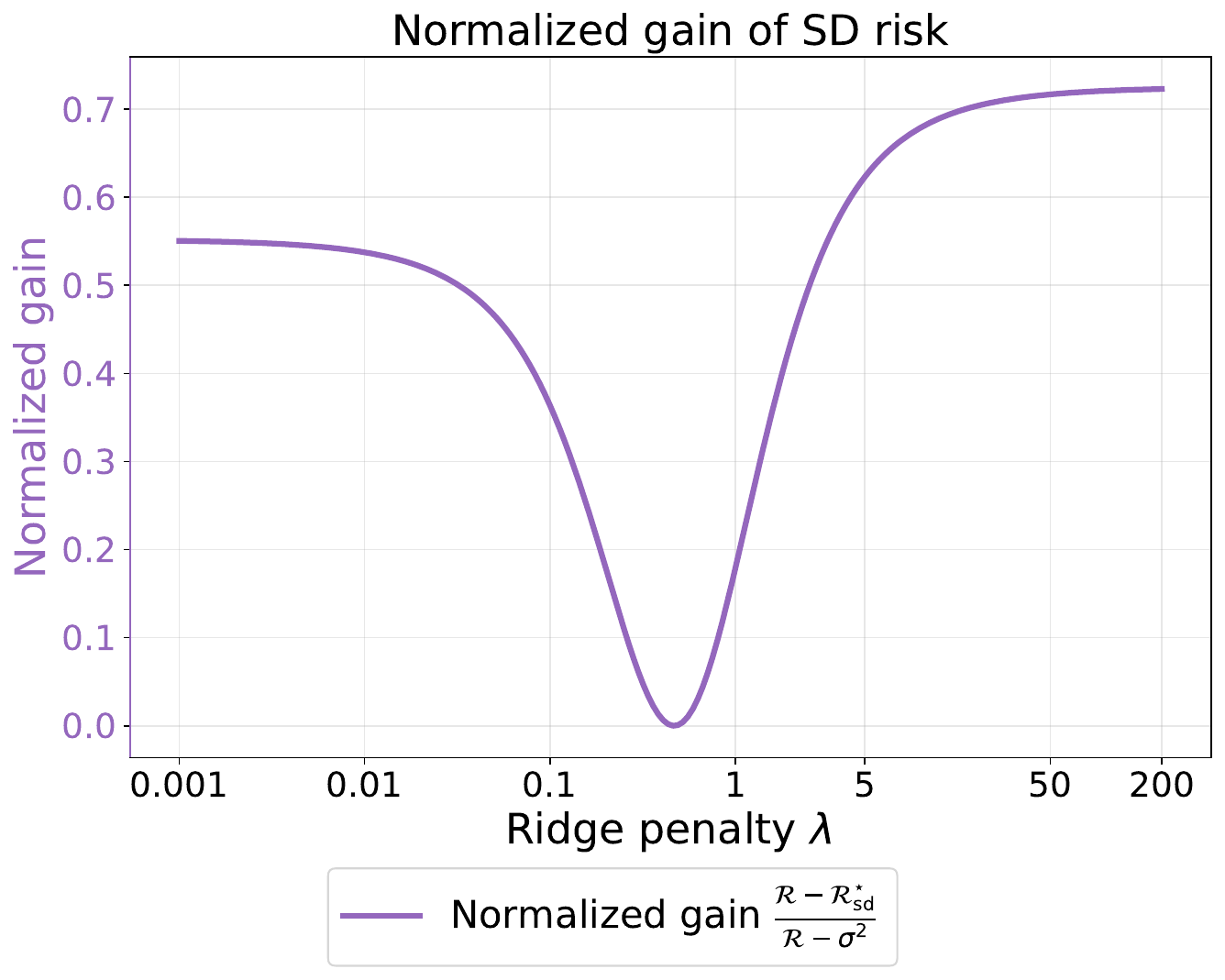}
    \caption{\textbf{Asymptotic gain over the teacher.} Deterministic limits of the risks and the gain $\sR(\lambda)-\sR_{\sd}^\star(\lambda)$. Same setting as \Cref{fig:asymptotic_ar1_top_aligned} with $r^2 = \sigma^2 = 1$.}
    \label{fig:asymptotic_gain_ar1_top_aligned}
\end{figure}

\subsection{Self-Distillation Risks with Extreme Regularization}
\label{sec:asymptotics_gains}

Next, we study how close optimal SD can get to the best possible predictor when the teacher is \emph{extremely} under- or over-regularized.
We focus on the isotropic design and isotropic signal setting ($\Sigma=I_p$ and $\beta\sim\mathcal{N}(0,(r^2/p)I_p)$), where the ridge-optimal predictor is Bayes-optimal and its asymptotic risk $\sR^\star$ is known in closed form \citep{dobriban2018high, hastie2022surprises}.
We compare the limiting SD risk $\sR_{\sd}^\star(\lambda)$ to $\sR^\star$ as $\lambda\to 0$ and $\lambda\to\infty$.

\begin{proposition}[Comparison with the optimal ridge]
\label{prop:compare_extreme_lambda}
        Assume $\Sigma=I_p$ and $\beta\sim\mathcal{N}(0,(r^2/p)I_p)$ and 
        let $S^{\star}(\SNR, \gamma) := \frac{1}{2 \gamma} (\, \SNR (\gamma - 1) - \gamma 
        +\sqrt{4 \SNR \gamma^2 + (\SNR(\gamma - 1) - \gamma)^2} \,)$. 
       Then
        \begin{align}
         \lim_{\lambda \to 0}
         \,
         \frac{\sR_{\sd}^{\star}(\lambda) - \sR^{\star}}{\sR^{\star}}
         &=
     \left\{
    \begin{array}{ll}
    \displaystyle
    \frac{\SNR (1 - \gamma)^2 + \gamma}{
    (\SNR(1 - \gamma)^3 + \gamma(1 - \gamma^2))(S^{\star} + 1)
    } - 1,
    & \gamma \in (0, 1), \\
    \displaystyle
    \frac{\SNR^2 (\gamma - 1)^4 + \SNR \gamma (2\gamma + 1)(\gamma -1)^2 + \gamma^4}{
    (\SNR \gamma (\gamma - 1)^3 + \gamma^2(\gamma^2 - 1))(S^{\star} + 1)
    } - 1,
    & \gamma \in (1, \infty), \\
    \end{array}
    \right.
        \nonumber 
        \\
         \lim_{\lambda \to \infty}
        \,
        \frac{\sR_{\sd}^{\star}(\lambda) - \sR^{\star}}{\sR^{\star}}
        &= 
        \frac{\SNR^2 \gamma + \SNR (2 \gamma + 1) + \gamma}{( \SNR (\gamma + 1) + \gamma ) (S^{\star} + 1)} - 1
        .
        \nonumber 
        \end{align}
\end{proposition}

\begin{figure}[!t]
    \centering
    \includegraphics[width=0.99\textwidth]{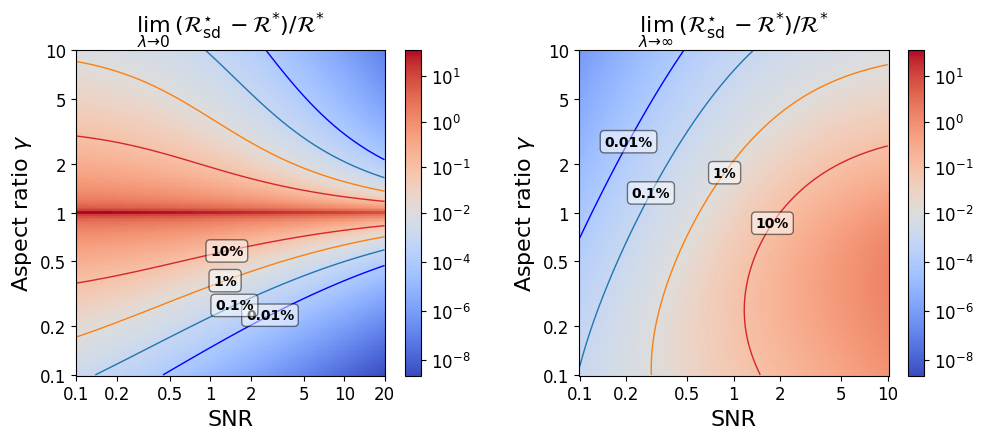}
    \caption{\textbf{Approximating the best predictor.} Percentage difference between $\sR_{\text{sd}}^{\star}(\SNR, \gamma)$ and the Bayes-optimal ridge risk $\sR^{\star}(\SNR, \gamma)$ in the isotropic design and signal setting.}
    \label{fig:heatmap_isotropic}
\end{figure}

 \Cref{prop:compare_extreme_lambda} expresses the relative suboptimality of optimal SD under extreme regularization scenarios explicitly in terms of $(\gamma,\SNR)$ under the isotropic design and signal setting.
The resulting percentage gaps are illustrated in \Cref{fig:heatmap_isotropic}.
Across a wide range of $(\gamma,\SNR)$, optimal SD can be remarkably close to Bayes-optimal performance even when the teacher is extremely under- or over-regularized (e.g., within $0.01\%$ for $(\SNR,\gamma)=(2,0.2)$ as $\lambda\to 0$).
The heatmaps also suggest a qualitative dichotomy: for low $\SNR$, over-regularization is broadly preferable across $\gamma$, while for high $\SNR$, under-regularization is preferable except near the square design regime when $\gamma\approx 1$.
For comparison, we also report the analogous extreme-$\lambda$ gaps between the original ridge risk $\sR(\lambda)$ and the ridge-optimal risk $\sR^{\star}$ in \Cref{app:additional-illustrations-extremereg}, along with experiments with other covariance structures.

\section{One-Shot Tuning and Risk Estimation}
\label{sec:tuning}

The results in \Cref{sec:asymptotics} quantify the benefits of optimal self-distillation.
However, the deterministic equivalents in \Cref{thm:risk-asymptotics} depend on population quantities (e.g., the covariance spectrum and signal--covariance alignment), which are unknown in practice.
For the practical question for tuning the mixing parameter $\xi$, a standard method involves grid search combined with split cross-validation (CV), but this is computationally expensive (requiring repeated retrainings over candidate $\xi$) and statistically inefficient in high dimensions due to sample splitting (see, e.g., \citet{rad2020scalable}).
In this section, we propose a computationally efficient one-shot procedure that estimates $\xi^\star(\lambda)$ from the training data without grid search or hold-out sets.

\subsection{Risk Estimators via Generalized Cross-Validation}
\label{sec:tuning_estimators}

Our starting point is the closed-form identity for optimal SD given in 
\Cref{prop:closed_form_xi_R1}, which expresses $\xi^\star(\lambda)$ and $R_{\sd}^\star(\lambda)$ in terms of the teacher risk $R(\lambda)$, the PD risk $R_{\pd}(\lambda)$, and the residual correlation term $C(\lambda)$.
We construct estimators for these three terms using training data by means of generalized cross-validation (GCV) and its variants, along with plug-in estimation of $\xi^\star(\lambda)$.

Let $\hat y_\lambda := f_\lambda(X)$ and $\hat y_{\pd,\lambda} := f_{\pd,\lambda}(X)$ be the fitted values of the ridge teacher and the pure-distilled predictors, respectively.
Our estimators involve the notion of \emph{effective degrees of freedom} from the theory of statistical optimism \citep{efron_1983,efron_1986}.
The degrees of freedom $\df(f)$ of a (possibly nonlinear) predictor $f$ is measured by the trace of the operator $y \mapsto (\partial/\partial y)f(y)$ \citep{hastie_tibshirani_1990,stein_1981}.
In particular, for linear smoothers, this corresponds to the trace of the smoothing matrix (the so-called ``hat'' matrix). 

Set $\df_\lambda := \df(f_\lambda)$ and $\df_{\pd,\lambda} := \df(f_{\pd,\lambda})$ and define the GCV-corrected residuals
\begin{equation}
\label{eq:gcv_residuals_compact}
  \hat r_\lambda:=\frac{y-\hat y_\lambda}{1-\df_\lambda/n},
  \quad
  \hat r_{\pd,\lambda}:=\frac{y-\hat y_{\pd,\lambda}}{1-\df_{\pd,\lambda}/n}.
\end{equation}
We estimate the teacher and PD prediction risks, as well as their residual correlation term, by
\begin{equation}
\label{eq:risk_estimators_compact}
  \widehat R(\lambda) := \frac{\|\hat r_\lambda\|_2^2}{n},
  \quad
  \widehat R_{\pd}(\lambda) := \frac{\|\hat r_{\pd,\lambda}\|_2^2}{n},
  \quad 
  \widehat C(\lambda) := \frac{\langle \hat r_\lambda,\hat r_{\pd,\lambda}\rangle}{n},
\end{equation}
respectively.
Here $\widehat R(\lambda)$ coincides with the standard ridge GCV estimator, while $\widehat R_{\pd}(\lambda)$ and $\widehat C(\lambda)$ extend the same $\df$ correction principle to the PD and cross-term quantities appearing in \Cref{prop:closed_form_xi_R1}.

\subsection{One-Shot Estimators for Optimal Mixing Weight and Optimal SD Risk}

Plugging the estimators \eqref{eq:risk_estimators_compact} into the exact identities in \eqref{eq:R1_star_short} yields the one-shot estimators
\begin{equation}
\label{eq:oneshot_xi_Rsd}
  \hat\xi^\star(\lambda):=\frac{\widehat R(\lambda)-\widehat C(\lambda)}{\widehat R(\lambda)+\widehat R_{\pd}(\lambda)-2\widehat C(\lambda)},
  \quad
  \widehat R_{\sd}^\star(\lambda)
  :=\widehat R(\lambda)-\frac{(\widehat R(\lambda)-\widehat C(\lambda))^2}{\widehat R(\lambda)+\widehat R_{\pd}(\lambda)-2\widehat C(\lambda)}.
\end{equation}
Note that the estimated denominator equals
\begin{equation}
\widehat D(\lambda):=\widehat R(\lambda)+\widehat R_{\pd}(\lambda)-2\widehat C(\lambda)
=\frac{1}{n}\|\hat r_\lambda-\hat r_{\pd,\lambda}\|_2^2\ge 0,
\end{equation}
mirroring the nonnegativity of its population counterpart in \eqref{eq:Dlambda}.
When $\widehat D(\lambda)$ is extremely small in finite samples, one may stabilize \eqref{eq:oneshot_xi_Rsd} by adding a small ridge term to the denominator; in our experiments this was not necessary.
Our next result shows that the one-shot estimates are consistent in the proportional regime.

\begin{theorem}[Consistency of one-shot SD tuning]
\label{thm:oneshot_consistency}
Under \Cref{def:dist}, as $n,p\to\infty$ with $p/n\to\gamma\in(0,\infty)$, for each fixed $\lambda>0$, we have
\begin{equation}
  \hat\xi^\star(\lambda) -\xi^\star(\lambda) \pto 0,
  \quad
  \widehat R_{\sd}^\star(\lambda) -R_{\sd}^\star(\lambda) \pto 0.
\end{equation}
\end{theorem}

Compared with grid-search CV over $\xi$, the one-shot procedure has two key advantages:
(i) CV requires retraining a student model for each candidate $\xi$, whereas \eqref{eq:oneshot_xi_Rsd} selects $\hat\xi^\star(\lambda)$ in closed form from a single set of fitted quantities at the given $\lambda$.
(ii) Split CV reduces the effective training sample size (e.g., to $4/5$) and suffers from nonzero bias in high-dimensional settings where $p$ is comparable to $n$; in contrast, the one-shot estimators use all $n$ samples without hold-out sets.

\subsection{Real Data Experiments}

To illustrate the utility of the one-shot tuning method, we apply \eqref{eq:oneshot_xi_Rsd} across a range of $\lambda$ values on several real datasets.
For regression, we consider UCI \dataset{BlogFeedback} and \dataset{Communities and Crime}.
For classification, we apply ridge regression on pretrained neural network features (ResNet-18/34) for \dataset{CIFAR10} and \dataset{CIFAR100} (details in \Cref{sec:additional_details}).
 
Across these tasks, \Cref{fig:exp_real_main,fig:cifar100} show that $\widehat R_{\sd}^\star(\lambda)$ closely tracks the test risk of the optimally distilled student, particularly in settings with small train and test discrepancy (e.g., CIFAR dataset benchmarks).
Moreover, when the teacher is over-regularized, the one-shot estimate correctly selects negative $\hat\xi^\star(\lambda)$, so that the SD predictor corrects excessive shrinkage.
This regime would be missed by restricting $\xi$ to $[0,1]$ (see \Cref{fig:exp_real_world_restricted} for constrained SD risks as a comparison).
Additional experiments and sample-size variations are provided in \Cref{app:additional-illustrations-real-world}, and for \dataset{CIFAR10}/\dataset{CIFAR100} we also report the corresponding test accuracies in \Cref{fig:cifar_test_accuracy}.

\section{Extensions and Variants}
\label{sec:extensions}

Our results so far focus on optimal \emph{one-round} SD for \emph{ordinary} ridge regression, where the student is refit on the \emph{same} design matrix $X$ as the teacher.
In this section, we briefly outline several extensions that are naturally captured by the same ``structural'' viewpoint from \Cref{sec:structural}.

\subsection{Multiple Rounds of Self-Distillation}
\label{sec:multiple-rounds}

Fix $\lambda>0$ and set $y^{(0)}:=y$.
A natural \emph{recursive} multi-round scheme iteratively self-distills using the previous round's ridge predictions.
For each $k\ge 0$, let $f^{(k)}_\lambda$ denote the teacher ridge regression \eqref{eq:ridge_predictor} trained on $(X,y^{(k)})$ with penalty $\lambda$.
and write $\hat y^{(k)}_{\lambda}:=f^{(k)}_\lambda(X)\in\RR^n$ for its fitted values.
Let $f^{(k)}_{\pd,\lambda}$ denote the corresponding PD refit of round $k$, i.e., ridge trained on
$(X,\hat y^{(k)}_{\lambda})$ with the same penalty $\lambda$.
Given a mixing weight $\xi_{k+1}\in\RR$, define the round-$(k+1)$ SD predictor
$f^{(k+1)}_{\sd,\lambda,\xi_{k+1}}$ as the ridge fit obtained from the same mixed-loss construction as in
\eqref{eq:sd_predictor}, but with base labels $y^{(k)}$ and teacher pseudo-labels $\hat y^{(k)}_\lambda$.
A simplification (see \Cref{sec:recursive-affine-path-proof}) shows this is equivalent to fitting ridge regression on the mixed labels (see \Cref{fig:multiple-rounds-setup} for an illustration):
\begin{equation}
\label{eq:recursive_multiround}
  y^{(k+1)}_{\lambda,\xi_{k+1}}
  :=
  (1-\xi_{k+1})\,y^{(k)}+\xi_{k+1}\,\hat y^{(k)}_{\lambda}.
\end{equation}
Because ridge is linear in the response, the round-$(k+1)$ predictor $f^{(k+1)}_{\sd,\lambda,\xi_{k+1}}$ lies on the affine path:
\begin{equation}
    \label{eq:recursive-affine-path}
  f^{(k+1)}_{\sd,\lambda,\xi_{k+1}}
  =(1-\xi_{k+1})\,f^{(k)}_{\lambda}+\xi_{k+1}\,f^{(k)}_{\pd,\lambda}.
\end{equation}

\begin{figure}[!t]
    \centering
    \includegraphics[width=0.99\textwidth]{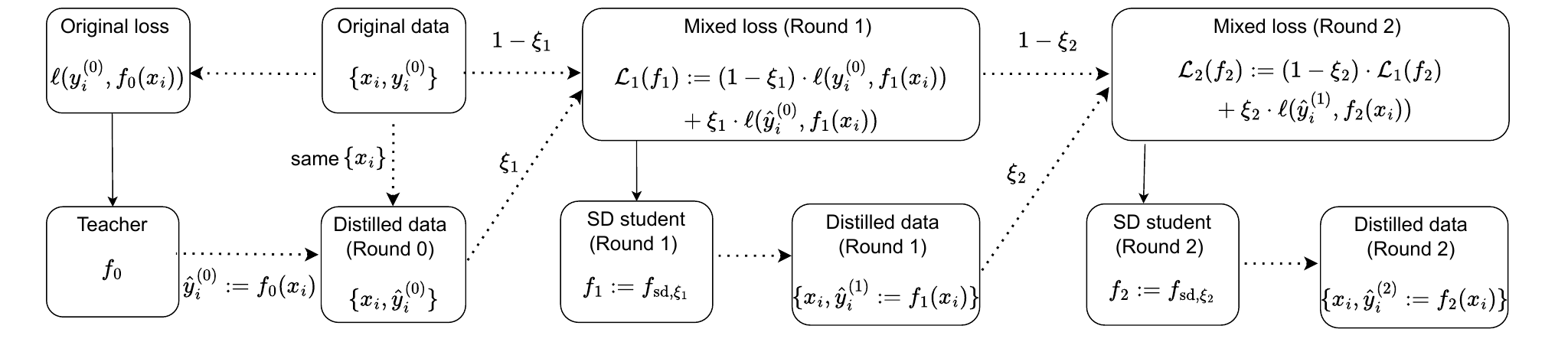}
    \caption{Visual illustration of the recursive multi-round self-distillation process for ridge regression.
    }
    \label{fig:multiple-rounds-setup}
\end{figure}

As in \Cref{sec:structural}, we evaluate any predictor $f$ by the conditional squared prediction risk $R(f):=\EE[(y_0-f(x_0))^2\mid \cD]$ for a (possibly out-of-distribution) test pair $(x_0,y_0)$ with finite conditional second moments.
Define the round-$k$ risk: $R_k(\lambda):=R(f^{(k)}_\lambda)$, round-$k$ correlation: $C_k(\lambda) :=\EE[(y_0-f^{(k)}_\lambda(x_0))(y_0-f^{(k)}_{\pd,\lambda}(x_0))\mid \cD]$, and round-$k$ discrepancy: $D_k(\lambda):=\EE[(f^{(k)}_\lambda(x_0)-f^{(k)}_{\pd,\lambda}(x_0))^2\mid \cD]\ge 0$.
For each round $k\ge 0$, let $\xi_{k+1}^\star(\lambda)\in\arg\min_{\xi_{k+1}\in\RR}R(f^{(k+1)}_{\sd,\lambda,\xi_{k+1}})$ denote the optimal (unconstrained) mixing weight.
Write $y^{(k+1)} = y^{(k+1)}_{\lambda,\xi_{k+1}^\star}$, $f^{(k+1)}_\lambda:=f^{(k+1)}_{\sd,\lambda,\xi_{k+1}^\star}$, and $R_{k+1}(\lambda):=R(f^{(k+1)}_\lambda)$.

All the one-round structural identities from \Cref{sec:structural} apply at each round $k$ by viewing $y^{(k)}$ as the ``base'' labels and $(f^{(k)}_\lambda,f^{(k)}_{\pd,\lambda})$ as the
round-$k$ (teacher, PD) pair.
In particular, whenever $D_k(\lambda)>0$, the round-wise optimizer and improvement satisfy the same formulas as
\Cref{prop:closed_form_xi_R1}:
\[
  \xi_{k+1}^\star(\lambda)=\frac{R_k(\lambda)- C_k(\lambda)}{D_k(\lambda)},
  \quad
  R_{k+1}(\lambda)
  =
  R_k(\lambda)-\frac{\big(R_k(\lambda)- C_k(\lambda)\big)^2}{D_k(\lambda)}.
\]
This leads to the following risk monotonicity result for optimal multi-round SD:

\begin{proposition}[Monotonicity of optimal recursive multi-round self-distillation]
\label{prop:multiround_monotone_main}
Fix $\lambda>0$.
The optimal SD risks are monotone (weakly) decreasing in the number of rounds $k$:
\[
  R_{k+1}(\lambda)\le R_k(\lambda)\quad\text{for all }k\ge 0.
\]
Assuming $D_{k}(\lambda) > 0$, the optimal SD mixing weights and risks admit the closed forms:
\begin{equation}
\label{eq:multiround_tangent_rule}
  \xi_{k+1}^\star(\lambda)
  =
  -\frac{\lambda}{2}\,\frac{R_k'(\lambda)}{D_k(\lambda)},
  \quad
  R_{k+1}(\lambda)
  =
  R_k(\lambda)-\frac{\lambda^2}{4}\,\frac{(R_k'(\lambda))^2}{D_k(\lambda)},
\end{equation}
where $R'_k(\lambda)$ is the derivative along the ridge path for the round-$k$ teacher with $y^{(k)}$ held fixed.
Thus, if $R_k'(\lambda)\neq 0$, then $R_{k+1}(\lambda) < R_k(\lambda)$,
and $\sign(\xi_{k+1}^\star(\lambda))=-\sign(R_k'(\lambda))$.
\end{proposition}

As with all results in \Cref{sec:structural}, \Cref{prop:multiround_monotone_main} is purely structural: it holds conditionally on $\cD=(X,y)$ with probability one and for any squared
prediction risk (including out-of-distribution).

\begin{figure}[!t]
\centering
\includegraphics[width=0.48\textwidth]{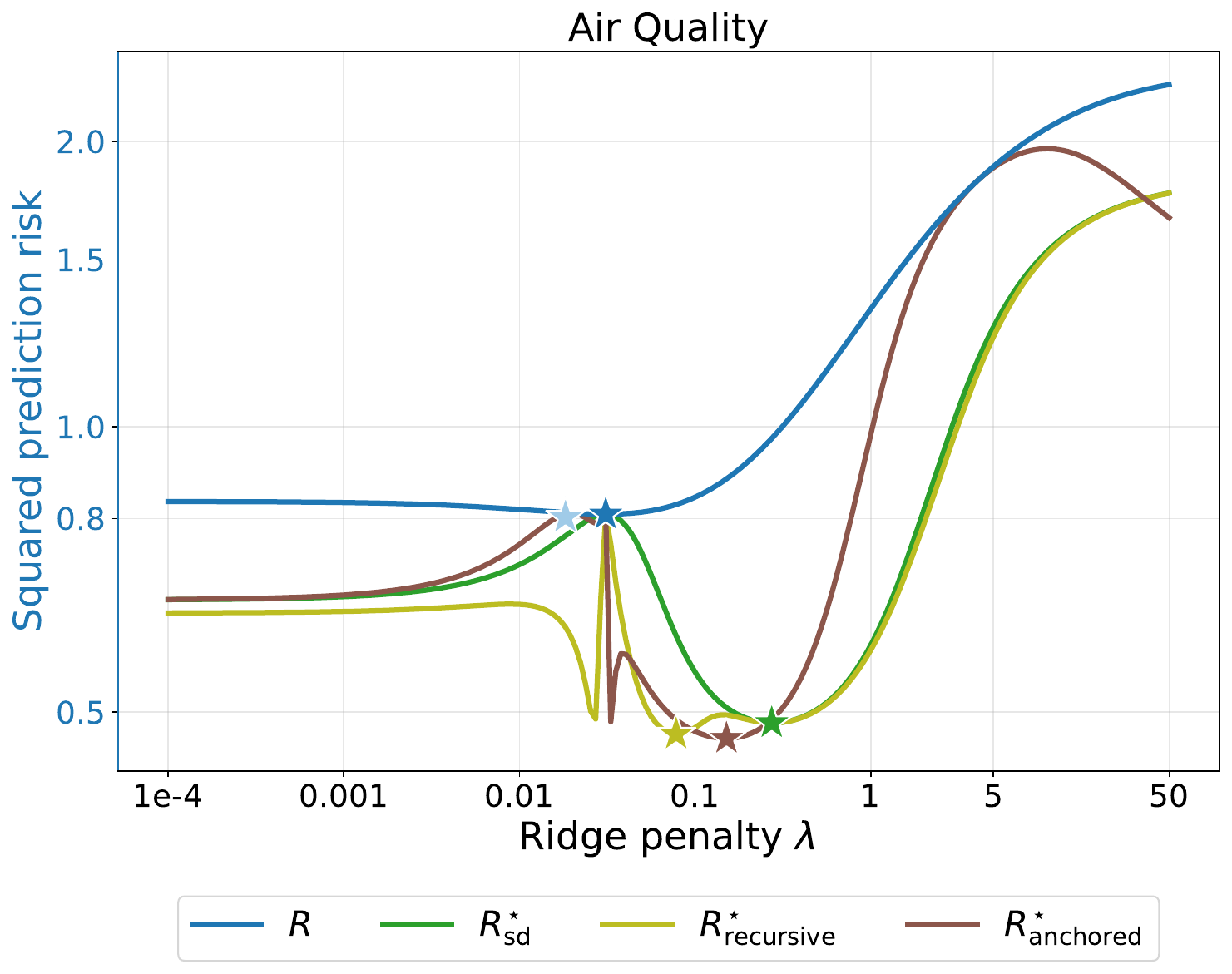}
\includegraphics[width=0.48\textwidth]{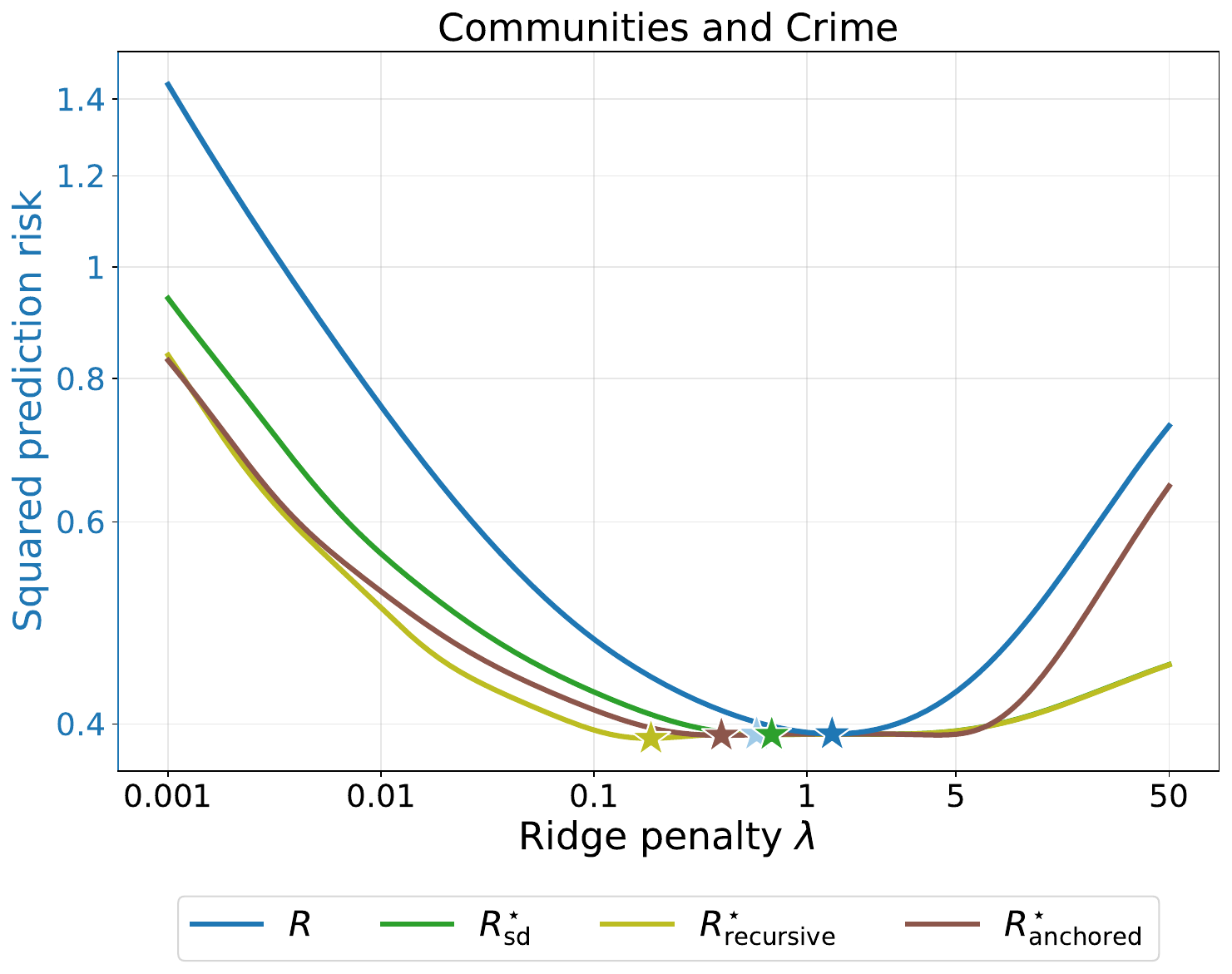}
\caption{\textbf{Recursive versus anchored multi-round self-distillation.} Test risks of the teacher ridge ($R$, in {\color{pyblue} blue}, trained using $y^{(0)}$), one-round ($k=1$) optimal self-distilled ridge ($R^{\star}_{\sd}$, in {\color{pygreen} green}, trained using $y^{(0)}$ and $\hat y^{(0)}$), two-round ($k=2$) recursive ($R^{\star}_{\text{recursive}}$, in {\color{pyolive} olive}, trained using $y^{(1)}$ and $\hat y^{(1)}$) and anchored ($R^{\star}_{\text{anchored}}$, in {\color{pybrown} brown}, trained using $y^{(0)}$ and $\hat y^{(1)}$) self-distillation.
Recursive mixing is monotone when $\xi_k$ is optimized each round (\Cref{prop:multiround_monotone_main}); the two-round risk curve $R^{\star}_{\text{recursive}}(\lambda)$ uniformly dominates the optimal one-round $R^{\star}_{\sd}(\lambda)$. 
Anchored mixing can be nonmonotone because the round-$k$ family need not contain the round-$(k-1)$ predictor; the two-round risk curve $R^{\star}_{\text{anchored}}(\lambda)$ may be larger than optimal one-round $R^{\star}_{\sd}(\lambda)$.}
\label{fig:multiround_monotonicity}
\end{figure}

Several repeated-distillation formulations instead remain \emph{anchored} to the original (ground-truth) labels
\citep{garg2025preventing, alemohammad2023self}.
Specifically, given $\xi_{k+1}\in\RR$, define the anchored
round-$(k+1)$ SD predictor as the ridge fit obtained by mixing the loss (as in \eqref{eq:sd_predictor}) against the original labels $y^{(0)}$
and the current pseudo-labels $\hat y^{(k)}_{\lambda}$.
As before, a simplification shows that this is equivalent to ordinary ridge
regression trained on the anchored mixed labels:
\begin{equation}
\label{eq:anchored_multiround}
  y^{(k+1),\mathrm{anch}}_{\lambda,\xi_{k+1}}
  :=
  (1-\xi_{k+1})\,y^{(0)}+\xi_{k+1}\,\hat y^{(k)}_{\lambda}.
\end{equation}
By linearity of ridge in the response, this anchored family is the affine path:
\[
  f^{(k+1),\mathrm{anch}}_{\sd,\lambda,\xi_{k+1}}
  =(1-\xi_{k+1})\,f^{(0)}_{\lambda}+\xi_{k+1}\,f^{(k)}_{\pd,\lambda},
\]
which in general need \emph{not} contain the previous-round predictor $f^{(k)}_{\lambda}$.
Consequently, the nesting argument behind \Cref{prop:multiround_monotone_main} breaks down and
monotonicity can fail; see \Cref{fig:multiround_monotonicity}.

\subsection{Self-Distillation with Fresh Unlabeled Features}
\label{sec:new-features}

Our structural results in \Cref{sec:structural} rely on the student refit using the \emph{same} design matrix $X$ as the teacher.
In particular, the same-$X$ mixed-loss SD predictor \eqref{eq:sd_predictor} reduces to an affine path
(\Cref{eq:affine_path_pred}) in $\xi$, so the risk $R_{\sd}(\lambda,\xi)$ is a quadratic function of $\xi$ and admits
closed-form characterizations (\Cref{prop:closed_form_xi_R1,thm:tangent_sign}).

A common pseudo-labeling variant instead refits using additional \emph{fresh} unlabeled covariates; see \Cref{fig:illustration-freshX}.
Let $\tilde X\in\RR^{m\times p}$ be independent of $\cD=(X,y)$, and define the teacher pseudo-labels on $\tilde X$ by $\tilde y_\lambda := f_\lambda(\tilde X)\in\RR^m$,
where $f_\lambda$ is the ridge teacher trained on $\cD$ as in \eqref{eq:ridge_predictor}.
In direct analogy with \eqref{eq:sd_predictor}, one may define the fresh-$X$ \emph{mixed-loss} SD predictor $f_{\sd,\lambda,\xi}^{\freshmix}(x)$ as
\begin{equation}
\label{eq:mixed_loss}
  x^\top\argmin_{\beta\in\RR^p}
  \big\{
  (1-\xi)\|y-X\beta\|_2^2/n
  + \xi\|\tilde y_\lambda-\tilde X\beta\|_2^2/m
  + \lambda\|\beta\|_2^2
  \big\}.
\end{equation}
Unlike the same-$X$ case, \eqref{eq:mixed_loss} does not reduce to ridge regression on a single mixed-label vector:
because the two quadratic losses involve different Gram matrices, the map
$\xi\mapsto f_{\sd,\lambda,\xi}^{\freshmix}$ is generally \emph{not} an affine path, and the risk $R_{\sd}^{\freshmix}(\lambda,\xi) := R(f_{\sd,\lambda,\xi}^{\freshmix})$
is no longer quadratic in $\xi$ (see \Cref{lem:fresh_mixedloss_matrix_weighted}).%
\footnote{When $\xi\notin[0,1]$, the objective in \eqref{eq:mixed_loss} need not be convex because it mixes two
different quadratic forms. Throughout, we interpret \eqref{eq:mixed_loss} only for values of $\xi$ such that the
Hessian in $\beta$ is positive definite, i.e., $(1-\xi)\,X^\top X/n + \xi\,\tilde X^\top\tilde X/m + \lambda I_p \succ 0$,
so the displayed $\argmin$ is well-defined.}
Consequently, the two-predictor identities for $\xi^\star(\lambda)$ and $R_{\sd}^\star(\lambda)$ in \Cref{prop:closed_form_xi_R1} and the derivative identities in \Cref{thm:tangent_sign} do not apply directly in this
fresh-$X$ mixed-loss setting.

\begin{figure}[!t]
    \centering
    \includegraphics[width=0.99\textwidth]{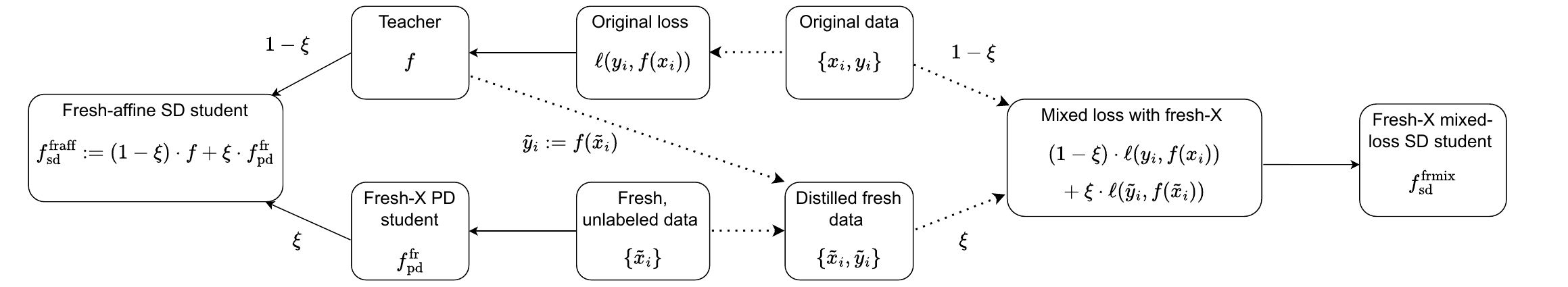}
    \caption{Visual illustration of self-distillation with fresh unlabeled features.}
    \label{fig:illustration-freshX}
\end{figure}

To retain an affine structure, one may instead define a fresh-$X$ PD refit and then mix predictors explicitly.
Let $f_{\pd,\lambda}^{\fresh}$ denote ridge regression trained on the pseudo-labeled dataset $(\tilde X,\tilde y_\lambda)$
with the same penalty $\lambda$.
Define the fresh-$X$ \emph{affine} SD family
\begin{equation}
\label{eq:fresh_affine_def}
  f_{\sd,\lambda,\xi}^{\freshaffine}(x)
  :=(1-\xi)\,f_\lambda(x)+\xi\,f_{\pd,\lambda}^{\fresh}(x).
\end{equation}
Because \eqref{eq:fresh_affine_def} is an explicit two-predictor affine path, the corresponding risk $R_{\sd}^{\freshaffine}(\lambda,\xi):=R(f_{\sd,\lambda,\xi}^{\freshaffine})$ is quadratic in $\xi$.
Thus, the identities in \Cref{prop:closed_form_xi_R1} apply directly (with $f_{\pd,\lambda}$ replaced by $f_{\pd,\lambda}^{\fresh}$) to characterize the optimal mixing $\xi_{\sd}^{\star,\freshaffine}(\lambda)\in\argmin_{\xi\in\RR}R_{\sd}^{\freshaffine}(\lambda,\xi)$ and the optimal risk $R_{\sd}^{\star,\freshaffine}(\lambda):=\min_{\xi\in\RR}R_{\sd}^{\freshaffine}(\lambda,\xi)$.
However, the same-$X$ coupling behind the derivative-based characterization \Cref{thm:tangent_sign} is absent here, so pointwise strict improvements are no longer structurally guaranteed.

Empirically, when the refit sample size matches the teacher sample size (i.e., $m=n$), the same-$X$ SD predictor tends to
dominate both the fresh-$X$ mixed-loss student \eqref{eq:mixed_loss} as well as the fresh-$X$ affine student \eqref{eq:fresh_affine_def} across $\lambda$; see \Cref{fig:same_x_fresh_x}.
Motivated by this behavior, we provide a prototype theoretical comparison in an isotropic in-distribution setting.
Specializing the asymptotic analysis of \Cref{sec:asymptotics} for same-$X$ (\Cref{lem:sameX_isotropic_limits}) and deriving
a corresponding asymptotic analysis for affine fresh-$X$ (\Cref{lem:freshX_isotropic_limits}) under isotropic covariance and
isotropic signal, we show that when $p/n\to\gamma$ and $p/m\to\gamma$, the same-$X$ optimal SD risk uniformly dominates the
fresh-affine optimal SD risk.
A full analysis under general covariance and signal structure, and for the mixed-loss fresh-$X$ student \eqref{eq:mixed_loss},
is left for future work.

\begin{figure}[!t]
    \centering
\includegraphics[width=0.48\textwidth]{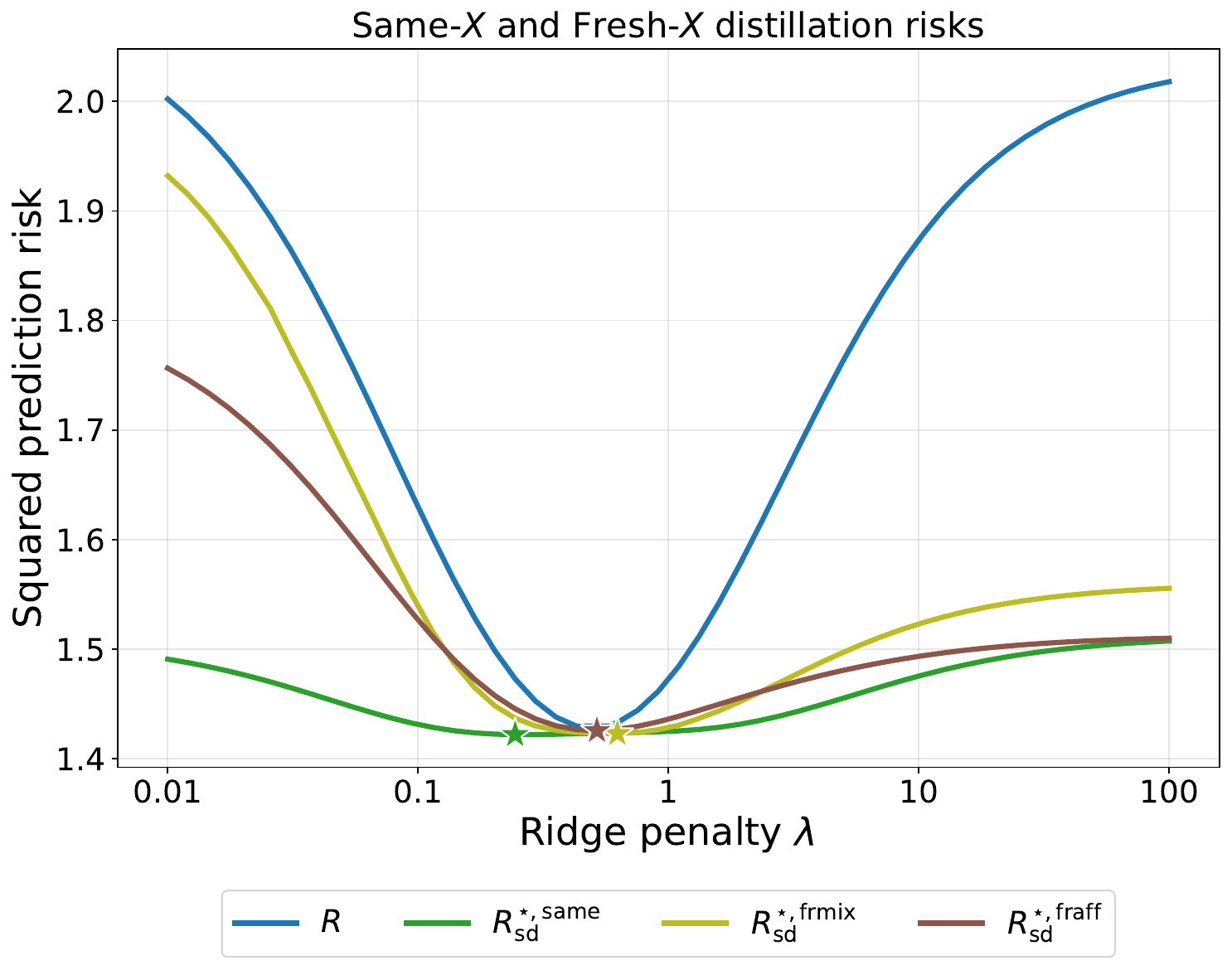}
\includegraphics[width=0.48\textwidth]{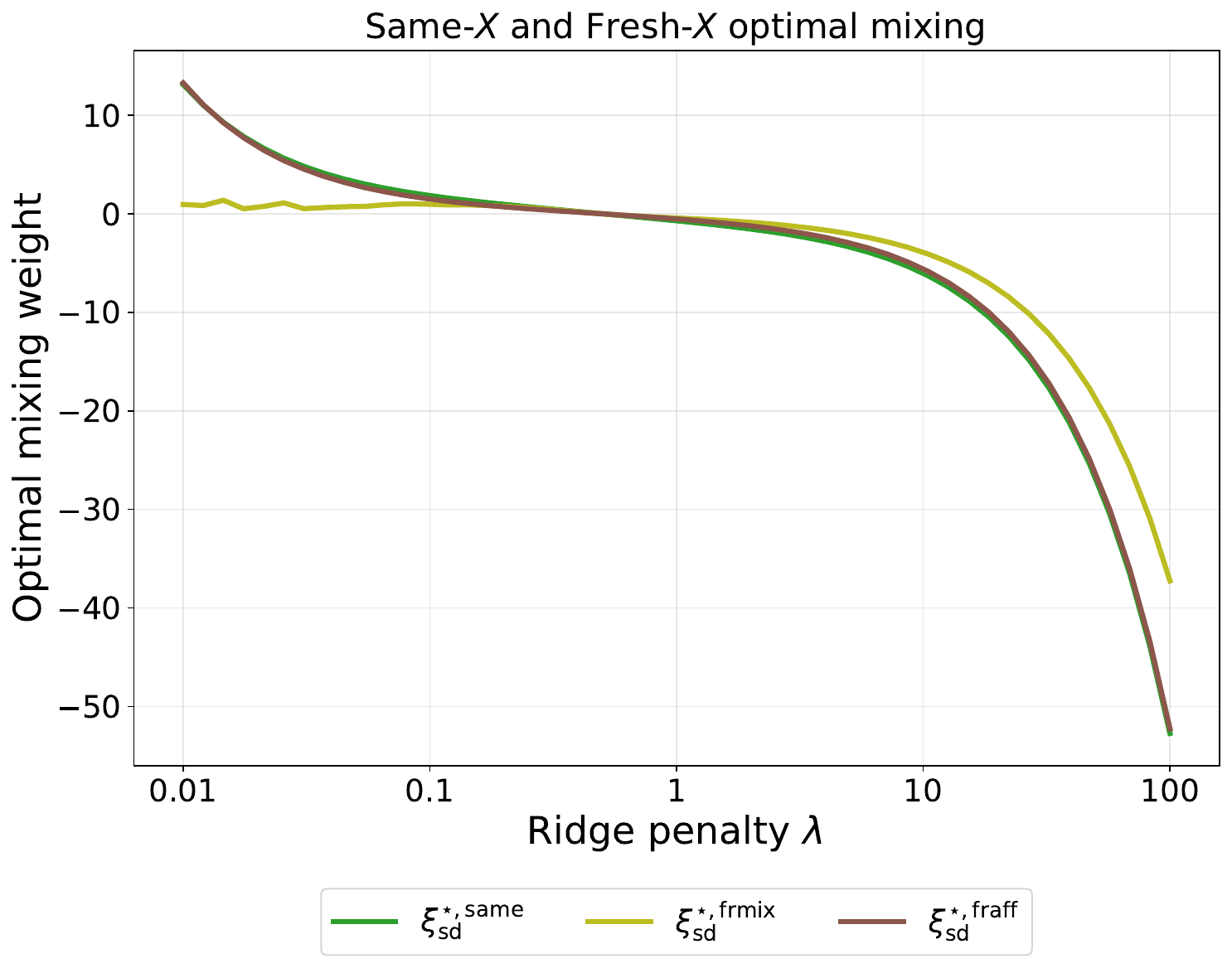}
    \caption{\textbf{Same-$X$ versus fresh-$X$ self-distillation risks.}
    Empirical test risks and optimal mixing weights (averaged over 30 simulations) in an isotropic setting with $n=200$,
    $p=100$, $r^2=1$, and $\sigma^2=1$. For each $\lambda$, we choose the fresh-$X$ mixing weight $\xi$ by grid search over
    $3{,}000$ values in $[-100,100]$, restricting to values for which the mixed-loss objective \eqref{eq:mixed_loss} is strictly
    convex, and picking the one with the lowest empirical test risk.}
    \label{fig:same_x_fresh_x}
\end{figure}

\begin{theorem}[Same-$X$ dominates fresh-$X$]
\label{thm:freshX_dominated_by_sameX_isotropic}
Under \Cref{def:dist} with $\Sigma=I_p$ and $\beta\sim\mathcal{N}(0,(r^2/p)I_p)$, as $m,n,p\to\infty$ with $p/m, p/n\to\gamma\in(0,\infty)$, for every fixed $\lambda>0$, we have
\[
    \sR^{\star,\same}_{\sd}(\lambda)
    \;\le\;
    \sR^{\star,\freshaffine}_{\sd}(\lambda),
\]
where $\sR^{\star,\same}_{\sd}(\lambda)$ and $\sR^{\star,\freshaffine}_{\sd}(\lambda)$ denote the limiting (deterministic) optimal SD risks for the same-$X$ SD predictor \eqref{eq:affine_path_pred} and the affine fresh-$X$ SD predictor \eqref{eq:fresh_affine_def}, respectively
\end{theorem}

This result is somewhat counterintuitive as one might expect that additional unlabeled data would help self-distillation. 
It highlights that the gains characterized in the main paper are
intimately tied to self-distillation on the same data. 
Understanding how the comparison changes when the amount
of unlabeled data grows (e.g., $m\gg n$) is an interesting direction for future work.

\subsection{Self-Distillation with Other Ridge Variants}
\label{sec:ridge-variants}

The structural results of \Cref{sec:structural} extend beyond ordinary ridge regression to a class of (ridge) resolvent-based smoothers.
At a high level, two properties drive the extension:
(i) the teacher is a \emph{linear smoother} in the labels, and
(ii) the student refit applies the same \emph{resolvent-based} smoothing family (at the same regularization level~$\lambda$).
Under (i), the SD family is again an affine path in the mixing weight~$\xi$, so the two-predictor identities
(e.g., \Cref{prop:closed_form_xi_R1}) apply under any squared prediction risk.
Under (ii), the teacher--PD gap admits a (surprising) derivative representation, resulting in strict improvement
and sign rule properties analogous to \Cref{thm:tangent_sign}.
We state the generic result first, then highlight two representative examples.

Fix training inputs $X=(x_1^\top,\dots,x_n^\top)^\top$ and labels $y\in\RR^n$.
Let $\{f_\lambda\}_{\lambda>0}$ be a teacher predictor family such that for each $\lambda$ there exists a vector-valued map
$s_\lambda(\cdot)\in\RR^n$ (depending on $X,\lambda$ but not on $y$) satisfying $f_\lambda(x)=s_\lambda(x)^\top y$ and $\hat y_\lambda:=f_\lambda(X)=S_\lambda y$ for some (smoothing) matrix $S_\lambda\in\RR^{n\times n}$ depending on $X,\lambda$ but not on $y$.
Define the (same-$\lambda$) PD student refit by reapplying the same smoother to the teacher fitted values: $f_{\pd,\lambda}(x):=s_\lambda(x)^\top \hat y_\lambda$.
Define the SD predictor by the same mixed-label construction as in \eqref{eq:sd_predictor} (now with teacher pseudo-labels
$\hat y_\lambda$), which is equivalent to applying the smoother to the mixed labels
$y{(\xi)}:=(1-\xi)y+\xi\hat y_\lambda$. 
Hence, the SD predictor can be expressed as (see \Cref{app:ridge_variants_affine_path}):
\begin{equation}
\label{eq:ridge_variant_affine_path}
  f_{\sd,\lambda,\xi}(x)
  =s_\lambda(x)^\top y{(\xi)}
  =(1-\xi)f_\lambda(x)+\xi f_{\pd,\lambda}(x).
\end{equation}

For any squared prediction risk $R(\cdot)$ as in \eqref{eq:pred-risk}, write $R(\lambda):=R(f_\lambda)$, $R_{\pd}(\lambda):=R(f_{\pd,\lambda})$, and define:
\[
  C(\lambda):=\EE\!\big[(y_0-f_\lambda(x_0))(y_0-f_{\pd,\lambda}(x_0))\mid \cD\big],
  \quad
  D(\lambda):=\EE\!\big[(f_\lambda(x_0)-f_{\pd,\lambda}(x_0))^2\mid \cD\big].
\]
(As in \Cref{sec:structural}, it is easy to check that $D(\lambda) = R(\lambda) + R_{\pd}(\lambda) - 2 C(\lambda)$.)
Assume in addition that the ridge-smoother family satisfies the derivative (tangent) identity
\begin{equation}
\label{eq:tangent_identity_ridge_variants_main}
  f_\lambda(x)-f_{\pd,\lambda}(x)
  =-\lambda\,\partial_\lambda f_\lambda(x)
  \qquad\text{for all }x,
\end{equation}
and that $\lambda\mapsto R(\lambda)$ is differentiable.
Then the derivative-based characterization from \Cref{thm:tangent_sign} extends as follows.

\begin{theorem}[Ridge-smoother strict improvement and sign rule]
\label{thm:ridge_variants_tangent_sign}
Fix $\lambda>0$ and assume $D(\lambda)>0$.
Under \eqref{eq:tangent_identity_ridge_variants_main}, the optimal SD mixing weight and risk along the regularization path satisfy
\[
  \xi^\star(\lambda)
  =-\frac{\lambda}{2}\,\frac{R'(\lambda)}{D(\lambda)},
  \qquad
  R_{\sd}^\star(\lambda)
  =
  R(\lambda)-\frac{\lambda^2}{4}\,\frac{(R'(\lambda))^2}{D(\lambda)}.
\]
In particular, if $R'(\lambda)\neq 0$, then $R_{\sd}^\star(\lambda)<R(\lambda)$ and
$\sign(\xi^\star(\lambda))=-\sign(R'(\lambda))$.
\end{theorem}

The derivative identity \eqref{eq:tangent_identity_ridge_variants_main} surprisingly holds for many resolvent-based ridge estimators,
including generalized or Tikhonov ridge (\Cref{lem:grr_tangent_identity}) and kernel ridge regression (\Cref{lem:krr_tangent_identity}):

\begin{itemize}
\item \emph{Generalized ridge.}
For a fixed penalty operator $\Omega\succ 0$ (independent of $y$), generalized ridge is a linear smoother with
\[
  s_\lambda^\Omega(x)=X\,(X^\top X+n\lambda\,\Omega)^{-1}x,
  \qquad
  S_\lambda^\Omega = X\,(X^\top X+n\lambda\,\Omega)^{-1}X^\top.
\]
This includes feature-wise shrinkage (diagonal $\Omega$), graph or Laplacian regularization, spline-type penalties, and other
quadratic Tikhonov penalties.

\item \emph{Kernel ridge.}
For a PSD kernel with kernel matrix $K\in\RR^{n\times n}$ and $k_x:=(k(x,x_1),\dots,k(x,x_n))^\top$,
kernel ridge is a linear smoother with
\[
  s_\lambda^{\mathrm{kern}}(x)=(K+n\lambda I_n)^{-1}k_x,
  \qquad
  S_\lambda^{\mathrm{kern}}=K\,(K+n\lambda I_n)^{-1}.
\]
This covers standard kernel ridge estimators (e.g., Gaussian and polynomial kernels) and their variants obtained by
changing $k$.
See \Cref{fig:kernel-air-quality} for empirical illustrations with a Gaussian kernel.
\end{itemize}

Both examples can be shown to satisfy \eqref{eq:tangent_identity_ridge_variants_main} (see \Cref{app:ridge_variants_verify_tangent}), hence inherit
\Cref{thm:ridge_variants_tangent_sign} under any squared prediction risk (including out-of-distribution).

\begin{figure}[!t]
  \centering
  \includegraphics[width=0.48\textwidth]{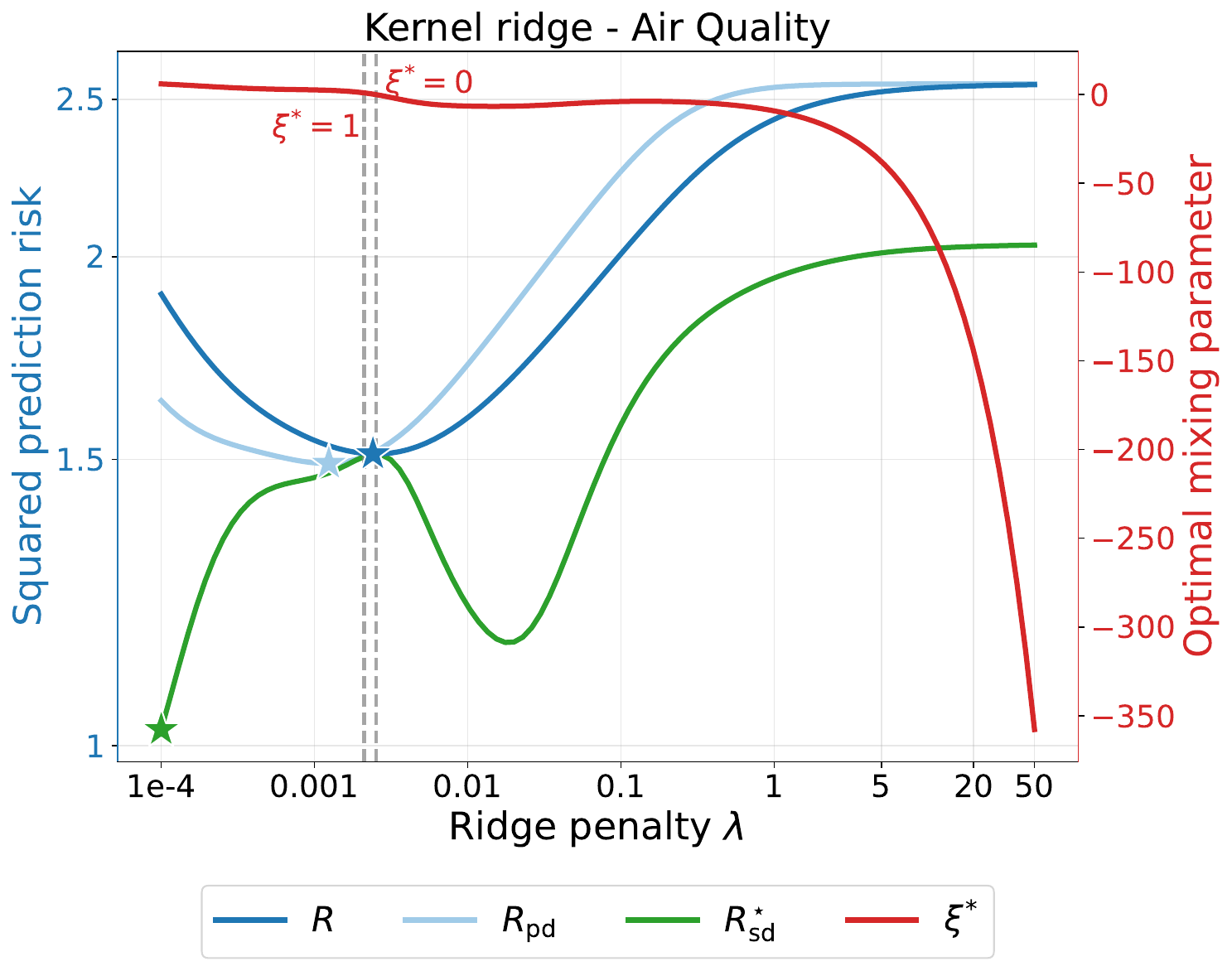}
  \includegraphics[width=0.48\textwidth]{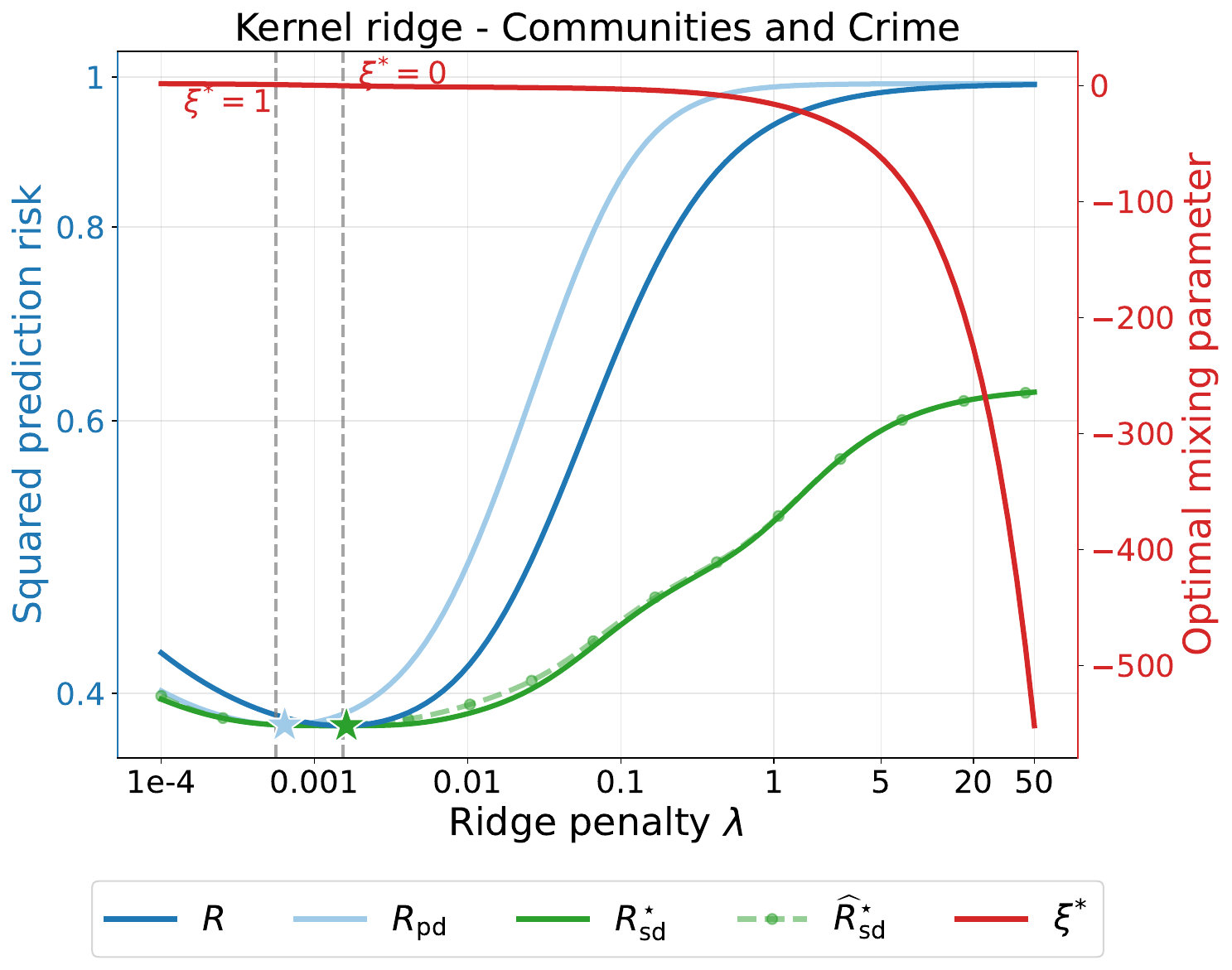}
  \caption{\textbf{Kernel ridge and kernel SD ridge regression.} Test risks of Gaussian kernel ridge and kernel SD ridge on two datasets.}
  \label{fig:kernel-air-quality}
\end{figure}

\section{Conclusion}

In this paper, we asked whether self-distillation can provably deliver strict gains (\Qs{1}), when it can rival an optimally tuned teacher (\Qs{2}), and how to tune it efficiently (\Qs{3}).
For ridge regression, and more broadly, for a class of resolvent-based ridge smoothers, we provide affirmative and sharp answers.

On the theory side, we give nonasymptotic structural identities that hold conditionally on the observed training data and for \emph{any} squared prediction risk, including out-of-distribution risks.
Under a mild nondegeneracy condition, we show that the optimally mixed student \emph{strictly} improves the teacher at every nonstationary $\lambda$ along the ridge path, and the optimal mixing weight obeys a simple sign rule: it has the \emph{opposite} sign of the ridge-risk derivative, so it can be negative in over-regularized regimes.
We then refine these identities in the proportional regime $p/n\to\gamma\in(0,\infty)$ by deriving deterministic equivalents for the optimal SD risk under anisotropic covariance and deterministic signals, which quantify how SD gains depend on $(\gamma,\SNR)$ and signal--covariance alignment.
On the algorithmic side, we propose a consistent one-shot GCV-based tuning method that avoids grid search, repeated retraining, and data splitting, while matching predicted risk curves quite well in several real-data experiments.

Taken together, our results suggest a simple conceptual message: at least in ridge-type problems, optimal self-distillation acts as a cheap and tractable correction for a mis-regularized teacher using the teacher’s own fitted values to move \emph{along} the regularization path in a direction determined by the risk slope.
An important direction is to understand how much of this regularization path geometry persists beyond ridge regression, including classification and more general learners, and to develop equally sharp guarantees for fresh-$X$ and other practically motivated distillation variants.

\section*{Acknowledgments}
We are grateful to Sujay Sanghavi for introducing us to various surprises in self-distillation and for several helpful discussions.
We also thank Adel Javanmard for an insightful seminar at the Institute for Foundations of Machine Learning (IFML) at UT Austin and the productive ensuing discussions.
Computing support is in part provided by the Texas Advanced Computing Center (TACC).

\bibliographystyle{plainnat}
\bibliography{references}

\newpage
\appendix

\newgeometry{left=0.5in,top=0.25in,right=0.5in,bottom=0.25in,head=.1in,foot=0.1in}

\begin{center}
\Large
{\bf \framebox{Supplement}}
\end{center}

\bigskip

This supplement serves as a companion to the paper titled ``Optimal Unconstrained Self-Distillation in Ridge Regression: Strict Improvements, Precise Asymptotics, and One-Shot Tuning''.
We provide an outline of the supplement in \Cref{tab:outline-supplement} and a summary of general notation used throughout the paper in \Cref{tab:notation}.

\section*{Organization}
\label{sec:organization_notation_supp}

\begin{table}[!ht]
\centering
\caption{Outline of the supplement.}
\label{tab:outline-supplement}
\begin{tabularx}{\textwidth}{L{1.75cm}L{2cm}L{20cm}}
\toprule
\textbf{Section} & \textbf{Subsection} & \textbf{Purpose} \\
\midrule
\multicolumn{2}{c}{Proofs in \Cref{sec:structural}} \\ \arrayrulecolor{black!20} \cmidrule(lr){1-2}
\addlinespace[0.5ex]
\multirow{4}{*}{\Cref{app:structural_proofs}} 
& \Cref{app:ridge_prelims} & Preliminaries \\
& \Cref{app:affine_identity} & Proof of \Cref{eq:affine_path_pred} \\
& \Cref{app:proof_prop_closed_form} & Proof of \Cref{prop:closed_form_xi_R1} \\
& \Cref{app:proof_tangent_sign} & Proof of \Cref{thm:tangent_sign} \\
& \Cref{app:proof_curvature_test} & Proof of \Cref{prop:curvature_test_main} \\
\arrayrulecolor{black!50}
\midrule 
\multicolumn{2}{c}{Proofs in \Cref{sec:asymptotics}} \\ \arrayrulecolor{black!20} \cmidrule(lr){1-2}
\addlinespace[0.5ex]
\multirow{1}{*}{\Cref{app:asymptotics_proofs}}
& \Cref{app:asymptotics_proofs-preliminaries} & Preliminaries \\
& \Cref{app:asymptotics_proofs-concentration-outline} & Helper results (concentration components) for the proof of \Cref{thm:risk-asymptotics} \\
& \Cref{app:asymptotics_proofs-equivalents-outline} & Helper results (deterministic equivalents) for the proof of \Cref{thm:risk-asymptotics} \\
& \Cref{proof:thm:risk-asymptotics} & Proof of \Cref{thm:risk-asymptotics} \\
& \Cref{proof:lem:response_concentration} & Proofs of helper results (concentration components) for the proof of \Cref{thm:risk-asymptotics} \\
& \Cref{app:DE_Q_U} & Proofs of helper results (deterministic equivalents) for the proof of \Cref{thm:risk-asymptotics} \\
& \Cref{app:thm:risk-asymptotics-isotropic-signal} & Proof of \Cref{cor:xi_asymptotic} \\
& \Cref{proof:prop:compare_extreme_lambda} & Proof of \Cref{prop:compare_extreme_lambda} \\
\arrayrulecolor{black!50}
\midrule 
\multicolumn{2}{c}{Proofs in \Cref{sec:tuning}} \\ \arrayrulecolor{black!20} \cmidrule(lr){1-2}
\addlinespace[0.5ex]
\multirow{4}{*}{\Cref{app:tuning_proofs}}
& \Cref{app:tuning_setup} & Preliminaries \\
& \Cref{app:tuning_pd_risk_corr} & Helper results (risk and correlation components consistency) for the proof of \Cref{thm:oneshot_consistency} \\
& \Cref{app:tuning_thm_proof} & Proof of \Cref{thm:oneshot_consistency} \\
& \Cref{app:tuning_pd_tech} & Technical lemmas \\
\arrayrulecolor{black!50}
\midrule 
\multicolumn{2}{c}{Proofs in \Cref{sec:extensions}} \\ \arrayrulecolor{black!20} \cmidrule(lr){1-2}
\addlinespace[0.5ex]
\multirow{3}{*}{\Cref{app:extensions_proofs}} 
& \Cref{app:multiround_proofs} & Proof of \Cref{prop:multiround_monotone_main} and other details in \Cref{sec:multiple-rounds} \\
& \Cref{sec:supp_freshX_isotropic} & Proof of \Cref{thm:freshX_dominated_by_sameX_isotropic} and other details in \Cref{sec:new-features} \\
& \Cref{sec:supp_extensions_ridge_variants} & Proof of \Cref{thm:ridge_variants_tangent_sign} and other details in \Cref{sec:ridge-variants} \\
\arrayrulecolor{black!50}
\midrule 
\multicolumn{2}{c}{Additional experiments} \\ \arrayrulecolor{black!20} \cmidrule(lr){1-2}
\addlinespace[0.5ex]
\multirow{4}{*}{\Cref{sec:additional_exps}}
& \Cref{app:additional-experiments-cifar} & Additional experiments on CIFAR datasets \\
& \Cref{app:vs_multi_round} & Illustrations for related works comparison in \Cref{sec:related_work} \\
& \Cref{app:additional-illustrations-real-world} & Additional illustrations in \Cref{sec:structural} \\
& \Cref{app:additional-illustrations-propasymp} & Additional illustrations in \Cref{sec:asymptotics_main} \\
& \Cref{app:additional-illustrations-extremereg} & Additional illustrations in \Cref{sec:asymptotics_gains} \\
\arrayrulecolor{black!50}
\midrule 
\multicolumn{2}{c}{Experimental details} \\ 
\arrayrulecolor{black!20} \cmidrule(lr){1-2}
\addlinespace[0.5ex]
\multirow{3}{*}{\Cref{sec:additional_details}}
& \Cref{sec:real-world-regression-tasks} & Real-world regression tasks \\
& \Cref{sec:resnet-cifar-details} & Pretrained ResNet features on CIFAR datasets \\
& \Cref{sec:synthetic_details} & Synthetic asymptotic experiments 
\\
\arrayrulecolor{black}
\bottomrule
\end{tabularx}
\end{table}

\clearpage
\section*{Notation}
\label{sec:notation}

\begin{table}[!ht]
\centering
\caption{Summary of general notation used throughout the paper.}
\label{tab:notation}
\begin{tabularx}{\textwidth}{L{5cm}L{15cm}} %
\toprule
\textbf{Notation} & \textbf{Description} \\
\midrule

\multicolumn{1}{c}{Typography} & \\ \arrayrulecolor{black!20}
\cmidrule(lr){1-1} \addlinespace[0.25ex]
\arrayrulecolor{black}
Lowercase (e.g., $x$) & Scalars or vectors \\
Uppercase (e.g., $X$) & Matrices or linear operators \\
Calligraphic (e.g., $\cD$, $\cR$) & Sets, $\sigma$-fields, events or certain limiting functions \\
Blackboard bold (e.g., $\RR$, $\NN$) & Standard number systems \\
\arrayrulecolor{black!50}\midrule

\multicolumn{1}{c}{Analysis} & \\ \arrayrulecolor{black!20} 
\cmidrule(lr){1-1} \addlinespace[0.25ex]
$\ZZ$, $\NN$, $\RR$, $\RR_{\ge 0}$, $\overline{\RR}$ & Integers, positive integers, reals, nonnegative reals, extended reals \\
$(a,b,c)$, $\{a,b,c\}$ & Ordered tuple and (unordered) set \\
$[n]$ & Set $\{1, \dots, n\}$ for a positive integer $n$ \\
$x \wedge y$, $x \vee y$ & $\min\{x, y\}$ and $\max\{x, y\}$ for real numbers $x, y$ \\
$\ind_A$ & Indicator random variable associated with event or set $A$ \\
$\sign(x)$ & Sign of a real number $x$ \\
$C^\infty$ & Function class of infinitely differentiable functions \\
\arrayrulecolor{black!50}\midrule

\multicolumn{1}{c}{Linear algebra} & \\ \arrayrulecolor{black!20}
\cmidrule(lr){1-1} \addlinespace[0.25ex]
\arrayrulecolor{black}
$\tr[\bA]$, $\otr[\bA]$, $\det(\bA)$ & Trace, normalized trace ($\tr[\bA] / p$), and determinant of a square matrix $\bA \in \RR^{p \times p}$ \\
$\bB^{-1}$ & Inverse of an invertible square matrix $\bB$ \\
$\bC^{\dagger}$ & Moore-Penrose inverse of a general rectangular matrix $C$ \\
$\diag(d_1,\dots,d_p)$ & Diagonal matrix with diagonal entries $d_1,\dots,d_p$ \\
$U^{1/2}$ & Principal square root of a positive semidefinite matrix $U\succeq 0$ \\
$f(W)$ & Functional calculus for positive semidefinite matrix $W$ (apply $f$ to eigenvalues of $W$) \\
$\bI$, $\bm{1}$, $\bm{0}$ & The identity matrix, the all-ones vector, the all-zeros vector \\
\arrayrulecolor{black!50}\midrule

\multicolumn{1}{c}{Inner products and norms} & \\ \arrayrulecolor{black!20}
\cmidrule(lr){1-1} \addlinespace[0.25ex]
\arrayrulecolor{black}
$\langle u,v\rangle$ & Euclidean inner product (or another inner product when specified) \\
$\|x\|_2$ (or simply $\|x\|$) & Euclidean norm of a vector $x$ \\
$\|x\|_q$ & $\ell_q$ norm of a vector ($q\ge 1$) \\
$\|x\|_A:=\sqrt{x^\top A x}$ & $A$-seminorm for $A\succeq 0$ \\
$\|A\|_{\oper}$ & Operator/spectral norm of a matrix $A$ \\
$\|A\|_{F}$ & Frobenius norm of a matrix $A$ \\
$\|A\|_{\tr}$ & Trace/nuclear norm (sum of singular values) \\
$\|f\|_{L_q}$ & $L_q$ norm of a function $f$ under the relevant measure ($q\ge 1$) \\
\arrayrulecolor{black!50}\midrule

\multicolumn{1}{c}{Probability} & \\ \arrayrulecolor{black!20}
\cmidrule(lr){1-1} \addlinespace[0.25ex]
\arrayrulecolor{black}
$\PP(\cdot)$, $\EE[\cdot]$ & Probability and expectation \\
$\EE[\cdot\mid \cG]$ & Conditional expectation given $\sigma$-field $\cG$ (or given data, depending on context) \\
$\Var(\cdot)$, $\Cov(\cdot,\cdot)$ & Variance and covariance \\
$X\sim \cN(\mu,\Sigma)$ & Gaussian random vector with mean $\mu$ and covariance $\Sigma$ \\
$\stackrel{\mathrm{d}}{=}$ & Equality in distribution \\
\arrayrulecolor{black!50}\midrule

\multicolumn{1}{c}{Orders and asymptotics} & \\ \arrayrulecolor{black!20}
\cmidrule(lr){1-1} \addlinespace[0.25ex]
\arrayrulecolor{black}
$X=\cO_\alpha(Y)$, $X\lesssim_\alpha Y$ & Deterministic upper bounds with constant possibly depending on parameter $\alpha$ \\
$o(\cdot)$, $\cO(\cdot)$ & Deterministic little-$o$ and big-$O$ \\
$\op(\cdot)$, $\Op(\cdot)$ & Probabilistic little-$o$ and big-$O$ \\
$\to$, $\pto$, $\asto$, $\dto$ & Convergence, in probability, almost surely, in distribution \\
$\asymp$ & Asymptotic equivalence (see \Cref{def:strong_DE} for more details) \\
$C, C', c, c'$ & Generic positive constants (may change from line to line) \\

\arrayrulecolor{black}\bottomrule
\end{tabularx}
\end{table}

\restoregeometry

\clearpage
\section{Proofs in Section~\ref{sec:structural}}
\label{app:structural_proofs}

\subsection{Preliminaries}
\label{app:ridge_prelims}

Define the empirical covariance matrix and empirical resolvent as
\begin{equation}
\label{eq:app_emp_cov_resolvent}
\widehat\Sigma:=\frac1n X^\top X\in\RR^{p\times p},
\quad \text{and} \quad 
Q_\lambda:=(\widehat\Sigma+\lambda I_p)^{-1}\in\RR^{p\times p},
\end{equation}
respectively.
For any label vector $u\in\RR^n$, define the ridge coefficient map and predictor as
\begin{equation}
\label{eq:app_ridge_map_general_u}
\beta_\lambda(u):=Q_\lambda\,\frac{X^\top u}{n},
\quad \text{and} \quad 
f_{\lambda,u}(x):=x^\top\beta_\lambda(u),
\end{equation}
respectively.
In particular, the teacher predictor is $f_\lambda:=f_{\lambda,y}$ with fitted values $\hat y_\lambda:=f_\lambda(X)=X\beta_\lambda(y)\in\RR^n$.
Note that $f_{\lambda,u}$ is linear in $u$.

The PD predictor $f_{\pd,\lambda}$ is the ridge predictor trained on $(X,\hat y_\lambda)$ with the same penalty $\lambda$, i.e., $f_{\pd,\lambda}:=f_{\lambda,\hat y_\lambda}$.
Its coefficient vector admits the closed form
\begin{equation}
\label{eq:app_beta_pd_closed_form}
\beta_{\pd,\lambda}
:=\beta_\lambda(\hat y_\lambda)
=Q_\lambda\,\frac{X^\top(X\beta_\lambda(y))}{n}
=Q_\lambda\,\widehat\Sigma\,\beta_\lambda(y).
\end{equation}
Since $\widehat\Sigma$ and $Q_\lambda$ commute (both are polynomials in $\widehat\Sigma$), we have
$\widehat\Sigma Q_\lambda=I_p-\lambda Q_\lambda$, and therefore
\begin{equation}
\label{eq:app_beta_diff_identity}
\beta_\lambda(y)-\beta_{\pd,\lambda}
=\big(I_p-Q_\lambda\widehat\Sigma\big)\beta_\lambda(y)
=\lambda Q_\lambda\,\beta_\lambda(y).
\end{equation}

\subsection{Proof of \Cref{eq:affine_path_pred}}
\label{app:affine_identity}

\begin{lemma}[SD ridge representation]
\label{lem:sd_is_mixed_ridge_app}
Fix $\lambda>0$ and $\xi\in\RR$.
Let $\hat y_\lambda=f_\lambda(X)$ and define the mixed label vector
\begin{equation}
\label{eq:app_mixed_labels}
y^{(1)}(\xi):=(1-\xi)y+\xi\hat y_\lambda\in\RR^n.
\end{equation}
Then the SD predictor \eqref{eq:sd_predictor} coincides with ridge regression trained on $(X,y^{(1)}(\xi))$ with penalty $\lambda$, i.e.,
$f_{\sd,\lambda,\xi}=f_{\lambda,\,y^{(1)}(\xi)}$.
\end{lemma}

\begin{proof}
Expanding the SD objective (and dropping terms independent of $\beta$), we obtain that
\begin{align*}
&(1-\xi)\frac1n\|y-X\beta\|_2^2+\xi\frac1n\|\hat y_\lambda-X\beta\|_2^2+\lambda\|\beta\|_2^2 \\
&=
\frac1n\|X\beta\|_2^2
-\frac{2}{n}\beta^\top X^\top\big((1-\xi)y+\xi\hat y_\lambda\big)
+\lambda\|\beta\|_2^2+\text{const},
\end{align*}
where $\text{const}$ is a term that does not depend on $\beta$. 
This is exactly the ridge objective with response $y^{(1)}(\xi)$ as defined in \eqref{eq:app_mixed_labels}.
\end{proof}

\begin{proof}[Proof of \Cref{eq:affine_path_pred}]
By Lemma~\ref{lem:sd_is_mixed_ridge_app}, $f_{\sd,\lambda,\xi}=f_{\lambda,\,(1-\xi)y+\xi\hat y_\lambda}$.
Using linearity of $f_{\lambda,u}$ in $u$ (cf.\ \eqref{eq:app_ridge_map_general_u}),
\[
f_{\sd,\lambda,\xi}
=(1-\xi)f_{\lambda,y}+\xi f_{\lambda,\hat y_\lambda}
=(1-\xi)f_\lambda+\xi f_{\pd,\lambda},
\]
which is \eqref{eq:affine_path_pred}.
\end{proof}

\subsection{Proof of Proposition~\ref{prop:closed_form_xi_R1}}
\label{app:proof_prop_closed_form}

Throughout, $\lambda > 0$ is fixed and we abbreviate $f:=f_\lambda$ and $g:=f_{\pd,\lambda}$.
Write
\[
  R:=R(f),\qquad R_{\pd}:=R(g),\qquad C:=\EE[(y_0-f(x_0))(y_0-g(x_0))\mid\cD],
\]
and define the nonnegative quantity
y
\[
  D:=R+R_{\pd}-2C=\EE[(f(x_0)-g(x_0))^2\mid\cD]\ge 0.
\]

\begin{proof}[Proof of Proposition~\ref{prop:closed_form_xi_R1}]
By \eqref{eq:affine_path_pred}, the SD predictor with mixing weight $\xi$ is $f_{\sd}(\xi)=(1-\xi)f+\xi g$.
Let the residuals be
$r_f:=y_0-f(x_0)$ and $r_g:=y_0-g(x_0)$.
Then
\[
  y_0-f_{\sd}(\xi)(x_0)=(1-\xi)r_f+\xi r_g,
\]
and thus
\begin{align*}
R_{\sd}(\lambda,\xi)
&=\EE[((1-\xi)r_f+\xi r_g)^2\mid\cD] \\
&=(1-\xi)^2\EE[r_f^2\mid\cD]+\xi^2\EE[r_g^2\mid\cD]+2\xi(1-\xi)\EE[r_f r_g\mid\cD] \\
&=(1-\xi)^2R+\xi^2R_{\pd}+2\xi(1-\xi)C \\
&=R-2\xi(R-C)+\xi^2D.
\end{align*}
If $D>0$, this is a strictly convex quadratic in $\xi$ with a unique minimizer given by
\[
0=\partial_\xi R_{\sd}(\lambda,\xi)=-2(R-C)+2\xi D
\quad\Longrightarrow\quad
\xi^\star=\frac{R-C}{D}.
\]
Substituting back yields
\[
R_{\sd}^\star
=R-\frac{(R-C)^2}{D}.
\]
If $D=0$, then $f(x_0)=g(x_0)$ a.s.\ under the test distribution, hence
$R_{\sd}(\lambda,\xi)\equiv R$ for all $\xi$ and every $\xi$ is optimal with $R_{\sd}^\star=R$.
\end{proof}

\subsection{Proof of Theorem~\ref{thm:tangent_sign}}
\label{app:proof_tangent_sign}

We first establish a derivative identity relating the PD predictor to the derivative of the teacher along the ridge path.

\begin{lemma}[Resolvent derivative rule]
\label{lem:resolvent_derivative_app}
For a fixed $\lambda>0$ and symmetric $A\succeq 0$, let $Q_\lambda:=(A+\lambda I)^{-1}$. Then,
\[
  \partial_\lambda Q_\lambda = -Q_\lambda^2.
\]
\end{lemma}

\begin{proof}
Differentiate $(A+\lambda I)Q_\lambda=I$ and left-multiply by $Q_\lambda$.
\end{proof}

\begin{lemma}[Derivative identity for ridge versus pure-distillation]
\label{lem:tangent_identity_app}
For any $\lambda>0$ and $x\in\RR^p$,
\begin{equation}
\label{eq:tangent_identity_main_app}
  f_\lambda(x)-f_{\pd,\lambda}(x)=-\lambda\,\partial_\lambda f_\lambda(x).
\end{equation}
\end{lemma}

\begin{proof}
From the definitions, $\beta_\lambda(y)=Q_\lambda X^\top y/n$ and $\beta_{\pd,\lambda}=Q_\lambda\widehat\Sigma\,\beta_\lambda(y)$; see \ \eqref{eq:app_ridge_map_general_u} and \eqref{eq:app_beta_pd_closed_form}.
By Lemma~\ref{lem:resolvent_derivative_app},
\[
  \partial_\lambda \beta_\lambda(y)
  =(\partial_\lambda Q_\lambda)\frac{X^\top y}{n}
  =-Q_\lambda^2\frac{X^\top y}{n}
  =-Q_\lambda\,\beta_\lambda(y).
\]
Thus $-\lambda\,\partial_\lambda \beta_\lambda(y)=\lambda Q_\lambda\beta_\lambda(y)$.
Using \eqref{eq:app_beta_diff_identity}, $\lambda Q_\lambda\beta_\lambda(y)=\beta_\lambda(y)-\beta_{\pd,\lambda}$.
Multiplying by $x^\top$ yields \eqref{eq:tangent_identity_main_app}.
\end{proof}

Next we relate $R(\lambda)-C(\lambda)$ to $R'(\lambda)$.
Before we proceed, we justify that $R'(\lambda)$ is indeed well-defined under our setup.

\begin{lemma}[Smoothness of ridge predictions and risk]
\label{lem:app_risk_smoothness}
Fix the training data $\cD=(X,y)$ and assume $\EE[y_0^2\mid\cD]<\infty$ and $\EE[\|x_0\|_2^2\mid\cD]<\infty$.
Then for every $\lambda>0$, the map $\lambda\mapsto f_\lambda(x_0)$ is $C^\infty$ and the risk
$R(\lambda)=\EE[(y_0-f_\lambda(x_0))^2\mid\cD]$ is $C^\infty$ on $(0,\infty)$.
Moreover,
\begin{equation}
\label{eq:app_Rprime_formula}
R'(\lambda)
=
-2\,\EE\!\big[(y_0-f_\lambda(x_0))\,\partial_\lambda f_\lambda(x_0)\mid\cD\big],
\end{equation}
where $\partial_\lambda f_\lambda(x)=-x^\top Q_\lambda \beta_\lambda(y)$ with $Q_\lambda=(\widehat\Sigma+\lambda I)^{-1}$.
\end{lemma}

\begin{proof}
Fix $\lambda_0>0$ and consider $\lambda\in[\lambda_0/2,3\lambda_0/2]$.
Then $\|Q_\lambda\|_{\text{op}}\le 2/\lambda_0$, and with $b:=X^\top y/n$ we have
$\beta_\lambda=Q_\lambda b$ and $\partial_\lambda \beta_\lambda=-Q_\lambda\beta_\lambda$.
Hence there exist constants (depending on $\cD$ and $\lambda_0$) such that uniformly over this interval,
\[
|f_\lambda(x_0)| \le c_1\|x_0\|_2,
\qquad
|\partial_\lambda f_\lambda(x_0)| \le c_2\|x_0\|_2.
\]
Therefore,
\[
\big|\partial_\lambda (y_0-f_\lambda(x_0))^2\big|
=2|y_0-f_\lambda(x_0)|\,|\partial_\lambda f_\lambda(x_0)|
\le c_3\big(|y_0|\|x_0\|_2+\|x_0\|_2^2\big),
\]
whose conditional expectation is finite by Cauchy--Schwarz and the assumed second-moment bounds.
Dominated convergence justifies differentiating under $\EE[\cdot\mid\cD]$, yielding \eqref{eq:app_Rprime_formula}.
Higher derivatives follow similarly since $f_\lambda(x_0)$ is a rational (hence analytic) function of $\lambda$ for $\lambda>0$.
\end{proof}

\begin{lemma}[Derivative identity for $R(\lambda)-C(\lambda)$]
\label{lem:derivative_identity_app_new}
For any $\lambda>0$, 
\begin{equation}
\label{eq:derivative_identity_main_app}
  R(\lambda)-C(\lambda)=-\frac{\lambda}{2}\,R'(\lambda),
\end{equation}
where $R(\lambda)=R(f_\lambda)$ and $C(\lambda)=\EE[(y_0-f_\lambda(x_0))(y_0-f_{\pd,\lambda}(x_0))\mid\cD]$.
\end{lemma}

\begin{proof}
Write $r_\lambda:=y_0-f_\lambda(x_0)$.
First note that
\[
R(\lambda)-C(\lambda)
=\EE[r_\lambda^2-r_\lambda(y_0-f_{\pd,\lambda}(x_0))\mid\cD]
=\EE[r_\lambda\,(f_{\pd,\lambda}(x_0)-f_\lambda(x_0))\mid\cD].
\]
Using Lemma~\ref{lem:tangent_identity_app},
$f_{\pd,\lambda}(x_0)-f_\lambda(x_0)=\lambda\,\partial_\lambda f_\lambda(x_0)$, hence
\begin{equation}
\label{eq:app_RminusC_as_inner}
R(\lambda)-C(\lambda)=\lambda\,\EE\!\big[r_\lambda\,\partial_\lambda f_\lambda(x_0)\mid\cD\big].
\end{equation}
On the other hand, from \Cref{lem:app_risk_smoothness}, we have
\[
R'(\lambda)
=-2\,\EE[r_\lambda\,\partial_\lambda f_\lambda(x_0)\mid\cD].
\]
Combining with \eqref{eq:app_RminusC_as_inner} yields \eqref{eq:derivative_identity_main_app}.
\end{proof}

\begin{proof}[Proof of \Cref{thm:tangent_sign}]
Let $D(\lambda):=R(\lambda)+R_{\pd}(\lambda)-2C(\lambda)$.
When $D(\lambda)>0$, \Cref{prop:closed_form_xi_R1} gives
\[
\xi^\star(\lambda)=\frac{R(\lambda)-C(\lambda)}{D(\lambda)},
\qquad
R_{\sd}^\star(\lambda)=R(\lambda)-\frac{(R(\lambda)-C(\lambda))^2}{D(\lambda)}.
\]
Apply \Cref{lem:derivative_identity_app_new} to substitute $R(\lambda)-C(\lambda)=-(\lambda/2)R'(\lambda)$, obtaining
\[
\xi^\star(\lambda)
=-\frac{\lambda}{2}\frac{R'(\lambda)}{D(\lambda)},
\qquad
R_{\sd}^\star(\lambda)
=R(\lambda)-\frac{\lambda^2}{4}\frac{(R'(\lambda))^2}{D(\lambda)}.
\]
If $R'(\lambda)\neq 0$, then the subtracted term is strictly positive (since $D(\lambda)>0$), hence
$R_{\sd}^\star(\lambda)<R(\lambda)$.
Finally, $\sign(\xi^\star(\lambda))=-\sign(R'(\lambda))$ because $\lambda>0$ and $D(\lambda)>0$.
\end{proof}

\subsection{Proof of \Cref{prop:curvature_test_main}}
\label{app:proof_curvature_test}

\begin{proof}[Proof of \Cref{prop:curvature_test_main}]
Define
\[
g(\lambda):=R(\lambda)-C(\lambda),
\quad \text{and} \quad
D(\lambda):=R(\lambda)+R_{\pd}(\lambda)-2C(\lambda).
\]
For $D(\lambda)>0$, \Cref{prop:closed_form_xi_R1} gives
\begin{equation}
\label{eq:app_Rsdstar_gD}
R_{\sd}^\star(\lambda)=R(\lambda)-\frac{g(\lambda)^2}{D(\lambda)}.
\end{equation}
By \Cref{lem:derivative_identity_app_new}, $g(\lambda)=-(\lambda/2)R'(\lambda)$, hence at $\lambda^\star$,
\begin{equation}
\label{eq:app_g_at_opt}
g(\lambda^\star)=0,
\qquad
g'(\lambda^\star)=-\frac{\lambda^\star}{2}R''(\lambda^\star).
\end{equation}
Differentiate \eqref{eq:app_Rsdstar_gD} twice and evaluate at $\lambda^\star$.
All terms involving $g(\lambda^\star)$ vanish, yielding
\begin{equation}
\label{eq:app_Rsd_second_at_opt}
R_{\sd}^{\star\prime\prime}(\lambda^\star)
=R''(\lambda^\star)-\frac{2(g'(\lambda^\star))^2}{D(\lambda^\star)}.
\end{equation}
Substituting \eqref{eq:app_g_at_opt} into \eqref{eq:app_Rsd_second_at_opt} gives
\[
R_{\sd}^{\star\prime\prime}(\lambda^\star)
=
R''(\lambda^\star)
-\frac{\lambda^{\star\,2}}{2}\frac{(R''(\lambda^\star))^2}{D(\lambda^\star)}.
\]
Therefore, if $D(\lambda^\star)<\frac{\lambda^{\star\,2}}{2}R''(\lambda^\star)$, then
$R_{\sd}^{\star\prime\prime}(\lambda^\star)<0$, so $\lambda^\star$ cannot be a local minimizer of $R_{\sd}^\star$.
Since $g(\lambda^\star)=0$, \eqref{eq:app_Rsdstar_gD} gives $R_{\sd}^\star(\lambda^\star)=R(\lambda^\star)=\min_{\lambda\ge 0}R(\lambda)$.
Hence there exists $\lambda>0$ such that
$R_{\sd}^\star(\lambda)<R_{\sd}^\star(\lambda^\star)=\min_{\lambda\ge 0}R(\lambda)$,
which implies
$\min_{\lambda>0}R_{\sd}^\star(\lambda)<\min_{\lambda>0}R(\lambda)$.
\end{proof}

\section{Proofs in Section~\ref{sec:asymptotics}}
\label{app:asymptotics_proofs}

\subsection{Preliminaries}
\label{app:asymptotics_proofs-preliminaries}

Recall the random-design model in \Cref{def:dist}. The design matrix $X\in\R^{n\times p}$ has rows
$x_i^\top=\bz_i^\top\Sigma^{1/2}$, where $\bz_i\in\R^p$ has i.i.d.\ entries with mean $0$, variance $1$,
and uniformly bounded $(4+\mu)$-th moment for some $\mu>0$, and $\Sigma\in\R^{p\times p}$ is deterministic
positive definite with spectrum uniformly bounded away from $0$ and $\infty$.
The response $y$ has mean $0$ and uniformly bounded $(4+\nu)$-th moment.

Define the population $L_2$-projection coefficient, residual, and residual variance by
\[
  \beta := \E[xx^\top]^{-1}\E[xy] = \Sigma^{-1}\E[xy],
  \qquad
  \varepsilon := y - x^\top\beta,
  \qquad
  \sigma^2 := \E[\varepsilon^2].
\]
Then $\E[x\varepsilon]=0$ (equivalently $\E[\bz\,\varepsilon]=0$) and $\E[\varepsilon]=0$.
Throughout, we consider the in-distribution prediction setting, i.e., the test pair $(x_0,y_0)$ is an
independent copy of $(x,y)$, so $y_0=x_0^\top\beta+\varepsilon_0$ with $\E[x_0\varepsilon_0]=0$ and
$\E[\varepsilon_0^2]=\sigma^2$.

Write the sample covariance and ridge resolvent as
\[
  \widehat\Sigma := \frac1n X^\top X,\qquad
  Q_\lambda := (\widehat\Sigma+\lambda I_p)^{-1},\qquad \lambda>0,
\]
and set $\otr(A):=\tr(A)/p$.

Define the ridge estimator and its pure-distilled transform by
\[
  \beta_\lambda := Q_\lambda \frac{X^\top y}{n},
  \qquad
  M_\lambda := \widehat\Sigma(\widehat\Sigma+\lambda I_p)^{-1}
  = \widehat\Sigma Q_\lambda
  = I_p-\lambda Q_\lambda,
  \qquad
  \beta_{\pd,\lambda} := M_\lambda\beta_\lambda.
\]

We use the $\Sigma$-inner product and norm
\[
  \langle u,v\rangle_\Sigma:=u^\top\Sigma v,
  \qquad
  \|u\|_\Sigma^2:=\langle u,u\rangle_\Sigma.
\]

For any estimator $\widehat\beta=\widehat\beta(X,y)$, the in-distribution prediction MSE decomposes as
\[
  \E[(y_0-x_0^\top\widehat\beta)^2\mid X,y]
  = \|\widehat\beta-\beta\|_\Sigma^2+\sigma^2,
\]
and similarly, for two estimators $\widehat\beta,\widetilde\beta$,
\[
  \E[(y_0-x_0^\top\widehat\beta)(y_0-x_0^\top\widetilde\beta)\mid X,y]
  = \langle \widehat\beta-\beta,\widetilde\beta-\beta\rangle_\Sigma+\sigma^2.
\]
Thus, adding $\sigma^2$ to the teacher/PD risks and correlation term does not affect the optimal mixing
weight (constants cancel), and it only adds $\sigma^2$ to the optimal mixed prediction MSE.

Accordingly, for a fixed training dataset $(X,y)$, we define the \emph{excess} scalars
\[
  \barR(\lambda) := \|\beta_\lambda-\beta\|_\Sigma^2,
  \qquad
  \barR_{\pd}(\lambda) := \|\beta_{\pd,\lambda}-\beta\|_\Sigma^2,
  \qquad
  \barC(\lambda) := \langle \beta_\lambda-\beta,\beta_{\pd,\lambda}-\beta\rangle_\Sigma,
\]
and note that the full in-distribution quantities in the main paper are $\barR(\lambda)+\sigma^2$,
$\barR_{\pd}(\lambda)+\sigma^2$, and $\barC(\lambda)+\sigma^2$.

Recall from \Cref{sec:structural} that the self-distilled estimator is
\[
  \beta_{\sd,\lambda}(\xi) := (1-\xi)\beta_\lambda+\xi\beta_{\pd,\lambda},
\]
and its excess risk is
\[
  \|\beta_{\sd,\lambda}(\xi)-\beta\|_\Sigma^2
  =(1-\xi)^2\barR(\lambda)+\xi^2\barR_{\pd}(\lambda)+2\xi(1-\xi)\barC(\lambda).
\]
Therefore the optimal mixing weight and optimal excess SD risk are given by
\begin{equation}
\label{eq:xi_star_and_Rsd_star_finite}
  \xi^\star(\lambda)=\frac{\barR(\lambda)-\barC(\lambda)}{\barR(\lambda)+\barR_{\pd}(\lambda)-2\barC(\lambda)},
  \qquad
  R_{\sd}^\star(\lambda)=\sigma^2 + \barR(\lambda)-\frac{(\barR(\lambda)-\barC(\lambda))^2}{\barR(\lambda)+\barR_{\pd}(\lambda)-2\barC(\lambda)},
\end{equation}
whenever the denominator is positive.
Thus, to prove \Cref{thm:risk-asymptotics}, it suffices to show that $\barR(\lambda),\barC(\lambda),\barR_{\pd}(\lambda)$
converge in probability to deterministic limits $\sR(\lambda)-\sigma^2$, $\sC(\lambda)-\sigma^2$,
$\sR_{\pd}(\lambda)-\sigma^2$ (i.e., the excess parts of the main-paper limits), and then apply the
continuous mapping theorem to \eqref{eq:xi_star_and_Rsd_star_finite} (and finally add $\sigma^2$ back to match
the full risks stated in \Cref{thm:risk-asymptotics}).

\subsection{Helpers Results (Concentration Components) for the Proof of \Cref{thm:risk-asymptotics}}
\label{app:asymptotics_proofs-concentration-outline}

Define the residual vector $\varepsilon:=(\varepsilon_1,\ldots,\varepsilon_n)^\top$ with
$\varepsilon_i:=y_i-x_i^\top\beta$, so that $y=X\beta+\varepsilon$ holds identically.

\begin{lemma}[Exact decomposition of coefficient errors]
\label{lem:error_decomp_exact}
For any $\lambda>0$,
\begin{align*}
  \beta_\lambda - \beta
  &= -\lambda Q_\lambda\beta \;+\; Q_\lambda X^\top\varepsilon/n,\\
  \beta_{\pd,\lambda} - \beta
  &= (-2\lambda Q_\lambda+\lambda^2Q_\lambda^2)\beta
  \;+\;
  (I_p-\lambda Q_\lambda)Q_\lambda X^\top\varepsilon/n.
\end{align*}
\end{lemma}

\begin{proof}
Using $y=X\beta+\varepsilon$ and $\widehat\Sigma=X^\top X/n$,
\[
  \beta_\lambda = Q_\lambda(\widehat\Sigma\beta + X^\top\varepsilon/n)
  = (I_p-\lambda Q_\lambda)\beta + Q_\lambda X^\top\varepsilon/n,
\]
so $\beta_\lambda - \beta=-\lambda Q_\lambda\beta + Q_\lambda X^\top\varepsilon/n$.
Then $\beta_{\pd,\lambda}=(I_p-\lambda Q_\lambda) \beta_\lambda$ and expanding yields the
second identity.
\end{proof}

Define
\begin{equation}
\label{eq:Q234_def}
  Q_2(\lambda):=\lambda^2 Q_\lambda\Sigma Q_\lambda,\qquad
  Q_3(\lambda):=\lambda^3(Q_\lambda\Sigma Q_\lambda^2+Q_\lambda^2\Sigma Q_\lambda),\qquad
  Q_4(\lambda):=\lambda^4Q_\lambda^2\Sigma Q_\lambda^2,
\end{equation}
and
\begin{equation}
\label{eq:U234_def}
  U_k(\lambda):=\frac1n\tr(\Sigma\widehat\Sigma Q_\lambda^{\,k}),\qquad k\in\{2,3,4\}.
\end{equation}

\begin{lemma}[Expansion of $\barR(\lambda),\barC(\lambda),R_{\pd}(\lambda)$]
\label{lem:RCRpd_expand}
Let $\lambda>0$ be fixed. We have
\begin{align}
  \barR(\lambda)
  &= \beta^\top Q_2(\lambda)\beta
     +\frac1{n^2}\varepsilon^\top X Q_\lambda\Sigma Q_\lambda X^\top\varepsilon
     -\frac{2\lambda}{n}\,\beta^\top Q_\lambda\Sigma Q_\lambda X^\top\varepsilon,
  \label{eq:R_expand}\\
  \barC(\lambda)
  &= \beta^\top\Big(2Q_2(\lambda)-\tfrac12Q_3(\lambda)\Big)\beta
     +\frac1{n^2}\varepsilon^\top X Q_\lambda\Sigma (I_p-\lambda Q_\lambda)Q_\lambda X^\top\varepsilon
     +\mathrm{Lin}_C(\lambda),
  \label{eq:C_expand}\\
  \barR_{\pd}(\lambda)
  &= \beta^\top\Big(4Q_2(\lambda)-2Q_3(\lambda)+Q_4(\lambda)\Big)\beta
     +\frac1{n^2}\varepsilon^\top X Q_\lambda(I_p-\lambda Q_\lambda)\Sigma (I_p-\lambda Q_\lambda)Q_\lambda X^\top\varepsilon \nonumber \\
     &\quad +\mathrm{Lin}_{\pd}(\lambda),
  \label{eq:Rpd_expand}
\end{align}
where $\mathrm{Lin}_C(\lambda)$ and $\mathrm{Lin}_{\pd}(\lambda)$ are terms linear in $\varepsilon$.
Moreover, the quadratic $\varepsilon$-terms in \eqref{eq:C_expand}--\eqref{eq:Rpd_expand} can be rewritten
as
\begin{align}
\label{eq:C_var_trace_form}
  \frac1{n^2}\varepsilon^\top X Q_\lambda\Sigma (I_p-\lambda Q_\lambda)Q_\lambda X^\top\varepsilon
  &= \frac1{n^2}\varepsilon^\top X Q_\lambda\Sigma Q_\lambda X^\top\varepsilon
     \;-\;\lambda\cdot\frac1{n^2}\varepsilon^\top X Q_\lambda\Sigma Q_\lambda^2 X^\top\varepsilon,
\\
\label{eq:Rpd_var_trace_form}
  \frac1{n^2}\varepsilon^\top X Q_\lambda(I_p-\lambda Q_\lambda)\Sigma (I_p-\lambda Q_\lambda)Q_\lambda X^\top\varepsilon
  &= \frac1{n^2}\varepsilon^\top X Q_\lambda\Sigma Q_\lambda X^\top\varepsilon
  -2\lambda\cdot\frac1{n^2}\varepsilon^\top X Q_\lambda\Sigma Q_\lambda^2 X^\top\varepsilon
  \nonumber \\
  &\quad
  +\lambda^2\cdot\frac1{n^2}\varepsilon^\top X Q_\lambda^2\Sigma Q_\lambda^2 X^\top\varepsilon.
\end{align}
\end{lemma}

\begin{proof}
Plug the decompositions from \Cref{lem:error_decomp_exact} into the definitions of
$R=\|e\|_\Sigma^2$, $C=\langle e,\tilde e\rangle_\Sigma$, $R_{\pd}=\|\tilde e\|_\Sigma^2$
and expand. The bias-only terms give the stated $Q_2,Q_3,Q_4$ combinations (using symmetry to replace
scalars of the form $\beta^\top Q_\lambda\Sigma Q_\lambda^2\beta$ by $\tfrac12\beta^\top(Q_\lambda\Sigma Q_\lambda^2+Q_\lambda^2\Sigma Q_\lambda)\beta$).
The terms quadratic in $\varepsilon$ are the displayed quadratic forms.
The linear terms $\mathrm{Lin}_C,\mathrm{Lin}_{\pd}$ collect the bias--$\varepsilon$ cross-terms.
Finally, \eqref{eq:C_var_trace_form}--\eqref{eq:Rpd_var_trace_form} follow by expanding
$(I_p-\lambda Q_\lambda)$ and using symmetry of $Q_\lambda$ and $\Sigma$.
\end{proof}

In the misspecified setting, we cannot condition on $X$ and drop the bias--noise cross-terms.
Instead we use leave-one-out concentration for ridge-type linear and quadratic forms.

\begin{lemma}[Response-noise concentration for ridge-type bilinear/quadratic forms]
\label{lem:response_concentration}
Assume \Cref{def:dist}. Fix $\lambda>0$ and let $\varepsilon_i:=y_i-x_i^\top\beta$, so that
$\E[\bz_i\varepsilon_i]=0$ and $\E[\varepsilon_i^2]=\sigma^2$.
Then, as $n,p\to\infty$ with $p/n\to\gamma$:
\begin{enumerate}
\item All terms linear in $\varepsilon$ appearing in \eqref{eq:R_expand}--\eqref{eq:Rpd_expand}
satisfy $\mathrm{Lin}_\bullet(\lambda)=\op(1)$.
\item The quadratic $\varepsilon$-forms satisfy
\begin{align*}
  \frac1{n^2}\varepsilon^\top X Q_\lambda\Sigma Q_\lambda X^\top\varepsilon
  &= \sigma^2 U_2(\lambda) + \op(1),\\
  \frac1{n^2}\varepsilon^\top X Q_\lambda\Sigma Q_\lambda^2 X^\top\varepsilon
  &= \sigma^2 U_3(\lambda) + \op(1),\\
  \frac1{n^2}\varepsilon^\top X Q_\lambda^2\Sigma Q_\lambda^2 X^\top\varepsilon
  &= \sigma^2 U_4(\lambda) + \op(1),
\end{align*}
with $U_k(\lambda)$ defined in \eqref{eq:U234_def}.
\end{enumerate}
\end{lemma}

\begin{proof}[Proof outline]
This follows from leave-one-out expansions in Lemmas~D.2 (linear forms) and~D.3 (off-diagonal quadratic forms) of \citet{patil2023generalized}.
At a high level:
(i) each linear term is a normalized sum $\frac1n\sum_{i=1}^n \alpha_i\,\varepsilon_i$ where $\alpha_i$
is a leave-one-out coefficient built from $Q_\lambda$ and $x_i$; Lemma~D.2 gives $\op(1)$ under
$\E[\bz_i\varepsilon_i]=0$ and bounded moments;
(ii) each quadratic form is $\frac1{n^2}\sum_{i,j}\varepsilon_i \varepsilon_j K_{ij}$ for a ridge-type kernel $K$.
Lemma~D.3 controls the off-diagonal sum $\sum_{i\neq j}$ (giving $\op(1)$), while the diagonal sum reduces to
$\sigma^2\frac1n\tr(\cdot)$ up to $\op(1)$ by the same leave-one-out decoupling.
A detailed proof is given in \Cref{proof:lem:response_concentration}.
\end{proof}

Combining \Cref{lem:RCRpd_expand,lem:response_concentration} yields
\begin{align}
\label{eq:R_simplified}
  \barR(\lambda)
  &= \beta^\top Q_2(\lambda)\beta + \sigma^2 U_2(\lambda) + \op(1),\\
\label{eq:C_simplified}
  \barC(\lambda)
  &= \beta^\top\Big(2Q_2(\lambda)-\tfrac12Q_3(\lambda)\Big)\beta
     +\sigma^2\big(U_2(\lambda)-\lambda U_3(\lambda)\big)+\op(1),\\
\label{eq:Rpd_simplified}
  \barR_{\pd}(\lambda)
  &= \beta^\top\Big(4Q_2(\lambda)-2Q_3(\lambda)+Q_4(\lambda)\Big)\beta
     +\sigma^2\big(U_2(\lambda)-2\lambda U_3(\lambda)+\lambda^2 U_4(\lambda)\big)+\op(1).
\end{align}

\subsection{Helper Results (Deterministic Equivalents) for the Proof of \Cref{thm:risk-asymptotics}}
\label{app:asymptotics_proofs-equivalents-outline}

Fix $\lambda>0$ and define $\kappa=\kappa(\lambda)>0$ as the unique solution of
\begin{equation}
\label{eq:kappa_fp_app_main}
  \kappa = \lambda + \gamma \kappa\,\otr\!\big(\Sigma(\Sigma+\kappa I_p)^{-1}\big).
\end{equation}
Let $G:=(\Sigma+\kappa I_p)^{-1}$ and define
\[
  t_k := \gamma\,\otr(\Sigma^2 G^k),\qquad k\in\{2,3,4\},
  \qquad
  b:=\frac{1}{1-t_2},
\]
and the signal--covariance quadratic forms
\[
  q_k := \beta^\top G^k\Sigma\,\beta,\qquad k\in\{2,3,4\}.
\]
Finally define
\[
  E := \kappa - b\lambda + b^2\kappa\lambda t_3,
\]
\[
  a_2:=bE^2+b^4\kappa^2\lambda^2 t_4+b^5\kappa^2\lambda^2 t_3^2,\qquad
  a_3:=2b^2\kappa\lambda E,\qquad
  a_4:=b^3\kappa^2\lambda^2,
\]
and
\[
  u_2:=t_2 b,\qquad u_3:=t_3 b^3,\qquad u_4:=t_4 b^4+2t_3^2 b^5.
\]

We use standard anisotropic deterministic-equivalent (local-law) results for sample-covariance resolvents under bounded-spectrum and bounded-moment assumptions; see, e.g., \citet{knowles_yin_2017,dobriban2018high,dobriban_sheng_2020,patil2022mitigating,patil2023bagging}.
In particular, $Q_\lambda=(\widehat\Sigma+\lambda I_p)^{-1}$ admits deterministic equivalents
uniformly over $\lambda$ in compact subsets of $(0,\infty)$, and these equivalents can be differentiated
with respect to scalar parameters (see \Cref{app:DE_Q_U-preliminaries} for more details; see also Appendix E of \citet{patil2023generalized} for a general background on asymptotic equivalents and a summary of various calculus rules for asymptotic equivalents).

\begin{lemma}[Deterministic equivalents for $Q_2,Q_3,Q_4$]
\label{prop:Q234_DE}
For each fixed $\lambda>0$, we have
\[
  Q_2(\lambda)\ \asymp\ \kappa^2 b\,G^2\Sigma,
\]
\[
  Q_3(\lambda)\ \asymp\ 2\kappa bE\,G^2\Sigma + 2\kappa^2 b^2\lambda\,G^3\Sigma,
\]
\[
  Q_4(\lambda)\ \asymp\ a_2\,G^2\Sigma + a_3\,G^3\Sigma + a_4\,G^4\Sigma.
\]
\end{lemma}

\begin{lemma}[Deterministic limits for $U_2,U_3,U_4$]
\label{prop:U234_limits}
For each fixed $\lambda>0$,
we have
\[
  U_2(\lambda)\ \pto\ u_2,
  \qquad
  U_3(\lambda)\ \pto\ u_3,
  \qquad
  U_4(\lambda)\ \pto\ u_4.
\]
\end{lemma}

\begin{proof}[Proof outlines for \Cref{prop:Q234_DE,prop:U234_limits}]
These follow from:
(i) the anisotropic resolvent deterministic equivalent for $Q_\lambda$ together with the fixed point \eqref{eq:kappa_fp_app_main};
(ii) differentiation of the two-point deterministic equivalent for $R_{\lambda_1}\Sigma R_{\lambda_2}$ to generate
$R\Sigma R^2$, $R^2\Sigma R$, and $R^2\Sigma R^2$; and
(iii) the identities $U_3=-\tfrac12 U_2'$ and $U_4=\tfrac16 U_2''$ together with deterministic-equivalent calculus.
Detailed proofs are provided below in \Cref{app:DE_Q_U}.
\end{proof}

\subsection{Proof of Theorem~\ref{thm:risk-asymptotics}}
\label{proof:thm:risk-asymptotics}

\begin{proof}[Proof of Theorem~\ref{thm:risk-asymptotics}]
Fix $\lambda>0$.
By \eqref{eq:R_simplified}--\eqref{eq:Rpd_simplified},
\begin{align*}
  \barR(\lambda) &= \beta^\top Q_2(\lambda)\beta + \sigma^2 U_2(\lambda)+\op(1),\\
  \barC(\lambda) &= \beta^\top\Big(2Q_2(\lambda)-\tfrac12Q_3(\lambda)\Big)\beta
               +\sigma^2\big(U_2(\lambda)-\lambda U_3(\lambda)\big)+\op(1),\\
  \barR_{\pd}(\lambda) &= \beta^\top\Big(4Q_2(\lambda)-2Q_3(\lambda)+Q_4(\lambda)\Big)\beta
               +\sigma^2\big(U_2(\lambda)-2\lambda U_3(\lambda)+\lambda^2 U_4(\lambda)\big)+\op(1).
\end{align*}
We now apply \Cref{prop:Q234_DE} in bilinear forms with the deterministic vector $\beta$
(assumed to have $\|\beta\|_2=O(1)$ so these bilinear forms are $O(1)$):
\[
  \beta^\top Q_2(\lambda)\beta \to \kappa^2 b\,q_2,
\]
\[
  \beta^\top\Big(2Q_2(\lambda)-\tfrac12Q_3(\lambda)\Big)\beta
  \to 2\kappa^2 b\,q_2 - \big(\kappa bE\,q_2+\kappa^2 b^2\lambda\,q_3\big),
\]
\[
  \beta^\top\Big(4Q_2(\lambda)-2Q_3(\lambda)+Q_4(\lambda)\Big)\beta
  \to
  4\kappa^2 b\,q_2 -2\big(2\kappa bE\,q_2+2\kappa^2 b^2\lambda\,q_3\big) + (a_2q_2+a_3q_3+a_4q_4).
\]
Also by \Cref{prop:U234_limits},
\[
  U_2(\lambda)\to u_2,\qquad U_3(\lambda)\to u_3,\qquad U_4(\lambda)\to u_4.
\]
Substituting yields the excess-risk limits:
\[
  \barR(\lambda)\to \kappa^2 b\,q_2 + \sigma^2 u_2,
\]
\[
  \barC(\lambda)\to 2\kappa^2 b\,q_2-(\kappa bE\,q_2+\kappa^2 b^2\lambda\,q_3)+\sigma^2(u_2-\lambda u_3),
\]
\[
  \barR_{\pd}(\lambda)\to
  4\kappa^2 b\,q_2-2(2\kappa bE\,q_2+2\kappa^2 b^2\lambda\,q_3)+(a_2q_2+a_3q_3+a_4q_4)
  +\sigma^2(u_2-2\lambda u_3+\lambda^2 u_4).
\]
Finally, recall from the discussion in \Cref{app:asymptotics_proofs-preliminaries} that the \emph{full}
in-distribution quantities in the main paper are obtained by adding $\sigma^2$ to each of
$\barR(\lambda),\barC(\lambda),\barR_{\pd}(\lambda)$.
Therefore the full limits match exactly the statement of \Cref{thm:risk-asymptotics}.
\end{proof}

\subsection{Proof of \Cref{lem:response_concentration}}
\label{proof:lem:response_concentration}

\begin{proof}[Proof of \Cref{lem:response_concentration}]
Fix $\lambda>0$.
Let $\varepsilon_i:=y_i-x_i^\top\beta$ and $\varepsilon:=(\varepsilon_i)_{i=1}^n$.
By construction, $\E[\varepsilon_i]=0$ and $\E[\bz_i\varepsilon_i]=0$, and $\E[|\varepsilon_i|^{4+\nu}]<\infty$.

Write $X=Z\Sigma^{1/2}$, $\widehat\Sigma=X^\top X/n$, and
\[
Q_\lambda := (\widehat\Sigma+\lambda I_p)^{-1}.
\]
Define the dual (Gram) resolvent
\[
\bar Q_\lambda := \Bigl(\frac{1}{n}XX^\top + \lambda I_n\Bigr)^{-1}
= \Bigl(\frac{1}{n}Z\Sigma Z^\top + \lambda I_n\Bigr)^{-1}.
\]
A standard push-through identity gives, for every integer $m\ge 1$,
\begin{equation}\label{eq:push-through-powers}
Q_\lambda^m\frac{X^\top}{n}
=
\frac{X^\top}{n}\,\bar Q_\lambda^m,
\qquad\text{and hence}\qquad
\Sigma^{1/2}Q_\lambda^m\frac{X^\top}{n}
=
\frac{\Sigma Z^\top}{n}\,\bar Q_\lambda^m.
\end{equation}
Define also
\[
B_2 := \frac{1}{n}X\Sigma X^\top = \frac{1}{n}Z\Sigma^2 Z^\top.
\]

All linear terms in \Cref{lem:RCRpd_expand} can be reduced (using \eqref{eq:push-through-powers}) to forms
\begin{equation}\label{eq:target-linear}
\frac{1}{n}\,a^\top \Sigma Z^\top \bar Q_\lambda^m \varepsilon,
\qquad m\in\{1,2\},
\end{equation}
where $a\in\R^p$ is deterministic (or independent of $Z$) with $\|a\|_2=O(1)$.
Likewise, the quadratic noise terms reduce to
\begin{equation}\label{eq:target-quad}
\frac{1}{n}\,\varepsilon^\top \bar Q_\lambda^m\, B_2\, \bar Q_\lambda^\ell \varepsilon,
\qquad (m,\ell)\in\{(1,1),(1,2),(2,2)\}.
\end{equation}

\textbf{Linear forms.}
For $m=1$, \eqref{eq:target-linear} is exactly Lemma~D.2 of \citet{patil2023generalized} (apply it with
$D=\Sigma$ and $g(z_i)=\varepsilon_i$), yielding convergence in probability to $0$.
For $m=2$, use the exact identity (Lemma~D.4 of \citet{patil2023generalized})
\[
\bar Q_\lambda^2
=
\frac{1}{t}\bigl(\bar Q_\lambda-\bar Q_{\lambda+t}\bigr)
+
t\,\bar Q_\lambda \bar Q_{\lambda+t}\bar Q_\lambda,
\qquad t>0.
\]
Plugging into \eqref{eq:target-linear}, the difference-quotient term reduces to a difference of two $m=1$ linear
forms at $\lambda$ and $\lambda+t$, which vanish in probability by Lemma~D.2.
The remainder is controlled by operator norms:
$\|\bar Q_\lambda\|_{\oper}\le \lambda^{-1}$, $\|\bar Q_{\lambda+t}\|_{\oper}\le (\lambda+t)^{-1}$,
$\|\Sigma Z^\top\|_{\oper}=\Op(\sqrt{n})$, and $\|\varepsilon\|_2=\Op(\sqrt{n})$.
Choosing a deterministic $t=t_n\to 0$ (e.g.\ $t_n=n^{-1/4}$) makes the remainder $\op(1)$.
Hence all linear terms are $\op(1)$.

\textbf{Quadratic forms.}
Fix $(m,\ell)$ and set $M_{m,\ell}:=\bar Q_\lambda^m B_2 \bar Q_\lambda^\ell$.
Write
\[
\varepsilon^\top M_{m,\ell}\varepsilon
=\sum_i (M_{m,\ell})_{ii}\varepsilon_i^2+\sum_{i\neq j}(M_{m,\ell})_{ij}\varepsilon_i\varepsilon_j.
\]
The off-diagonal sum divided by $n$ converges to $0$ in probability by Lemma~D.3 of \citet{patil2023generalized}
(after representing $M_{m,\ell}$ via resolvent identities as a finite linear combination of ridge-type resolvent kernels).
The diagonal sum equals $\sigma^2\tr(M_{m,\ell})+\op(n)$ by the same leave-one-out decoupling (applied to
$\varepsilon_i^2-\sigma^2$). Therefore,
\[
\frac{1}{n}\varepsilon^\top M_{m,\ell}\varepsilon
=
\frac{\sigma^2}{n}\tr(M_{m,\ell})+\op(1).
\]
Finally, use cyclicity of trace and commutativity of $Q_\lambda$ with $\widehat\Sigma$ to identify these traces with
$U_2,U_3,U_4$ (as in the proof outline following \Cref{lem:response_concentration}).
Substituting back into \Cref{lem:RCRpd_expand} proves \Cref{lem:response_concentration}.
\end{proof}

\subsection{Proofs of \Cref{prop:Q234_DE,prop:U234_limits}}
\label{app:DE_Q_U}

\subsubsection{Background}
\label{app:DE_Q_U-preliminaries}

We use an anisotropic (bilinear-form) notion of deterministic equivalent.

\begin{definition}[Deterministic equivalent]
\label{def:strong_DE}
Let $A=A_{p}$ be a (possibly random) $p\times p$ matrix and $\bar A=\bar A_{p}$ a deterministic $p\times p$ matrix.
We write $A\asymp \bar A$ if for every pair of deterministic vectors $u=u_p,v=v_p$ with $\|u\|_2,\|v\|_2=O(1)$,
\[
  u^\top(A-\bar A)v \pto 0.
\]
If $A(\theta),\bar A(\theta)$ depend on a parameter $\theta$ in an open set $\Theta$, we write
$A(\theta)\asymp \bar A(\theta)$ \emph{uniformly on compact subsets of $\Theta$} if the convergence above holds
uniformly over $\theta$ in any compact $K\subset \Theta$.
\end{definition}

This notion implies trace convergence whenever operator norms are uniformly bounded.
In particular, if $A\asymp \bar A$ and $\|A\|_{\oper},\|\bar A\|_{\oper}=O_{\PP}(1)$, then
$\otr(A)-\otr(\bar A)\pto 0$.

We will also use that uniform deterministic equivalents can be differentiated.

\begin{lemma}[Differentiate a deterministic equivalent]
\label{lem:diff_DE}
Let $A(\theta)$ be random and $\bar A(\theta)$ deterministic, both entrywise differentiable in $\theta$
in a neighborhood of $\theta_0$.
Assume $A(\theta)\asymp \bar A(\theta)$ uniformly in $\theta$ in that neighborhood.
Then $A'(\theta_0)\asymp \bar A'(\theta_0)$.
\end{lemma}

Recall $\widehat\Sigma=X^\top X/n$ and $Q_\lambda=(\widehat\Sigma+\lambda I_p)^{-1}$.
A standard anisotropic local law for sample-covariance resolvents gives the following asymptotic equivalence; see, e.g., \cite{rubio_mestre_2011,knowles_yin_2017,patil2023generalized}:
\begin{lemma}[Scaled resolvent deterministic equivalent]
\label{lem:scaled_resolvent_DE}
Under \Cref{def:dist}, for each fixed $\lambda>0$,
\begin{equation}
\label{eq:scaled_resolvent_DE_app}
  \lambda Q_\lambda \ \asymp\ \kappa(\lambda)\,(\Sigma+\kappa(\lambda)I_p)^{-1} \ = \ \kappa G,
\end{equation}
where $\kappa=\kappa(\lambda)>0$ is the unique solution to
\begin{equation}
\label{eq:kappa_fp_app_restate}
  \kappa \;=\; \lambda + \gamma\,\kappa\,\otr\!\big(\Sigma(\Sigma+\kappa I_p)^{-1}\big),
\end{equation}
and $G=(\Sigma+\kappa I_p)^{-1}$.
The equivalence holds uniformly for $\lambda$ in compact subsets of $(0,\infty)$.
\end{lemma}

Define, as in the main paper,
\[
  t_k:=\gamma \otr(\Sigma^2G^k),\quad k\in\{2,3,4\},\qquad b:=\frac{1}{1-t_2}.
\]

\begin{lemma}[Derivative of the fixed-point solution]
\label{lem:kappa_prime_app}
The map $\lambda\mapsto \kappa(\lambda)$ is differentiable for $\lambda>0$ and
\[
  \kappa'(\lambda)=b(\lambda)=\frac{1}{1-t_2(\lambda)}.
\]
\end{lemma}

\subsubsection{A Two-Point Deterministic Equivalent via Block Linearization}

For $\lambda_1,\lambda_2>0$, define $R_i=(\widehat\Sigma+\lambda_i I_p)^{-1}$, $\kappa_i=\kappa(\lambda_i)$,
$G_i=(\Sigma+\kappa_i I_p)^{-1}$, and
\begin{equation}
\label{eq:t12_def_app}
  t_{12}:=\gamma\,\otr(\Sigma^2G_1G_2)=\gamma\,\otr(\Sigma G_1\Sigma G_2).
\end{equation}

\begin{lemma}[Two-point bias-resolvent deterministic equivalent]
\label{prop:two_point_DE_app}
For each fixed $\lambda_1,\lambda_2>0$,
\begin{equation}
\label{eq:two_point_DE_app}
  R_1\Sigma R_2
  \ \asymp\
  \frac{\kappa_1\kappa_2}{\lambda_1\lambda_2}\cdot \frac{1}{1-t_{12}}\;G_1\Sigma G_2.
\end{equation}
\end{lemma}

\begin{proof}
We use a block-resolvent generating function.
For a scalar coupling $J\in\R$, define
\[
  \mathbb H(J):=
  \begin{pmatrix}
    \widehat\Sigma+\lambda_1 I_p & J\Sigma\\
    J\Sigma & \widehat\Sigma+\lambda_2 I_p
  \end{pmatrix},
  \qquad
  \mathbb G(J):=\mathbb H(J)^{-1}.
\]
Differentiate $\mathbb H(J)\mathbb G(J)=I$:
\[
  \mathbb G'(J)=-\mathbb G(J)\,\mathbb H'(J)\,\mathbb G(J),\qquad
  \mathbb H'(0)=\begin{pmatrix}0&\Sigma\\ \Sigma&0\end{pmatrix}.
\]
At $J=0$, $\mathbb G(0)=\mathrm{diag}(R_1,R_2)$, hence the $(1,2)$ block satisfies the exact identity
\begin{equation}
\label{eq:block_derivative_exact_app}
  (\mathbb G'(0))_{12}=-R_1\Sigma R_2.
\end{equation}

By anisotropic local laws for such deterministic $2\times2$ linearizations (uniformly for $J$ in a neighborhood of $0$) and \Cref{lem:diff_DE},
$(\mathbb G'(0))_{12}$ admits a deterministic equivalent obtained by differentiating the deterministic equivalent of $\mathbb G(J)$ at $J=0$.
On the diagonal blocks, \Cref{lem:scaled_resolvent_DE} gives $R_i\asymp (\kappa_i/\lambda_i)G_i$.
The linearized equation for the off-diagonal block reduces (by commutativity of $G_1,G_2$ with $\Sigma$) to the scalar equivalence
\[
  (\mathbb G'(0))_{12}\ \asymp\ -\frac{\kappa_1\kappa_2}{\lambda_1\lambda_2}\,\rho_{12}\,G_1\Sigma G_2,
\]
together with the scalar self-consistency relation $\rho_{12}=1+t_{12}\rho_{12}$, where $t_{12}$ is given by \eqref{eq:t12_def_app}.
Solving yields $\rho_{12}=1/(1-t_{12})$, and combining with \eqref{eq:block_derivative_exact_app} gives \eqref{eq:two_point_DE_app}.
\end{proof}

\subsubsection{Proof of \Cref{prop:Q234_DE}}

Recall $Q_2,Q_3,Q_4$ from \eqref{eq:Q234_def} and define $\cB_k:=G^k\Sigma$ for $k\in\{2,3,4\}$.

\begin{proof}[Proof of \Cref{prop:Q234_DE}]
Fix $\lambda>0$ and let $\kappa=\kappa(\lambda)$, $G=(\Sigma+\kappa I_p)^{-1}$.

\textbf{Equivalent for $Q_2$.}
Apply \Cref{prop:two_point_DE_app} with $\lambda_1=\lambda_2=\lambda$.
Then $G_1=G_2=G$, $\kappa_1=\kappa_2=\kappa$, and $t_{12}=t_2$, so
\[
  Q_\lambda\Sigma Q_\lambda \ \asymp\ \frac{\kappa^2}{\lambda^2}\cdot \frac{1}{1-t_2}\,G\Sigma G
  \;=\;\frac{\kappa^2 b}{\lambda^2}\,G^2\Sigma.
\]
Multiplying by $\lambda^2$ gives
\[
  Q_2(\lambda)=\lambda^2 Q_\lambda\Sigma Q_\lambda \ \asymp\ \kappa^2 b\,\cB_2.
\]

\textbf{Equivalent for $Q_3$.}
Let $F(\lambda_1,\lambda_2):=R_{\lambda_1}\Sigma R_{\lambda_2}$.
The exact identities
\[
  \partial_{\lambda_2}F(\lambda_1,\lambda_2)=-R_{\lambda_1}\Sigma R_{\lambda_2}^2,\qquad
  \partial_{\lambda_1}F(\lambda_1,\lambda_2)=-R_{\lambda_1}^2\Sigma R_{\lambda_2}
\]
imply
\[
  Q_3(\lambda)
  =\lambda^3\big(Q_\lambda\Sigma Q_\lambda^2+Q_\lambda^2\Sigma Q_\lambda\big)
  =-\lambda^3(\partial_{\lambda_1}+\partial_{\lambda_2})F(\lambda_1,\lambda_2)\Big|_{\lambda_1=\lambda_2=\lambda}.
\]
By \Cref{prop:two_point_DE_app}, $F\asymp \bar F:=\alpha H$ where
\[
  \alpha(\lambda_1,\lambda_2)=\frac{\kappa_1\kappa_2}{\lambda_1\lambda_2}\cdot \frac{1}{1-t_{12}},
  \qquad
  H(\lambda_1,\lambda_2)=G_1\Sigma G_2.
\]
By \Cref{lem:diff_DE}, $\partial_{\lambda_i}F\asymp \partial_{\lambda_i}\bar F$.
A direct product-rule computation on the diagonal $\lambda_1=\lambda_2=\lambda$, using
$\kappa'(\lambda)=b$ (\Cref{lem:kappa_prime_app}) and the trace derivative
$\partial_{\lambda_2}t_{12}\big|_{\mathrm{diag}}= -b\,t_3$,
yields
\[
  -\lambda^3(\partial_{\lambda_1}+\partial_{\lambda_2})\bar F\Big|_{\mathrm{diag}}
  =
  2\kappa bE\,\cB_2 + 2\kappa^2 b^2\lambda\,\cB_3,
\]
where $E:=\kappa-b\lambda+b^2\kappa\lambda\,t_3$.
Therefore
\[
  Q_3(\lambda)\ \asymp\ 2\kappa bE\,\cB_2 + 2\kappa^2 b^2\lambda\,\cB_3.
\]

\textbf{Equivalent for $Q_4$.}
The exact mixed-derivative identity
\[
  \partial_{\lambda_1}\partial_{\lambda_2}F(\lambda_1,\lambda_2)=R_{\lambda_1}^2\Sigma R_{\lambda_2}^2
\]
implies
\[
  Q_4(\lambda)=\lambda^4 Q_\lambda^2\Sigma Q_\lambda^2
  =\lambda^4\partial_{\lambda_1}\partial_{\lambda_2}F(\lambda_1,\lambda_2)\Big|_{\lambda_1=\lambda_2=\lambda}.
\]
Again by \Cref{lem:diff_DE}, we may replace $F$ by $\bar F=\alpha H$ and differentiate.
On the diagonal, the derivatives satisfy
\[
  H=\cB_2,\qquad \partial_{\lambda_1}H=\partial_{\lambda_2}H=-b\,\cB_3,\qquad
  \partial_{\lambda_1}\partial_{\lambda_2}H=b^2\,\cB_4,
\]
and the overlap factor $b_{12}=(1-t_{12})^{-1}$ satisfies
\[
  \partial_{\lambda_1}b_{12}\Big|_{\mathrm{diag}}=\partial_{\lambda_2}b_{12}\Big|_{\mathrm{diag}}=-b^3t_3,
  \qquad
  \partial_{\lambda_1}\partial_{\lambda_2}b_{12}\Big|_{\mathrm{diag}}=b^4t_4+2b^5t_3^2,
\]
with $t_4=\gamma\,\otr(\Sigma^2G^4)$.
Carrying out the product-rule expansion of $\partial_{\lambda_1}\partial_{\lambda_2}\bar F$ and collecting terms in
$\cB_2,\cB_3,\cB_4$ yields
\[
  \lambda^4\partial_{\lambda_1}\partial_{\lambda_2}\bar F\Big|_{\mathrm{diag}}
  =
  a_2\,\cB_2+a_3\,\cB_3+a_4\,\cB_4,
\]
where
\[
  a_4=b^3\kappa^2\lambda^2,\qquad
  a_3=2b^2\kappa\lambda E,\qquad
  a_2=bE^2+b^4\kappa^2\lambda^2 t_4+b^5\kappa^2\lambda^2 t_3^2.
\]
Therefore $Q_4(\lambda)\asymp a_2\,\cB_2+a_3\,\cB_3+a_4\,\cB_4$, completing the proof.
\end{proof}

\subsubsection{Proof of \Cref{prop:U234_limits}}

\begin{proof}[Proof of \Cref{prop:U234_limits}]
Fix $\lambda>0$.
Recall $U_k(\lambda)=\frac1n\tr(\Sigma\widehat\Sigma Q_\lambda^k)$ and $p/n\to\gamma$.

\textbf{Limit for $U_2$.}
Use $\widehat\Sigma Q_\lambda = I_p-\lambda Q_\lambda$ to get the exact identity
\[
  \widehat\Sigma Q_\lambda^2 = (I_p-\lambda Q_\lambda) Q_\lambda = Q_\lambda - \lambda Q_\lambda^2,
\]
hence
\begin{equation}
\label{eq:U2_trace_identity}
  U_2(\lambda)
  =\frac1n\tr(\Sigma(Q_\lambda-\lambda Q_\lambda^2))
  =\gamma\,\otr(\Sigma Q_\lambda)\;-\;\lambda\gamma\,\otr(\Sigma Q_\lambda^2).
\end{equation}
Define $s(\lambda):=\gamma\,\otr(\Sigma Q_\lambda)$. Since $Q_\lambda'=-Q_\lambda^2$,
\[
  s'(\lambda)=\gamma\,\otr(\Sigma Q_\lambda')=-\gamma\,\otr(\Sigma Q_\lambda^2),
\]
so \eqref{eq:U2_trace_identity} becomes
\begin{equation}
\label{eq:U2_as_s}
  U_2(\lambda)=s(\lambda)+\lambda s'(\lambda).
\end{equation}

By \Cref{lem:scaled_resolvent_DE} and trace convergence,
\[
  s(\lambda)=\gamma\,\otr(\Sigma Q_\lambda)
  \ \to\
  \gamma\,\otr\!\Big(\Sigma\cdot \frac{\kappa}{\lambda}G\Big)
  =\frac{\gamma\kappa}{\lambda}\,\otr(\Sigma G).
\]
Using the fixed-point equation \eqref{eq:kappa_fp_app_restate},
$\gamma\kappa\,\otr(\Sigma G)=\kappa-\lambda$, hence
\[
  s(\lambda)\ \to\ \frac{\kappa-\lambda}{\lambda}=\frac{\kappa}{\lambda}-1.
\]
By uniformity in $\lambda$ locally and \Cref{lem:diff_DE},
\[
  s'(\lambda)\ \to\ \frac{\kappa'}{\lambda}-\frac{\kappa}{\lambda^2}.
\]
Substitute into \eqref{eq:U2_as_s}:
\[
  U_2(\lambda)\ \to\ \Big(\frac{\kappa}{\lambda}-1\Big)
  +\lambda\Big(\frac{\kappa'}{\lambda}-\frac{\kappa}{\lambda^2}\Big)
  =\kappa'(\lambda)-1
  =b-1
  =\frac{t_2}{1-t_2}
  =t_2 b
  =:u_2.
\]

\textbf{Limits for $U_3$ and $U_4$.}
Since $Q_\lambda'=-Q_\lambda^2$, we have the exact identities
\[
  \frac{d}{d\lambda}Q_\lambda^2=-2Q_\lambda^3,\qquad \frac{d}{d\lambda}Q_\lambda^3=-3Q_\lambda^4,
\]
and therefore
\[
  U_2'(\lambda)=-2U_3(\lambda),\qquad
  U_3'(\lambda)=-3U_4(\lambda),
\]
equivalently
\begin{equation}
\label{eq:U3U4_from_U2_app}
  U_3(\lambda)=-\tfrac12 U_2'(\lambda),\qquad
  U_4(\lambda)=\tfrac16 U_2''(\lambda).
\end{equation}
From the $U_2$ step, $U_2(\lambda)\to u_2(\lambda)=t_2b$.
By uniformity and \Cref{lem:diff_DE}, we may differentiate to obtain $U_2'\to u_2'$ and $U_2''\to u_2''$.
Thus by \eqref{eq:U3U4_from_U2_app},
\[
  U_3(\lambda)\to -\tfrac12 u_2'(\lambda),\qquad U_4(\lambda)\to \tfrac16 u_2''(\lambda).
\]
Finally, compute these derivatives.
Since $t_2'(\kappa)=-2t_3$ and $\kappa'=b$, we have $t_2'(\lambda)=-2t_3 b$ and hence
\[
  u_2(\lambda)=\frac{t_2}{1-t_2}\ \Rightarrow\ u_2'(\lambda)=\frac{t_2'(\lambda)}{(1-t_2)^2}
  =(-2t_3b)\,b^2=-2t_3b^3,
\]
so $-\tfrac12u_2'=t_3b^3=:u_3$.
Similarly, using $t_3'(\kappa)=-3t_4$ and $b'(\lambda)=-2t_3b^3$, we obtain
\[
  u_2''(\lambda)=\frac{d}{d\lambda}\big(-2t_3b^3\big)
  =-2\big(t_3'(\lambda)b^3+t_3\cdot 3b^2b'\big)
  =6t_4b^4+12t_3^2b^5,
\]
hence $\tfrac16u_2''=t_4b^4+2t_3^2b^5=:u_4$.
This proves $U_3\to u_3$ and $U_4\to u_4$.
\end{proof}

\subsection{Proof of \Cref{cor:xi_asymptotic}}
\label{app:thm:risk-asymptotics-isotropic-signal}

\begin{proof}[Proof of \Cref{cor:xi_asymptotic}]

Fix $\lambda > 0$.
Recall from \Cref{thm:risk-asymptotics} that the optimal mixing weight and optimal SD risk satisfy
\[
  \xi^\star(\lambda)
  \pto
  \frac{\sR(\lambda)-\sC(\lambda)}{\sR(\lambda)+\sR_{\pd}(\lambda)-2\sC(\lambda)},
  \qquad
  R_{\sd}^\star(\lambda)
  \pto
  \sR_{\sd}^\star(\lambda)
  :=
  \sR(\lambda)
  -\frac{(\sR(\lambda)-\sC(\lambda))^2}{\sR(\lambda)+\sR_{\pd}(\lambda)-2\sC(\lambda)}.
\]
In particular, writing $\sD(\lambda):=\sR(\lambda)+\sR_{\pd}(\lambda)-2\sC(\lambda) \ge 0$, we have $\sR_{\sd}^\star(\lambda)<\sR(\lambda)$ whenever $\sR(\lambda)-\sC(\lambda)\neq 0$ and $\sD(\lambda)>0$, and moreover $\sign(\xi^\star(\lambda))=\sign(\sR(\lambda)-\sC(\lambda))$ whenever $\sD(\lambda)>0$. 
Thus it remains to characterize the sign of $\sR(\lambda)-\sC(\lambda)$ under the isotropic signal prior.

As before, let $\hat\Sigma:=X^\top X/n$ and $Q_\lambda:=(\hat\Sigma+\lambda I_p)^{-1}$.
Under the isotropic prior $\beta\sim\cN(0,(r^2/p)I_p)$, for any random matrix $M$ (measurable with respect to $X$), $\EE[\beta^\top M\beta \mid X] = \tr(M\EE[\beta\beta^\top])= r^2 \tr(M) / p$.
Likewise, for a noise vector $\varepsilon$ independent of $X$ with $\EE[\varepsilon]=0$ and $\EE[\varepsilon\varepsilon^\top]=\sigma^2 I_n$, for any random matrix $M$, $\EE[\varepsilon^\top M \varepsilon \mid X]=\sigma^2 \tr(M)$.
Applying these $X$-conditional limits on the conditional risk and correlation expansions
\eqref{eq:R_simplified}--\eqref{eq:Rpd_simplified} yields (with $\op(1)$ remainders absorbed since they are mean-zero
after conditioning on $X$):
\begin{align}
\label{eq:R_isotropic_signal}
\barR_X(\lambda)
&:=
\EE\!\left[\barR(\lambda)\,\middle|\,X\right]
=
\frac{r^2}{p}\,\tr\!\big(\Sigma\,\lambda^2 Q_\lambda^2\big)
+\frac{\sigma^2}{n}\,\tr\!\big(\Sigma Q_\lambda (I_p-\lambda Q_\lambda)\big),
\\
\label{eq:Rpd_isotropic_signal}
\barR_{\pd,X}(\lambda)
&:=
\EE\!\left[\barR_{\pd}(\lambda)\,\middle|\,X\right]
=
\frac{r^2}{p}\,\tr\!\Big(\Sigma\,\lambda^2 Q_\lambda^2\,(2I_p-\lambda Q_\lambda)^2\Big)
+\frac{\sigma^2}{n}\,\tr\!\big(\Sigma Q_\lambda (I_p-\lambda Q_\lambda)^3\big),
\\
\label{eq:C_isotropic_signal}
\barC_X(\lambda)
&:=
\EE\!\left[\barC(\lambda)\,\middle|\,X\right]
=
\frac{r^2}{p}\,\tr\!\Big(\Sigma\,\lambda^2 Q_\lambda^2\,(2I_p-\lambda Q_\lambda)\Big)
+\frac{\sigma^2}{n}\,\tr\!\big(\Sigma Q_\lambda (I_p-\lambda Q_\lambda)^2\big).
\end{align}
Here $\barR,\barR_{\pd},\barC$ denote the corresponding excess components as above, i.e.,
$R(\lambda)=\sigma^2+\barR(\lambda)$, $R_{\pd}(\lambda)=\sigma^2+\barR_{\pd}(\lambda)$, and
$C(\lambda)=\sigma^2+\barC(\lambda)$.

Define the nonasymptotic critical value $\lambda^\star_{n,p}:=\frac{p}{n}\,\frac{\sigma^2}{r^2}$ so that $\lambda^\star_{n,p}\to \lambda^\star:=\gamma\frac{\sigma^2}{r^2}$.
A direct simplification of \eqref{eq:R_isotropic_signal} and \eqref{eq:C_isotropic_signal} gives
\begin{align*}
  \barR_X(\lambda)-\barC_X(\lambda)
  &=
  \lambda\Big(\frac{\sigma^2}{n}-\frac{r^2}{p}\lambda\Big)
  \Big(\tr(\Sigma Q_\lambda^2)-\lambda\,\tr(\Sigma Q_\lambda^3)\Big)\\
  &=
  \lambda(\lambda^\star_{n,p}-\lambda)\,\frac{r^2}{p}\,
  \tr\!\big(\Sigma Q_\lambda^2(I_p-\lambda Q_\lambda)\big)\\
  &=
  \lambda(\lambda^\star_{n,p}-\lambda)\,\frac{r^2}{p}\,
  \tr\!\big(\Sigma Q_\lambda^2\,\hat\Sigma Q_\lambda\big),
\end{align*}
where in the last step we used $I_p-\lambda Q_\lambda=\hat\Sigma Q_\lambda$, which follows from
$(\hat\Sigma+\lambda I_p)Q_\lambda=I_p$.
The trace factor is strictly positive for $\lambda>0$ with probability one because $\Sigma\succ 0$ and $Q_\lambda\succ 0$ for $\lambda>0$ and hence the trace vanishes if and only if $\hat\Sigma=0$ (which under under \Cref{def:dist} is a measure zero event).
Therefore, for each fixed $\lambda>0$, $\sign(\barR_X(\lambda)-\barC_X(\lambda)) = \sign(\lambda^\star_{n,p}-\lambda)$ with probability one.

Now, by the same concentration arguments of \Cref{lem:response_concentration} used to prove \Cref{thm:risk-asymptotics}, the random quantities $\barR(\lambda)-\barC(\lambda)$ concentrate around their conditional means, and $\barR(\lambda)-\barC(\lambda) - \big(\barR_X(\lambda)-\barC_X(\lambda)\big) \pto 0$.
Moreover, $\lambda^\star_{n,p}\to\lambda^\star$ and $\barR_X(\lambda)-\barC_X(\lambda)$ converges to
$\sR(\lambda)-\sC(\lambda)$ (since $\sigma^2$ cancels in the difference).
Hence, $\sR(\lambda)-\sC(\lambda)=0$ if and only if $\lambda=\lambda^\star$, and $\sign(\sR(\lambda)-\sC(\lambda)\big)=\sign(\lambda^\star-\lambda)$.
Finally, for $\lambda>0$, note that the limiting discrepancy $\sD(\lambda)=\sR(\lambda)+\sR_{\pd}(\lambda)-2\sC(\lambda)$ is
strictly positive with probability one in this case, as it is the limit of $\EE[(f_\lambda(x_0)-f_{\pd,\lambda}(x_0))^2 \mid \cD]$,
which is nondegenerate in this case (see also the proof of \Cref{thm:freshX_dominated_by_sameX_isotropic} for a more direct argument).
\end{proof}

\subsection{Proof of \Cref{prop:compare_extreme_lambda}}
\label{proof:prop:compare_extreme_lambda}

We provide two proofs.
The first evaluates the needed trace functionals by diagonalizing $\widehat\Sigma$ and invoking the Marchenko--Pastur law (including negative moments) in the limits $\lambda\to 0$ and $\lambda\to\infty$.
The second is a specialization of the general deterministic equivalent in \Cref{thm:risk-asymptotics} to $\Sigma=I_p$.

\subsubsection{First proof}

For the first approach, we characterize separately the extreme-$\lambda$ limits of the ridge risk $\sR(\lambda)$, the ridge-optimal risk $\sR^\star$, and the optimal SD risk $\sR_{\sd}^\star(\lambda)$, and then take ratios as in \Cref{prop:compare_extreme_lambda}.

We first recall some known results from the literature; see, e.g., \cite{dobriban2018high}.

\begin{lemma}
\label{lem:extreme_R}
Assume $\Sigma=I_p$ and $\beta\sim \cN(0,(r^2/p)I_p)$, and consider proportional asymptotics $p,n\to\infty$ with $p/n\to\gamma$.
Then the asymptotic ridge prediction risk $\sR(\lambda)$ satisfies:
\begin{align}
\lim_{\lambda \to 0} \sR(\lambda)
&=
\left\{\begin{matrix}
\frac{\sigma^2 \gamma}{1 - \gamma} + \sigma^2, & \gamma \in (0, 1),  \\
\frac{r^2(\gamma - 1)}{\gamma} + \frac{\sigma^2}{\gamma - 1} + \sigma^2, & \gamma \in (1, \infty),
\end{matrix} \right.
\\
\lim_{\lambda \to \infty} \sR(\lambda)
&= r^2 + \sigma^2,\quad \forall\,\gamma\in(0,\infty).
\end{align}
\end{lemma}

\begin{lemma}
\label{lem:extreme_R_star}
Assume $\Sigma=I_p$.
Then the (proportional) asymptotic ridge-optimal risk is:
\[
\sR^{\star}
=
\sigma^2 + \frac{1}{2 \gamma}  \left(
- \gamma \sigma^2 + r^2(\gamma - 1) + \sqrt{4 \gamma^2 r^2 \sigma^2 + (\gamma \sigma^2 - r^2(\gamma - 1))^2}
\right).
\]
\end{lemma}

Next we characterize the extreme-$\lambda$ limits of $\sR_{\sd}^\star(\lambda)$.

\begin{lemma}
\label{lem:R1_asymptotic_extreme}
Assume $\Sigma=I_p$ and $\beta\sim \cN(0,(r^2/p)I_p)$, and consider proportional asymptotics $p,n\to\infty$ with $p/n\to\gamma$.
Then the asymptotic optimal SD prediction risk $\sR_{\sd}^{\star}(\lambda)$ satisfies
\begin{align}
\lim_{\lambda \to 0} \sR_{\sd}^{\star}(\lambda)
&=
\left\{\begin{matrix}
\dfrac{r^2 \sigma^2 \gamma (1 - \gamma)^2 + \sigma^4 \gamma^3}{r^2 (1 - \gamma)^3 + \sigma^2 \gamma(1 - \gamma^2)} + \sigma^2, & \gamma \in (0, 1),  \\
\dfrac{r^4 (\gamma - 1)^4 + r^2 \sigma^2 \gamma (\gamma + 2)(\gamma - 1)^2 + \sigma^4 \gamma^2}{r^2 \gamma (\gamma - 1)^3 + \sigma^2 \gamma^2 (\gamma^2 - 1)} + \sigma^2, & \gamma \in (1, \infty),
\end{matrix} \right.
\\
\lim_{\lambda \to \infty} \sR_{\sd}^{\star}(\lambda)
&= \frac{r^4 \gamma + r^2 \sigma^2 \gamma}{r^2(\gamma+1) + \sigma^2 \gamma} + \sigma^2,\quad \gamma\in(0,\infty).
\end{align}
\end{lemma}

\begin{proof}[Proof of \Cref{lem:R1_asymptotic_extreme}]
As before, let $Q_{\lambda} := (\widehat{\Sigma} + \lambda  I_p)^{-1}$ with $\widehat{\Sigma} :=  X^{\top}  X/n$.
Recall from \eqref{eq:xi_star_and_Rsd_star_finite} the closed-form expression of the optimal SD risk, in terms of $X$-conditional quantities:
\begin{align*}
\sR_{\sd}^\star(\lambda)= \sigma^2 +  \lim_{n, p \to \infty} \left(\frac{\barR_{X}(\lambda) \barR_{\pd, X}(\lambda) - \barC_{X}(\lambda)^2}{\barR_{X}(\lambda)+\barR_{\pd, X}(\lambda)-2\barC_{X}(\lambda)} \right).
\end{align*}
Recall the expressions for these terms from equations \eqref{eq:R_isotropic_signal}--\eqref{eq:C_isotropic_signal}. Based on the assumption $\Sigma =  I$, we have
\begin{align*}
    &\barR_{X}(\lambda) \barR_{\pd, X}(\lambda)  - \barC_{X}(\lambda)^2
    \\
    &= 
    \frac{r^4}{p^2}
    \bigg[
    \tr\!\big(\lambda^2\, Q_{\lambda}^2 \big) 
    \tr\!\Big(  \lambda^2 Q_{\lambda}^2 (2  I - \lambda Q_{\lambda})^2 \Big)
    - 
    \tr \big( \lambda^2 Q_{\lambda}^2 (2 I - \lambda Q_{\lambda}) \big)^2
    \bigg]
    \\
    &+ 
    \frac{r^2 \sigma^2}{p n \lambda}
    \bigg[
    \tr\!\big( \lambda^2\, Q_{\lambda}^2 \big)
    \tr\!\big( \lambda Q_{\lambda} ( I - \lambda Q_{\lambda})^3 \big) 
    + 
    \tr\!\big( \lambda Q_{\lambda}\,( I - \lambda Q_{\lambda}) \, \big)
    \tr\!\Big(  \lambda^2 Q_{\lambda}^2 (2  I - \lambda Q_{\lambda})^2 \Big)
    \\
    &\hspace{7cm}
    - 2 \tr \big( \lambda^2 Q_{\lambda}^2 (2 I - \lambda Q_{\lambda}) \big) 
    \tr\!\big( \lambda Q_{\lambda} ( I -\lambda Q_{\lambda})^2 \big)
    \bigg]
    \\
    &+ \frac{\sigma^4}{n^2 \lambda^2}
    \bigg[
    \tr\!\big(\lambda Q_{\lambda}\,( I - \lambda Q_{\lambda}) \, \big) 
    \tr\!\big( \lambda Q_{\lambda} ( I - \lambda Q_{\lambda})^3 \big) 
    - \tr\!\big( \lambda Q_{\lambda} ( I -\lambda Q_{\lambda})^2 \big)^2
    \bigg],
    \\
    &:= \frac{r^4}{p^2} A_1 
    + \frac{r^2 \sigma^2}{pn \lambda} A_2 
    + \frac{\sigma^4}{n^2 \lambda^2} A_3,
    \\ %
    &\barR_{X}(\lambda) + \barR_{\pd, X}(\lambda) - 2\barC_{X}(\lambda)
    = \frac{r^2}{p} \tr \big( \lambda^2 Q_{\lambda}^2 ( I - \lambda Q_{\lambda})^2 \big) 
    + \frac{\sigma^2}{n \lambda} 
    \tr \big(  \lambda^3 Q_{\lambda}^3 ( I - \lambda Q_{\lambda}) \big).
\end{align*}

Denote the eigenvalues of the sample covariance matrix $\hSigma$ as $\{ s_i \}_{i=1}^{p}$. Then we can rewritten the traces in the above as:
\begin{align*}
    A_1 &= \tr\!\big( \lambda^2\, Q_{\lambda}^2 \big) 
    \tr\!\Big(  \lambda^2 Q_{\lambda}^2 (2  I - \lambda Q_{\lambda})^2 \Big)
    - 
    \tr \big( \lambda^2 Q_{\lambda}^2 (2 I - \lambda Q_{\lambda}) \big)^2
    \\
    &=
    \left( \sum_{i=1}^{p} \frac{\lambda^2}{(s_i + \lambda)^2} \right)
    \left( \sum_{i=1}^{p}  \frac{\lambda^2}{(s_i + \lambda)^2}
    \left( 2 - \frac{\lambda}{s_i + \lambda} 
    \right)^2
    \right)
    - 
    \left( \sum_{i=1}^{p}  \frac{\lambda^2}{(s_i + \lambda)^2}
    \left( 2 - \frac{\lambda}{s_i + \lambda} 
    \right) \right)^2
    \\
    &\overset{(1)}{=} \frac{1}{2}
    \sum_{i, j = 1}^{p}
    \left(
    \frac{\lambda}{s_i + \lambda} 
    \frac{\lambda}{s_j + \lambda}
    \left( 2 - \frac{\lambda}{s_j + \lambda} 
    \right)
    - 
    \frac{\lambda}{s_j + \lambda} 
    \frac{\lambda}{s_i + \lambda}
    \left( 2 - \frac{\lambda}{s_i + \lambda} 
    \right)
    \right)^2
    \\
    &= 
    \sum_{i, j = 1}^{p}
    \frac{1}{2}
    \frac{\lambda^2}{(s_i + \lambda)^2 (s_j + \lambda)^2}
    \left( \frac{\lambda}{s_i + \lambda}  - \frac{\lambda}{s_j + \lambda}   \right)^2 
    \\
    &= \frac{1}{2}
    \sum_{i, j = 1}^{p}
    \frac{\lambda^6 (s_i - s_j)^2}{(s_i + \lambda)^4 (s_j + \lambda)^4},
    \\ %
    A_2 &= 
    \left( \sum_{i=1}^{p} \frac{\lambda^2}{(s_i + \lambda)^2} \right)
    \left( \sum_{i=1}^{p} \frac{\lambda}{s_i + \lambda}
    \frac{s_i^3}{(s_i + \lambda)^3}
    \right)
    + 
    \left( \sum_{i=1}^{p} \frac{\lambda}{s_i + \lambda}
    \frac{s_i}{s_i + \lambda}
    \right)
    \left( \sum_{i=1}^{p} \frac{\lambda^2}{(s_i + \lambda)^2}
    \left( 2 - \frac{\lambda}{s_i + \lambda} \right)^2
    \right)
    \\
    &\hspace{2cm}
    - 
    2 
     \left( \sum_{i=1}^{p} \frac{\lambda^2}{(s_i + \lambda)^2}
    \left( 2 - \frac{\lambda}{s_i + \lambda} \right)
    \right)
    \left( \sum_{i=1}^{p} \frac{\lambda}{s_i + \lambda}
    \frac{s_i^2}{(s_i + \lambda)^2}
    \right)
    \\
    &\overset{(2)}{=}
    \sum_{i,j=1}^{p}
    \left(
    \frac{\lambda}{s_i + \lambda}
    \sqrt{\frac{\lambda}{s_j + \lambda}}
    \frac{s_j^{3/2}}{(s_j + \lambda)^{3/2}}
    - \frac{\lambda}{s_i + \lambda}
    \left( 2 - \frac{\lambda}{s_i + \lambda} \right)
    \sqrt{\frac{\lambda}{s_j + \lambda}
    \frac{s_j}{s_j + \lambda}}
    \right)^2
    \\
    &= 
    \sum_{i,j=1}^{p}
    \frac{\lambda^2}{(s_i + \lambda)^2}
     \frac{\lambda}{s_j + \lambda}
     \frac{s_j}{s_j + \lambda}
     \left( 
     \frac{s_j}{s_j + \lambda} - 2 + \frac{\lambda}{s_i + \lambda}
     \right)^2
     \\
     &= \sum_{i,j=1}^{p}
     \frac{\lambda^3 s_j}{(s_i + \lambda)^2 (s_j + \lambda)^2} 
     \left(  \frac{-\lambda}{s_j + \lambda} +  \frac{\lambda}{s_i + \lambda} - 1 \right)^2,
    \\ %
    A_3 
    &= 
    \left( \sum_{i=1}^{p} \frac{\lambda}{s_i + \lambda}
    \frac{s_i}{s_i + \lambda}
    \right) 
     \left( \sum_{i=1}^{p} \frac{\lambda}{s_i + \lambda}
    \frac{s_i^3}{(s_i + \lambda)^3}
    \right) 
    - 
     \left( \sum_{i=1}^{p} \frac{\lambda}{s_i + \lambda}
    \frac{s_i^2}{(s_i + \lambda)^2}
    \right)^2
    \\
    &\overset{(1)}{=}
    \frac{1}{2}
    \sum_{i,j=1}^{p}
    \left(
    \sqrt{\frac{\lambda}{s_i + \lambda} \frac{s_i}{s_i + \lambda}}
    \sqrt{\frac{\lambda}{s_j + \lambda}} \frac{s_j^{3/2}}{(s_j + \lambda)^{3/2}}
    - 
    \sqrt{\frac{\lambda}{s_j + \lambda} \frac{s_j}{s_j + \lambda}}
    \sqrt{\frac{\lambda}{s_i + \lambda}} \frac{s_i^{3/2}}{(s_i + \lambda)^{3/2}}
    \right)^2
    \\
    &= \frac{1}{2}
    \sum_{i,j=1}^{p}
    \frac{\lambda^2 s_i s_j}{(s_i + \lambda)^2 (s_j + \lambda)^2}
    \left(
    \frac{s_i}{s_i + \lambda} - 
     \frac{s_j}{s_j + \lambda} 
    \right)^2
    \\
    &= \frac{1}{2} 
      \sum_{i,j=1}^{p}
      \frac{\lambda^4 s_i s_j (s_i - s_j)^2}{(s_i + \lambda)^4 (s_j + \lambda)^4},
\end{align*}
and
\begin{align}
    \barR_{X}(\lambda) + \barR_{\pd, X}(\lambda) - 2 \barC_{X}(\lambda)
    &= 
    \frac{r^2}{p} \sum_{i=1}^{p}
    \frac{\lambda^2 s_i^2}{(s_i + \lambda)^4} + \frac{\sigma^2}{n} 
    \sum_{i=1}^{p}
    \frac{\lambda^2 s_i}{(s_i + \lambda)^4},
\end{align}
where $(1)$ and $(2)$ are from the following equalities:
\begin{align*}
    &\bigg( \sum_{i=1}^{p} x_i^2 \bigg)
    \bigg(
    \sum_{i=1}^{p} y_i^2 \bigg) - 
    \bigg(\sum_{i=1}^{p} x_i y_i \bigg)^2
    = \frac{1}{2} \sum_{i, j=1}^{p}
    \bigg(x_i y_j - x_j y_i \bigg)^2,
    \\
    &\bigg(\sum_{i=1}^{p} x_i^2 \bigg)
    \bigg(\sum_{i=1}^{p} y_i^2 \bigg)
    + 
    \bigg(\sum_{i=1}^{p} z_i^2 \bigg)
    \bigg(\sum_{i=1}^{p} t_i^2 \bigg)
    - 
    2 
    \bigg(\sum_{i=1}^{p} x_i t_i \bigg)
    \bigg(\sum_{i=1}^{p} y_i z_i \bigg)
    = \sum_{i, j=1}^{p}
    \bigg(x_i y_j - z_j t_i \bigg)^2.
\end{align*}

\textbf{Case 1: $\lambda \to \infty$ and $\gamma \in (0, \infty)$.}
For this case, we have
\begin{align*} 
    \lambda^2 \frac{r^4}{p^2} A_1 
    &= \frac{r^4}{2p^2}
    \sum_{i, j = 1}^{p}
    \frac{\lambda^8 (s_i - s_j)^2}{(s_i + \lambda)^4 (s_j + \lambda)^4}
    = 
    \frac{r^4}{2p^2}
    \sum_{i, j = 1}^{p}
    \frac{(s_i - s_j)^2}{(s_i/\lambda + 1)^4 (s_j / \lambda + 1)^4}
    \\
    &\overset{\lambda \to \infty}{\xrightarrow{\hspace{1cm}}}
    \frac{r^4}{2p^2} \sum_{i,j=1}^{p}
    (s_i - s_j)^2 
    \overset{n, p \to \infty}{\xrightarrow{\hspace{1cm}}}
    r^4 \gamma,
\end{align*}
where we used the fact that $\sum_{i,j=1}^{p}
    (s_i - s_j)^2 /p^2 = 
    2 \sum_i s_i^2 / p - 2 (\sum_i s_i/p)^2
    \to 2(1 + \gamma) - 2 = 2 \gamma$ based on the Marchenko--Pastur law. Similarly,
    \begin{align*}
        \lambda^2 \frac{r^2 \sigma^2}{p n \lambda} A_2
        &= 
        \frac{r^2 \sigma^2}{p n} 
        \sum_{i,j=1}^{p}
     \frac{\lambda^4 s_j}{(s_i + \lambda)^2 (s_j + \lambda)^2} 
     \left(  \frac{-\lambda}{s_j + \lambda} +  \frac{\lambda}{s_i + \lambda} - 1 \right)^2
     \\
     &= 
     \frac{r^2 \sigma^2}{p n} 
        \sum_{i,j=1}^{p}
     \frac{s_j}{(s_i / \lambda + 1)^2 (s_j / \lambda + 1)^2} 
     \left(  \frac{-1}{s_j/\lambda + 1} +  \frac{1}{s_i / \lambda + 1} - 1 \right)^2
     \\
     &\overset{\lambda \to \infty}{\xrightarrow{\hspace{1cm}}}
     \frac{r^2 \sigma^2}{p n}
       \sum_{i,j=1}^{p} s_j =  \frac{r^2 \sigma^2}{n}
       \sum_{j=1}^{p} s_j
    \overset{n, p \to \infty}{\xrightarrow{\hspace{1cm}}}
    r^2 \sigma^2 \gamma,
    \\ %
    \lambda^2 \frac{\sigma^4}{n^2 \lambda^2} A_3
    &= \frac{\sigma^2}{2n^2}
    \sum_{i,j=1}^{p}
      \frac{\lambda^4 s_i s_j (s_i - s_j)^2}{(s_i + \lambda)^4 (s_j + \lambda)^4} 
      \overset{\lambda \to \infty}{\xrightarrow{\hspace{1cm}}}
      0,
      \\ %
      \lambda^2 (\barR_{X}(\lambda) + \barR_{\pd, X}(\lambda) - 2\barC_{X}(\lambda))
      &= 
      \frac{r^2}{p} \sum_{i=1}^{p}
    \frac{\lambda^4 s_i^2}{(s_i + \lambda)^4} + \frac{\sigma^2}{n} 
    \sum_{i=1}^{p}
    \frac{\lambda^4 s_i}{(s_i + \lambda)^4}
    \\ %
    &\overset{\lambda \to \infty}{\xrightarrow{\hspace{1cm}}} 
    \frac{r^2}{p} \sum_{i=1}^{p} s_i^2 + 
    \frac{\sigma^2}{n} \sum_{i=1}^{p} s_i
    \overset{n, p \to \infty}{\xrightarrow{\hspace{1cm}}}
    \frac{r^2}{p}( 1 + \gamma) + \sigma^2 \gamma.
    \end{align*}
    From the above limits, we have
    \begin{align*}
        \lim_{\lambda \to \infty} \sR_{\sd}^\star(\lambda)= \sigma^2 +  \lim_{\lambda \to \infty} \lim_{ n, p \to \infty} \left(\frac{\lambda^{2}(\barR_{X}(\lambda) \barR_{\pd, X}(\lambda) - \barC_{X}(\lambda)^2)}{\lambda^{2}(\barR_{X}(\lambda)+\barR_{\pd, X}(\lambda)-2\barC_{X}(\lambda))} \right)
       = 
        \sigma^2 + \frac{r^4 \gamma + r^2 \sigma^2 \gamma}{r^2(1 + \gamma) + \sigma^2 \gamma}.
    \end{align*}

    \textbf{Case 2: $\lambda \to 0$ and $\gamma < 1$.} 
    First, we state the following results for negative moments of Marchenko--Pastur law. Assume $X \sim \text{MP}(\gamma)$ with $\gamma < 1$, then from \Cref{lem:MP_moment}, we have
    \begin{align*}
        \EE[X^{-1}] = \frac{1}{1 - \gamma}, 
        \quad
        \EE[X^{-2}] = \frac{1}{(1 - \gamma)^3},
        \quad
        \EE[X^{-3}] = \frac{1 + \gamma}{(1 - \gamma)^5}.
    \end{align*}
    Back to this case, we have
    \begin{align*}
        \lambda^{-2} \frac{r^4}{p^2} A_1
        &= \frac{r^4}{2p^2}
        \sum_{i, j = 1}^{p}
        \frac{\lambda^4 (s_i - s_j)^2}{(s_i + \lambda)^4 (s_j + \lambda)^4}
         \overset{\lambda \to 0}{\xrightarrow{\hspace{1cm}}} 
         0,
         \\ %
         \lambda^{-2} \frac{r^2 \sigma^2}{p n \lambda} A_2
         &=
         \frac{r^2 \sigma^2}{p n} 
        \sum_{i,j=1}^{p}
     \frac{s_j}{(s_i + \lambda)^2 (s_j + \lambda)^2} 
     \left(  \frac{-\lambda}{s_j + \lambda} +  \frac{\lambda}{s_i + \lambda} - 1 \right)^2
     \\
     &\overset{\lambda \to 0}{\xrightarrow{\hspace{1cm}}} 
     \frac{r^2 \sigma^2}{p n}
     \sum_{i,j=1}^{p} 
     \frac{1}{s_i^2 s_j}
     = 
     \frac{r^2 \sigma^2}{p n}
     \left( \sum_{i=1}^{p} \frac{1}{s_i^2}\right)
     \left( \sum_{i=1}^{p} \frac{1}{s_i} \right)
     \\
     &= \frac{r^2 \sigma^2 p}{n} 
      \left( \frac{1}{p} \sum_{i=1}^{p} \frac{1}{s_i^2}\right)
     \left(  \frac{1}{p}  \sum_{i=1}^{p} \frac{1}{s_i}\right)
     \overset{n, p \to \infty}{\xrightarrow{\hspace{1cm}}}
     \frac{r^2 \sigma^2 \gamma}
     {(1 - \gamma)^4},
     \\ %
     \lambda^{-2} \frac{\sigma^4}{n^2 \lambda^2} A_3
     &= \frac{\sigma^4}{2n^2}
     \sum_{i,j=1}^{p}
     \frac{s_i s_j(s_i - s_j)^2}{(s_i + \lambda)^4 (s_j + \lambda)^4}
     \overset{\lambda \to 0}{\xrightarrow{\hspace{1cm}}}
     \frac{\sigma^4}{2n^2}
     \sum_{i,j=1}^{p}
     \frac{(s_i - s_j)^2}{s_i^3 s_j^3}
     \\
     &= 
     \frac{\sigma^4}{n^2}
     \left(
    \left( \sum_{i=1}^p \frac{1}{s_i} \right)
    \left( \sum_{i=1}^p \frac{1}{s_i^3} \right)
    -
    \left( \sum_{i=1}^p \frac{1}{s_i^2} \right)^2
     \right)
     \\
     &\overset{n, p \to \infty}{\xrightarrow{\hspace{1cm}}}
     \sigma^4 \gamma^2 
     \left( \frac{1 + \gamma}{(1 - \gamma)^6} - \frac{1}{(1 - \gamma)^6} \right)
     = \sigma^4 \gamma^2 \frac{\gamma}{(1 - \gamma)^6},
     \\ %
    \lambda^{-2} (\barR_{X}(\lambda) + \barR_{\pd, X}(\lambda) - 2\barC_{X}(\lambda))
      &= 
      \frac{r^2}{p} \sum_{i=1}^{p}
    \frac{s_i^2}{(s_i + \lambda)^4} + \frac{\sigma^2}{n} 
    \sum_{i=1}^{p}
    \frac{s_i}{(s_i + \lambda)^4}
    \\ %
    &\overset{\lambda \to 0}{\xrightarrow{\hspace{1cm}}} 
    \frac{r^2}{p} \sum_{i=1}^{p} \frac{1}{s_i^2} + 
    \frac{\sigma^2}{n} \sum_{i=1}^{p} 
    \frac{1}{s_i^3}
    \overset{n, p \to \infty}{\xrightarrow{\hspace{1cm}}}
    \frac{r^2}{(1 - \gamma)^3} 
    + \frac{\sigma^2 \gamma ( 1 + \gamma)}{(1 - \gamma)^5}
    \end{align*}
    From the above limits, we have
    \begin{align*}
        \lim_{\lambda \to 0} \sR_{\sd}^\star(\lambda)= \sigma^2 +  \lim_{\lambda \to 0} \lim_{n, p \to \infty}  \left(\frac{\lambda^{-2}(\barR_{X}(\lambda) \barR_{\pd, X}(\lambda) - \barC_{X}(\lambda)^2)}{\lambda^{-2}(\barR_{X}(\lambda)+\barR_{\pd, X}(\lambda)-2\barC_{X}(\lambda))} \right)
       = 
        \sigma^2 + \frac{r^2 \sigma^2 \gamma (1 - \gamma)^2 + \sigma^4 \gamma^3}{r^2 (1 - \gamma)^3 + \sigma^2 \gamma (1 - \gamma^2)}.
    \end{align*}

    \textbf{Case 3: $\lambda \to 0$ and $\gamma > 1$.}
    For this case, the sample covariance is at most rank $n < p$, so we have $s_i = 0$ for $i > n$. 
    Thus
    \begin{align*}
        \lambda^{-2} \frac{r^4}{p^2} A_1
        &= \frac{r^4}{2p^2}
        \sum_{i, j = 1}^{p}
        \frac{\lambda^4 (s_i - s_j)^2}{(s_i + \lambda)^4 (s_j + \lambda)^4}
        \\
        &= \frac{r^4}{2 p^2}
        \left(
        \sum_{i = 1}^{n}
        \sum_{j = 1}^{n}
        \frac{\lambda^4 (s_i - s_j)^2}{(s_i + \lambda)^4 (s_j + \lambda)^4}
        +
        2 
        \sum_{i = 1}^{n}
        \sum_{j = n + 1}^{p}
        \frac{\lambda^4 (s_i - s_j)^2}{(s_i + \lambda)^4 (s_j + \lambda)^4}
        \right)
        \\
        &= 
        \frac{r^4}{2 p^2}
        \left(
        \sum_{i = 1}^{n}
        \sum_{j = 1}^{n}
        \frac{\lambda^4 (s_i - s_j)^2}{(s_i + \lambda)^4 (s_j + \lambda)^4}
        +
        2 (p - n)
        \sum_{i = 1}^{n}
        \frac{\lambda^4 s_i^2}{(s_i + \lambda)^4 \lambda^4}
        \right)
        \\
         &\overset{\lambda \to 0}{\xrightarrow{\hspace{1cm}}} 
        \frac{r^4}{2 p^2}
        \left(
        0
        +
        2(p-n) 
        \sum_{i=1}^{n}
        \frac{1}{s_i^2}
        \right)
        =
        r^4 \frac{p - n}{p} 
        \left( \frac{1}{p} \sum_{i=1}^{n} 
        \frac{1}{s_i^2} \right)
        \\
        &\overset{n, p \to \infty}{\xrightarrow{\hspace{1cm}}}
        r^4 \left(1 - \frac{1}{\gamma} \right)
        \frac{1}{(\gamma - 1)^3} =
        \frac{r^4}{\gamma(\gamma - 1)^2},
    \\ %
     \lambda^{-2} \frac{r^2 \sigma^2}{p n \lambda} A_2
     &=
     \frac{r^2 \sigma^2}{p n} 
     \sum_{i,j=1}^{p}
     \frac{s_j}{(s_i + \lambda)^2 (s_j + \lambda)^2} 
     \left(  \frac{-\lambda}{s_j + \lambda} +  \frac{\lambda}{s_i + \lambda} - 1 \right)^2
     \\
     &=
     \frac{r^2 \sigma^2}{p n} 
     \left(
     \sum_{i = 1}^{n}
     \sum_{j = 1}^{n}
     \frac{s_j}{(s_i + \lambda)^2 (s_j + \lambda)^2} 
     \left(  \frac{-\lambda}{s_j + \lambda} +  \frac{\lambda}{s_i + \lambda} - 1 \right)^2
     +
     \sum_{i = n+1}^{p}
     \sum_{j = 1}^{n}
     \frac{s_j}{\lambda^2 (s_j + \lambda)^2} 
     \frac{\lambda^2}{(s_j+\lambda)^2}
     \right)
     \\
     &\overset{\lambda \to 0}{\xrightarrow{\hspace{1cm}}} 
     \frac{r^2 \sigma^2}{p n}
     \left(
     \sum_{i=1}^{n}  \sum_{j=1}^{n}
     \frac{1}{s_i^2 s_j}
     +
     (p-n) \sum_{j=1}^{n} \frac{1}{s_j^3}
     \right)
     \\
     &= \frac{r^2 \sigma^2 p}{n} 
     \bigg[
     \left( \frac{1}{p}  \sum_{i=1}^{n} \frac{1}{s_i^2} \right)
     \left(  \frac{1}{p}  \sum_{i=1}^{n} \frac{1}{s_i}\right)
     + 
     \frac{p - n}{p} 
     \left( \frac{1}{p}  \sum_{i=1}^{n} \frac{1}{s_i^3}
     \right)
     \bigg]
     \\&
     \overset{n, p \to \infty}{\xrightarrow{\hspace{1cm}}}
    r^2 \sigma^2 \gamma
    \bigg[ 
    \frac{1}{\gamma ( \gamma - 1)^4} + 
    \left(1 - \frac{1}{\gamma} \right)
    \frac{\gamma + 1}{(\gamma - 1)^5}
    \bigg]
    =
    \frac{r^2 \sigma^2 (\gamma+2)}{(\gamma - 1)^4},
     \\ %
     \lambda^{-2} \frac{\sigma^4}{n^2 \lambda^2} A_3
     &= \frac{\sigma^4}{2n^2}
     \sum_{i,j=1}^{p}
     \frac{s_i s_j(s_i - s_j)^2}{(s_i + \lambda)^4 (s_j + \lambda)^4}
     =
     \frac{\sigma^4}{2n^2}
     \sum_{i,j=1}^{n}
     \frac{s_i s_j(s_i - s_j)^2}{(s_i + \lambda)^4 (s_j + \lambda)^4}
     \\
     &\overset{\lambda \to 0}{\xrightarrow{\hspace{1cm}}}
     \frac{\sigma^4}{2n^2}
     \sum_{i,j=1}^{n}
     \frac{(s_i - s_j)^2}{s_i^3 s_j^3}
     \\
     &= 
     \frac{\sigma^4 p^2 }{n^2}
     \left(
    \left( \frac{1}{p} \sum_{i=1}^n \frac{1}{s_i} \right)
    \left( \frac{1}{p} \sum_{i=1}^n \frac{1}{s_i^3} \right)
    -
    \left( \frac{1}{p} \sum_{i=1}^n \frac{1}{s_i^2} \right)^2
     \right)
     \\
     &\overset{n, p \to \infty}{\xrightarrow{\hspace{1cm}}}
     \sigma^4 \gamma^2 
     \left( \frac{\gamma + 1}{\gamma(\gamma - 1)^6} - \frac{1}{(\gamma - 1)^6} \right)
     = \frac{\sigma^4 \gamma}{(\gamma - 1)^6},
    \end{align*}
    where we used the following facts about the MP law's negative moments when $\gamma > 1$:
    \begin{align*}
        \frac{1}{p}
        \sum_{i=1}^{p} \frac{1}{s_i}
        \to \frac{1}{\gamma ( \gamma - 1)},
        \quad
        \frac{1}{p}
        \sum_{i=1}^{p} \frac{1}{s_i^2}
        \to \frac{1}{(\gamma - 1)^3}
        \quad
        \frac{1}{p}
        \sum_{i=1}^{p} \frac{1}{s_i^3}
        \to 
        \frac{\gamma + 1}{(\gamma - 1)^5}.
    \end{align*}
    Similarly, we have
    \begin{align*}
    \lambda^{-2} (\barR_{X}(\lambda) + \barR_{\pd, X}(\lambda) - 2\barC_{X}(\lambda))
      &= 
      \frac{r^2}{p} \sum_{i=1}^{p}
    \frac{s_i^2}{(s_i + \lambda)^4} + \frac{\sigma^2}{n} 
    \sum_{i=1}^{p}
    \frac{s_i}{(s_i + \lambda)^4}
    \\
    &= \frac{r^2}{p} \sum_{i=1}^{n} 
     \frac{s_i^2}{(s_i + \lambda)^4} + \frac{\sigma^2}{n} 
    \sum_{i=1}^{n}
    \frac{s_i}{(s_i + \lambda)^4}
    \\
    &\overset{\lambda \to 0}{\xrightarrow{\hspace{1cm}}} 
    \frac{r^2}{p} \sum_{i=1}^{n} \frac{1}{s_i^2}
    + \frac{\sigma^2}{n} \sum_{i=1}^{n} \frac{1}{s_i^3}
    \\
    & \overset{n, p \to \infty}{\xrightarrow{\hspace{1cm}}}
    \frac{r^2}{(\gamma - 1)^3} + \frac{\sigma^2 \gamma (\gamma + 1)}{(\gamma - 1)^5}.
    \end{align*}
    Thus, for $\gamma > 1$, we have
    \begin{align*}
        \lim_{\lambda \to 0} \sR_{\sd}^{\star}(\lambda) = \sigma^2 
        + \frac{r^4 (\gamma - 1)^4 + r^2 \sigma^2 \gamma (\gamma + 2)(\gamma - 1)^2 + \sigma^4 \gamma^2}{r^2 \gamma (\gamma - 1)^3 + \sigma^2 \gamma^2 (\gamma^2 - 1)}.
    \end{align*}

Finally, the ratios in \Cref{prop:compare_extreme_lambda} are obtained by plugging in the asymptotic limits of the corresponding risks using \Cref{lem:extreme_R,lem:extreme_R_star,lem:R1_asymptotic_extreme}.
\end{proof}

\begin{lemma}
    \label{lem:MP_moment}
    Assume the random variable $X \sim \text{MP}(\gamma)$ with $\gamma < 1$.
    Then we have
    \begin{align*}
        \EE[X^{-1}] = \frac{1}{1 - \gamma}, 
        \quad
        \EE[X^{-2}] = \frac{1}{(1 - \gamma)^3},
        \quad
        \EE[X^{-3}] = \frac{1 + \gamma}{(1 - \gamma)^5}.
    \end{align*}
\end{lemma}

\begin{proof}
    Denote $b = (1 + \sqrt{\gamma})^2$ and $a = (1 - \sqrt{\gamma})^2$. Let $R := \sqrt{(b - x)(x - a)}$,
    from Equation 2.265 of \citet{gradshteyn2007}, for $m \geq 2$ we have
    \begin{align}
        \int_{a}^{b} 
        \frac{\sqrt{R}}{x^m} dx
        = \frac{(2m - 5)(b + a)}{2(m-1) ba} \int_{a}^{b} \frac{\sqrt{R}}{x^{m-1}} dx
        - \frac{m - 4}{(m-1)ba} \int_{a}^{b} \frac{\sqrt{R}}{x^{m-2}} dx.
    \end{align} 
    Recall the density of the MP law $f(x) = \frac{\sqrt{R}}{2 \pi \gamma x}$ for $x \in (a, b)$. For $k \geq 1$, using the above identity, we have
    \begin{align*}
        \EE(X^{-k}) 
        &= \int_{a}^{b} \frac{1}{x^k} f(x) dx
        = \frac{1}{2 \pi \gamma} 
        \int_{a}^{b} \frac{\sqrt{R}}{x^{k+1}} dx
        \\
        &= 
        \frac{(2k-3)(1 + \gamma)}{k( 1 - \gamma)^2} \EE(X^{-k + 1}) 
        - \frac{k-3}{k ( 1 - \gamma)^2} \EE(X^{-k + 2}).
    \end{align*}
    Plug-in $k = 1, 2, 3$ and noting that $\EE(X) = 1$, we obtain the desired results.
\end{proof}

\subsubsection{Second proof}

For the second approach, we will specialize the general deterministic equivalent in \Cref{thm:risk-asymptotics} to the isotropic case.

\begin{proof}[Proof of \Cref{prop:compare_extreme_lambda}]
We specialize \Cref{thm:risk-asymptotics} to $\Sigma = I_p$.
In this case, we have $G=(I_p+\kappa I_p)^{-1}=(1+\kappa)^{-1}I_p$, and thus
\[
t_k=\gamma\,\otr(I_p^2 G^k)=\frac{\gamma}{(1+\kappa)^k},
\qquad
q_k=\beta^\top G^k I_p\,\beta=\frac{\|\beta\|_2^2}{(1+\kappa)^k}.
\]
With $\beta\sim \cN(0,(r^2/p)I_p)$, $\|\beta\|_2^2\to r^2$ in probability, so $q_k\to r^2/(1+\kappa)^k$ in probability.
The fixed-point equation becomes
\begin{equation}
\label{eq:kappa_fp_iso}
\kappa=\lambda+\frac{\gamma\kappa}{1+\kappa},
\end{equation}
whose unique nonnegative solution is
\begin{equation}
\label{eq:kappa_closed_form_iso}
\kappa(\lambda)
=\frac{1}{2}\Big((\lambda+\gamma-1)+\sqrt{(\lambda+\gamma-1)^2+4\lambda}\Big).
\end{equation}
In particular,
\[
\kappa(\lambda)\xrightarrow[\lambda\to\infty]{}\infty,
\qquad
\kappa(\lambda)\xrightarrow[\lambda \to 0]{}(\gamma-1)_+.
\]

Let $\bar \sR_{\sd}^{\star}(\lambda):=\sR_{\sd}^{\star}(\lambda)-\sigma^2$ denote the excess prediction risk.
Substituting the isotropic specializations into the general formula and simplifying yields the explicit rational form
(as a function of $\kappa=\kappa(\lambda)$):
\begin{equation}
\label{eq:Rsd_star_iso_kappa}
\bar \sR_{\sd}^{\star}(\lambda)
=
\frac{
\gamma\Big(
r^4\kappa^4
+ r^2\sigma^2(1+\kappa)^4
+ \gamma^2\sigma^2(r^2+\sigma^2)
-2\gamma r^2\sigma^2(2\kappa+1)
\Big)
}{
\big((1+\kappa)^2-\gamma\big)
\Big(
r^2\big((1+\gamma)(1+\kappa)^2-4\gamma(1+\kappa)+\gamma(1+\gamma)\big)
+\sigma^2\gamma\big((1+\kappa)^2+\gamma\big)
\Big)
}.
\end{equation}

\textbf{Case 1: $\lambda\to\infty$.}
Using $\kappa(\lambda)\to\infty$, the leading terms in \eqref{eq:Rsd_star_iso_kappa} give
\[
\lim_{\lambda\to\infty}\bar \sR_{\sd}^{\star}(\lambda)
=
\frac{\gamma r^2(r^2+\sigma^2)}{r^2(\gamma+1)+\gamma\sigma^2},
\]
and adding back $\sigma^2$ and writing in terms of $\SNR$ yields the stated $\lambda\to\infty$ limit.

\textbf{Case 2: $\lambda\to 0$.}
If $\gamma\in(0,1)$ then $\kappa(\lambda)\to 0$; substituting $\kappa=0$ in \eqref{eq:Rsd_star_iso_kappa} yields
the stated $\gamma<1$ limit.
If $\gamma\in(1,\infty)$ then $\kappa(\lambda)\to \gamma-1$; substituting $\kappa=\gamma-1$ yields the stated $\gamma>1$ limit.
Adding back $\sigma^2$ and expressing in terms of $\SNR$ completes the proof.
\end{proof}

\section{Proofs in \Cref{sec:tuning}}
\label{app:tuning_proofs}

\subsection{Preliminaries}
\label{app:tuning_setup}

Fix $\lambda>0$. Let $X\in\R^{n\times p}$ and $y\in\R^n$, and define the sample covariance
$\widehat\Sigma:=X^\top X/n$ and ridge resolvent $Q_\lambda:=(\widehat\Sigma+\lambda I_p)^{-1}$.
Define the ridge hat (smoother) matrix
\[
  H_\lambda := XQ_\lambda X^\top/n \in \R^{n\times n}.
\]
The ridge teacher coefficient and fitted values are
\[
  \beta_\lambda := Q_\lambda X^\top y/n,
  \qquad
  \hat y_\lambda := X\beta_\lambda = H_\lambda y.
\]
Define the pure-distilled ridge coefficient (ridge refit on pseudo-labels $\hat y_\lambda$)
\[
  \beta_{\pd,\lambda}
  :=
  Q_\lambda X^\top \hat y_\lambda/n,
  \qquad
  \hat y_{\pd,\lambda}:=X\beta_{\pd,\lambda}.
\]
Since $\hat y_\lambda=H_\lambda y$ and $H_\lambda=XQ_\lambda X^\top/n$, a direct calculation yields the identity used repeatedly below:
\begin{equation}
\label{eq:pd_hat_matrix_identity}
  \hat y_{\pd,\lambda}=H_\lambda^2\,y.
\end{equation}
Thus both the teacher and PD predictors are linear smoothers in $y$, with smoothing matrices
$H_\lambda$ and $H_\lambda^2$, respectively. Consequently, their degrees of freedom equal the traces
of these matrices:
\[
  \df_\lambda = \tr(H_\lambda),
  \qquad
  \df_{\pd,\lambda}=\tr(H_\lambda^2).
\]
Recall the GCV residuals and risk/correlation estimators defined in
\Cref{eq:gcv_residuals_compact,eq:risk_estimators_compact}:
\[
  \hat r_\lambda:=\frac{y-\hat y_\lambda}{1-\df_\lambda/n},
  \qquad
  \hat r_{\pd,\lambda}:=\frac{y-\hat y_{\pd,\lambda}}{1-\df_{\pd,\lambda}/n},
\]
and
\[
  \widehat R(\lambda):=\frac{\|\hat r_\lambda\|_2^2}{n},\quad
  \widehat R_{\pd}(\lambda):=\frac{\|\hat r_{\pd,\lambda}\|_2^2}{n},\quad
  \widehat C(\lambda):=\frac{\langle \hat r_\lambda,\hat r_{\pd,\lambda}\rangle}{n}.
\]

Let $(x_0,y_0)$ be an independent test pair distributed as in \Cref{def:dist} and independent of the training data.
Define the out-of-sample errors
\[
  e_\lambda := y_0 - x_0^\top\beta_\lambda,
  \qquad
  e_{\pd,\lambda} := y_0 - x_0^\top\beta_{\pd,\lambda},
\]
and the corresponding conditional (oracle) quantities
\[
  R(\lambda):=\E[e_\lambda^2\mid X,y],\qquad
  R_{\pd}(\lambda):=\E[e_{\pd,\lambda}^2\mid X,y],\qquad
  C(\lambda):=\E[e_\lambda e_{\pd,\lambda}\mid X,y].
\]
Finally define
\begin{equation}
\label{eq:D_def}
  D(\lambda):=R(\lambda)+R_{\pd}(\lambda)-2C(\lambda)
  =\E\big[(e_\lambda-e_{\pd,\lambda})^2\mid X,y\big]\ge 0.
\end{equation}

We will use the following ridge functional estimation result (Theorem~4 of \citet{patil2022estimating}):

\begin{lemma}[Ridge functional estimation via GCV/LOOCV]
\label{thm:prt_functional}
Assume the conditions of \Cref{def:dist} and $p/n\to\gamma$. Fix $\lambda>0$.
Let $t:\R\to\R$ be continuous and satisfy a quadratic growth condition.
Define $T_\lambda:=\E[t(e_\lambda)\mid X,y]$.
Let $T^{\gcv}_\lambda$ and $T^{\loo}_\lambda$ be the plug-in estimators based on the usual ridge
GCV and LOOCV residuals.
Then
\[
  T^{\gcv}_\lambda - T_\lambda \pto 0,
  \qquad
  T^{\loo}_\lambda - T_\lambda \pto 0.
\]
\end{lemma}

\subsection{Helper Results (Risk and Correlation Components Consistency) for the Proof of \Cref{thm:oneshot_consistency}}
\label{app:tuning_pd_risk_corr}

\begin{lemma}[Consistency of the ridge-risk estimator]
\label{prop:consistency_R}
Under \Cref{def:dist} and $p/n\to\gamma$, for each fixed $\lambda>0$,
\[
  \widehat R(\lambda) - R(\lambda) \pto 0.
\]
\end{lemma}

\begin{proof}
This simply follows from \Cref{thm:prt_functional} with $t(z)=z^2$.
Then $T_\lambda=\E[e_\lambda^2\mid X,y]=R(\lambda)$, and the ridge GCV plug-in estimator is precisely
$\widehat R(\lambda)=\|\hat r_\lambda\|_2^2/n$.
\end{proof}

\begin{lemma}[Consistency of the pure-distilled risk and correlation estimators]
\label{prop:consistency_Rpd_C}
Under \Cref{def:dist} and $p/n\to\gamma$, for each fixed $\lambda>0$,
\[
  \widehat R_{\pd}(\lambda)-R_{\pd}(\lambda)\pto 0,
  \qquad
  \widehat C(\lambda)-C(\lambda)\pto 0.
\]
\end{lemma}

\begin{proof}
We follow the same three-part argument used by \citet[Proof of Theorem~3]{patil2022estimating}:
(i) relate the target risks/correlation to leave-one-out (LOO) errors,
(ii) show a stability approximation that replaces true LOO errors by diagonal-corrected residuals
computed from the full-sample smoother, and
(iii) replace the diagonal correction by the GCV trace correction using diagonal--trace equivalence.

\emph{Part (i): LOO concentration around $R_{\pd}(\lambda)$ and $C(\lambda)$.}
Let $\beta_{\lambda}^{(-i)}$ and $\beta_{\pd,\lambda}^{(-i)}$ be the teacher and PD coefficients trained
without sample $i$, and define the corresponding LOO prediction errors
\[
  e_{i,\lambda}^{\loo}:=y_i-x_i^\top\beta_{\lambda}^{(-i)},
  \qquad
  e_{i,\pd,\lambda}^{\loo}:=y_i-x_i^\top\beta_{\pd,\lambda}^{(-i)}.
\]
Since $(x_i,y_i)$ is independent of $\sigma\{(x_j,y_j):j\neq i\}$ and identically distributed to $(x_0,y_0)$,
we have
\[
  \E[(e_{i,\pd,\lambda}^{\loo})^2\mid X^{(-i)},y^{(-i)}] = R_{\pd}^{(-i)}(\lambda),
  \qquad
  \E[e_{i,\lambda}^{\loo}e_{i,\pd,\lambda}^{\loo}\mid X^{(-i)},y^{(-i)}] = C^{(-i)}(\lambda),
\]
where $R_{\pd}^{(-i)}(\lambda)$ and $C^{(-i)}(\lambda)$ denote the analogues of $R_{\pd}(\lambda)$ and
$C(\lambda)$ for the $(n-1)$-sample-trained predictors.

Using the same martingale-difference decomposition and Burkholder inequality arguments as in
\citet[Supplement~S.1.2]{patil2022estimating}, together with the moment bounds in \Cref{def:dist},
one obtains
\[
  \left|
  \frac1n\sum_{i=1}^n (e_{i,\pd,\lambda}^{\loo})^2
  -
  \frac1n\sum_{i=1}^n R_{\pd}^{(-i)}(\lambda)
  \right|\pto 0,
  \qquad
  \left|
  \frac1n\sum_{i=1}^n e_{i,\lambda}^{\loo}e_{i,\pd,\lambda}^{\loo}
  -
  \frac1n\sum_{i=1}^n C^{(-i)}(\lambda)
  \right|\pto 0.
\]
Moreover, by LOO stability of ridge and the PD stability in \Cref{lem:pd_stability} (see Part (ii) below),
\[
  \frac1n\sum_{i=1}^n \big(R_{\pd}^{(-i)}(\lambda)-R_{\pd}(\lambda)\big)\pto 0,
  \qquad
  \frac1n\sum_{i=1}^n \big(C^{(-i)}(\lambda)-C(\lambda)\big)\pto 0,
\]
because (conditionally on the training data) the test risks are Lipschitz in the coefficient vector:
for instance,
\[
  \big|R_{\pd}^{(-i)}(\lambda)-R_{\pd}(\lambda)\big|
  \;=\;
  \big|\E[(x_0^\top(\beta_{\pd,\lambda}^{(-i)}-\beta_{\pd,\lambda}))^2\mid X,y]\big|
  \;\le\;
  \|\Sigma\|_{\oper}\,\|\beta_{\pd,\lambda}^{(-i)}-\beta_{\pd,\lambda}\|_2^2.
\]
Combining yields
\begin{equation}
\label{eq:loo_targets}
  \frac1n\sum_{i=1}^n (e_{i,\pd,\lambda}^{\loo})^2 \pto R_{\pd}(\lambda),
  \qquad
  \frac1n\sum_{i=1}^n e_{i,\lambda}^{\loo}e_{i,\pd,\lambda}^{\loo} \pto C(\lambda).
\end{equation}

\emph{Part (ii): Equivalence of diagonal-corrected residuals and LOO errors.}
Define the diagonal-corrected residuals based on the full-sample smoothing matrices:
\[
  \bar r_{\lambda,i}:=\frac{y_i-(H_\lambda y)_i}{1-(H_\lambda)_{ii}},
  \qquad
  \bar r_{\pd,i}:=\frac{y_i-(H_\lambda^2 y)_i}{1-(H_\lambda^2)_{ii}}.
\]
For ridge, the shortcut formula is exact: $e_{i,\lambda}^{\loo}=\bar r_{\lambda,i}$ for all $i$.
For the PD two-stage procedure, leaving out one sample changes both the teacher and student fits;
however, by the stability statement in \Cref{lem:pd_stability} (together with ridge stability),
these perturbations are negligible on average. In particular, one obtains
\begin{equation}
\label{eq:pd_loo_diag_approx}
  \frac1n\sum_{i=1}^n\big(e_{i,\pd,\lambda}^{\loo}-\bar r_{\pd,i}\big)^2 \pto 0.
\end{equation}
Combining \eqref{eq:loo_targets} and \eqref{eq:pd_loo_diag_approx} yields
\begin{equation}
\label{eq:diag_targets}
  \frac1n\sum_{i=1}^n \bar r_{\pd,i}^2 \pto R_{\pd}(\lambda),
  \qquad
  \frac1n\sum_{i=1}^n \bar r_{\lambda,i}\bar r_{\pd,i} \pto C(\lambda),
\end{equation}
where for the cross term we use $\bar r_{\lambda,i}=e_{i,\lambda}^{\loo}$ exactly and Cauchy--Schwarz.

\emph{Part (iii): Equivalence of trace-corrected residuals to diagonal-corrected residuals.}
By \Cref{lem:H2_diag_trace},
\[
  \max_{1\le i\le n}\left|
    \frac{1}{1-\tr(H_\lambda^2)/n}
    -\frac{1}{1-(H_\lambda^2)_{ii}}
  \right|\pto 0,
\]
and the analogous ridge diagonal--trace equivalence (see \citet[Supplement~S.1.3]{patil2022estimating})
gives the same statement with $H_\lambda$ in place of $H_\lambda^2$.
Since $\hat r_{\pd,\lambda,i}=(y_i-(H_\lambda^2y)_i)/(1-\tr(H_\lambda^2)/n)$ and
$\hat r_{\lambda,i}=(y_i-(H_\lambda y)_i)/(1-\tr(H_\lambda)/n)$, it follows that
\[
  \frac1n\sum_{i=1}^n(\hat r_{\pd,\lambda,i}-\bar r_{\pd,i})^2\pto 0,
  \qquad
  \frac1n\sum_{i=1}^n(\hat r_{\lambda,i}-\bar r_{\lambda,i})^2\pto 0.
\]
Combining with \eqref{eq:diag_targets} and using Cauchy--Schwarz yields
\[
  \frac1n\sum_{i=1}^n \hat r_{\pd,\lambda,i}^2 \pto R_{\pd}(\lambda),
  \qquad
  \frac1n\sum_{i=1}^n \hat r_{\lambda,i}\hat r_{\pd,\lambda,i} \pto C(\lambda),
\]
which is exactly $\widehat R_{\pd}(\lambda)\pto R_{\pd}(\lambda)$ and
$\widehat C(\lambda)\pto C(\lambda)$.
\end{proof}

\subsection{Proof of Theorem~\ref{thm:oneshot_consistency}}
\label{app:tuning_thm_proof}

\begin{proof}[Proof of Theorem~\ref{thm:oneshot_consistency}]
By \Cref{prop:consistency_R,prop:consistency_Rpd_C},
\[
  \big(\widehat R(\lambda),\widehat R_{\pd}(\lambda),\widehat C(\lambda)\big)
  \pto
  \big(R(\lambda),R_{\pd}(\lambda),C(\lambda)\big).
\]
In particular,
\[
  \widehat D(\lambda)
  :=\widehat R(\lambda)+\widehat R_{\pd}(\lambda)-2\widehat C(\lambda)
  \pto
  D(\lambda):=R(\lambda)+R_{\pd}(\lambda)-2C(\lambda).
\]
Under the proportional regime, \Cref{thm:risk-asymptotics} implies that $D(\lambda)$ converges in probability
to a deterministic limit $\sD(\lambda)$ (given by $\sR(\lambda)+\sR_{\pd}(\lambda)-2\sC(\lambda)$).
In all nondegenerate settings, $\sD(\lambda)>0$ (equivalently, the teacher and PD predictions do not coincide
asymptotically), hence $D(\lambda)$ and $\widehat D(\lambda)$ are bounded away from $0$ in probability.
Therefore, the plug-in maps
\[
  (a,b,c)\mapsto \frac{a-c}{a+b-2c},
  \qquad
  (a,b,c)\mapsto a-\frac{(a-c)^2}{a+b-2c}
\]
are continuous with high probability in a neighborhood of $(R(\lambda),R_{\pd}(\lambda),C(\lambda))$,
and the continuous mapping theorem yields
\[
  \hat\xi^\star(\lambda)-\xi^\star(\lambda)\pto 0,
  \qquad
  \widehat R_{\sd}^\star(\lambda)-R_{\sd}^\star(\lambda)\pto 0.
\]
(If $\sD(\lambda)=0$, then the oracle SD objective is asymptotically flat in $\xi$; in this case
$R_{\sd}^\star(\lambda)=R(\lambda)$ and the risk-consistency conclusion remains valid, while $\xi^\star(\lambda)$
is not identifiable.)
\end{proof}

\subsection{Technical Lemmas}
\label{app:tuning_pd_tech}

We establish two technical ingredients for the pure-distilled smoother $H_\lambda^2$:
(i) leave-one-out stability and (ii) diagonal-trace equivalence.

\subsubsection{Leave-One-Out Stability}

For $i\in\{1,\dots,n\}$, let $(X^{(-i)},y^{(-i)})$ denote the dataset with observation $i$ removed and
$\widehat\Sigma^{(-i)}:=(X^{(-i)})^\top X^{(-i)}/(n-1)$.
Let $Q_\lambda^{(-i)}:=(\widehat\Sigma^{(-i)}+\lambda I_p)^{-1}$ and
$\beta_\lambda^{(-i)}:=Q_\lambda^{(-i)}(X^{(-i)})^\top y^{(-i)}/(n-1)$.
Define the PD coefficient trained without $i$ by
\[
  \beta_{\pd,\lambda}^{(-i)}
  :=
  Q_\lambda^{(-i)}(X^{(-i)})^\top \hat y_\lambda^{(-i)}/(n-1),
  \qquad
  \hat y_\lambda^{(-i)}:=X^{(-i)}\beta_\lambda^{(-i)}.
\]

\begin{lemma}[Operator-Lipschitz property of $M_\lambda(\cdot)$]
\label{lem:M_lipschitz}
For $\lambda>0$, define $M_\lambda(A):=A(A+\lambda I)^{-1}=I-\lambda(A+\lambda I)^{-1}$.
Then for any symmetric $A,B\succeq 0$,
\[
  \|M_\lambda(A)-M_\lambda(B)\|_{\oper}\le \frac{1}{\lambda}\,\|A-B\|_{\oper}.
\]
\end{lemma}

\begin{proof}
Using $M_\lambda(A)=I-\lambda(A+\lambda I)^{-1}$ and the resolvent identity,
\[
  M_\lambda(A)-M_\lambda(B)
  =-\lambda\big[(A+\lambda I)^{-1}-(B+\lambda I)^{-1}\big]
  =-\lambda(A+\lambda I)^{-1}(B-A)(B+\lambda I)^{-1}.
\]
Taking operator norms and using $\|(A+\lambda I)^{-1}\|_{\oper}\le 1/\lambda$ and
$\|(B+\lambda I)^{-1}\|_{\oper}\le 1/\lambda$ yields the claim.
\end{proof}

\begin{lemma}[Average stability of the pure-distilled coefficient]
\label{lem:pd_stability}
Under \Cref{def:dist} and $p/n\to\gamma$, for each fixed $\lambda>0$,
\[
  \frac1n\sum_{i=1}^n \big\|\beta_{\pd,\lambda}-\beta_{\pd,\lambda}^{(-i)}\big\|_2^2 \pto 0.
\]
\end{lemma}

\begin{proof}
Write $\beta_{\pd,\lambda}=M_\lambda(\widehat\Sigma)\beta_\lambda$ and
$\beta_{\pd,\lambda}^{(-i)}=M_\lambda(\widehat\Sigma^{(-i)})\beta_\lambda^{(-i)}$.
Then
\[
  \beta_{\pd,\lambda}-\beta_{\pd,\lambda}^{(-i)}
  =
  M_\lambda(\widehat\Sigma)\big(\beta_\lambda-\beta_\lambda^{(-i)}\big)
  +
  \big(M_\lambda(\widehat\Sigma)-M_\lambda(\widehat\Sigma^{(-i)})\big)\beta_\lambda^{(-i)}.
\]
Using $(a+b)^2\le 2a^2+2b^2$ and $\|Au\|_2\le \|A\|_{\oper}\|u\|_2$, we obtain
\begin{align}
\label{eq:pd_stab_split}
\frac1n\sum_{i=1}^n \|\beta_{\pd,\lambda}-\beta_{\pd,\lambda}^{(-i)}\|_2^2
&\le
\frac{2}{n}\sum_{i=1}^n \|M_\lambda(\widehat\Sigma)\|_{\oper}^2\,\|\beta_\lambda-\beta_\lambda^{(-i)}\|_2^2
+
\frac{2}{n}\sum_{i=1}^n
\|(M_\lambda(\widehat\Sigma)-M_\lambda(\widehat\Sigma^{(-i)}))\beta_\lambda^{(-i)}\|_2^2.
\end{align}
Since $\|M_\lambda(\cdot)\|_{\oper}\le 1$, the first term in \eqref{eq:pd_stab_split} is bounded by
$\frac{2}{n}\sum_i \|\beta_\lambda-\beta_\lambda^{(-i)}\|_2^2$, which converges to $0$ in probability by the
ridge stability result of \citet{patil2022estimating} (used there to control LOOCV/GCV).

For the second term, use the exact identity
$M_\lambda(A)=I-\lambda(A+\lambda I)^{-1}$ to write
\[
  M_\lambda(\widehat\Sigma)-M_\lambda(\widehat\Sigma^{(-i)})
  =
  -\lambda\big(Q_\lambda-Q_\lambda^{(-i)}\big).
\]
By the resolvent identity,
\[
  Q_\lambda-Q_\lambda^{(-i)}
  =
  Q_\lambda(\widehat\Sigma^{(-i)}-\widehat\Sigma)Q_\lambda^{(-i)}.
\]
Moreover,
\(
  \widehat\Sigma = \frac{n-1}{n}\widehat\Sigma^{(-i)}+\frac{1}{n}x_i x_i^\top
\),
so $\widehat\Sigma^{(-i)}-\widehat\Sigma=\frac{1}{n}\big(\widehat\Sigma^{(-i)}-x_i x_i^\top\big)$.
Define $v_i:=Q_\lambda^{(-i)}\beta_\lambda^{(-i)}$. Then
\[
  (M_\lambda(\widehat\Sigma)-M_\lambda(\widehat\Sigma^{(-i)}))\beta_\lambda^{(-i)}
  =
  \frac{\lambda}{n}\,Q_\lambda\big(x_i x_i^\top-\widehat\Sigma^{(-i)}\big)v_i.
\]
Since $\|Q_\lambda\|_{\oper}\le 1/\lambda$, it suffices to control
$n^{-1}\sum_i \|(x_i x_i^\top-\widehat\Sigma^{(-i)})v_i\|_2^2/n^2$.
Condition on $\sigma\{(x_j,y_j):j\neq i\}$, under which $v_i$ is deterministic and independent of $x_i$.
Using $\|x_i x_i^\top v_i\|_2^2=\|x_i\|_2^2(x_i^\top v_i)^2$ and bounded-spectrum $\Sigma$ from \Cref{def:dist},
\[
  \E\!\left[\|x_i x_i^\top v_i\|_2^2 \,\middle|\, X^{(-i)},y^{(-i)}\right]
  =
  \E[\|x_i\|_2^2(x_i^\top v_i)^2\mid X^{(-i)},y^{(-i)}]
  \;\lesssim\; p\,\|v_i\|_2^2,
\]
while $\|\widehat\Sigma^{(-i)}v_i\|_2^2\lesssim \|v_i\|_2^2$ since $\|\widehat\Sigma^{(-i)}\|_{\oper}=O_p(1)$ for
proportional random design.
Therefore,
\[
  \E\!\left[\|(x_i x_i^\top-\widehat\Sigma^{(-i)})v_i\|_2^2 \,\middle|\, X^{(-i)},y^{(-i)}\right]
  \;\lesssim\; p\,\|v_i\|_2^2,
\]
and hence
\[
  \E\!\left[\|(M_\lambda(\widehat\Sigma)-M_\lambda(\widehat\Sigma^{(-i)}))\beta_\lambda^{(-i)}\|_2^2\right]
  \;\lesssim\; \frac{p}{n^2}\,\E[\|v_i\|_2^2].
\]
Finally, $\|v_i\|_2=\|Q_\lambda^{(-i)}\beta_\lambda^{(-i)}\|_2\le \|Q_\lambda^{(-i)}\|_{\oper}\|\beta_\lambda^{(-i)}\|_2
\le \lambda^{-1}\|\beta_\lambda^{(-i)}\|_2$, and for fixed $\lambda>0$ we have $\|\beta_\lambda^{(-i)}\|_2=O_p(1)$
uniformly in $i$ under \Cref{def:dist}. Since $p/n\to\gamma$, we have $p/n^2=O(1/n)$, and averaging over $i$ yields
\[
  \frac{1}{n}\sum_{i=1}^n
  \|(M_\lambda(\widehat\Sigma)-M_\lambda(\widehat\Sigma^{(-i)}))\beta_\lambda^{(-i)}\|_2^2
  \pto 0.
\]
Together with the first term in \eqref{eq:pd_stab_split}, this proves the claim.
\end{proof}

\subsubsection{Diagonal-Trace Equivalence}

Let $S:=XX^\top/n\in\R^{n\times n}$ and $G_\lambda:=(S+\lambda I_n)^{-1}$.
Recall the standard identity
\[
  H_\lambda = S(S+\lambda I_n)^{-1}=I_n-\lambda G_\lambda,
\]
so
\begin{equation}
\label{eq:H2_poly_G}
  H_\lambda^2 = (I_n-\lambda G_\lambda)^2
  = I_n - 2\lambda G_\lambda + \lambda^2 G_\lambda^2.
\end{equation}

\begin{lemma}[Diagonal-trace equivalence for $H_\lambda^2$]
\label{lem:H2_diag_trace}
Fix $\lambda>0$. Under \Cref{def:dist} and $p/n\to\gamma$,
\[
  \max_{1\le i\le n}\left|(H_\lambda^2)_{ii} - \frac1n\tr(H_\lambda^2)\right| \xrightarrow{p} 0,
\]
and hence
\[
  \max_{1\le i\le n}\left|
    \frac{1}{1-\tr(H_\lambda^2)/n}
    -\frac{1}{1-(H_\lambda^2)_{ii}}
  \right| \xrightarrow{p} 0.
\]
\end{lemma}

\begin{proof}
By \eqref{eq:H2_poly_G}, for each $i$,
\[
  (H_\lambda^2)_{ii} - \frac1n\tr(H_\lambda^2)
  =
  -2\lambda\Big((G_\lambda)_{ii}-\frac1n\tr(G_\lambda)\Big)
  +\lambda^2\Big((G_\lambda^2)_{ii}-\frac1n\tr(G_\lambda^2)\Big).
\]
Thus it suffices to prove
\[
  \max_i\Big|(G_\lambda)_{ii}-\frac1n\tr(G_\lambda)\Big|\xrightarrow{p}0,
  \qquad
  \max_i\Big|(G_\lambda^2)_{ii}-\frac1n\tr(G_\lambda^2)\Big|\xrightarrow{p}0.
\]

\begin{enumerate}
    \item 
    \emph{Control for $G_\lambda$:}
    Since $1-(H_\lambda)_{ii}=\lambda(G_\lambda)_{ii}$ and
    $1-\tr(H_\lambda)/n=\lambda\,\tr(G_\lambda)/n$, the diagonal--trace equivalence for ridge
    (see \citet[Supplement~S.1.3]{patil2022estimating}) implies
    $\max_i|(G_\lambda)_{ii}-\tr(G_\lambda)/n|\pto 0$.

    \item
    \emph{Control for $G_\lambda^2$:}
    Fix any $t>0$ and define $G_{\lambda+t}:=(S+(\lambda+t)I_n)^{-1}$.
    Since $G_\lambda$ and $G_{\lambda+t}$ are both functions of $S$, they commute, and we have the exact identity
    \begin{equation}
    \label{eq:G2_id}
      G_\lambda^2
      =
      \frac{G_\lambda-G_{\lambda+t}}{t}
      +t\,G_{\lambda+t}\,G_\lambda^2.
    \end{equation}
    Taking diagonal--trace differences and maxima gives
    \[
    \max_i\Big|(G_\lambda^2)_{ii}-\tfrac1n\tr(G_\lambda^2)\Big|
    \le
    \frac{1}{t}\max_i\Big|(G_\lambda)_{ii}-\tfrac1n\tr(G_\lambda)\Big|
    +
    \frac{1}{t}\max_i\Big|(G_{\lambda+t})_{ii}-\tfrac1n\tr(G_{\lambda+t})\Big|
    +
    2t\,\|G_{\lambda+t}G_\lambda^2\|_{\oper}.
    \]
    Since $\|G_\lambda\|_{\oper}\le 1/\lambda$ and $\|G_{\lambda+t}\|_{\oper}\le 1/(\lambda+t)$,
    the last term is bounded by $2t/\lambda^3$.
    Letting $n\to\infty$ and using Step~1 at $\lambda$ and at $\lambda+t$ yields
    \[
    \limsup_{n\to\infty}
    \max_i\Big|(G_\lambda^2)_{ii}-\tfrac1n\tr(G_\lambda^2)\Big|
    \;\le\; \frac{2t}{\lambda^3}
    \qquad\text{in probability.}
    \]
    Since $t>0$ was arbitrary, sending $t\to 0$ gives the desired diagonal--trace control for $G_\lambda^2$.
\end{enumerate}

This proves the first claim. For the second claim, note that $0\preceq H_\lambda^2\preceq I_n$, so
$1-(H_\lambda^2)_{ii}>0$ and $1-\tr(H_\lambda^2)/n>0$.
Moreover, for fixed $\lambda>0$ these denominators are bounded away from $0$ with high probability.
Hence
\[
\left|
\frac{1}{1-\tr(H_\lambda^2)/n}-\frac{1}{1-(H_\lambda^2)_{ii}}
\right|
=
\frac{
\left|(H_\lambda^2)_{ii}-\tr(H_\lambda^2)/n\right|
}{
\big(1-\tr(H_\lambda^2)/n\big)\big(1-(H_\lambda^2)_{ii}\big)
},
\]
and the second claim follows from the first.
\end{proof}

\section{Proofs in \Cref{sec:extensions}}
\label{app:extensions_proofs}

\subsection{Proofs and Details in \Cref{sec:multiple-rounds}}
\label{app:multiround_proofs}

\subsubsection{Proof of \Cref{eq:recursive-affine-path}}
\label{sec:recursive-affine-path-proof}

\begin{proof}[Proof of \Cref{eq:recursive-affine-path}]
Fix $k\ge 0$ and $\xi\in\RR$.
Recall that the round-$(k+1)$ SD predictor
$f^{(k+1)}_{\sd,\lambda,\xi_{k+1}}$ is defined as:
\begin{equation}
\label{eq:sd_predictor_multiround}
  f^{(k+1)}_{\sd,\lambda,\xi}(x)
  :=
  x^\top\argmin_{\beta\in\RR^p}
  \big\{
  (1-\xi)\|y^{(k)}-X\beta\|_2^2/n
  +\xi\|\hat y^{(k)}_{\lambda}-X\beta\|_2^2/n
  +\lambda\|\beta\|_2^2
  \big\}.
\end{equation}
Expand the objective in \eqref{eq:sd_predictor_multiround} (dropping constants independent of $\beta$):
\begin{align*}
  &(1-\xi)\|y^{(k)}-X\beta\|_2^2/n+\xi\|\hat y^{(k)}_{\lambda}-X\beta\|_2^2/n+\lambda\|\beta\|_2^2 \\
  &\equiv
  \|(1-\xi)y^{(k)}+\xi\hat y^{(k)}_{\lambda}-X\beta\|_2^2/n+\lambda\|\beta\|_2^2
  \;=\;
  \|y^{(k+1)}_{\lambda,\xi}-X\beta\|_2^2/n+\lambda\|\beta\|_2^2,
\end{align*}
where $y^{(k+1)}_{\lambda,\xi}=(1-\xi)y^{(k)}+\xi\hat y^{(k)}_{\lambda}$ is exactly the mixed-label vector
in \eqref{eq:recursive_multiround}.
Thus $f^{(k+1)}_{\sd,\lambda,\xi}$ coincides with ridge regression trained on $(X,y^{(k+1)}_{\lambda,\xi})$ with penalty $\lambda$.
Since ridge is linear in its response vector, we have
\[
  f^{(k+1)}_{\sd,\lambda,\xi}
  =
  (1-\xi)\,f^{(k)}_\lambda+\xi\,f^{(k)}_{\pd,\lambda},
\]
because $f^{(k)}_\lambda$ is ridge on response $y^{(k)}$ and $f^{(k)}_{\pd,\lambda}$ is ridge on response
$\hat y^{(k)}_{\lambda}=f^{(k)}_\lambda(X)$.
This is \eqref{eq:recursive-affine-path}.
\end{proof}

\subsubsection{Proof of \Cref{prop:multiround_monotone_main}}

\begin{proof}[Proof of \Cref{prop:multiround_monotone_main}]
Fix $k\ge 0$.
By \Cref{eq:recursive-affine-path}, the family $\{f^{(k+1)}_{\sd,\lambda,\xi}:\xi\in\RR\}$ is the affine path between
$f^{(k)}_\lambda$ and $f^{(k)}_{\pd,\lambda}$.

\emph{Monotonicity.}
Since $\xi=0$ is feasible and $f^{(k+1)}_{\sd,\lambda,0}=f^{(k)}_\lambda$, we have
\[
  R_{k+1}(\lambda)
  =\min_{\xi\in\RR} R\!\big(f^{(k+1)}_{\sd,\lambda,\xi}\big)
  \le R\!\big(f^{(k+1)}_{\sd,\lambda,0}\big)
  =R\!\big(f^{(k)}_\lambda\big)
  =R_k(\lambda).
\]

\emph{Closed forms.}
When $D_k(\lambda)>0$, applying \Cref{prop:closed_form_xi_R1} to the affine path between
$f^{(k)}_\lambda$ and $f^{(k)}_{\pd,\lambda}$ yields the displayed round-wise closed forms
(with $(R,C,D)$ replaced by $(R_k,C_k,D_k)$).
Likewise, \eqref{eq:multiround_tangent_rule} follows by applying \Cref{thm:tangent_sign} at round $k$,
again with $(f_\lambda,f_{\pd,\lambda},R,C,D)$ replaced by $(f^{(k)}_\lambda,f^{(k)}_{\pd,\lambda},R_k,C_k,D_k)$,
and interpreting $R_k'(\lambda)$ as the derivative with respect to the ridge penalty while holding the base labels
$y^{(k)}$ fixed.
\end{proof}

\subsection{Proofs and Details in \Cref{sec:new-features}}
\label{sec:supp_freshX_isotropic}

\subsubsection{Fresh-X Mixed-Loss Student Representation}
\label{sec:freshX-mixed-nonaffine}

Recall the fresh-$X$ mixed-loss SD predictor $f_{\sd,\lambda,\xi}^{\freshmix}$ from \eqref{eq:mixed_loss}.
Writing $f_{\sd,\lambda,\xi}^{\freshmix}(x)=x^\top \beta_{\sd,\lambda}^{\freshmix}(\xi)$, we derive an explicit
closed form for $\beta_{\sd,\lambda}^{\freshmix}(\xi)$ and show that, unlike the same-$X$ setting, the map
$\xi\mapsto f_{\sd,\lambda,\xi}^{\freshmix}$ is generally not an affine path.

Let $\hat\Sigma:=X^\top X/n$, $\tilde\Sigma:=\tilde X^\top\tilde X/m$, $\hat v:=X^\top y/n$, and define
the ridge resolvents:
\[
  Q_\lambda := (\hat\Sigma+\lambda I_p)^{-1},
  \qquad
  \tilde Q_\lambda := (\tilde\Sigma+\lambda I_p)^{-1}.
\]
The ridge teacher coefficient is $\beta_\lambda=Q_\lambda \hat v$.
The fresh-$X$ PD refit trained on $(\tilde X,\tilde y_\lambda)$ with $\tilde y_\lambda=\tilde X\beta_\lambda$
has coefficient
\[
  \beta_{\pd,\lambda}^{\fresh}
  :=(\tilde\Sigma+\lambda I_p)^{-1}\frac{1}{m}\tilde X^\top \tilde y_\lambda
  =(\tilde\Sigma+\lambda I_p)^{-1}\tilde\Sigma\,\beta_\lambda
  =\tilde Q_\lambda\,\tilde\Sigma\,\beta_\lambda.
\]
(Here $\tilde Q_\lambda\,\tilde\Sigma=\tilde\Sigma\,\tilde Q_\lambda$ since both are functions of $\tilde\Sigma$.)

\begin{lemma}[Fresh-$X$ mixed-loss student representation]
\label{lem:fresh_mixedloss_matrix_weighted}
Fix $\lambda>0$ and let $\xi\in\RR$ be such that the mixed-loss objective \eqref{eq:mixed_loss} is strictly
convex in $\beta$, equivalently  $A_\xi := (1-\xi)\hat\Sigma+\xi\tilde\Sigma+\lambda I_p \succ 0$.
(In particular, this holds for all $\xi\in[0,1]$.) Then the coefficient of the fresh-$X$ mixed-loss student
satisfies
\begin{equation}
\label{eq:beta_fresh_closed_form}
  \beta_{\sd,\lambda}^{\freshmix}(\xi)
  =
  A_\xi^{-1}\big((1-\xi)\hat v+\xi\,\tilde\Sigma\,\beta_\lambda\big)
  =
  \big((1-\xi)Q_\lambda^{-1}+\xi \tilde Q_\lambda^{-1}\big)^{-1}
  \big((1-\xi)\hat v+\xi\,\tilde\Sigma\,\beta_\lambda\big).
\end{equation}
Moreover, for $\xi\notin\{0,1\}$, writing $S:=Q_\lambda\,\tilde Q_\lambda^{-1}$ (so $S^{-1}=\tilde Q_\lambda Q_\lambda^{-1}$),
we have the matrix-weighted decomposition
\begin{equation}
\label{eq:beta_fresh_matrix_affine}
  \beta_{\sd,\lambda}^{\freshmix}(\xi)
  =
  \Big(I+\tfrac{\xi}{1-\xi}S\Big)^{-1}\beta_\lambda
  +
  \Big(I+\tfrac{1-\xi}{\xi}S^{-1}\Big)^{-1}\beta_{\pd,\lambda}^{\fresh},
\end{equation}
with the boundary cases $\xi\in\{0,1\}$ obtained by continuity.
\end{lemma}
\begin{proof}
Expanding \eqref{eq:mixed_loss} and setting the gradient with respect to $\beta$ to zero yields the normal equations
\[
  \big((1-\xi)\hat\Sigma+\xi\tilde\Sigma+\lambda I_p\big)\beta
  =
  (1-\xi)\hat v+\xi\,\tilde\Sigma\,\beta_\lambda,
\]
which gives the first identity in \eqref{eq:beta_fresh_closed_form}; the second follows from
$Q_\lambda^{-1}=\hat\Sigma+\lambda I_p$ and $\tilde Q_\lambda^{-1}=\tilde\Sigma+\lambda I_p$.
For \eqref{eq:beta_fresh_matrix_affine}, substitute $\hat v=Q_\lambda^{-1}\beta_\lambda$ and
$\tilde\Sigma\,\beta_\lambda=\tilde Q_\lambda^{-1}\beta_{\pd,\lambda}^{\fresh}$ into \eqref{eq:beta_fresh_closed_form},
then split the two terms and factor $Q_\lambda^{-1}$ (respectively $\tilde Q_\lambda^{-1}$) from the left.
\end{proof}

The upshot of \Cref{lem:fresh_mixedloss_matrix_weighted} is that, unless $S=I_p$ (equivalently $Q_\lambda=\tilde Q_\lambda$),
the map $\xi\mapsto \beta_{\sd,\lambda}^{\freshmix}(\xi)$ is not a scalar affine combination
$(1-\xi)\beta_\lambda+\xi\beta_{\pd,\lambda}^{\fresh}$.
In particular, $\xi\mapsto f_{\sd,\lambda,\xi}^{\freshmix}$ is generally not an affine path, unlike the same-$X$
case in \eqref{eq:affine_path_pred}.

\subsubsection[Same-X SD under Isotropic Setting]{Same-$X$ SD under Isotropic Setting}
\label{sec:sameX_isotropic_specialization}

Here we specialize the general deterministic equivalents from
\Cref{thm:risk-asymptotics} to the isotropic setting $\Sigma=I_p$ with isotropic signal $\beta\sim\cN(0,(r^2/p)I_p)$.

Define the companion Stieltjes transform $v=v(\lambda)>0$ as the unique solution of
\begin{equation}
\label{eq:v_fp_iso}
  \frac{1}{v}
  \;=\;
  \lambda + \frac{\gamma}{1+v},
  \qquad \lambda>0,
\end{equation}
and set $\kappa(\lambda):=1/v(\lambda)$.
In the isotropic case, these admit the explicit closed forms
\[
  \kappa(\lambda)
  =\frac{(\lambda+\gamma-1)+\sqrt{(\lambda+\gamma-1)^2+4\lambda}}{2},
  \qquad
  v(\lambda)
  =\frac{\sqrt{(\lambda+\gamma-1)^2+4\lambda}-(\lambda+\gamma-1)}{2\lambda}.
\]

Since $\Sigma=I_p$, we have $G=(\Sigma+\kappa I_p)^{-1}=(1+\kappa)^{-1}I_p$. Hence for $k\in\{2,3,4\}$,
\begin{align}
  t_k
  &:= \gamma\,\frac{\tr(\Sigma^2 G^k)}{p}
  = \gamma\Big(\frac{1}{1+\kappa}\Big)^k
  = \gamma\Big(\frac{v}{1+v}\Big)^k,
  \label{eq:tq_iso}
  \\
  q_k
  &:= \E\!\big[\beta^\top G^k\Sigma\,\beta\big]
  = r^2\Big(\frac{1}{1+\kappa}\Big)^k
  = r^2\Big(\frac{v}{1+v}\Big)^k .
  \nonumber
\end{align}
(Equivalently, $q_k=\lim_{p\to\infty}\beta^\top G^k\Sigma\,\beta$ in probability, by isotropy of $\beta$.)

Define
\[
  b := (1-t_2)^{-1},
  \qquad
  E := \kappa - b\lambda + b^2 \kappa \lambda t_3,
\]
the variance-trace limits
\[
  u_2 := t_2 b,
  \qquad
  u_3 := t_3 b^3,
  \qquad
  u_4 := t_4 b^4 + 2 t_3^2 b^5,
\]
and
\[
  a_2:=bE^2+b^4\kappa^2\lambda^2 t_4+b^5\kappa^2\lambda^2 t_3^2,
  \qquad
  a_3:=2b^2\kappa\lambda E,
  \qquad
  a_4:=b^3\kappa^2\lambda^2.
\]

\begin{lemma}[Isotropic limits for same-$X$ SD]
\label{lem:sameX_isotropic_limits}
Fix $\lambda>0$.
Under \Cref{def:dist} with $\Sigma=I_p$ and $\beta\sim\cN(0,(r^2/p)I_p)$, as $n,p\to\infty$ with
$p/n\to\gamma\in(0,\infty)$, the same-$X$ optimal SD risk satisfies
\[
  R_{\sd}^{\star,\same}(\lambda)
  \pto
  \sR_{\sd}^{\star,\same}(\lambda)
  :=
  \sR(\lambda)
  -\frac{\big(\sR(\lambda)-\sC^{\same}(\lambda)\big)^2}{\sR(\lambda)+\sR_{\pd}^{\same}(\lambda)-2\sC^{\same}(\lambda)}.
\]
Here $\sR(\lambda)$, $\sR_{\pd}^{\same}(\lambda)$, and $\sC^{\same}(\lambda)$ are the deterministic limits of the
(one-round) teacher risk $R(\lambda)$, same-$X$ PD risk $R_{\pd}^{\same}(\lambda)$, and residual correlation
$C^{\same}(\lambda)$, respectively (in the same-$X$ setting of \Cref{sec:structural}, these coincide with
$R(\lambda)$, $R_{\pd}(\lambda)$, and $C(\lambda)$), given by:
\begin{align*}
  \sR(\lambda)
  &=
  \sigma^2
  +\kappa^2 b\,q_2 + \sigma^2 u_2,
  \\
  \sC^{\same}(\lambda)
  &=
  \sigma^2
  +2\kappa^2 b\,q_2-\big(\kappa bE\,q_2+\kappa^2 b^2\lambda\,q_3\big)
     + \sigma^2\,(u_2-\lambda u_3),
  \\
  \sR_{\pd}^{\same}(\lambda)
  &=
  \sigma^2
  +4\kappa^2 b\,q_2-2\big(2\kappa bE\,q_2+2\kappa^2 b^2\lambda\,q_3\big)
     +(a_2q_2+a_3q_3+a_4q_4)
     + \sigma^2\,(u_2-2\lambda u_3+\lambda^2 u_4).
\end{align*}
\end{lemma}

\begin{proof}
This follows by specializing \Cref{thm:risk-asymptotics} to $\Sigma=I_p$ and isotropic $\beta$.
\end{proof}

For an illustration of the empirical versus theoretical risks in \Cref{lem:sameX_isotropic_limits}, see the second
panel of \Cref{fig:same_x_fresh_x_affine}.

\subsubsection[Fresh-X (affine) SD under Isotropic Setting]{Fresh-$X$ (affine) SD under Isotropic Setting}
\label{sec:freshX_affine_isotropic}

Throughout, we work in the balanced proportional regime $m,n,p\to\infty$ with
$p/m\to\gamma$ and $p/n\to\gamma$, and under the isotropic assumptions $\Sigma=I_p$ and
$\beta\sim\cN(0,(r^2/p)I_p)$ from \Cref{sec:supp_freshX_isotropic}.
Recall the companion Stieltjes transform $v(\lambda)$ and $\kappa(\lambda)=1/v(\lambda)$ from
\Cref{sec:sameX_isotropic_specialization}, as well as $b(\lambda)=\kappa'(\lambda)=(1-t_2(\lambda))^{-1}$.

Let $\tilde\Sigma:=\tilde X^\top\tilde X/m$ and define the (fresh) shrinkage matrix $M_\lambda := \tilde\Sigma(\tilde\Sigma+\lambda I_p)^{-1}$.
Let
\[
  s(\lambda):=\lim_{m,p\to\infty}\frac{1}{p}\tr(M_\lambda),
  \qquad
  s_2(\lambda):=\lim_{m,p\to\infty}\frac{1}{p}\tr(M_\lambda^2).
\]
Writing $\tilde S:=\tilde X\tilde X^\top/m$ and $\tilde G_\lambda:=(\tilde S+\lambda I_m)^{-1}$, we have
$\tilde S(\tilde S+\lambda I_m)^{-1}=I_m-\lambda \tilde G_\lambda$, and therefore
\[
  \frac{1}{p}\tr(M_\lambda)
  =\frac{m}{p}\cdot \frac{1}{m}\tr\!\big(\tilde S(\tilde S+\lambda I_m)^{-1}\big)
  =\frac{1}{\gamma}\Big(1-\lambda\,\frac{1}{m}\tr(\tilde G_\lambda)\Big).
\]
Since $p/m\to\gamma$, we have $\frac{1}{m}\tr(\tilde G_\lambda)\to v(\lambda)$, where $v(\lambda)$ solves \eqref{eq:v_fp_iso}. 
Hence
\begin{equation}
\label{eq:s_iso}
  s(\lambda)=\frac{1-\lambda v(\lambda)}{\gamma}
  =\frac{1-\lambda/\kappa(\lambda)}{\gamma}.
\end{equation}
Similarly, using $M_\lambda=I_p-\lambda(\tilde\Sigma+\lambda I_p)^{-1}$ and the identity
$\frac{1}{m}\tr(\tilde G_\lambda^2)=-v'(\lambda)$, we obtain
\begin{equation}
\label{eq:s2_iso}
  s_2(\lambda)
  = \frac{1-2\lambda v(\lambda) -\lambda^2 v'(\lambda)}{\gamma}
  = \frac{1-2\lambda v(\lambda) + \lambda^2 b(\lambda) v(\lambda)^2}{\gamma}
  =\frac{1-2\lambda/\kappa(\lambda)+\lambda^2 b(\lambda)/\kappa(\lambda)^2}{\gamma}.
\end{equation}

\begin{lemma}[Isotropic limits for fresh-$X$ (affine) SD]
\label{lem:freshX_isotropic_limits}
Fix $\lambda>0$.
Under \Cref{def:dist} with $\Sigma=I_p$ and $\beta\sim\cN(0,(r^2/p)I_p)$, as $m,n,p\to\infty$ with
$p/m,p/n\to\gamma\in(0,\infty)$, the optimal fresh-$X$ SD risk satisfies
\[
  R_{\sd}^{\star,\freshaffine}(\lambda)
  \pto
  \sR_{\sd}^{\star,\freshaffine}(\lambda)
  :=
  \sR(\lambda)
  -\frac{\big(\sR(\lambda)-\sC^{\fresh}(\lambda)\big)^2}{\sR(\lambda)+\sR_{\pd}^{\fresh}(\lambda)-2\sC^{\fresh}(\lambda)}.
\]
Here $\sR(\lambda)$ is the same teacher risk limit from \Cref{lem:sameX_isotropic_limits}, and
$\sR_{\pd}^{\fresh}(\lambda)$ and $\sC^{\fresh}(\lambda)$ are the deterministic limits of the fresh-$X$ PD risk
$R_{\pd}^{\fresh}(\lambda)$ and the corresponding residual correlation $C^{\fresh}(\lambda)$, given by:
\begin{align*}
  \sC^{\fresh}(\lambda)
  &=
  \sigma^2
  + s(\lambda)\big(\sR(\lambda)-\sigma^2\big)
  + r^2\big(1-s(\lambda)\big)^2,\\
  \sR_{\pd}^{\fresh}(\lambda)
  &=
  \sigma^2
  + s_2(\lambda)\big(\sR(\lambda)-\sigma^2\big)
  + r^2\big(1-2s(\lambda)^2+(2s(\lambda)-1)s_2(\lambda)\big).
\end{align*}
\end{lemma}

\begin{proof}
Let $\beta_{\pd,\lambda}^{\fresh}:=M_\lambda\beta_\lambda$ be the coefficient of the fresh-$X$ PD refit
(trained on $(\tilde X,\tilde y_\lambda)$ with $\tilde y_\lambda=\tilde X\beta_\lambda$),
and write $\Delta_\lambda:=\beta_\lambda-\beta$, and  $\Delta_{\pd,\lambda}^{\fresh}:=\beta_{\pd,\lambda}^{\fresh}-\beta$.
In the isotropic in-distribution setting,
the teacher and PD risks and their residual correlation satisfy the identities:
\[
  R(\lambda)=\sigma^2+\|\Delta_\lambda\|_2^2,
  \qquad
  R_{\pd}^{\fresh}(\lambda)=\sigma^2+\|\Delta_{\pd,\lambda}^{\fresh}\|_2^2,
  \qquad
  C^{\fresh}(\lambda)=\sigma^2+\langle \Delta_\lambda,\Delta_{\pd,\lambda}^{\fresh}\rangle.
\]
Since $\Delta_{\pd,\lambda}^{\fresh}=M_\lambda\beta_\lambda-\beta=M_\lambda\Delta_\lambda+(M_\lambda-I_p)\beta$, we have
\begin{equation}
\label{eq:fresh_expand_C}
  \langle \Delta_\lambda,\Delta_{\pd,\lambda}^{\fresh}\rangle
  =
  \Delta_\lambda^\top M_\lambda \Delta_\lambda
  +
  \Delta_\lambda^\top(M_\lambda-I_p)\beta,
\end{equation}
and
\begin{equation}
\label{eq:fresh_expand_B}
  \|\Delta_{\pd,\lambda}^{\fresh}\|_2^2
  =
  \Delta_\lambda^\top M_\lambda^2 \Delta_\lambda
  +
  \beta^\top(M_\lambda-I_p)^2\beta
  +
  2\,\Delta_\lambda^\top M_\lambda(M_\lambda-I_p)\beta.
\end{equation}
Thus, we need asymptotics for linear and quadratic forms involving $\Delta_\lambda$, which we obtain below.

\emph{Linear and quadratic forms of $\Delta_\lambda$.}
Write the ridge teacher error as
\[
  \Delta_\lambda
  = -\lambda(\hat\Sigma+\lambda I_p)^{-1}\beta
    +(\hat\Sigma+\lambda I_p)^{-1}\frac{X^\top\eps}{n}
  =: \Delta_\lambda^{\mathrm{bias}}+\Delta_\lambda^{\mathrm{var}},
\]
where $\hat\Sigma=X^\top X/n$ and $\eps:=y-X\beta$.
As in the proof of \Cref{thm:risk-asymptotics}, the cross term
$\langle \Delta_\lambda^{\mathrm{bias}},\Delta_\lambda^{\mathrm{var}}\rangle$ is $o_p(1)$, and $\|\Delta_\lambda\|_2^2=\|\Delta_\lambda^{\mathrm{bias}}\|_2^2+\|\Delta_\lambda^{\mathrm{var}}\|_2^2 + o_p(1)$.
Moreover, by \Cref{lem:sameX_isotropic_limits}, $\|\Delta_\lambda\|_2^2\pto \sR(\lambda)-\sigma^2$.

Because $\tilde X$ is independent of $(X,\beta,\eps)$, the matrix $M_\lambda$ is independent of
$(\Delta_\lambda,\beta)$ and is orthogonally invariant, with $\|M_\lambda\|_{\oper}\le 1$.
By standard concentration for quadratic forms of an independent orthogonally invariant matrix, similar to \Cref{lem:response_concentration} in the proof of \Cref{thm:risk-asymptotics}, we have
\begin{align}
  \Delta_\lambda^\top M_\lambda \Delta_\lambda
  &= \Big(\frac{1}{p}\tr(M_\lambda)\Big)\,\|\Delta_\lambda\|_2^2 + o_p(1), \label{eq:quad_M}\\
  \Delta_\lambda^\top M_\lambda^2 \Delta_\lambda
  &= \Big(\frac{1}{p}\tr(M_\lambda^2)\Big)\,\|\Delta_\lambda\|_2^2 + o_p(1), \label{eq:quad_M2}\\
  \Delta_\lambda^\top M_\lambda \beta
  &= \Big(\frac{1}{p}\tr(M_\lambda)\Big)\,\langle \Delta_\lambda,\beta\rangle + o_p(1), \label{eq:bilin_M}\\
  \Delta_\lambda^\top M_\lambda^2 \beta
  &= \Big(\frac{1}{p}\tr(M_\lambda^2)\Big)\,\langle \Delta_\lambda,\beta\rangle + o_p(1), \label{eq:bilin_M2}
\end{align}
and similarly,
\begin{equation}
\label{eq:quad_MI}
  \beta^\top(M_\lambda-I_p)^2\beta
  =
  \Big(\frac{1}{p}\tr\big((M_\lambda-I_p)^2\big)\Big)\,\|\beta\|_2^2 + o_p(1).
\end{equation}

Now, note that
\[
  \langle \Delta_\lambda,\beta\rangle
  =
  -\lambda\,\beta^\top(\hat\Sigma+\lambda I_p)^{-1}\beta
  + \beta^\top(\hat\Sigma+\lambda I_p)^{-1}\frac{X^\top\eps}{n}.
\]
The second term is mean-zero conditional on $(X,\beta)$ and is $o_p(1)$ by \Cref{lem:response_concentration}.
For the first term, isotropy of $\beta$ and standard quadratic-form concentration give
\[
  \beta^\top(\hat\Sigma+\lambda I_p)^{-1}\beta
  =
  \frac{\|\beta\|_2^2}{p}\tr(\hat\Sigma+\lambda I_p)^{-1}+o_p(1).
\]
Using $\|\beta\|_2^2\pto r^2$ and $\hat\Sigma(\hat\Sigma+\lambda I_p)^{-1}=I_p-\lambda(\hat\Sigma+\lambda I_p)^{-1}$,
we obtain
\[
  \lambda\cdot \frac{1}{p}\tr(\hat\Sigma+\lambda I_p)^{-1}
  = 1-\frac{1}{p}\tr\!\big(\hat\Sigma(\hat\Sigma+\lambda I_p)^{-1}\big).
\]
In the isotropic proportional regime with $p/n\to\gamma$, the term
$\frac{1}{p}\tr(\hat\Sigma(\hat\Sigma+\lambda I_p)^{-1})$ converges to the same deterministic limit as
$\frac{1}{p}\tr(M_\lambda)$ (since $\hat\Sigma$ and $\tilde\Sigma$ are independent Wishart matrices with the same aspect ratio),
namely $s(\lambda)$ from \eqref{eq:s_iso}. Therefore,
\begin{equation}
\label{eq:Delta_beta_limit}
  \langle \Delta_\lambda,\beta\rangle
  \pto -r^2\big(1-s(\lambda)\big).
\end{equation}
We are now ready to obtain the asymptotics for the fresh-$X$ PD risk and residual correlation.

\emph{Residual correlation asymptotics.}
Combining \eqref{eq:fresh_expand_C} with \eqref{eq:quad_M} and \eqref{eq:bilin_M} gives
\[
  \langle \Delta_\lambda,\Delta_{\pd,\lambda}^{\fresh}\rangle
  =
  \Big(\frac{1}{p}\tr(M_\lambda)\Big)\|\Delta_\lambda\|_2^2
  +\Big(\frac{1}{p}\tr(M_\lambda)-1\Big)\langle \Delta_\lambda,\beta\rangle
  +o_p(1).
\]
Using $\frac{1}{p}\tr(M_\lambda)\to s(\lambda)$, $\|\Delta_\lambda\|_2^2\pto \sR(\lambda)-\sigma^2$, and
\eqref{eq:Delta_beta_limit}, we obtain
\[
  \langle \Delta_\lambda,\Delta_{\pd,\lambda}^{\fresh}\rangle
  \pto
  s(\lambda)\big(\sR(\lambda)-\sigma^2\big) + r^2\big(1-s(\lambda)\big)^2.
\]
Therefore,
\[
  C^{\fresh}(\lambda)
  =\sigma^2+\langle \Delta_\lambda,\Delta_{\pd,\lambda}^{\fresh}\rangle
  \pto
  \sigma^2
  + s(\lambda)\big(\sR(\lambda)-\sigma^2\big)
  + r^2\big(1-s(\lambda)\big)^2
  = \sC^{\fresh}(\lambda).
\]

\emph{PD risk asymptotics.}
Similarly, combining \eqref{eq:fresh_expand_B} with \eqref{eq:quad_M2}, \eqref{eq:quad_MI}, and \eqref{eq:bilin_M2} yields:
\[
  \|\Delta_{\pd,\lambda}^{\fresh}\|_2^2
  =
  \Big(\frac{1}{p}\tr(M_\lambda^2)\Big)\|\Delta_\lambda\|_2^2
  + \Big(\frac{1}{p}\tr((M_\lambda-I_p)^2)\Big)\|\beta\|_2^2
  + 2\Big(\frac{1}{p}\tr(M_\lambda^2)-\frac{1}{p}\tr(M_\lambda)\Big)\langle \Delta_\lambda,\beta\rangle
  + o_p(1).
\]
Using $\frac{1}{p}\tr(M_\lambda^2)\to s_2(\lambda)$, $\frac{1}{p}\tr((M_\lambda-I_p)^2)\to 1-2s(\lambda)+s_2(\lambda)$,
$\|\beta\|_2^2\pto r^2$, $\|\Delta_\lambda\|_2^2\pto \sR(\lambda)-\sigma^2$, and \eqref{eq:Delta_beta_limit}, we get
\[
  \|\Delta_{\pd,\lambda}^{\fresh}\|_2^2
  \pto
  s_2(\lambda)\big(\sR(\lambda)-\sigma^2\big)
  + r^2\big(1-2s(\lambda)^2+(2s(\lambda)-1)s_2(\lambda)\big).
\]
Hence
\[
  R_{\pd}^{\fresh}(\lambda)
  =\sigma^2+\|\Delta_{\pd,\lambda}^{\fresh}\|_2^2
  \pto
  \sigma^2
  + s_2(\lambda)\big(\sR(\lambda)-\sigma^2\big)
  + r^2\big(1-2s(\lambda)^2+(2s(\lambda)-1)s_2(\lambda)\big)
  = \sR_{\pd}^{\fresh}(\lambda).
\]

Finally, because $f_{\sd,\lambda,\xi}^{\freshaffine}=(1-\xi)f_\lambda+\xi f_{\pd,\lambda}^{\fresh}$ is a
two-predictor affine path, combining the oracle formula in \Cref{prop:closed_form_xi_R1} with the component convergences proved above finishes the proof.
\end{proof}

For an illustration of the empirical versus theoretical risks in \Cref{lem:freshX_isotropic_limits}, see the second panel of \Cref{fig:same_x_fresh_x_affine}.

\subsubsection{Proof of \Cref{thm:freshX_dominated_by_sameX_isotropic}}
\label{proof:thm:freshX_dominated_by_sameX_isotropic}

\begin{proof}[Proof of \Cref{thm:freshX_dominated_by_sameX_isotropic}]
As before, let $v=v(\lambda)>0$ denote the (unique) solution of \eqref{eq:v_fp_iso}, and define the signal-to-noise ratio
${\SNR}:=r^2/\sigma^2>0$.
Set $D_0:=(1+v)^2-\gamma v^2=(1+v)^2\bigl(1-t_2\bigr)$,
where $t_2=\gamma\bigl(v/(1+v)\bigr)^2$ as in \eqref{eq:tq_iso}.
Note that $D_0>0$.
This is because differentiating \eqref{eq:v_fp_iso} shows that
\[
  v'(\lambda)
  =-\frac{v(\lambda)^2}{1-\gamma\bigl(v(\lambda)/(1+v(\lambda))\bigr)^2}
  =-\frac{v(\lambda)^2}{1-t_2(\lambda)}<0,
\]
so $1-t_2(\lambda)>0$ and hence $D_0>0$ for all $\lambda>0$.

Using the explicit isotropic expressions from \Cref{lem:sameX_isotropic_limits,lem:freshX_isotropic_limits} and simplifying,
one obtains the factorization
\begin{equation}
\label{eq:gap_factorization_v}
  \sR_{\sd}^{\star,\freshaffine}(\lambda)-\sR_{\sd}^{\star,\same}(\lambda)
  =
  \sigma^2\,
  \frac{\gamma v\;\Delta(v)^2\;\Xi(v)}
  {(1+v)^3\,D_0\,\Pi_1(v)\,\Pi_2(v)},
\end{equation}
where
\begin{align}
\Delta(v)
&:=\gamma\,{\SNR}\,v + \gamma v^2 + \gamma v - {\SNR}\,v - {\SNR},\\
\Pi_1(v)
&:= -\gamma\,{\SNR}\,v+\gamma\,{\SNR}+\gamma v+\gamma+{\SNR}\,v+{\SNR},\\
\Pi_2(v)
&:=\gamma^2\,{\SNR}\,v^2+\gamma^2 v^2-2\gamma\,{\SNR}\,v^2-2\gamma\,{\SNR}\,v+\gamma\,{\SNR} \nonumber \\
& \quad +\gamma v^2+2\gamma v+\gamma+{\SNR}\,v^2+2\,{\SNR}\,v+{\SNR},\\
\Xi(v)
&:=\gamma^2 v^3+{\SNR}\Bigl(\frac{(D_0-1)^2}{v}+D_0\,v+v+2\Bigr).
\end{align}

All factors on the right-hand side of \eqref{eq:gap_factorization_v} are nonnegative:
$\sigma^2>0$, $\gamma>0$, $v>0$, and $\Delta(v)^2\ge 0$.
Moreover, $D_0>0$ as shown above, and $\Xi(v)>0$ since it is a sum of strictly positive terms.

It remains to note that $\Pi_1(v)$ and $\Pi_2(v)$ are strictly positive.
Indeed, in both the same-$X$ and fresh-$X$ settings, the two-predictor oracle formula
\Cref{prop:closed_form_xi_R1} involves the discrepancy denominator
\[
  \sD^{\bullet}(\lambda)
  :=\sR(\lambda)+\sR_{\pd}^{\bullet}(\lambda)-2\sC^{\bullet}(\lambda)
  \qquad \bullet\in\{\same,\fresh\},
\]
which is strictly positive for every $\lambda>0$ (the teacher and the corresponding PD refit do not coincide in this
isotropic model).
A direct identification shows that $\Pi_1(v)$ and $\Pi_2(v)$ are proportional to $\sD^{\fresh}(\lambda)$ and
$\sD^{\same}(\lambda)$, respectively, with positive proportionality constants; hence $\Pi_1(v),\Pi_2(v)>0$.

Therefore the right-hand side of \eqref{eq:gap_factorization_v} is nonnegative for all $\lambda>0$, showing the desired domination.
\end{proof}

\begin{figure}[!t]
    \centering
    \includegraphics[width=0.48\linewidth]{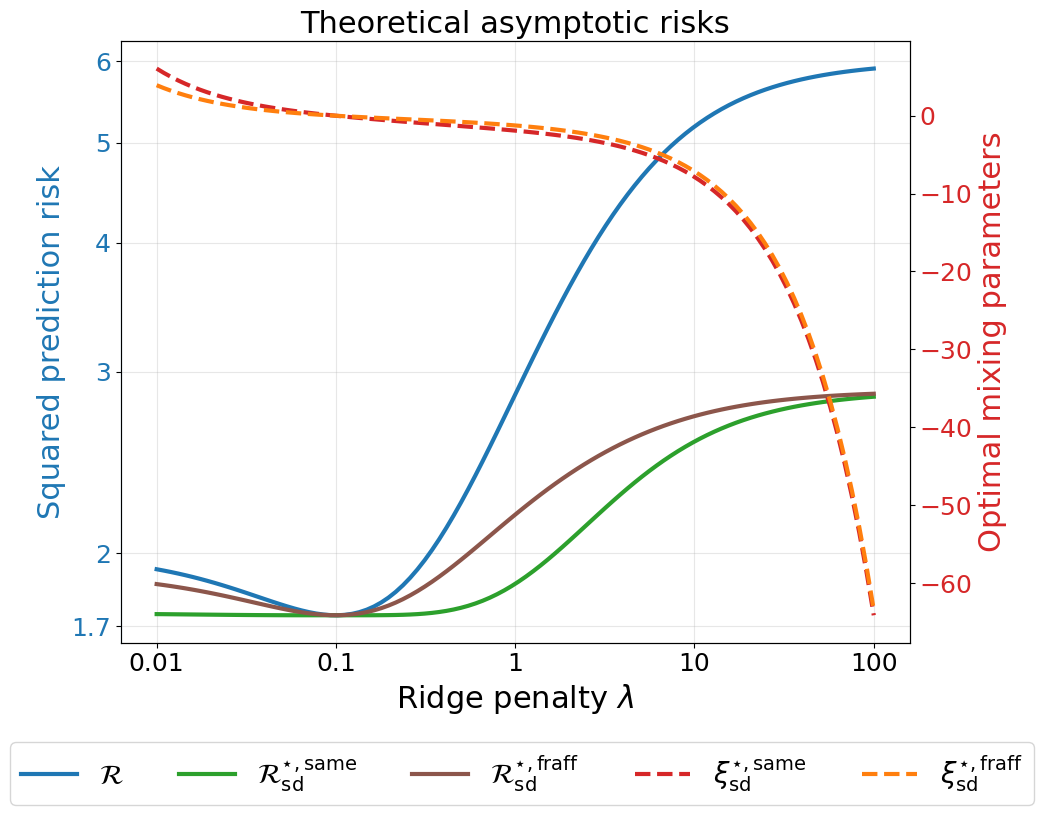}
    \includegraphics[width=0.44\linewidth]{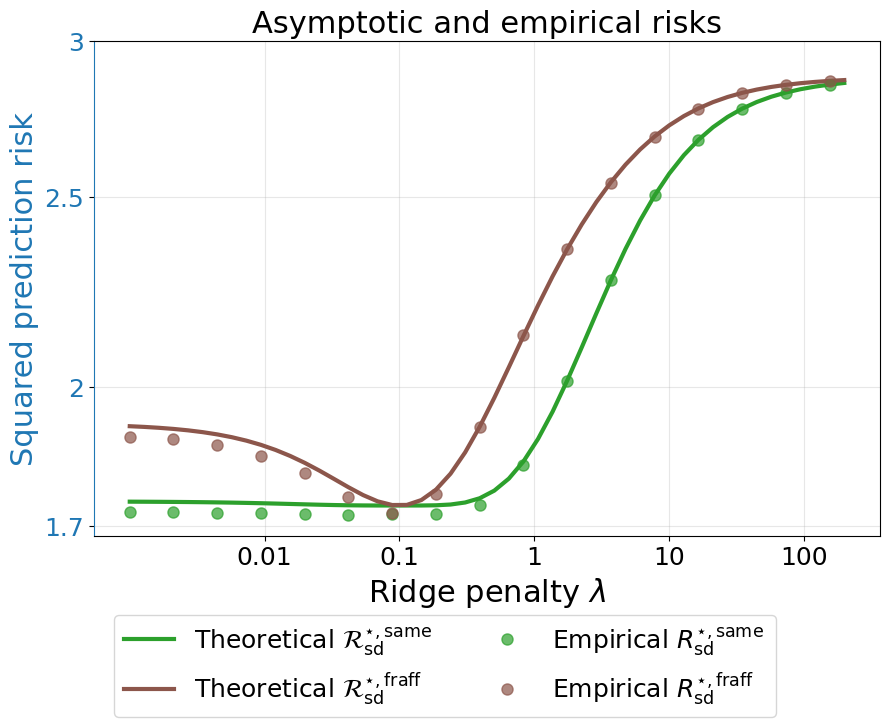}
    \caption{\textbf{Same-$X$ versus fresh-$X$ self-distillation risks.}
    Theoretical asymptotic (\Cref{thm:freshX_dominated_by_sameX_isotropic}) and empirical risks (averaged over 20 simulations)
    and optimal mixing weights for an isotropic setting with $n=400$, $p=200$, $r^2=5$, and $\sigma^2=1$.}
    \label{fig:same_x_fresh_x_affine}
\end{figure}

For a visual illustration of the risk asymptotics of the same-$X$ versus (affine) fresh-$X$ optimal SD risks, see the first panel
of \Cref{fig:same_x_fresh_x_affine}. 
Consistent with \Cref{thm:freshX_dominated_by_sameX_isotropic}, the same-$X$ optimal SD risk uniformly dominates the fresh-$X$ optimal SD risk across $\lambda$.

\subsection{Proofs and Details in \Cref{sec:ridge-variants}}
\label{sec:supp_extensions_ridge_variants}

\subsubsection{Proof of \Cref{eq:ridge_variant_affine_path}}
\label{app:ridge_variants_affine_path}

\begin{proof}[Proof of \Cref{eq:ridge_variant_affine_path}]
Recall that $f_\lambda(x)=s_\lambda(x)^\top y$ and $\hat y_\lambda=S_\lambda y$, and that the SD refit applies the
same smoother to the mixed labels $y^{(\xi)}:=(1-\xi)y+\xi\hat y_\lambda$.
Then, for every $x$,
\[
  f_{\sd,\lambda,\xi}(x)
  = s_\lambda(x)^\top y^{(\xi)}
  = (1-\xi)s_\lambda(x)^\top y + \xi s_\lambda(x)^\top \hat y_\lambda
  = (1-\xi)f_\lambda(x)+\xi f_{\pd,\lambda}(x),
\]
where $f_{\pd,\lambda}(x):=s_\lambda(x)^\top \hat y_\lambda$ by definition.
\end{proof}

\subsubsection{Proof of \Cref{thm:ridge_variants_tangent_sign}}
\label{app:ridge_variants_tangent_sign}

\begin{proof}[Proof of \Cref{thm:ridge_variants_tangent_sign}]
Fix $\lambda>0$.
By \eqref{eq:ridge_variant_affine_path}, the SD family is the affine path
$f_{\sd,\lambda,\xi}=(1-\xi)f_\lambda+\xi f_{\pd,\lambda}$.
Let $e_\lambda:=y_0-f_\lambda(x_0)$ and $e_{\pd}:=y_0-f_{\pd,\lambda}(x_0)$.
Expanding the square gives, for every $\xi\in\RR$,
\begin{align*}
  R_{\sd}(\lambda,\xi)
  &= \EE\!\Big[\big((1-\xi)e_\lambda+\xi e_{\pd}\big)^2 \,\Big|\, \cD\Big] \\
  &= (1-\xi)^2 R(\lambda) + \xi^2 R_{\pd}(\lambda) + 2\xi(1-\xi)C(\lambda) \\
  &= R(\lambda) - 2\xi\big(R(\lambda)-C(\lambda)\big) + \xi^2 D(\lambda),
\end{align*}
where $D(\lambda)=R(\lambda)+R_{\pd}(\lambda)-2C(\lambda)$.
Under $D(\lambda)>0$, this is a strictly convex quadratic in $\xi$, hence the unique minimizer is
\[
  \xi^\star(\lambda)=\frac{R(\lambda)-C(\lambda)}{D(\lambda)},
  \qquad
  R_{\sd}^\star(\lambda)=R(\lambda)-\frac{(R(\lambda)-C(\lambda))^2}{D(\lambda)}.
\]

Next, assume \eqref{eq:tangent_identity_ridge_variants_main} holds.
Since $\lambda\mapsto R(\lambda)$ is differentiable, we may differentiate inside the conditional expectation to get
\[
  R'(\lambda)
  = \partial_\lambda \EE[e_\lambda^2\mid \cD]
  = \EE\!\big[2e_\lambda\,\partial_\lambda e_\lambda \mid \cD\big]
  = -2\,\EE\!\big[e_\lambda\,\partial_\lambda f_\lambda(x_0)\mid \cD\big].
\]
On the other hand,
\begin{align*}
  R(\lambda)-C(\lambda)
  &= \EE\!\big[e_\lambda(e_\lambda-e_{\pd})\mid \cD\big]
   = \EE\!\big[e_\lambda\,(f_{\pd,\lambda}(x_0)-f_\lambda(x_0))\mid \cD\big] \\
  &= -\EE\!\big[e_\lambda\,(f_\lambda(x_0)-f_{\pd,\lambda}(x_0))\mid \cD\big]
   = -\EE\!\big[e_\lambda\,(-\lambda\,\partial_\lambda f_\lambda(x_0))\mid \cD\big] \\
  &= \lambda\,\EE\!\big[e_\lambda\,\partial_\lambda f_\lambda(x_0)\mid \cD\big]
   = -\frac{\lambda}{2}\,R'(\lambda).
\end{align*}
Substituting $R(\lambda)-C(\lambda)=-(\lambda/2)R'(\lambda)$ into the closed forms above yields
\[
  \xi^\star(\lambda)
  =-\frac{\lambda}{2}\,\frac{R'(\lambda)}{D(\lambda)},
  \qquad
  R_{\sd}^\star(\lambda)
  =
  R(\lambda)-\frac{\lambda^2}{4}\,\frac{(R'(\lambda))^2}{D(\lambda)}.
\]
Finally, if $R'(\lambda)\neq 0$ then $R_{\sd}^\star(\lambda)<R(\lambda)$ and
$\sign(\xi^\star(\lambda))=-\sign(R'(\lambda))$.
\end{proof}

\subsubsection[Verifying Derivative Property for Common Ridge Variants]{Verifying Derivative Property \eqref{eq:tangent_identity_ridge_variants_main} for Common Ridge Variants}
\label{app:ridge_variants_verify_tangent}

\begin{lemma}[Generalized ridge satisfies \eqref{eq:tangent_identity_ridge_variants_main}]
\label{lem:grr_tangent_identity}
Fix $\Omega\succ 0$ and $\lambda>0$, and let
\[
  f_\lambda^\Omega(x):=x^\top(X^\top X+n\lambda\,\Omega)^{-1}X^\top y.
\]
Let $f_{\pd,\lambda}^\Omega$ denote the PD refit obtained by training generalized ridge at the same $(X,\Omega,\lambda)$
on pseudo-labels $\hat y_\lambda^\Omega=f_\lambda^\Omega(X)$.
Then, for all $x$,
\[
  f_\lambda^\Omega(x)-f_{\pd,\lambda}^\Omega(x)=-\lambda\,\partial_\lambda f_\lambda^\Omega(x).
\]
\end{lemma}

\begin{proof}
Let $A_\lambda:=X^\top X+n\lambda\,\Omega$.
Then $f_\lambda^\Omega(x)=x^\top A_\lambda^{-1}X^\top y$ and $f_{\pd,\lambda}^\Omega(x)=x^\top A_\lambda^{-1}X^\top X A_\lambda^{-1}X^\top y$.
Hence
\[
  f_\lambda^\Omega(x)-f_{\pd,\lambda}^\Omega(x)
  = x^\top A_\lambda^{-1}(A_\lambda-X^\top X)A_\lambda^{-1}X^\top y
  = n\lambda\,x^\top A_\lambda^{-1}\Omega A_\lambda^{-1}X^\top y.
\]
Moreover, $\partial_\lambda A_\lambda^{-1}=-A_\lambda^{-1}(\partial_\lambda A_\lambda)A_\lambda^{-1}
=-nA_\lambda^{-1}\Omega A_\lambda^{-1}$, so
\[
  \partial_\lambda f_\lambda^\Omega(x)
  = x^\top(\partial_\lambda A_\lambda^{-1})X^\top y
  = -n\,x^\top A_\lambda^{-1}\Omega A_\lambda^{-1}X^\top y,
\]
and multiplying by $-\lambda$ gives the claim.
\end{proof}

\begin{lemma}[Kernel ridge satisfies \eqref{eq:tangent_identity_ridge_variants_main}]
\label{lem:krr_tangent_identity}
Fix $\lambda>0$ and a PSD kernel with kernel matrix $K\in\RR^{n\times n}$ and
$k_x:=(k(x,x_1),\dots,k(x,x_n))^\top$.
Let
\[
  f_\lambda^{\mathrm{kern}}(x):=k_x^\top(K+n\lambda I_n)^{-1}y,
\]
and let $f_{\pd,\lambda}^{\mathrm{kern}}$ denote the PD refit trained at the same $\lambda$ on pseudo-labels
$\hat y_\lambda^{\mathrm{kern}}=f_\lambda^{\mathrm{kern}}(X)$.
Then, for all $x$,
\[
  f_\lambda^{\mathrm{kern}}(x)-f_{\pd,\lambda}^{\mathrm{kern}}(x)
  =-\lambda\,\partial_\lambda f_\lambda^{\mathrm{kern}}(x).
\]
\end{lemma}

\begin{proof}
Let $B_\lambda:=K+n\lambda I_n$.
Then $f_\lambda^{\mathrm{kern}}(x)=k_x^\top B_\lambda^{-1}y$ and
$f_{\pd,\lambda}^{\mathrm{kern}}(x)=k_x^\top B_\lambda^{-1}KB_\lambda^{-1}y$.
Thus
\[
  f_\lambda^{\mathrm{kern}}(x)-f_{\pd,\lambda}^{\mathrm{kern}}(x)
  = k_x^\top B_\lambda^{-1}(B_\lambda-K)B_\lambda^{-1}y
  = n\lambda\,k_x^\top B_\lambda^{-2}y.
\]
Since $\partial_\lambda B_\lambda^{-1}=-B_\lambda^{-1}(\partial_\lambda B_\lambda)B_\lambda^{-1}
=-nB_\lambda^{-2}$, we have
\[
  \partial_\lambda f_\lambda^{\mathrm{kern}}(x)
  = k_x^\top(\partial_\lambda B_\lambda^{-1})y
  = -n\,k_x^\top B_\lambda^{-2}y,
\]
and multiplying by $-\lambda$ yields the claim.
\end{proof}

\clearpage
\newgeometry{left=0.8in,top=0.8in,right=0.8in,bottom=0.8in,head=.1in,foot=0.1in}

\section{Additional Numerical Illustrations}
\label{sec:additional_exps}

The source code for reproducing the results of this paper can be found at the following location: \url{https://github.com/hhd357/optimal_self_distillation_ridge}

\subsection{Additional Experiments on CIFAR Datasets}
\label{app:additional-experiments-cifar}

\subsubsection{CIFAR100 with Pretrained ResNet-34 Features}

\begin{figure}[!ht]
    \centering
    \includegraphics[width=0.5\textwidth]{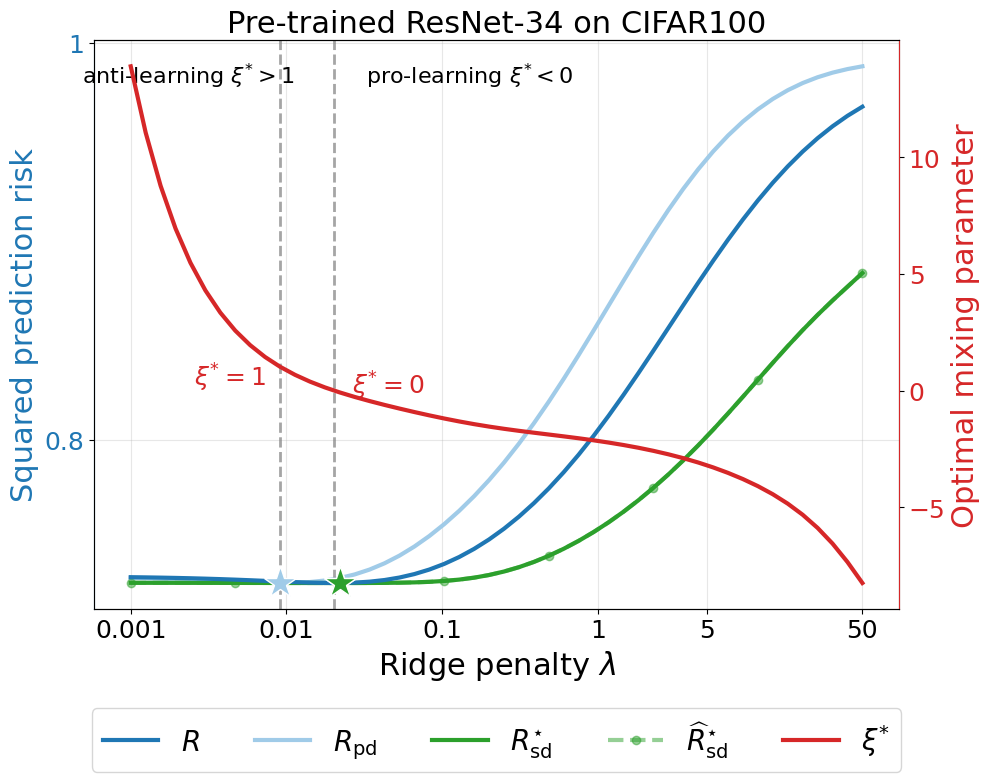}
    \caption{Squared prediction risks and optimal mixing parameter for ridge, pure-distilled and self-distilled ridge on pretrained ResNet-34 features on CIFAR100 dataset.}
    \label{fig:cifar100}
\end{figure}

\subsubsection{CIFAR10 and CIFAR100 Classification Accuracies}

\begin{figure}[!ht]
    \centering
    \includegraphics[width=0.48\textwidth]{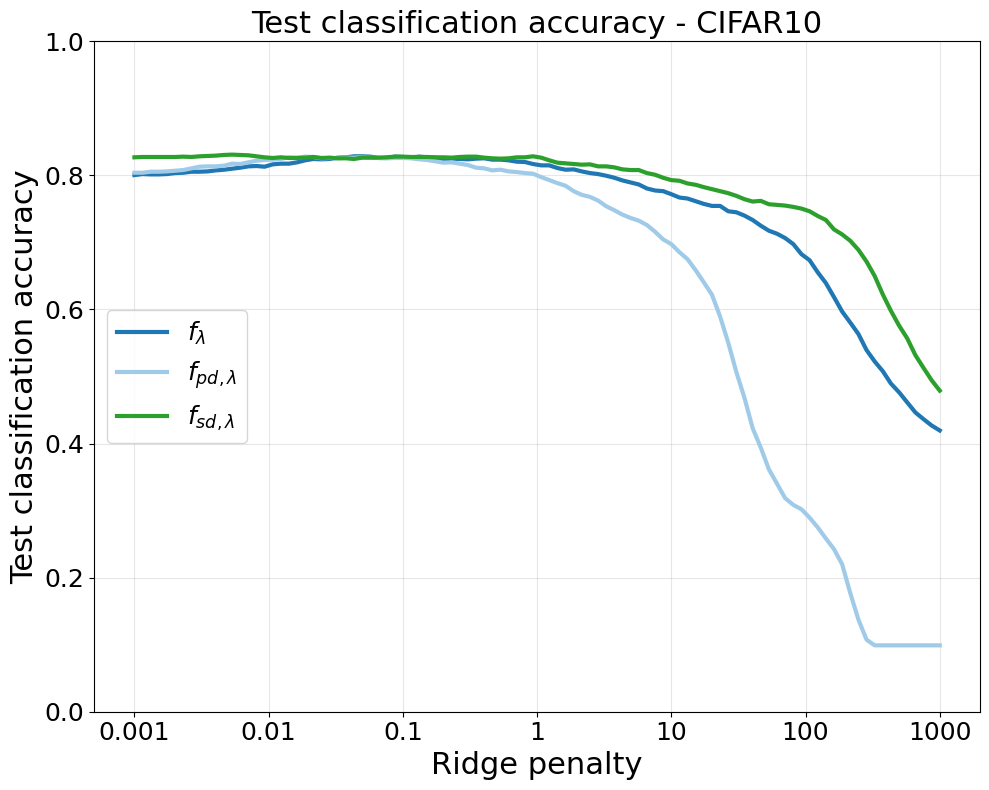}
    \includegraphics[width=0.48\textwidth]{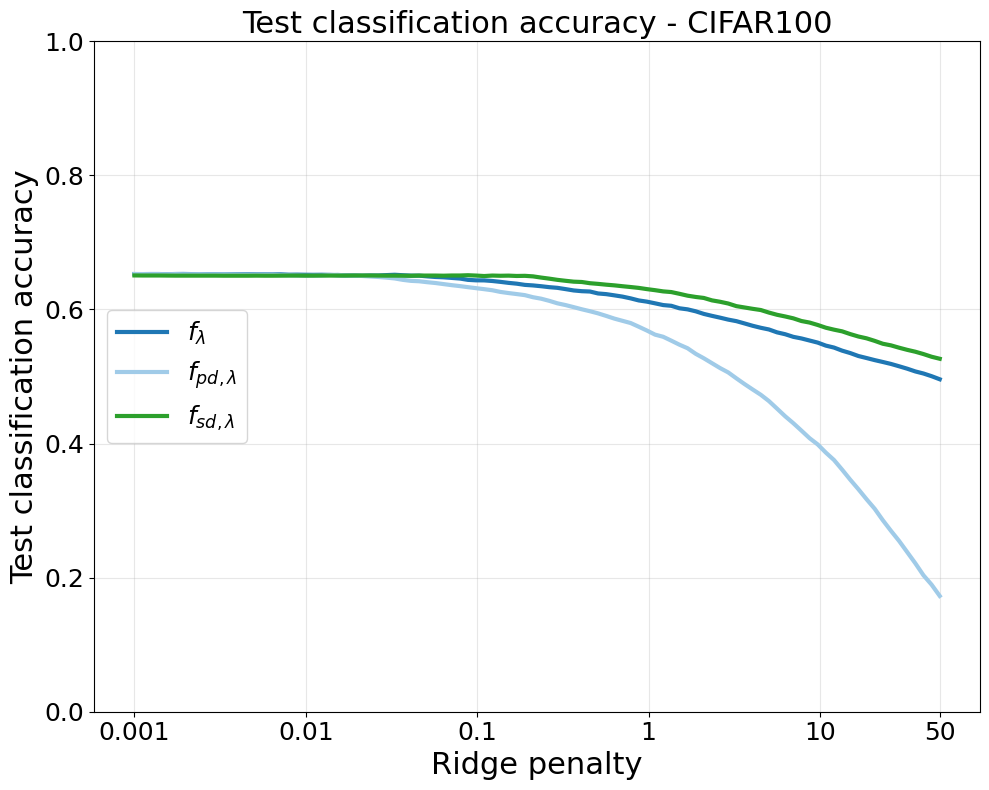}
    \caption{Test classification accuracy of the ridge and self-distilled ridges on CIFAR10 and CIFAR100 experiments. Although the optimal mixing is chosen to minimize the squared prediction risk, the strict improvement property still somewhat holds for test accuracy as well.}
    \label{fig:cifar_test_accuracy}
\end{figure}

\clearpage

\subsection{Illustration for Related Works Comparison in \Cref{sec:related_work}}
\label{app:vs_multi_round}

\begin{figure}[!ht]
    \centering
    \includegraphics[width=0.6\textwidth]{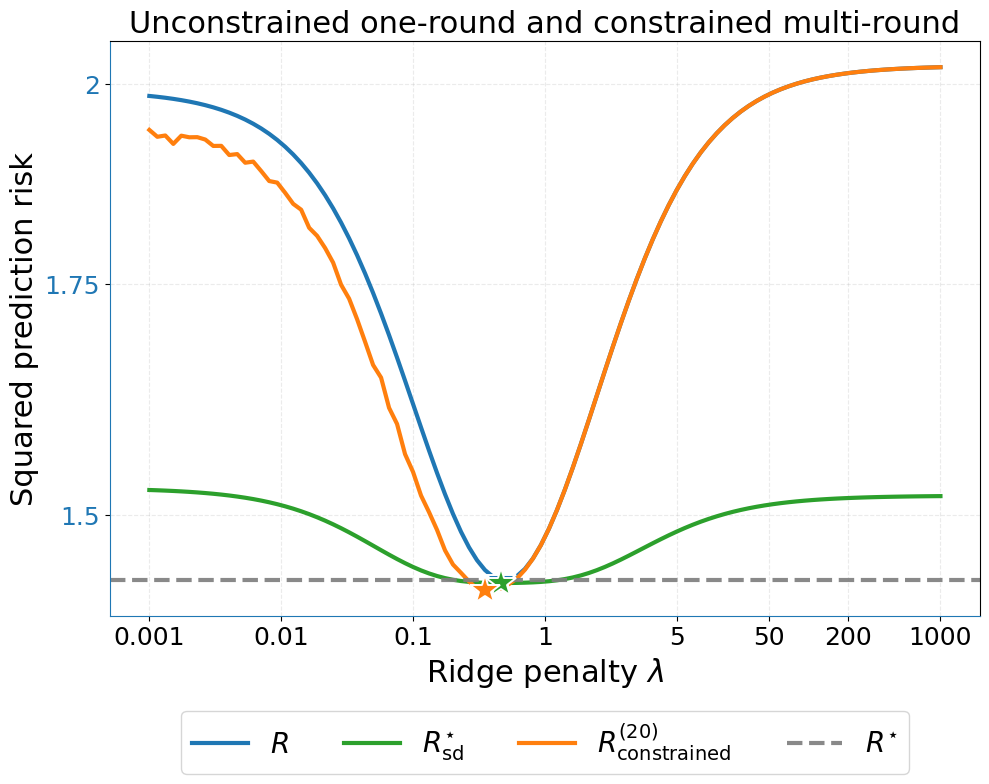}
    \caption{\textbf{One-round unconstrained versus multi-round constrained mixing weight.} Squared prediction risk empirical curves averaged over 20 simulations, $n = 400, p = 200$ with isotropic design and isotropic signal. We choose the optimal $\xi$ over a grid of 200 values in $[0,1]$ and pick the one with the lowest risk at the $20$-th round.}
    \label{fig:vs_multi_round}
\end{figure}

\clearpage
\subsection{Additional Illustrations Real-World Regression Tasks}
\label{app:additional-illustrations-real-world}

\subsubsection{Comparison with Constrained Self-Distillation}

\begin{figure}[!ht]
    \centering
    \includegraphics[width=0.48\textwidth]{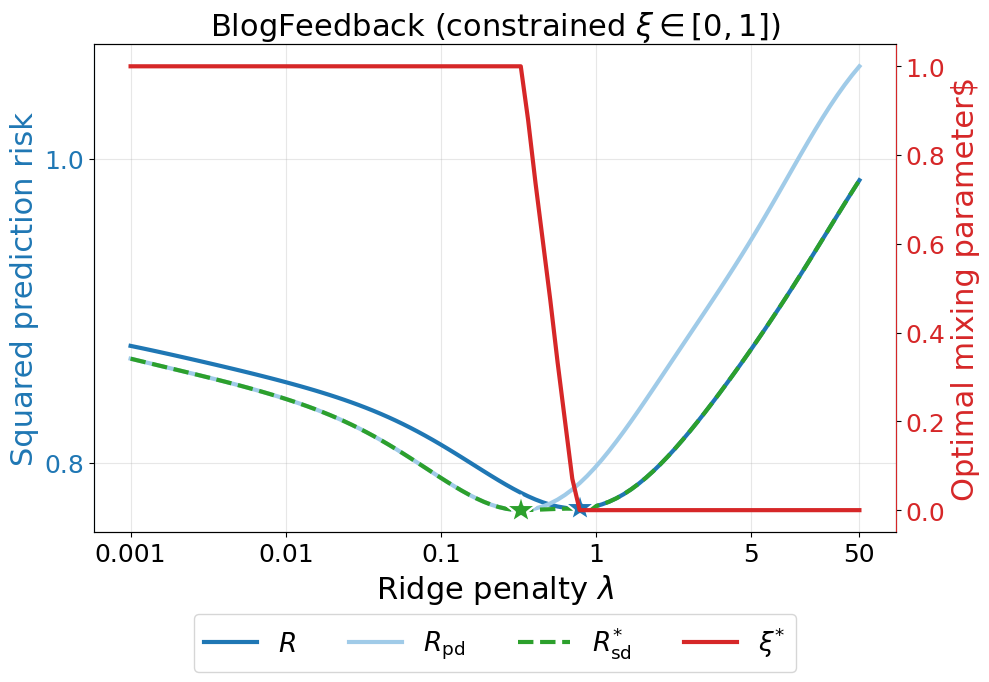}
    \includegraphics[width=0.48\textwidth]{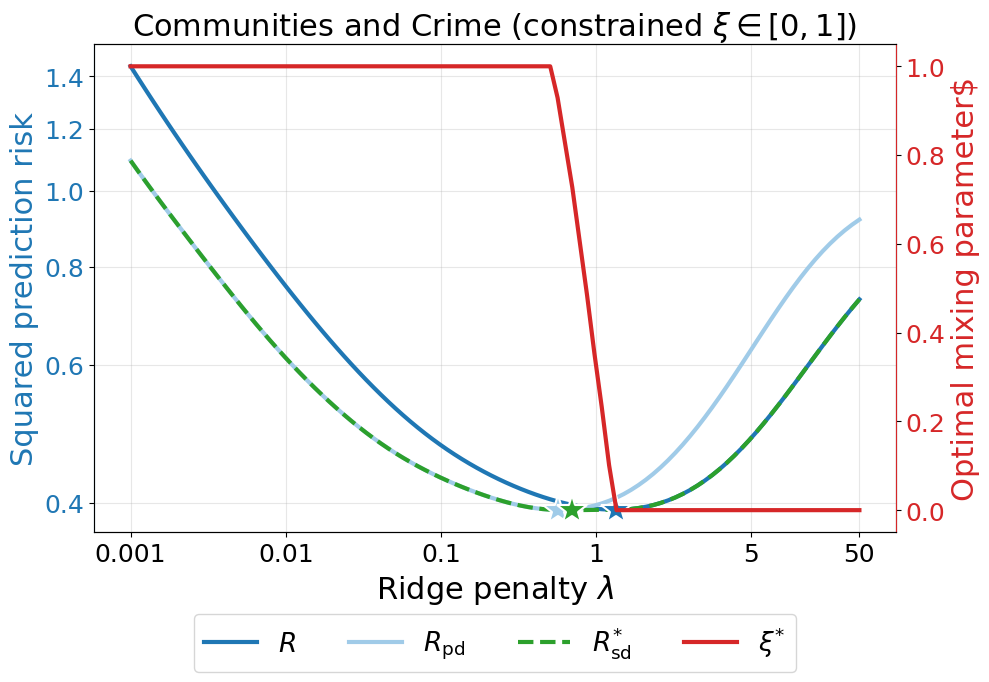}
    \includegraphics[width=0.48\textwidth]{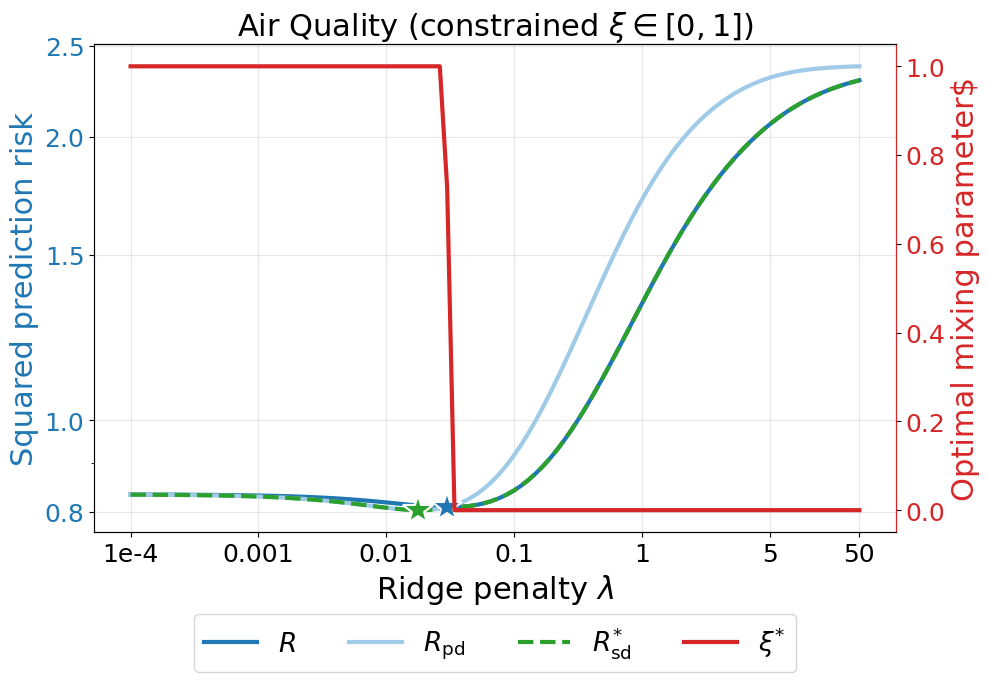}
    \caption{Test-set squared prediction risk and optimal mixing parameter, being constrained in $[0,1]$. We choose the optimal $\xi$ over a grid of 100 values in $[0,1]$ and pick the one with the lowest SD risk. The linestyle for the SD risk is changed only for this figure since it mostly overlaps with either the original ridge or the pure-distilled ridge curves. When $\xi = 1$, the SD risk exactly matches the PD risk and when $\xi = 0$, it exactly matches the teacher ridge risk.}
    \label{fig:exp_real_world_restricted}
\end{figure}

\clearpage
\subsubsection{Varying Train-to-Test Split Ratios and Gains}

BlogFeedback dataset:

\begin{figure*}[!ht]
    \centering
    \includegraphics[width=0.49 \textwidth, height = 70mm]{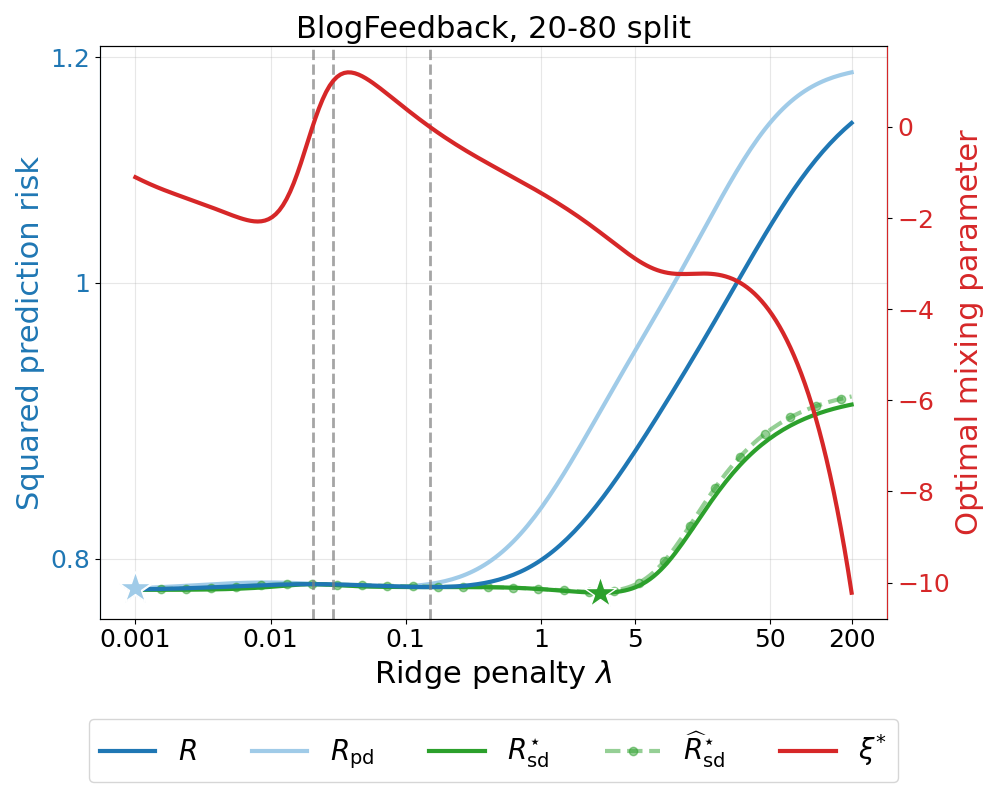}
    \includegraphics[width=0.49 \textwidth, height = 70mm]{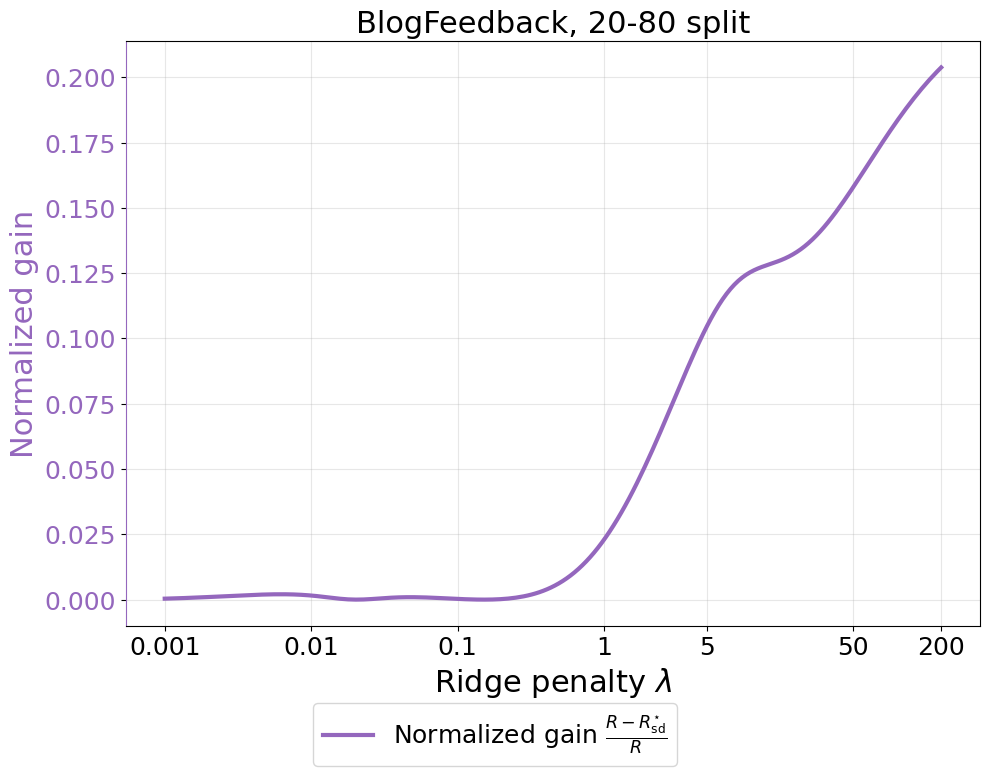}
    \includegraphics[width=0.49 \textwidth, height = 70mm]{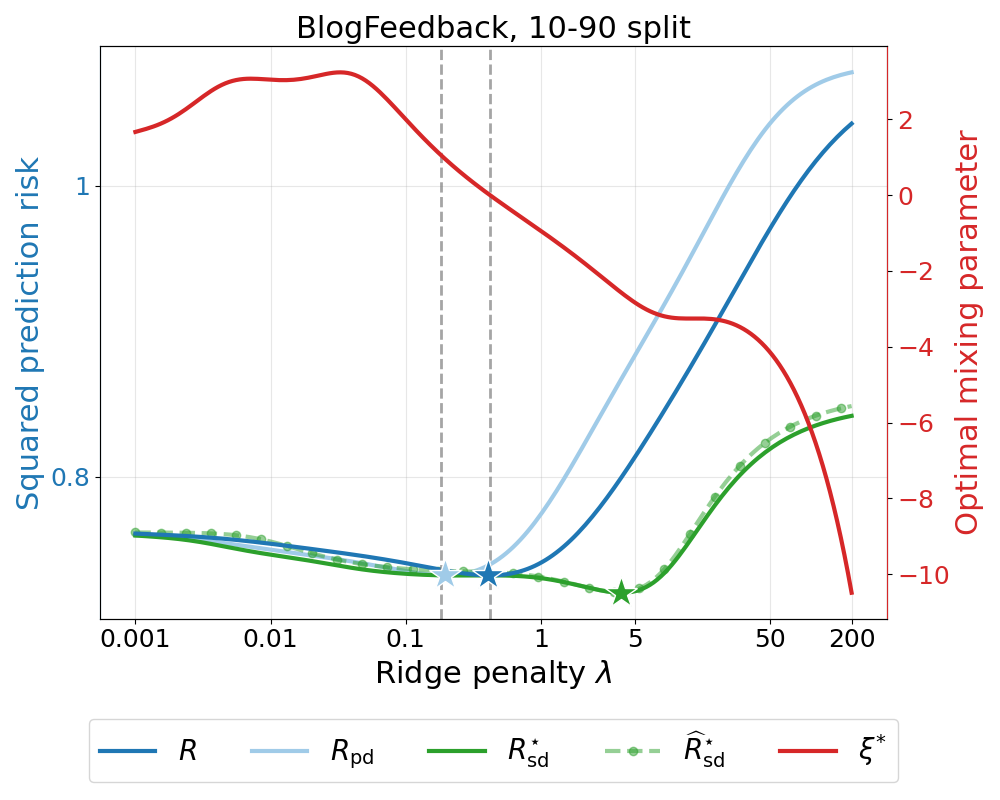}
    \includegraphics[width=0.49 \textwidth, height = 70mm]{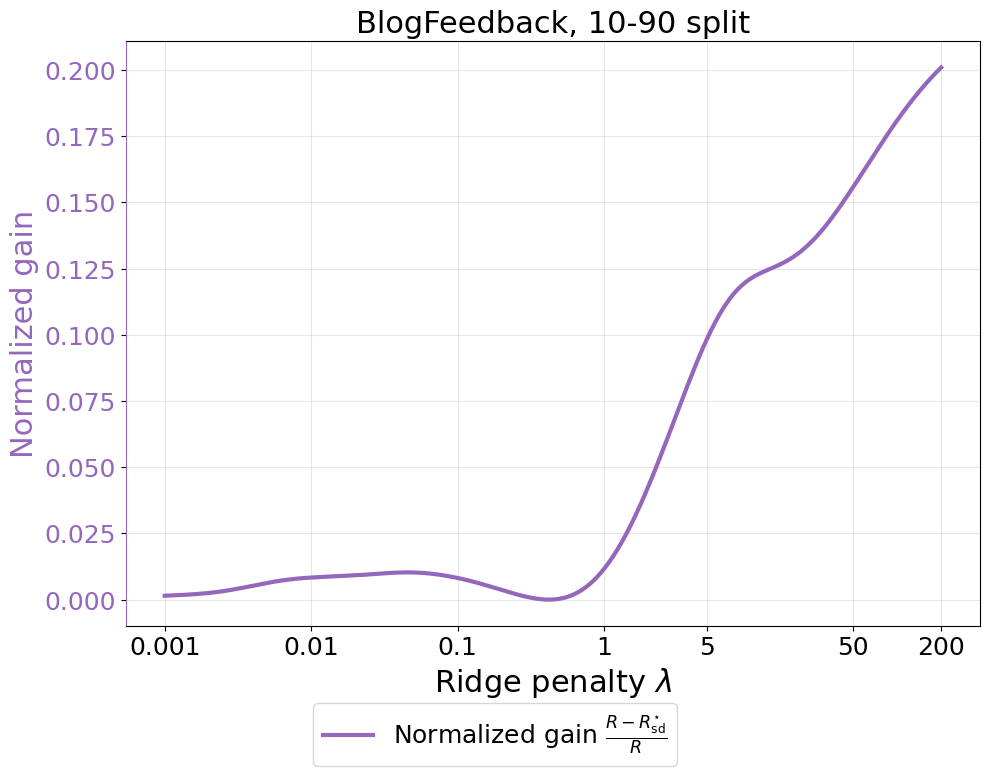}
    \includegraphics[width=0.49 \textwidth, height = 70mm]{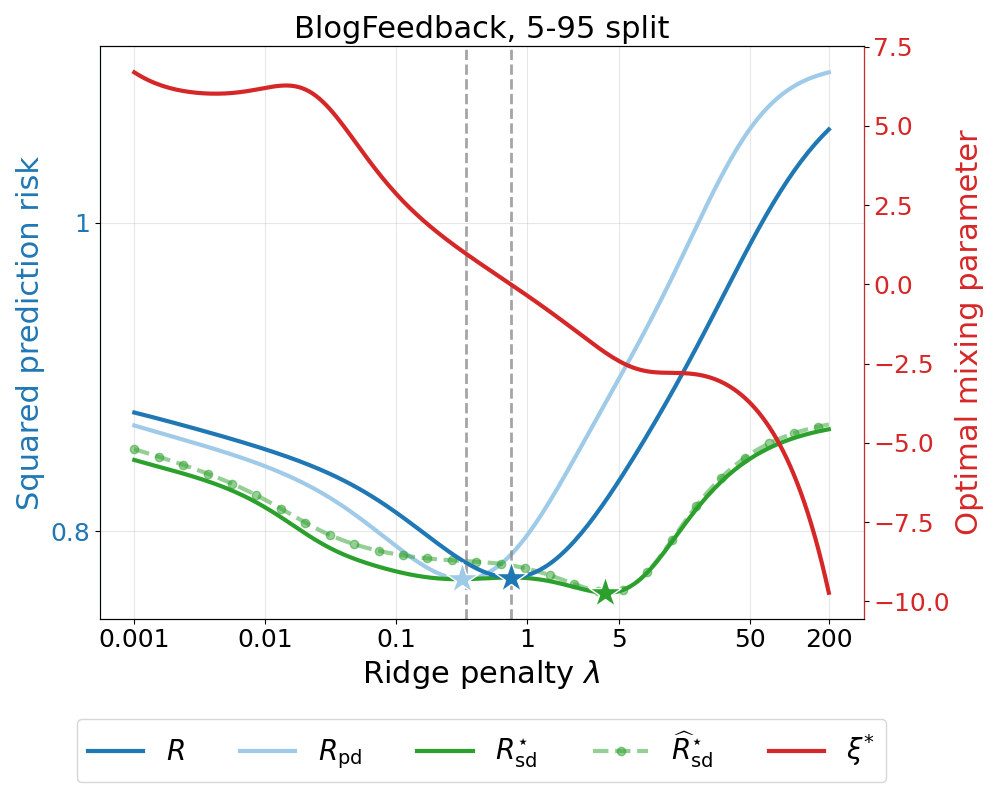}
    \includegraphics[width=0.49 \textwidth, height = 70mm]{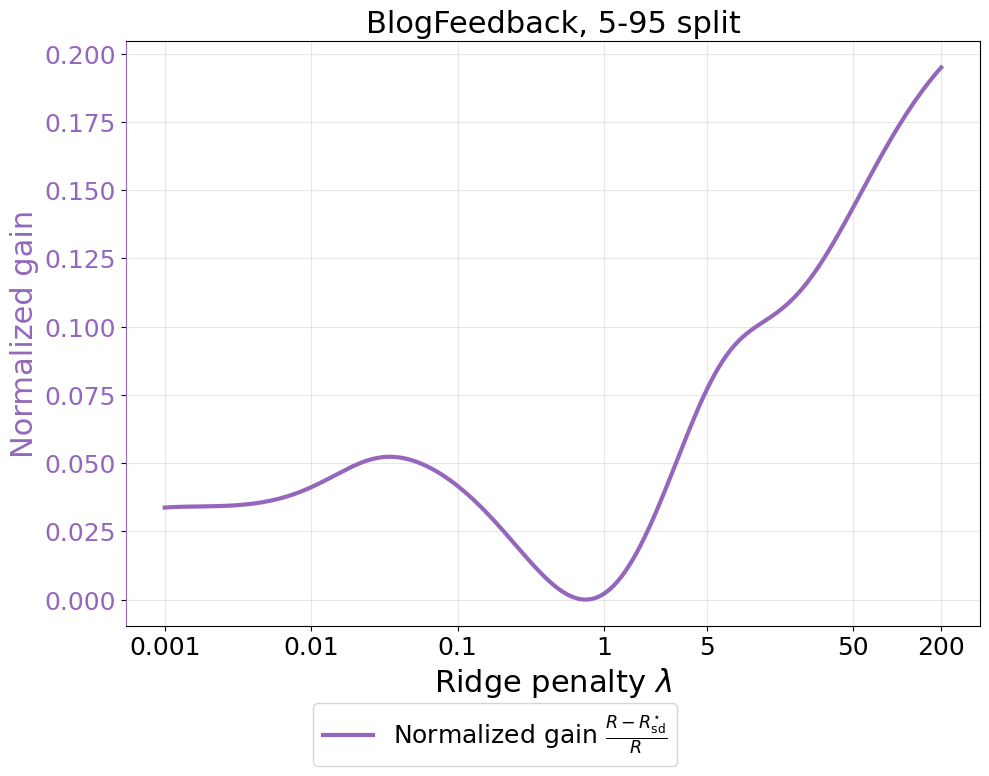}
    \caption{Squared prediction risks and gain curve on BlogFeedback dataset with different split ratios.}
    \label{fig:blog_appendix}
\end{figure*}

Communities and Crime dataset:

\begin{figure*}[!ht]
    \centering
    \includegraphics[width=0.49 \textwidth, height = 70mm]{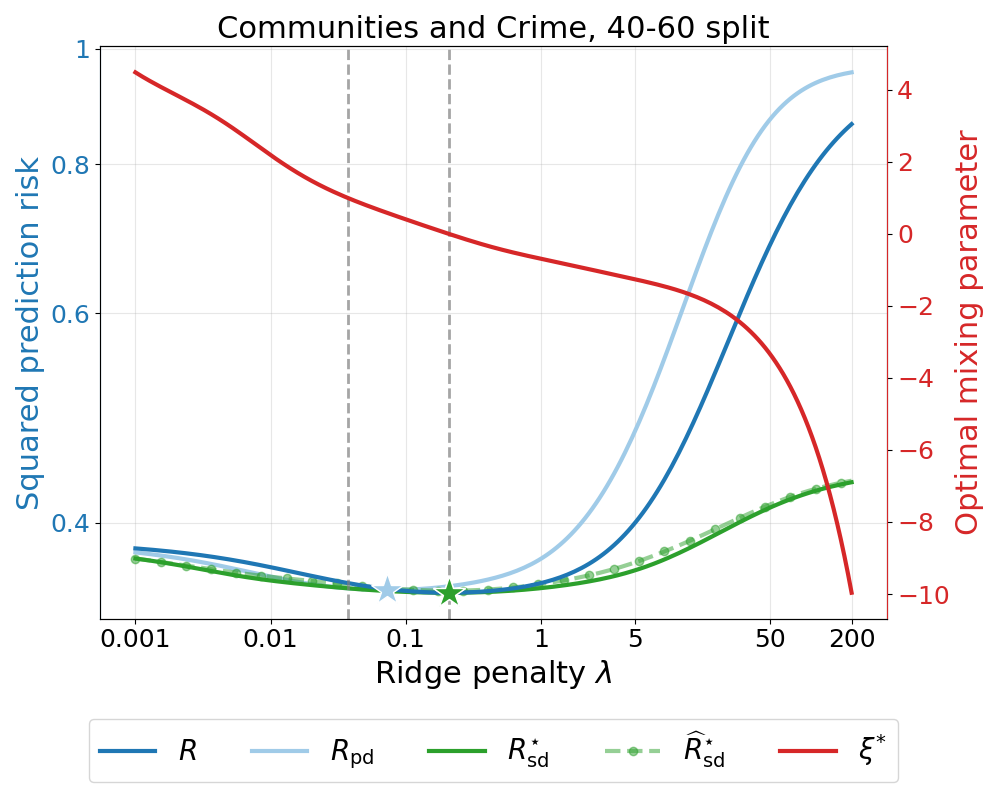}
    \includegraphics[width=0.49 \textwidth, height = 70mm]{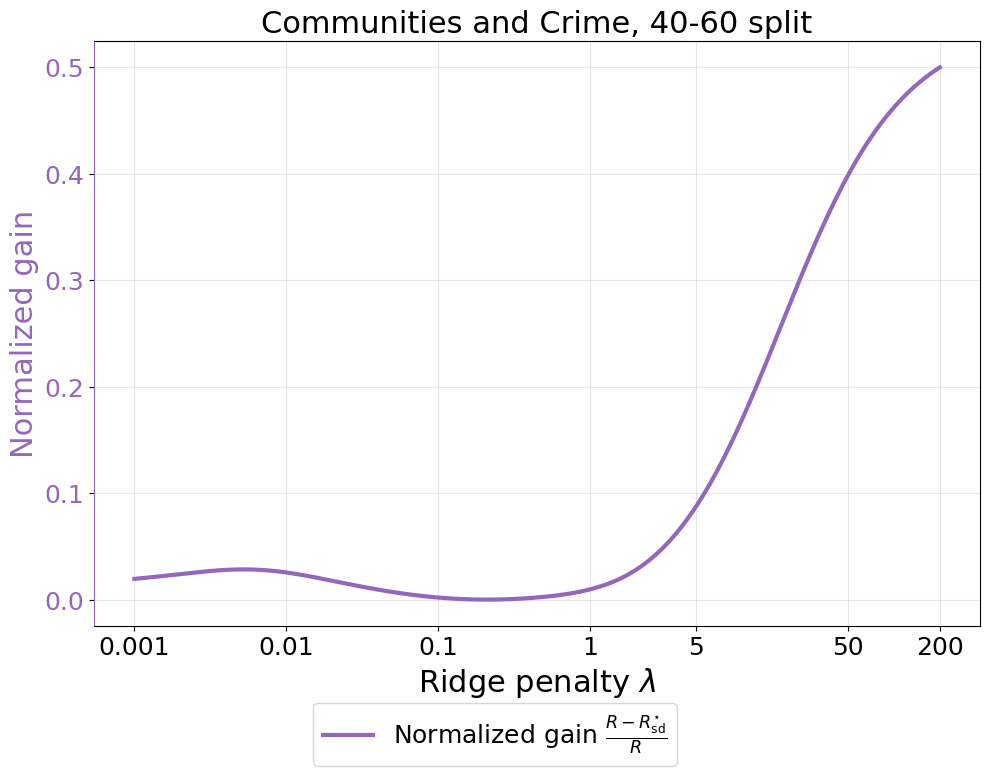}
    \includegraphics[width=0.49 \textwidth, height = 70mm]{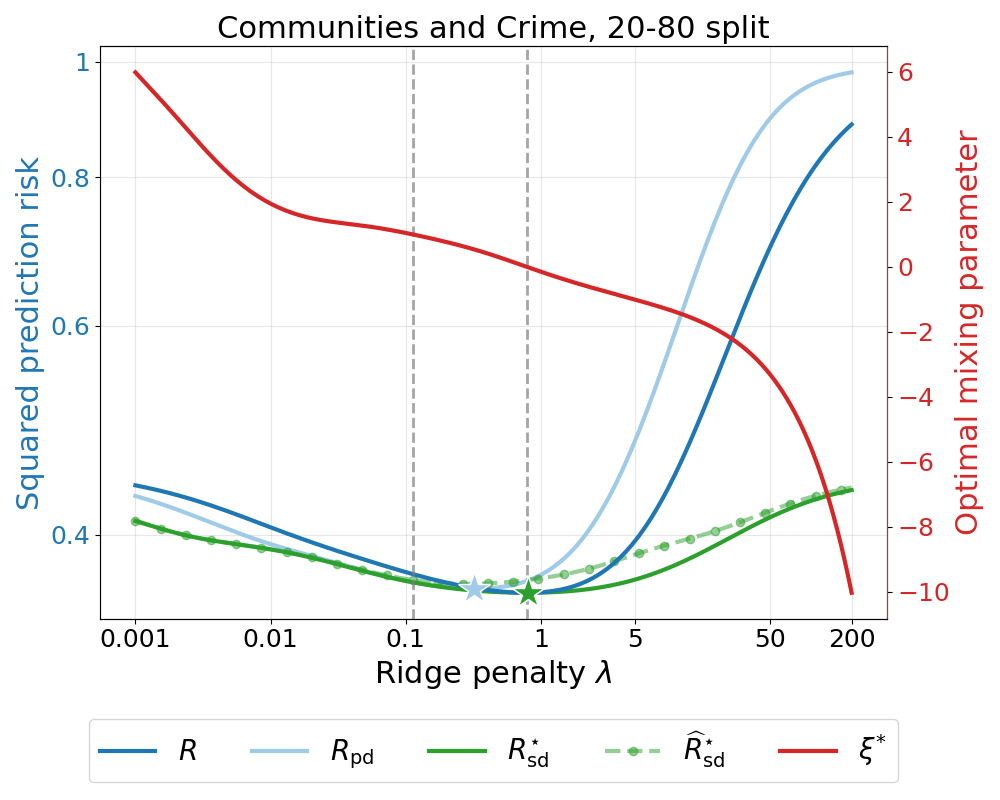}
    \includegraphics[width=0.49 \textwidth, height = 70mm]{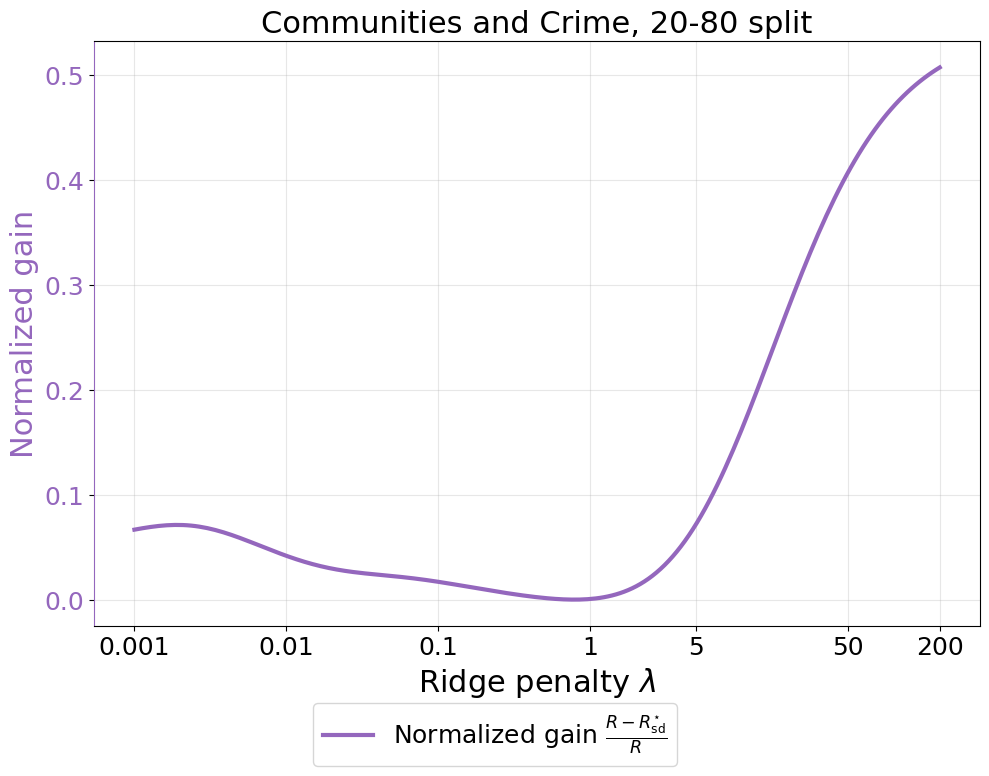}
    \includegraphics[width=0.49 \textwidth, height = 70mm]{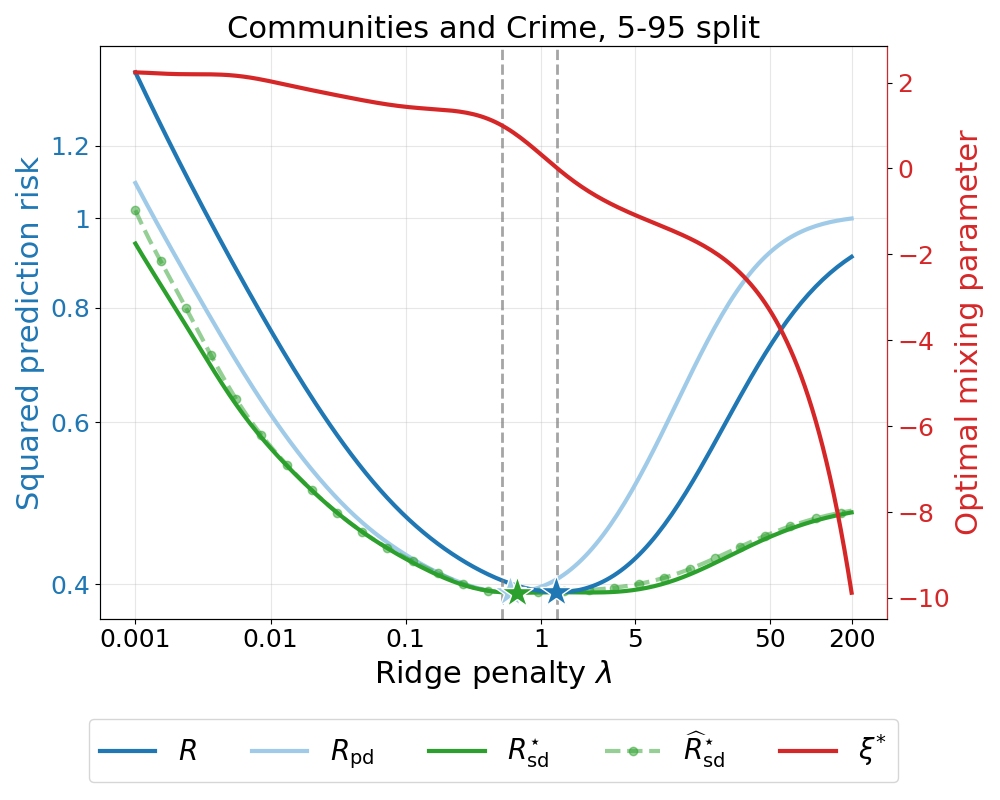}
    \includegraphics[width=0.49 \textwidth, height = 70mm]{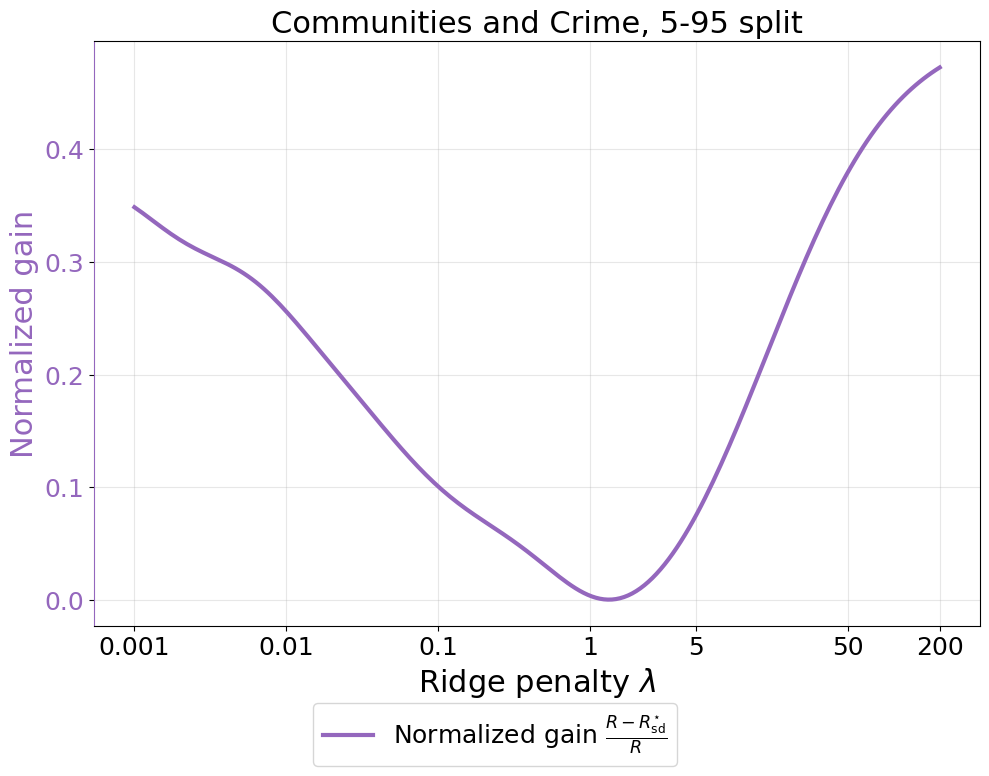}
    \caption{Squared prediction risks and gain curve on Communities and Crime dataset with different split ratios.}
    \label{fig:communites_appendix}
\end{figure*}

Air Quality dataset:

\begin{figure*}[!ht]
    \centering
    \includegraphics[width=0.49 \textwidth, height = 70mm]{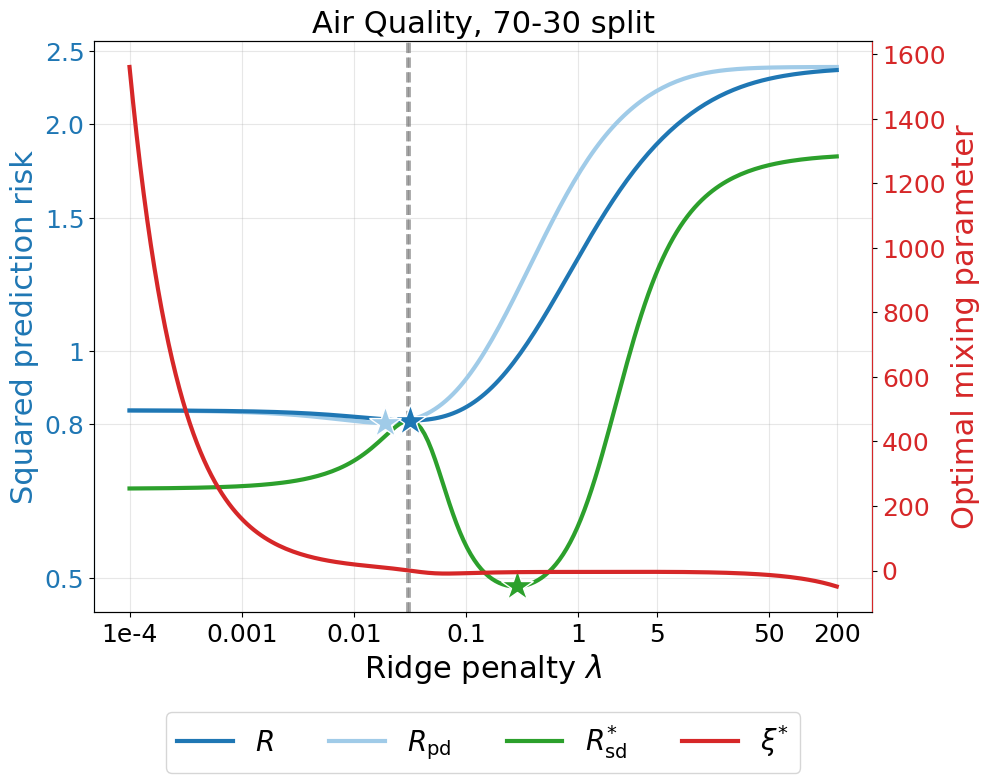}
    \includegraphics[width=0.49 \textwidth, height = 70mm]{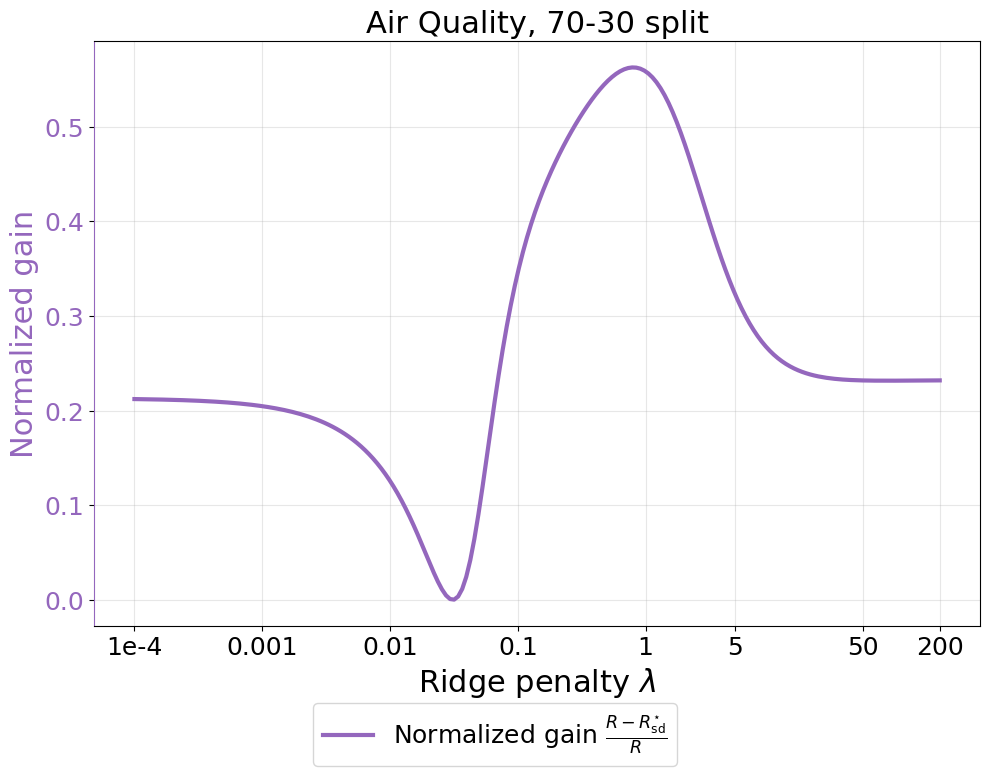}
    \includegraphics[width=0.49 \textwidth, height = 70mm]{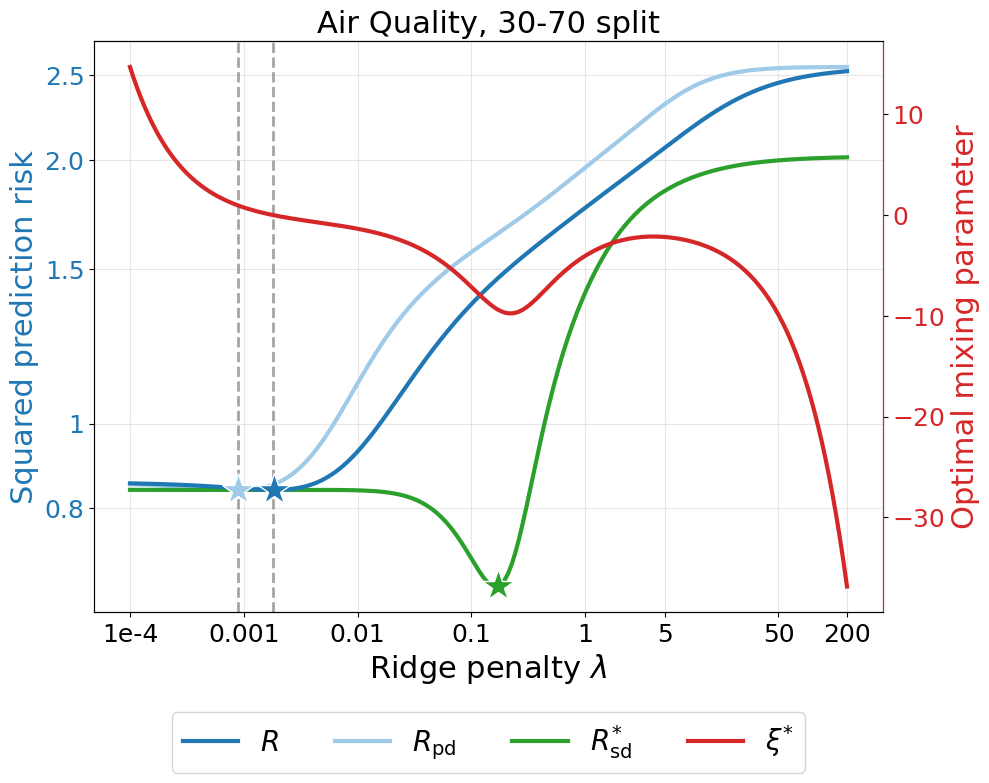}
    \includegraphics[width=0.49 \textwidth, height = 70mm]{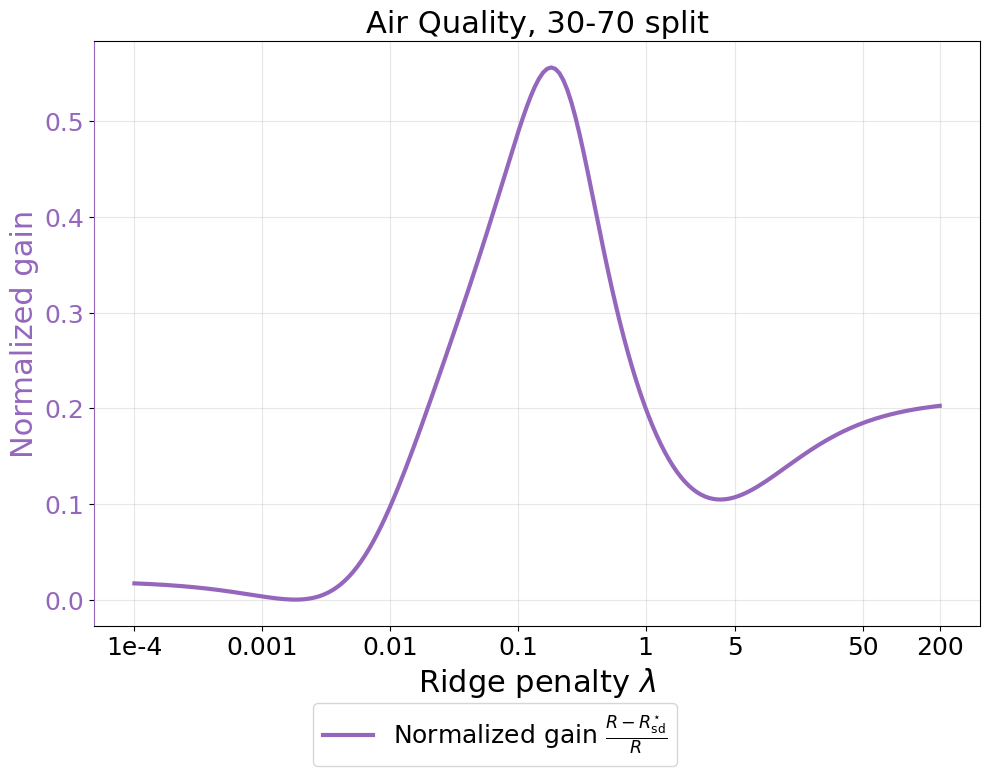}
    \caption{Squared prediction risks and gain curve on Air Quality dataset with different split ratios.}
    \label{fig:air_quality_appendix}
\end{figure*}
\clearpage

\subsection{Additional Illustrations on Proportional Asymptotic Risks}
\label{app:additional-illustrations-propasymp}

\subsubsection{Varying Signal-to-Noise Ratios}

\begin{figure}[!ht]
    \centering
    \includegraphics[width=0.8\textwidth]{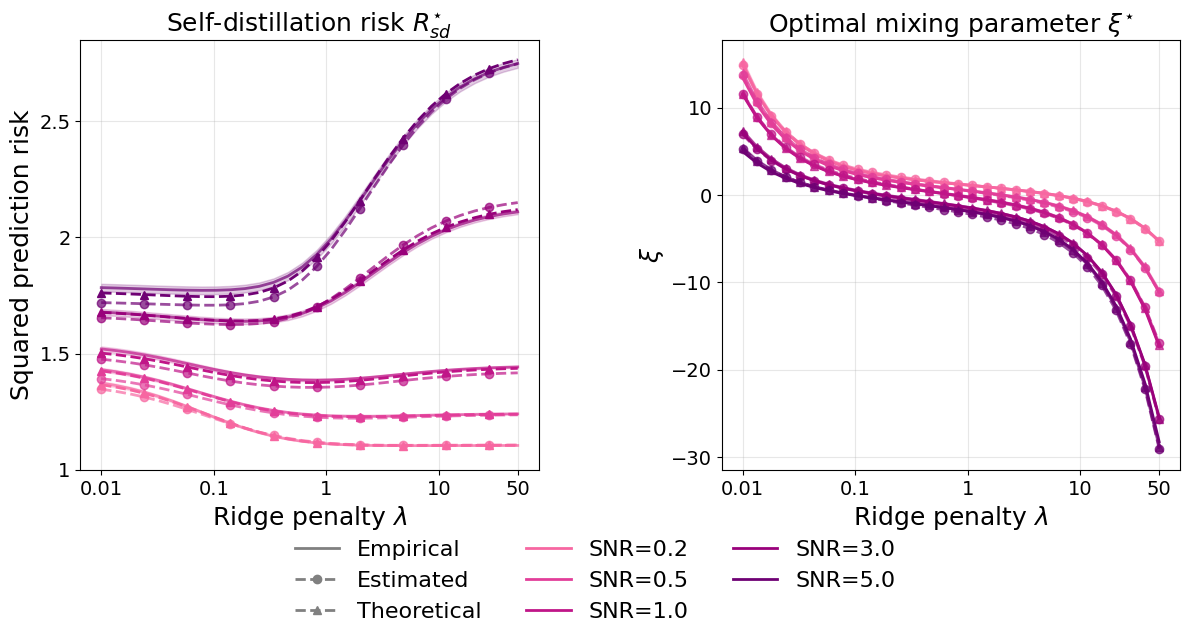}
    \includegraphics[width=0.8\textwidth]{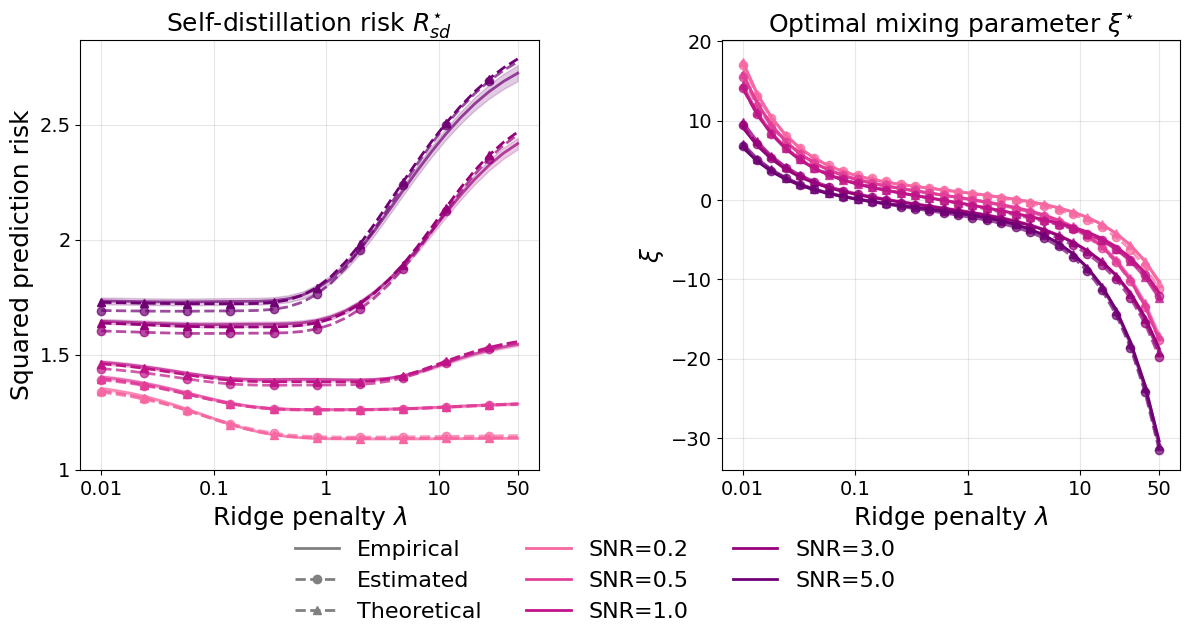}
    \caption{Asymptotic SD risk and optimal mixing parameter over various SNR ratios. Empirical curves are averaged over 30 numerical simulations, and the shaded band represents one standard deviation around the mean. The estimated are obtained using proposed tuning method (\Cref{sec:tuning}, averaged over 30 runs) and theoretical curves are from Theorem \ref{thm:risk-asymptotics}. $n = 400$, $p =200$, $\sigma^2 = 1$ and $r^2 = \sigma^2 \SNR$. \textit{Top row:} Data covariance $\Sigma$ is AR1, deterministic ground-truth signal is aligned with the bottom $10\%$ eigenvalues of $\Sigma$, with alignment factor 0.9. \textit{Bottom row:} Data covariance $\Sigma$ is a spiked covariance matrix, deterministic ground-truth signal is aligned with the top $10\%$ eigenvalues of $\Sigma$, with alignment factor 0.9.}
\end{figure}

\clearpage
\subsubsection{Varying Aspect Ratios}

\begin{figure}[!ht]
    \centering
    \includegraphics[width=0.8\textwidth]{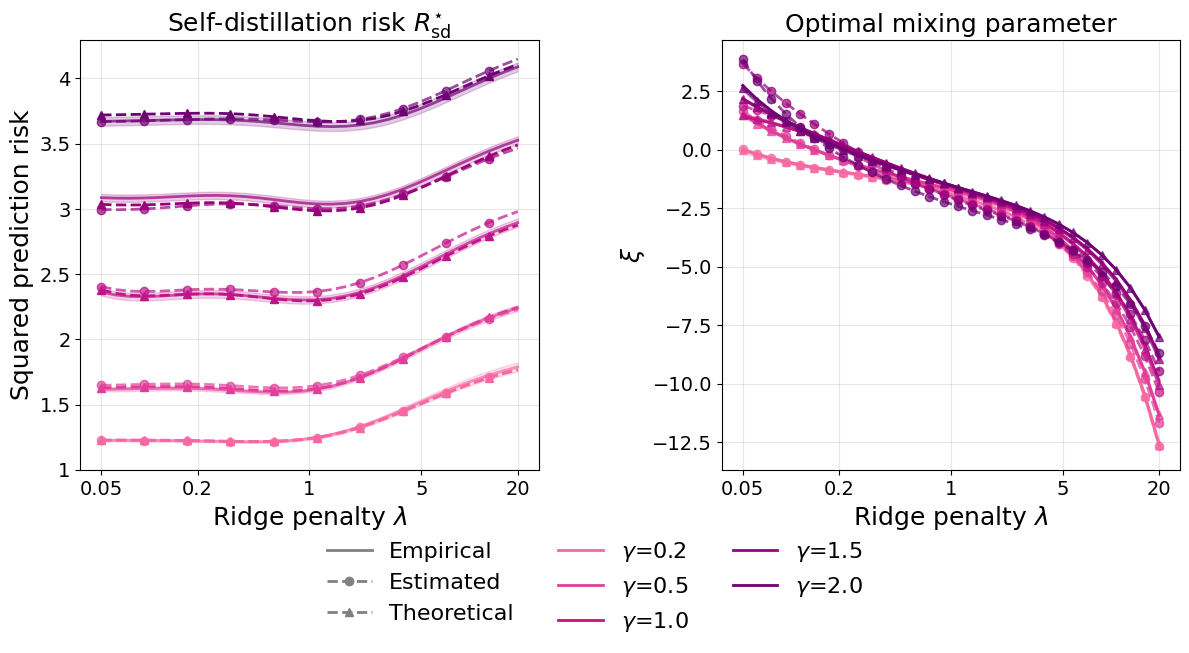}
    \includegraphics[width=0.8\textwidth]{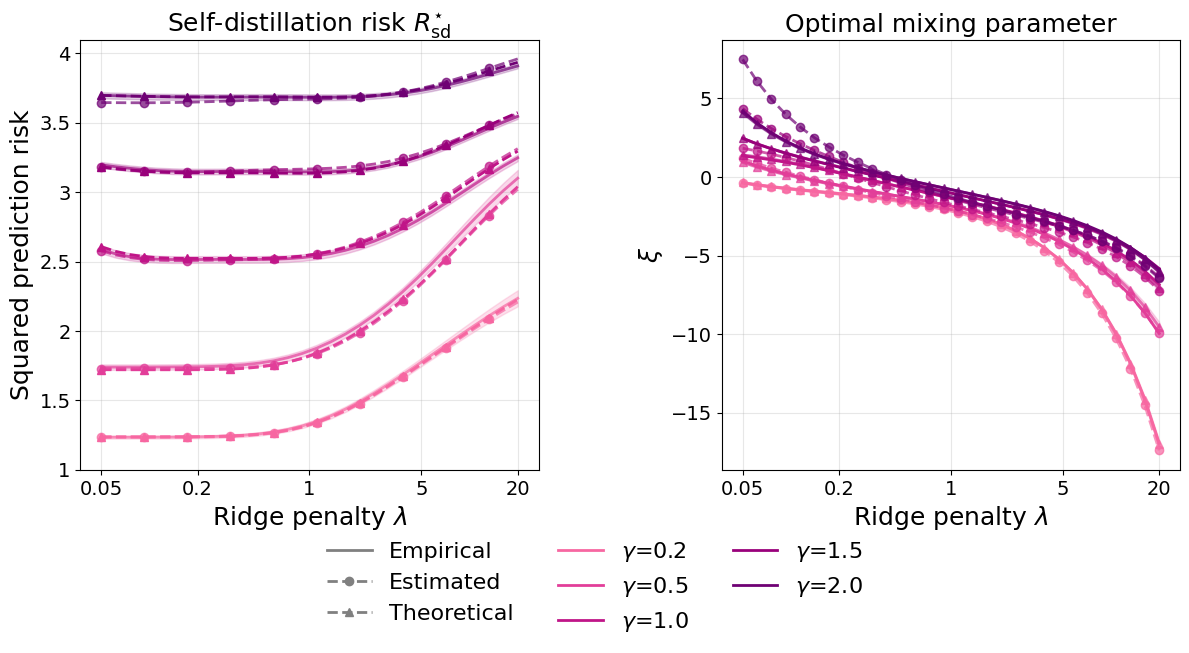}
    \caption{Asymptotic SD risk and optimal mixing parameter over different aspect ratios $\gamma = p/n$. Empirical curves are averaged over 30 numerical simulations, and the shaded band represents one standard deviation around the mean. The estimated are obtained using proposed tuning method (\Cref{sec:tuning}, averaged over 30 runs) and theoretical curves are from Theorem \ref{thm:risk-asymptotics}. $n = 300$, $p = n \gamma$, $\sigma^2 = 1$ and $r^2 = 5$. \textit{Top row:} Data covariance $\Sigma$ is AR1, deterministic ground-truth signal is aligned with the top $10\%$ eigenvalues of $\Sigma$, with alignment factor 0.9. \textit{Bottom row:} Data covariance $\Sigma$ is spiked covariance matrix with $\Sigma =  I + 5 v v^{\top}$ with $v$ is a random isotropic Gaussian vector, deterministic ground-truth signal is aligned with the top $10\%$ eigenvalues of $\Sigma$, with alignment factor 0.9.}
\end{figure}

\clearpage
\subsubsection{Additional Risk and Gain Curves}

\begin{figure*}[!ht]
    \centering
    \includegraphics[width=0.45 \textwidth]{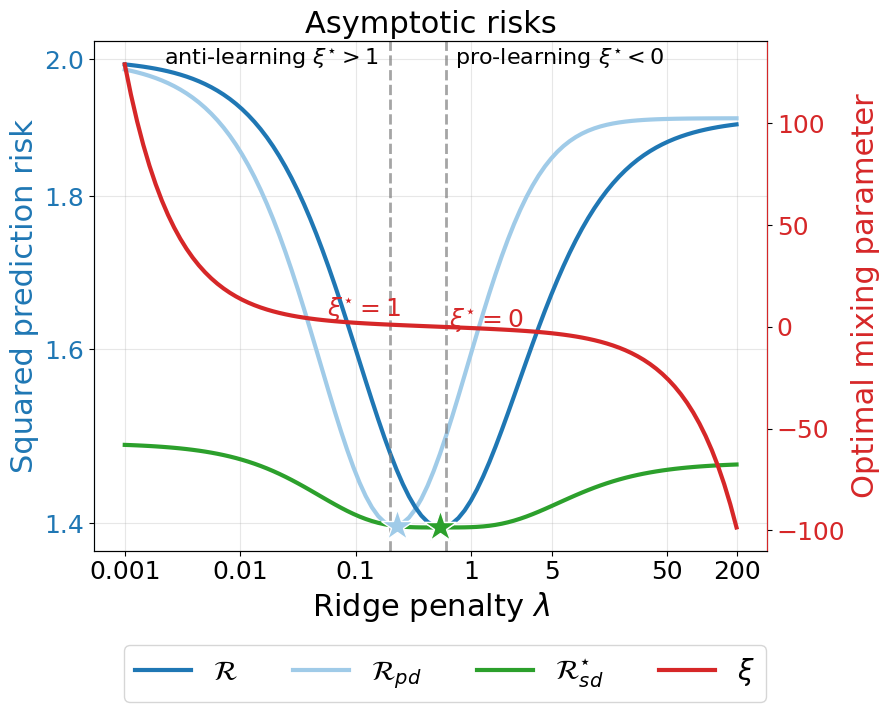}
    \includegraphics[width=0.45 \textwidth]{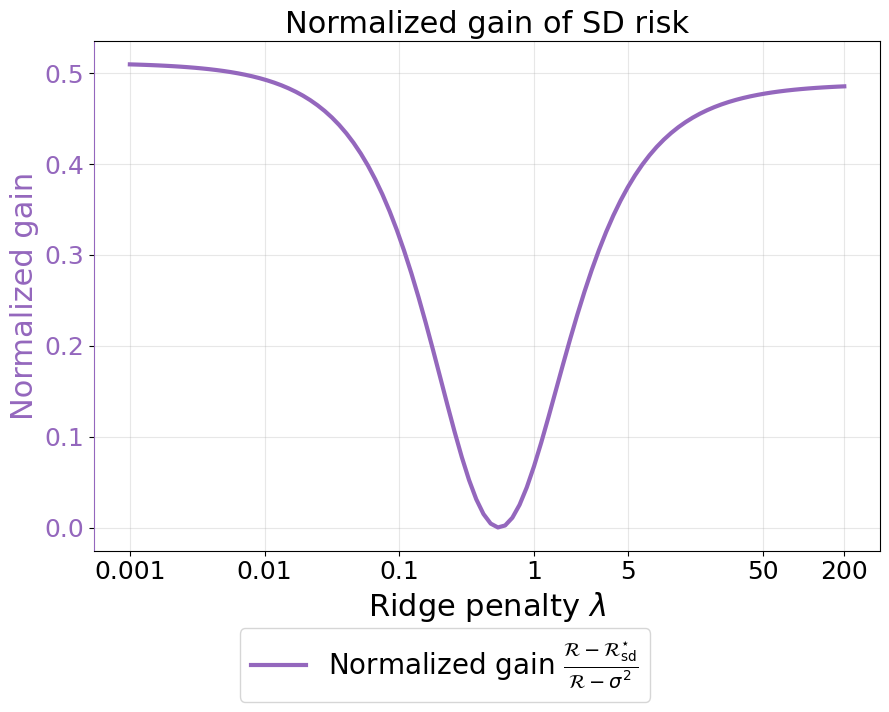}
    \includegraphics[width=0.45 \textwidth]{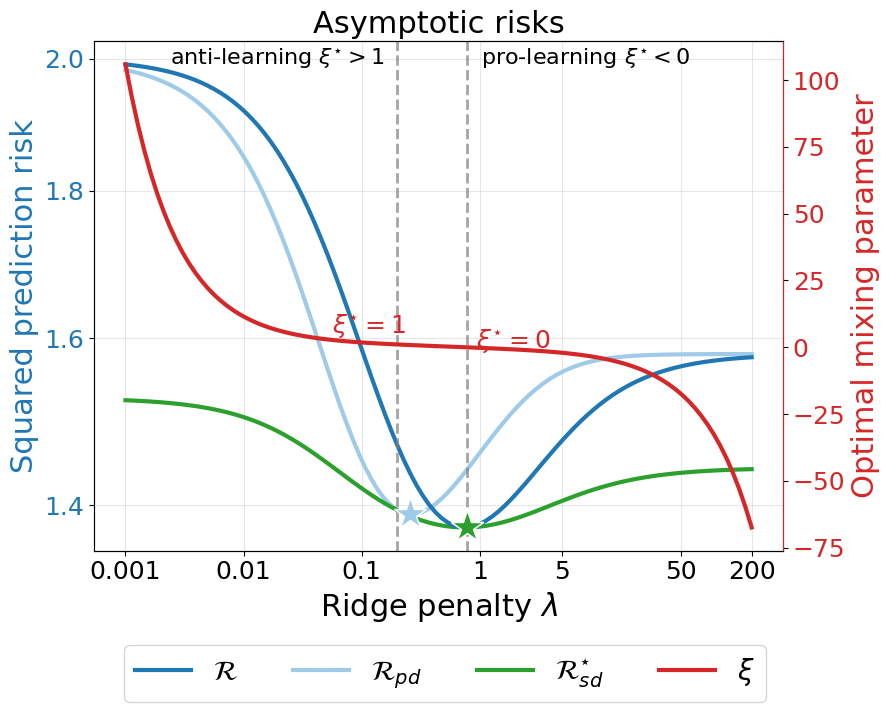}
    \includegraphics[width=0.45 \textwidth]{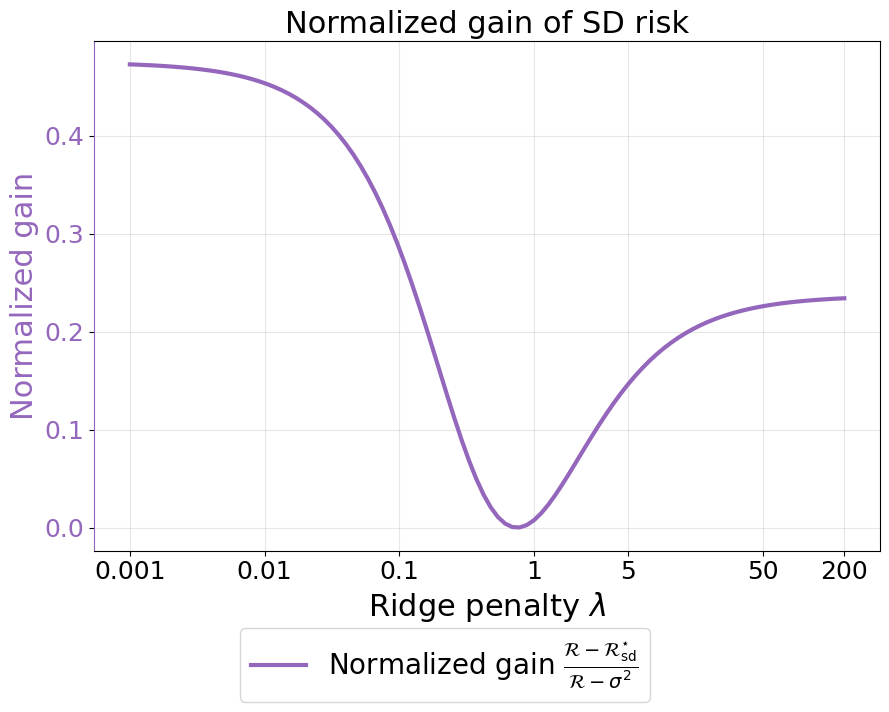}
    \caption{Asymptotic risk curves, asymptotic optimal mixing parameter and gain curves over the penalty $\lambda$. \textit{Top row:} Data covariance $\Sigma$ is isotropic, isotropic signal, $n = 400, p = 200, r^2 = \sigma^2 = 1$. \textit{Bottom row:} Data covariance $\Sigma$ is AR-1(0.25), deterministic ground-truth signal is aligned with the bottom $10\%$ eigenvalues of $\Sigma$, with alignment factor 0.9, $n = 400, p = 200, r^2 = \sigma^2 = 1$.}
    \label{fig:asym_gain_curve_appendix}
\end{figure*}

\clearpage

\subsection{Additional Illustrations on Extreme Regularized Risks}
\label{app:additional-illustrations-extremereg}

\subsubsection{Isotropic Covariance, Isotropic Signal}
\begin{figure}[!ht]
    \centering
    \includegraphics[width=\textwidth]{figures/heatmap_compare_with_optimal_iso_iso.png} \includegraphics[width=\textwidth]{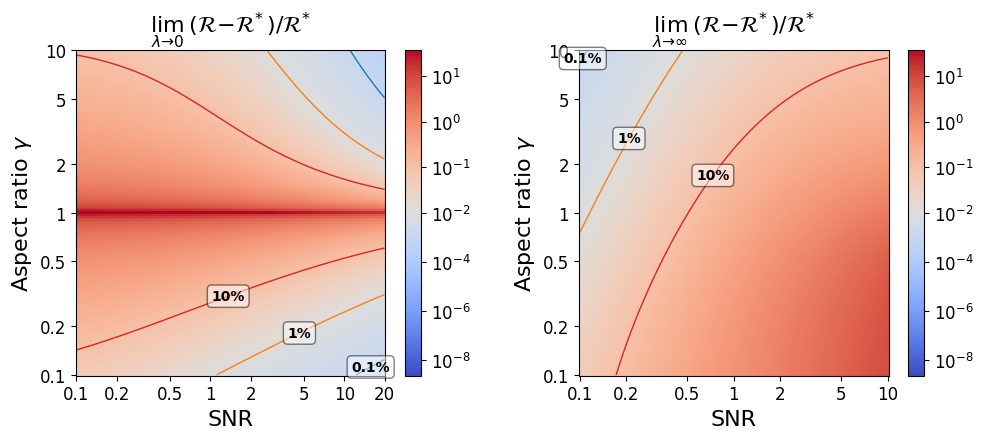}
    \caption{Percentage difference between $\sR(\text{SNR}, \gamma)$  and $\sR_{sd}^{\star}(\text{SNR}, \gamma)$ to the optimal ridge $\sR^{\star}(\text{SNR}, \gamma)$ (also the best predictor in this setting), and with isotropic design $\Sigma = I_p$, random signal follow an isotropic Gaussian distribution. The values plotted are calculated from \Cref{prop:compare_extreme_lambda}.}
    \label{fig: heatmap_r0_isotropic}
\end{figure}

\clearpage
\subsubsection{AR1 Covariance, Isotropic Signal}

\begin{figure}[!ht]
    \centering
    \includegraphics[width=\textwidth]{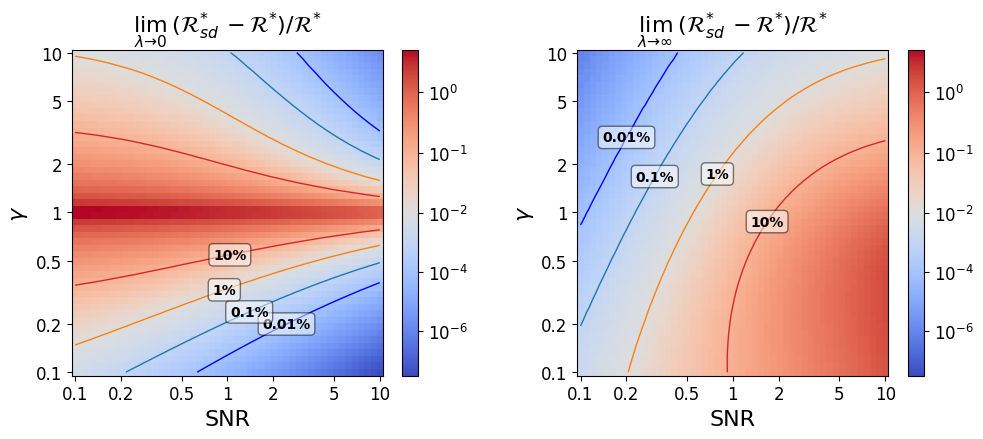} \includegraphics[width=\textwidth]{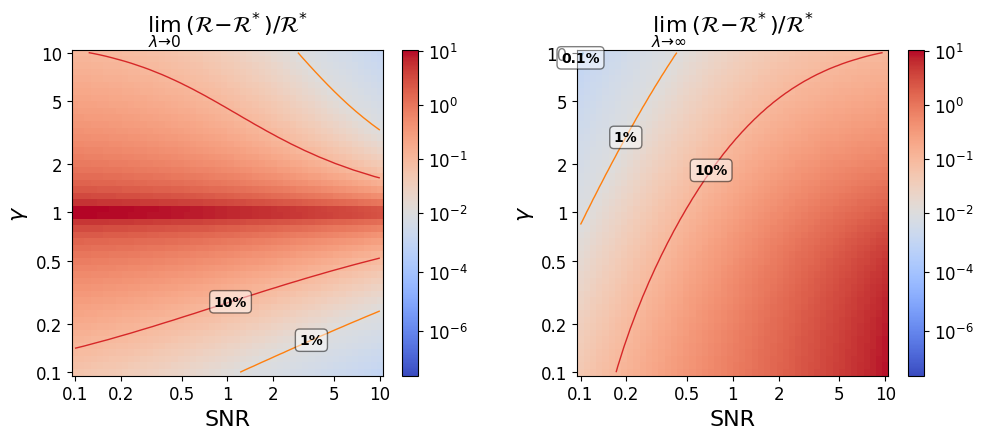}
    \caption{Percentage difference between $\sR(\text{SNR}, \gamma)$  and $\sR_{sd}^{\star}(\text{SNR}, \gamma)$ to the optimal ridge $\sR^{\star}(\text{SNR}, \gamma)$ (also the best predictor in this setting), with AR1 covariance design and random signal follows an isotropic Gaussian distribution. The ratios plotted are calculated using risk formulations at Theorem \ref{thm:risk-asymptotics} at $\lambda = 10^{-3}$ and $\lambda = 10^{6}$.}
    \label{fig: heatmap_r0_ar1}
\end{figure}

\clearpage
\subsubsection{Spiked Covariance, Isotropic Signal}

\begin{figure}[!ht]
    \centering
    \includegraphics[width=\textwidth]{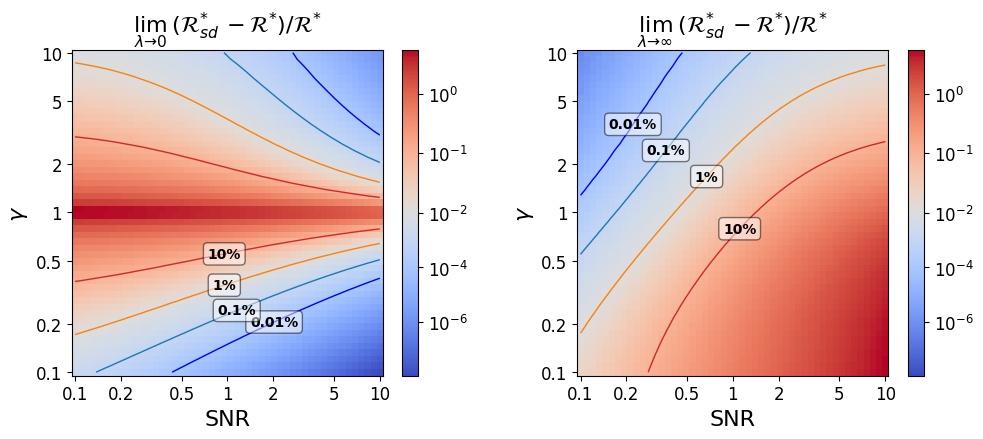} \includegraphics[width=\textwidth]{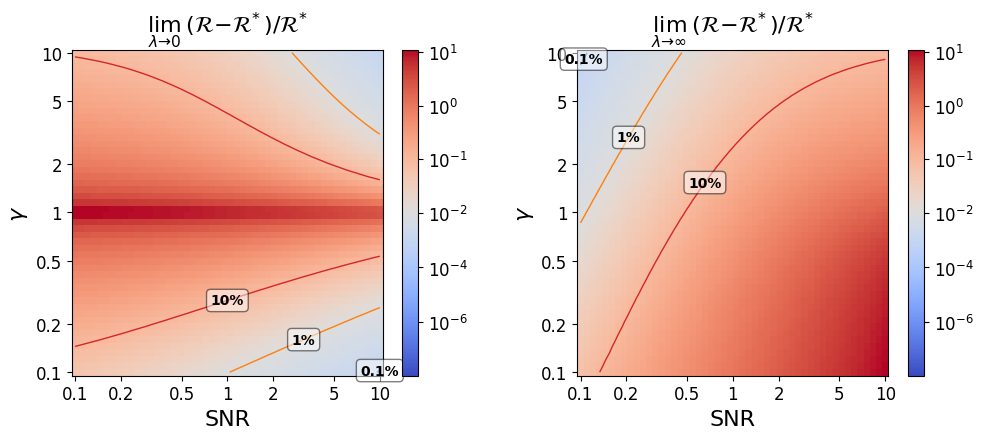}
    \caption{Percentage difference between $\sR(\text{SNR}, \gamma)$  and $\sR_{sd}^{\star}(\text{SNR}, \gamma)$ to the optimal ridge $\sR^{\star}(\text{SNR}, \gamma)$ (also the best predictor in this setting), with spiked covariance design and random signal follow an isotropic Gaussian distribution. The ratios plotted are calculated using risk formulations at Theorem \ref{thm:risk-asymptotics} at $\lambda = 10^{-3}$ and $\lambda = 10^{6}$.}
    \label{fig: heatmap_r0_spiked}
\end{figure}

\restoregeometry
\clearpage

\section{Experiment Details}
\label{sec:additional_details}

\subsection{Real-World Regression Tasks: Datasets Description, Splitting and Pre-processing}
\label{sec:real-world-regression-tasks}

For the experiments on real-world regression task, including UCI Blog Feedback \citep{blogfeedback}, UCI Communities and Crime \citep{communities_and_crime}, UCI Air Quality \citep{air_quality} datasets, we split the data into the training and test set (detailed below). Then, we process the data by removing records with missing values and we centered and standardized all the variables using the training set's mean and standard deviation. All the variables in the test set are also centered and standardized using the information from training set only.

All the risk curves shown at \Cref{fig:exp_real_main,fig:air_quality} are the squared risk measured on the test set. The optimal mixing parameter is computed using the formula from \Cref{eq:R1_star_short} where $R, R_{\pd}$ and $C$ are calculated using the test data. 

\textbf{UCI Blogfeedback.} The UCI Blogfeedback dataset \citep{blogfeedback} originates from blog posts, where the raw HTML-documents of the blog posts were crawled and processed. We use only the training set of this dataset, which contains 52,397 samples and perform a random 5-95 split ($5 \%$ for training and $95 \%$ for test set). The dataset consists of $280$ features that capture many aspects of blog content and metadata, such as post length, number of links, number of comments in the first 24 hours after the publication of the blog post. The target variable is the number of comments in the next 24 hours (relative to base time).
Thus, the training set in our experiment contains $p = 280$ covariates and $n = 2{,}619$ samples.
The additional results of different split ratios are shown at Figure \ref{fig:blog_appendix}.

\textbf{UCI Communities and Crime.}
The UCI Communities and Crime dataset \citep{communities_and_crime}
are authentic data that combines socio-economic data from the 1990 US Census, law enforcement data from the 1990 US LEMAS survey, and crime data from the 1995 FBI UCR. The dataset consists of 1994 samples with 127 covariates. The goal is to predict the value of ``ViolentCrimesPerPop'', which is the rate of violent crimes per 100,000 population. For this dataset, we removed 5 nonpredictive identifiers and 23 covariates from the LEMAS survey that contain 1,675 missing values out of 1,994 instances. The remaining $100$ covariates have no missing values and used in our experiments. We perform a random 20-80 split ($20 \%$ for training and $80 \%$ for test set). Thus, the training set in our experiment contains $p = 99$ covariates and $n = 398$ samples. The additional results of different split ratios are shown at Figure \ref{fig:communites_appendix}.

\textbf{UCI Air Quality.} The UCI Air Quality dataset \citep{air_quality} contains records of hourly averaged responses from an array of 5 metal oxide chemical sensors. Data were recorded at an Italy city from March 2004 to February 2005. Similar as \cite{pareek2024understanding}, we use $p = 8$ covariates for this task, which include 5 metal oxide chemical readings \texttt{PT08.S1(CO), PT08.S2(NMHC), PT08.S3(NOx), PT08.S4(NO2), PT08.S5(O3)} and 3 other covariates including Temperature \texttt{T}, Relative Humidity \texttt{RH}, and Absolute Humidity \texttt{AH}. The goal is to predict the Nitrogen Dioxide \texttt{NO2(GT)}. After removing missing records for these variables, we are left with $n = 7{,}393$ samples. We perform a \textit{sequential} 70-30 split since the data is heavily time-dependent. Thus, the training set in our experiment contains $p = 8$ covariates and $n = 5{,}175$ samples. The additional results of different split ratios are shown at Figure \ref{fig:air_quality_appendix}.

For the multi-round distillation experiment shown in \Cref{fig:multiround_monotonicity}, they are performed on a sequential 70-30 split of Air Quality dataset and a random 20-80 split of Communities and Crime dataset.

For the kernel ridge regression experiment shown in \Cref{fig:kernel-air-quality}, we use a Gaussian kernel with bandwidth estimated as the median of the $\ell_2$ distances between covariates in the training set.

\subsection{CIFAR10 and CIFAR100 Experiments}
\label{sec:resnet-cifar-details}

For the experiment on CIFAR10 and CIFAR100 datasets, we extract the pretrained ResNet-18 and ResNet-34 features \citep{he2016deep}, that trained on ImageNet dataset (available in Pytorch). For CIFAR10, we randomly sample 2,000 samples for training and 2,000 samples for the test set. For CIFAR100, we randomly sample 20,000 samples for training and 10,000 samples for the test set. Let $K$ is the number of classes. We then perform ridge regression on the last-layer features of the pretrained models. The predictor is now defined as the vector-valued function $f: \RR^{512} \to \RR^{K}$. To aggregate the risk, for an one-hot label vector $ y \in \RR^{K}$ where , we simply sum the mean squared error over the $K$ input dimensions,
\begin{align}
    R(f) = \EE_{(x,  y)} \bigg[
    \sum_{k=1}^{K} (y_k - f(x)_k)^2
    \bigg],
\end{align}
where $f_k: \RR^{512} \to \RR$ is a ridge predictor that predict the probability that the input belong to class $k$. The optimal mixing parameter is calculated from \Cref{eq:R1_star_short}.

\subsection{Synthetic Asymptotic Experiments}
\label{sec:synthetic_details}

We give more details here about the data covariance $\Sigma$ and signal $\beta$ (defined in \Cref{def:dist}) used in the proportional asymptotic experiments in \Cref{sec:asymptotics,app:additional-illustrations-propasymp,app:additional-illustrations-extremereg}.

\begin{itemize}[leftmargin=5mm,nosep,itemsep=2pt]
    \item Isotropic covariance: $\Sigma = I_p$
    
    \item AR1 covariance: $\rho$-autoregressive covariance with $\rho = 0.25$, $\Sigma_{ij} = \rho^{| i -j |}$ for all $i,j$.

    \item Spiked covariance: $\Sigma = I_p + 5 v v^{\top}$ where $v \in \mathbb{R}^{p}$ is a random isotropic Gaussian vector.

    \item Isotropic signal: $\beta \sim \mathcal{N}(\bm{0}, (r^2/p)  I_p)$.

    \item Top-aligned signal with alignment ratio $m \%$ and alignment factor of $a$: let $k = \frac{m}{100} \cdot p$, then $\beta \sim \mathcal{N}(\bm{0}, \Sigma_{\beta})$ where $\Sigma_{\beta} = p V \operatorname{diag}(\frac{a}{k}, \frac{a}{k}, \ldots, \frac{a}{k}, \frac{1 - a}{p - k}, \ldots, \frac{1 - a}{p - k}) V^{\top}$ where $V = [v_1, \ldots, v_p]$ contain the eigenvectors of $\Sigma$ with $v_k$ corresponds to $k$-th largest eigenvalue.

    \item Bottom-aligned signal with alignment ratio $m \%$ and alignment factor of $a$: let $k = \frac{m}{100} \cdot p$, then $\beta \sim \mathcal{N}(\bm{0}, \Sigma_{\beta})$ where $\Sigma_{\beta} = p V \operatorname{diag}(\frac{1 - a}{p - k}, \ldots, \frac{1 - a}{p - k}, \frac{a}{k}, \ldots, \frac{a}{k}) V^{\top}$ where $V = [v_1, \ldots, v_p]$ contain the eigenvectors of $\Sigma$ with $v_k$ corresponds to $k$-th largest eigenvalue.
\end{itemize}

\end{document}